\documentclass[11pt,a4paper]{article}
\usepackage[utf8]{inputenc}
\usepackage[british,UKenglish,USenglish,english,american]{babel}
\usepackage[all]{xy}
\usepackage{enumitem}
\usepackage[dvipsnames]{xcolor}
\usepackage{graphicx}
\usepackage{stmaryrd}
\usepackage{amsfonts,amstext,amssymb,amsmath,amsthm,pspicture,bigints}
\usepackage{fullpage}
\usepackage{moreverb}
\usepackage{pifont}
\usepackage{pst-all}
\usepackage[left=2cm,right=2cm,top=2cm,bottom=2cm]{geometry}
\usepackage{esint}
\usepackage{fourier-orns}
\usepackage{mathrsfs}
\usepackage{array}
\newcommand{\T}{\mathbb{T}}
\usepackage{hyperref}
\usepackage{authblk}
\newcommand{\ii }{{\rm i} }
\newcolumntype{C}[1]{>{\centering\arraybackslash}b{#1}}
\newcolumntype{R}[1]{>{\raggedleft\arraybackslash}b{#1}}
\newcolumntype{L}[1]{>{\raggedright\arraybackslash}b{#1}}
\newcolumntype{M}[1]{>{\centering}m{#1}}
\newtheorem{theo}{Theorem}[section]
\newtheorem{defin}{Definition}[section]
\newtheorem{lem}{Lemma}[section]
\newtheorem{prop}{Proposition}[section]

\newtheorem{remark}{Remark}[section]

\newtheorem{cor}{Corollary}[section]
\usepackage{bbm}
\numberwithin{equation}{section}

\date{}
\title{Boundary effects on the emergence of\\
quasi-periodic
solutions
for Euler equations}
\author{Zineb Hassainia\thanks{NYUAD Research Institute, New York University Abu Dhabi, PO Box 129188, Abu Dhabi, United Arab Emirates.\\
E-mail address: zh14@nyu.edu} \hspace{0.5cm} and \hspace{0.5cm}Emeric Roulley\thanks{Univ Rennes, CNRS, IRMAR – UMR 6625, F-35000 Rennes, France\\
E-mail address : emeric.roulley@univ-rennes1.fr}}
\begin{document}
\maketitle
\begin{abstract}
In this paper, we highlight the importance of the boundary effects on the construction  of quasi-periodic vortex patches solutions close to Rankine vortices and whose existence  is not known in the whole space due to the resonances of the linear frequencies. Availing of the lack of invariance by radial dilation of  Euler equations  in the unit disc and using a Nash-Moser implicit function iterative scheme we show the existence  of such structures
when the radius  of the Rankine vortex  belongs to a suitable massive Cantor-like set with almost full Lebesgue measure. 
\end{abstract}
\tableofcontents
\section{Introduction}
We consider the Euler system set in a domain $D\subset\mathbb{R}^{2}$ written in the velocity-vorticity formulation
\begin{equation}\label{Euler velocity-vorticity}
	\left\lbrace\begin{array}{ll}
		\partial_{t}\boldsymbol{\omega}+\mathbf{v}\cdot\nabla\boldsymbol{\omega}=0 & \textnormal{in }\mathbb{R}_{+}\times D\\
		\mathbf{v}=\nabla^{\perp}\boldsymbol{\Psi} & \\
		\boldsymbol{\omega}(0,\cdot)=\boldsymbol{\omega}_0 & \textnormal{in }D
	\end{array}\right.
\end{equation}
where $\nabla^{\perp}=(-\partial_{2},\partial_{1}).$ Here,  we are particularly interested in the cases where $D$ is the full plane, $D=\mathbb{R}^{2}$,  or the unit disc $D=\mathbb{D}:=\Big\{(x_1,x_2)\in\mathbb{R}^{2}\quad\textnormal{s.t.}\quad x_1^2+x_2^2\leqslant 1\Big\}.$ 
First consider the whole space dynamics. In this case,  the stream function is given by 
$$\boldsymbol{\Psi}(t,w)=\frac{1}{2\pi}\int_{\mathbb{R}^2}\log\big(|w-\xi|\big)\boldsymbol{\omega}(t,\xi)dA(\xi),$$
where $dA$ is the planar Lebesgue measure.
The global existence and uniqueness for  weak solutions bounded and integrable follows from Yudovich's theory \cite{Y63}. In particular, if the initial datum is a vortex patch, that is, the characteristic function of a smooth bounded planar domain $D_0$, then the solution keeps a patch form $\mathbf{1}_{D_t}$ for any time $t>0$, where $D_{t}$ is the transported domain $D_{0}$ by the flow map associated to $\mathbf{v}$, namely
$$D_{t}=\boldsymbol{\Phi}_{t}(D_0),\quad\partial_{t}\boldsymbol{\Phi}_{t}(x)=\mathbf{v}(t,\boldsymbol{\Phi}_{t}(x)),\quad\boldsymbol{\Phi}_{0}=\textnormal{Id}_{\mathbb{R}^2}.$$ 
The persistence of the regularity of the boundary was proved in \cite{BC93,C95}.  Notice that any radial profile generates a stationary solution.
It is a classical fact to look for periodic or quasi-periodic solutions close to these equilibrium state solutions for Hamiltonian systems like \eqref{Euler velocity-vorticity}. 
 A particular class of periodic solutions is given by the rigid body rotating vortex patches around the origin described by
$$D_t=e^{\ii\Omega t}D_0,$$
where $\Omega$ is the time independent angular velocity. Such solutions are called V-states according to the terminology introduced by Deem and Zabusky in \cite{DZ78}. The first explicit example, discovered by Kirchhoff in \cite{K74}, is the ellipse which rotates about its
center of mass with the constant angular velocity
$$\Omega=\frac{ab}{(a+b)^2},$$
where $a$ and $b$ are the semi-axes of the ellipse. Further families  of implicit solutions with higher symmetries  were established by Burbea in \cite{B82} using  local bifurcation tools and complex analysis. More precisely, he  proved the existence of branches of $\mathbf{m}$-fold rotating  solutions bifurcating from the discs at angular velocities
\begin{equation}\label{Burbea frequencies}
	\boldsymbol{\Omega}_{\mathbf{m}}=\frac{\mathbf{m}-1}{2\mathbf{m}},\quad\mathbf{m}\geqslant 1.
\end{equation}
Notice that the mode $\mathbf{m}=1$ corresponds to a translation of the trivial solution  and the second branch, emerging at $\boldsymbol{\Omega}_2=\tfrac{1}{4},$ describes the Kirchhoff ellipses. Moreover,  all the bifurcation angular velocities $\boldsymbol{\Omega}_{\mathbf{m}}$  are in the range  $(0,\tfrac{1}{2}).$ Outside this interval, the only uniformly rotating solutions are the radial ones, as proved in the series of papers \cite{F00,GPSY20,H15}. The boundary regularity was first discussed in \cite{CCG16,HMW20,HMV13} and the global bifurcation diagram was studied in \cite{HMW20}. Note also that countable branches of rotating patches bifurcating from the ellipses at implicit angular velocities were found in \cite{HM16}, however, the shapes have in fact less symmetry and being at most two-folds. 
  It is worthy to point out that Burbea's approach has been extended in the past few years to different topological structures for the V-states and for various nonlinear transport equations.
  For instance, we mention the existence results for the multiply-connected patches  \cite{HHMV16,HM16-2,HMV15} or the  multipole vortex patch patches obtained by desingularization of the point vortex system \cite{G20,G21,HH21,HW21,HM17}. Let us now turn to the case where Euler equations \eqref{Euler velocity-vorticity} are set in the unit disc. In this setting, the stream function $\mathbf{\Psi}$ solves the following Dirichlet problem in the unit disc $\mathbb{D}$ 
$$\left\lbrace\begin{array}{l}
	\Delta\mathbf{\Psi}=\boldsymbol{\omega}\\
	\mathbf{\Psi}|_{\partial\mathbb{D}}=0.
\end{array}\right.$$
Thus, by using the Green function of the unit disc, we get the expression
\begin{equation}\label{expression of the stream function}
\forall w\in \mathbb{C},\quad 	\mathbf{\Psi}(t,w)=\frac{1}{4\pi}\int_{\mathbb{D}}\log\left(\left|\frac{w-\xi}{1-w\overline{\xi}}\right|^{2}\right)\boldsymbol{\omega}(t,\xi)dA(\xi).
\end{equation}
The theory of weak solutions and vortex patches is still valid in this context and the persistence of the boundary   regularity of vortex  patches remains true, as proved in \cite{D99}. The existence of  V-states close to the discs $b\mathbb{D}$ ($b\in(0,1)$), also called Rankine vortices,  were obtained in \cite{HHHM15}. These curves of solutions have $\mathbf{m}$-fold symmetry, perform a uniform rotation and emerge at the angular velocities
\begin{equation}\label{def omegajb intro}
	\boldsymbol{\Omega}_{\mathbf{m}}(b)=\frac{\mathbf{m}-1+b^{2\mathbf{m}}}{2\mathbf{m}},\quad\mathbf{m}\geqslant 1.
\end{equation}
It is of  paramount importance to highlight different boundary effects observable at this periodic level. First, Burbea's frequencies \eqref{Burbea frequencies} are shifted to the right, implying in particular that the $\mathbf{1}$-fold patches, which are not centered at the origin, are no longer associated to the trivial solution. Second, the numerical observations in \cite{HHHM15} show that the bifurcation curves have oscillations.\\

The purpose of the current paper is to prove the existence of time quasi-periodic patch structures close to Rankine vortices. Remind that a function $f:\mathbb{R}\rightarrow\mathbb{R}$ is said to be \textit{quasi-periodic} if there exist a frequency vector $\omega\in\mathbb{R}^{d}$ ($d\in\mathbb{N}^*$) which is non-degenerate, that is
\begin{equation}\label{nonresonnace omega}
	\forall\, l\in\mathbb{Z}^{d}\setminus\{0\},\quad \omega\cdot l\neq 0
\end{equation}
and a function $F:\mathbb{T}^{d}\rightarrow\mathbb{R}$, where $\mathbb{T}^d$ denotes the flat torus of dimension $d$, such that
$$\forall\, t\in\mathbb{R},\quad f(t)=F(\omega t).$$
Remark that the case $d=1$ corresponds to the definition of periodic functions with frequency $\omega\in\mathbb{R}^*.$ The variable of $F$ living in $\mathbb{T}^d$ is denoted $\varphi.$ The existence of quasi-periodic vortex patch solutions has been initiated very recently  in \cite{BHM21,HHM21,HR21} using KAM techniques and Nash-Moser theory. We emphasize that in the papers   \cite{HHM21, HR21},  respectively devoted to  the generalized surface quasi-geostrophic equations and  quasi-geostrophic shallow-water equations, the authors proved the existence of quasi-periodic patches  close to Rankine vortices for suitable selected values of the exterior  parameters offered by the equations.  The situation for  Euler equations in the whole plane is quite delicate and  the search of quasi-periodic solutions near the discs is not clear due to the resonances of the linear frequencies \eqref{Burbea frequencies} and the absence of an exterior parameter. However, in \cite{BHM21}, the authors show the existence of quasi-periodic solutions for Euler equations close to the ellipses and the parameter used there is the aspect ratio of the ellipse. In the same spirit,  we aim here to take advantage of the lack  of invariance by radial dilation  to create a natural geometrical parameter $b$ describing a family of stationary solutions. \\

Before stating our main theorem, we shall briefly recall some  results related to the use of KAM theory in PDE.
 Notice that KAM theory is named after Kolmogorov \cite{K54}, Arnold \cite{A63} and Moser \cite{M62} works where they proved, for both  the analytic and the finitely many differentiable cases, the persistence of invariant tori supporting quasi-periodic motions under a small perturbation of integrable finite dimensional Hamiltonian systems. We mention that in the differentiable case, Moser used a modified version of a regularizing Newton method developed by Nash for the isometric embedding problem \cite{N54}; commonly known as Nash-Moser scheme. KAM theory was extended and refined  for  several Hamiltonian PDE with small divisors problems.
 For instance, it has been implemented  for the 1-d semilinear wave and Schr\"odinger equations in several papers \cite{B94,CY00,CW93,K87,P96,P96-1,W90}. Many results were also obtained for   semi-linear perturbations of PDE  \cite{BBP13,BBP14,B05,EK10,GYZ11,KP03,K98,K00,LY11}. However the case of quasi-linear or fully nonlinear perturbations  were explored in \cite{BBM14,BBM16,BBM16-1,BB15,BKM21,FP14}. 
Many interesting  results have also been obtained in the past few years on the periodic and quasi-periodic settings for  the water-waves equations as in   \cite{AB11,BBMH18, BFM21,BFM21-1,BM18,IPT05,PT01}.
Very recently,   a quasi-periodic forcing term was used in  \cite{BM21} to generate quasi-periodic solutions for 3D Euler equations. In the current context,  we have no any forcing to use, and  we rather use the internal structure to find  quasi-periodic solutions.\\

We shall now present the main result of this work and discuss the key ideas of its proof. We first consider a polar parametrization of a patch boundary close to the stationary solution $b\mathbb{D}$, namely
$$z(t,\theta)=R(b,t,\theta)e^{\ii\theta},\quad R(b,t,\theta)=\sqrt{b^{2}+2r(t,\theta)}.$$
The quantity of interest is the radial deformation $r$ assumed to be of small size. We emphasize that our ansatz is slightly different from the one in the papers \cite{BHM21,HHM21,HR21} where the parametrization is written in a rotating frame with an angular velocity $\Omega$ to remedy to the degeneracy of the first frequency. This is not the case in our context due to the non-degeneracy of the first frequency according to \eqref{def omegajb intro}. As explained in Lemma \ref{lem eq ED r} and Proposition \ref{prop Ham eq r}, the radial deformation solves a non-linear and non-local transport PDE which admits a Hamiltonian formulation in the form
\begin{equation}\label{Ham eq ED r intro}
	\partial_{t}r=\tfrac{1}{2}\partial_{\theta}\nabla E(r),
\end{equation}
where $E$ is the kinetic energy related to the stream function given by \eqref{expression of the stream function}. In view of Lemma \ref{lemma general form of the linearized operator}, the linearized operator at a state $r$ close to the Rankine patch $b\mathbb{D}$ takes the form
\begin{equation}\label{def Lr intro}
	\mathcal{L}_{r}=\partial_{t}+\partial_{\theta}\Big(V_{r}\cdot+\mathbf{L}_{r}-\mathbf{S}_{r}\Big),
\end{equation}
where
\begin{align*}
	V_{r}(b,t,\theta)&=-\tfrac{1}{2}\int_{\mathbb{T}}\tfrac{R^{2}(b,t,\eta)}{R^{2}(b,t,\theta)}d\eta
	-\tfrac{1}{R(b,t,\theta)}\int_{\mathbb{T}}\log\big(A_{r}(b,t,\theta,\eta)\big)\partial_{\eta}\big(R(b,t,\eta)\sin(\eta-\theta)\big)d\eta\\
	&\quad-\tfrac{1}{R^{3}(b,t,\theta)}\int_{\mathbb{T}}\log\big(B_{r}(b,t,\theta,\eta)\big)\partial_{\eta}\big(R(b,t,\eta)\sin(\eta-\theta)\big)d\eta,
\end{align*} 
$\mathbf{L}_{r}$ is a non-local operator in the form
\begin{align*}
	\mathbf{L}_{r}(\rho)(b,t,\theta)=\int_{\mathbb{T}}\rho(t,\eta)\log\left(A_{r}(b,t,\theta,\eta)\right)d\eta,\quad A_r(b,t,\theta,\eta)=\big|R(b,t,\theta)e^{\ii\theta}-R(b,t,\eta)e^{\ii\eta}\big|
\end{align*}
and $\mathbf{S}_{r}$ is a smoothing non-local operator in the form
\begin{align*}
	\mathbf{S}_{r}(\rho)(b,t,\theta)&=\int_{\mathbb{T}}\rho(t,\eta)\log\left(B_{r}(b,t,\theta,\eta)\right)d\eta,\quad B_r(b,t,\theta,\eta)=\big|1-R(b,t,\theta)R(b,t,\eta)e^{\ii(\eta-\theta)}\big|.
\end{align*}
The operator $\mathbf{L}_{r}$ is of order zero and reflects the planar Euler action. Moreover, we observe two boundary effects of $\mathbb{D}.$ The first one is quasi-linear in the transport part through the last term of $V_r$, but  with a smoothing action. The second one is given by the operator $\mathbf{S}_r$ which is smoothing  since it involves a smooth kernel. At the equilibrium state $r=0$, the linearized operator is a Fourier multiplier given by
$$\mathcal{L}_{0}=\partial_{t}+\tfrac{1}{2}\partial_{\theta}+\partial_{\theta}\mathcal{K}_{1,b}\ast\cdot-\partial_{\theta}\mathcal{K}_{2,b}\ast\cdot,$$
where
\begin{align*}
	\mathcal{K}_{1,b}(\theta)=\tfrac{1}{2}\log\big(\sin^2\left(\tfrac{\theta}{2}\right)\big)\quad\textnormal{and}\quad\mathcal{K}_{2,b}(\theta)=\log\big(|1-b^2e^{\ii\theta}|\big).
\end{align*}
Notice that the convolution with the kernel $\partial_{\theta}\mathcal{K}_{1,b}$ is exactly the Hilbert transform in the periodic setting. 
From direct computations, we may show that the  kernel of $\mathcal{L}_0$ is given by the set of functions in the form
$$(t,\theta)\mapsto\sum_{j\in\mathbb{Z}^{*}}r_j e^{\ii(j\theta-\Omega_j(b)t)},$$
where
\begin{equation}\label{omegajb intro}
	\forall j\in\mathbb{Z}\backslash\{0\},\quad\Omega_{j}(b)=\tfrac{{sgn}(j)}{2}\big(|j|-1+b^{2|j|}\big),
\end{equation}
where we denote by ${sgn}$  the sign function.
Notice that here and in the sequel, we shall use the notation
$$\mathbb{N}=\{0,1,2,\ldots\}\quad\textnormal{and}\quad\mathbb{N}^*=\{1,2,\ldots\}.$$
Consider a finite number of Fourier modes
$$\mathbb{S}=\big\{j_1,\ldots,j_d\big\}\subset\mathbb{N}^*\quad\textnormal{with}\quad1\leqslant j_1<\ldots<j_d,\quad (d\in\mathbb{N}^*).$$
Then, from Proposition \ref{lemma sol Eq}, we deduce that, for any $0<b_0<b_1<1$, for almost all $b\in[b_0,b_1],$ any function in the form
$$r:(t,\theta)\mapsto\sum_{j\in\mathbb{S}}r_j\cos(j\theta-\Omega_j(b)t),\quad r_j\in\mathbb{R}$$
is a quasi-periodic solution with frequency $\omega_{\textnormal{Eq}}(b)=(\Omega_j(b))_{j\in\mathbb{S}}$ of the equation $\mathcal{L}_0 r=0$ which is reversible, namely $r(-t,-\theta)=r(t,\theta).$ The measure of the Cantor set in $b$ generating these solutions is estimated using Rüssemann Lemma \ref{lemma useful for measure estimates} requiring a lower bound on the maximal derivative of a given function up to order $q_0.$ In our case, the value of $q_0$ is explicit, namely $q_0=2j_d+2$ which is due to the polynomial structure of the $\Omega_j(b).$
The aim of this work is to prove that these structures persist at the non-linear level, more precisely, our main result reads as follows.
\begin{theo}\label{thm QPS ED}
		Let $0<b_0<b_1<1$,  $d\in\mathbb{N}^{*}$ and  $\mathbb{S}\subset\mathbb{N}^*$ with $|\mathbb{S}|=d.$
		There exists  $\varepsilon_{0}\in(0,1)$ small enough with the following properties :  For every amplitudes ${\mathtt{a}}=(\mathtt{a}_{j})_{j\in\mathbb{S}}\in(\mathbb{R}_{+}^{*})^{d}$ satisfying
		$$|{\mathtt{a}}|\leqslant\varepsilon_{0},$$ 
		there exists a Cantor-like set $\mathcal{C}_{\infty}\subset(b_{0},b_{1})$ with asymptotically full Lebesgue  measure as ${\mathtt{a}}\rightarrow 0,$ i.e.
		$$\lim_{{\mathtt{a}}\rightarrow 0}|\mathcal{C}_{\infty}|=b_{1}-b_{0},$$
		such that for any $b\in\mathcal{C}_{\infty}$, the  equation \eqref{Ham eq ED r intro} admits a time quasi-periodic solution with diophantine frequency vector ${\omega}_{\tiny{\textnormal{pe}}}(b,{\mathtt{a}}):=(\omega_{j}(b,{\mathtt{a}}))_{j\in\mathbb{S}}\in\mathbb{R}^{d}$ and taking  the form
		$$r(t,\theta)=\sum_{j\in\mathbb{S}}{\mathtt{a}_{j}}\cos\big(j\theta+\omega_{j}(b,{\mathtt{a}})t\big)+\mathtt{p}\big({\omega}_{\tiny{\textnormal{pe}}}(b,\mathtt{a})t,\theta\big),$$
		with
		$${\mathtt{\omega}}_{\tiny{\textnormal{pe}}}(b,{\mathtt{a}})\underset{{\mathtt{a}}\rightarrow 0}{\longrightarrow}(-\Omega_{j}(b))_{j\in\mathbb{S}},$$
		where $\Omega_{j}(b)$ are the equilibrium frequencies defined in \eqref{omegajb intro} and  the perturbation $\mathtt{p}:\T^{d+1}\to\mathbb{R}$ is an even function satisfying
		$$\| \mathtt{p}\|_{H^{{s}}(\mathbb{T}^{d+1},\mathbb{R})}\underset{{\mathtt{a}}\rightarrow 0}{=}o(|{\mathtt{a}}|)$$
		for some large   index of regularity $s.$
\end{theo}
We shall now sketch  the main steps  used to prove  the previous theorem. First remark that small divisors problems already appear in the proof of Proposition \ref{lemma sol Eq} to find quasi-periodic structures at the linear level from the equilibrium. We can invert the linearized operator at the equilibrium with some fixed loss of regularity. Hence, we need to use a Nash-Moser scheme to find quasi-periodic solutions for the non-linear model. To do so, we must invert the linearized operator in a neighborhood of the equilibrium state. Since $\mathcal{L}_{r}$ has non constant coefficients, the task is more delicate. The basic idea consists in diagonalizing, namely to conjugate to constant coefficients operator. Actually, we may follow the procedure presented in \cite{BB15}, slightly modified in \cite{HHM21,HR21}, where the dynamics is decoupled into tangential and normal parts. On the tangential modes, we introduce action-angles variables $(I,\vartheta)$ allowing to reformulate the problem in terms of embedded tori. More precisely, we shall look for the zeros of the following functional
$$\mathcal{F}(i,\alpha,b,\omega,\varepsilon)=\left(\begin{array}{c}
	\omega\cdot\partial_{\varphi}\vartheta(\varphi)-\alpha-\varepsilon\partial_{I}\mathcal{P}_{\varepsilon}(i(\varphi))\\
	\omega\cdot\partial_{\varphi}I(\varphi)+\varepsilon\partial_{\vartheta}\mathcal{P}_{\varepsilon}(i(\varphi))\\
	\omega\cdot\partial_{\varphi}z(\varphi)-\partial_{\theta}\big[\mathrm{L}(b)z(\varphi)+\varepsilon\nabla_{z}\mathcal{P}_{\varepsilon}\big(i(\varphi)\big)\big]
\end{array}\right).$$
It turns out that it is more convenient to introduce one degree of freedom through the parameter $\alpha$ which provides at the end of the scheme a solution for the original problem when it is fixed  to $-\omega_{\textnormal{Eq}}(b).$ Given any small reversible embedded torus $i_0:\varphi\mapsto(\vartheta_0(\varphi),I_0(\varphi),z_0(\varphi))$ and any $\alpha_0\in\mathbb{R}^d,$ conjugating the linearized operator $d_{i,\alpha}\mathcal{F}(i_0,\alpha_0)$ via a suitable linear diffeomorphism of the toroidal phase space $\mathbb{T}^d\times\mathbb{R}^d\times L_{\perp}^{2}$, we obtain a triangular system in the action-angle-normal variables up to error terms. To solve the triangular system, we only have to invert the linearized operator in the normal directions, which is denoted by  $\widehat{\mathcal{L}}_{\omega}$.
This is done using KAM reducibility techniques in a similar way to \cite{BBMH18,BM18,HHM21,HR21}.
According to Proposition \ref{prop hat L omega}, we can write
$$\widehat{\mathcal{L}}_{\omega}=\Pi_{\mathbb{S}_0}^{\perp}\Big(\mathcal{L}_{\varepsilon r}-\varepsilon\partial_{\theta}\mathcal{R}\Big)\Pi_{\mathbb{S}_0}^{\perp},$$
where $\Pi_{\mathbb{S}_0}^\perp$ is the projector in the normal directions, $\mathcal{R}$ is an integral operator and $\mathcal{L}_{\varepsilon r}$ is defined by \eqref{def Lr intro}. First, following the KAM reducibility scheme in \cite{BM21,FGMP19,HR21}, we can reduce the transport part and the zero order part by conjugating by a quasi-periodic symplectic invertible change of variables in the form
$$\mathscr{B}\rho(\mu,\varphi,\theta)=\Big(1+\partial_{\theta}\beta(\mu,\varphi,\theta)\Big)\rho\big(\mu,\varphi,\theta+\beta(\mu,\varphi,\theta)\big).$$
More precisely, as stated in Proposition \ref{reduction of the transport part}, we can find a function $V_{i_0}^{\infty}=V_{i_0}^{\infty}(b,\omega)$ and a Cantor set
$$\mathcal{O}_{\infty,n}^{\gamma,\tau_1}(i_0)=\bigcap_{(l,j)\in\mathbb{Z}^{d}\times\mathbb{Z}\setminus\{(0,0)\}\atop|l|\leqslant N_{n}}\Big\{(b,\omega)\in\mathcal{O}\quad\textnormal{s.t.}\quad\big|\omega\cdot l+jV_{i_0}^{\infty}(b,\omega)\big|>\tfrac{4\gamma^{\upsilon}\langle j\rangle}{\langle l\rangle^{\tau_1}}\Big\}$$
in which the following decomposition holds
$$\mathscr{B}^{-1}\mathcal{L}_{\varepsilon r}\mathscr{B}=\omega\cdot\partial_{\varphi}+V_{i_0}^{\infty}\partial_{\theta}+\partial_{\theta}\mathcal{K}_{1,b}\ast\cdot-\partial_{\theta}\mathcal{K}_{2,b}\ast\cdot+\partial_{\theta}\mathfrak{R}_{\varepsilon r}+\mathtt{E}_{n}^{0},$$
where $\mathfrak{R}_{\varepsilon r}$ is a real and reversibility preserving Toeplitz in time integral operator enjoying good smallness properties. The operator $\mathtt{E}_n^0$ is an error term of order one  associated to the time troncation of the Cantor set $\mathcal{O}_{\infty,n}^{\gamma,\tau_1}(i_0).$   Notice that $N_n$  is defined by
$$N_n=N_0^{(\frac{3}{2})^n}\quad\textnormal{with}\quad N_0\gg 1.$$
Then, we project in the normal directions by considering the  operator
$$\mathscr{B}_{\perp}=\Pi_{\mathbb{S}_0}^{\perp}\mathscr{B}\Pi_{\mathbb{S}_0}^{\perp}.$$
Therefore, in view of Proposition \ref{projection in the normal directions}, we obtain the following decomposition in $\mathcal{O}_{\infty,n}^{\gamma,\tau_1}(i_0)$
$$\mathscr{B}_{\perp}^{-1}\widehat{\mathcal{L}}_{\omega}\mathscr{B}_{\perp}=\omega\cdot\partial_{\varphi}+\mathscr{D}_0+\mathscr{R}_0+\mathtt{E}_n^1:=\mathscr{L}_0+\mathtt{E}_n^1,$$
where $\mathscr{D}_0=(\ii\mu_{j}^{0}(b,\omega))_{j\in\mathbb{S}_0^c}$ is a diagonal and reversible operator and $\mathscr{R}_0=\Pi_{\mathbb{S}_0}^{\perp}\mathscr{R}_0\Pi_{\mathbb{S}_0}^{\perp}$ is a real and reversible Toeplitz in time remainder integral operator in $OPS^{-\infty}$ in space and satisfying nice smallness properties. The term $\mathtt{E}_n^1$ plays a similar role as the previous one  $\mathtt{E}_n^0$. The next goal is to reduce the remainder term $\mathscr{R}_0.$ For this aim, we implement a KAM reduction process in the Toeplitz topology as in \cite[Prop. 6.5]{HR21}. The result is stated in Proposition \ref{reduction of the remainder term} and provides two operators $\Phi_{\infty}$ and $\mathscr{D}_{\infty}=(\ii \mu_{j}^{\infty}(b,\omega))_{j\in\mathbb{S}_0^c}$, with $\mathscr{D}_{\infty}$ a diagonal and reversible operator whose spectrum is described by
$$\forall j\in\mathbb{S}_0^c,\quad\mu_{j}^{\infty}(b,\omega)=\Omega_{j}(b)+j\big(V_{i_0}^{\infty}(b,\omega)-\tfrac{1}{2}\big)+r_{j}^{\infty}(b,\omega),$$
such that in the Cantor set
$$\mathscr{O}_{\infty,n}^{\gamma,\tau_1,\tau_2}(i_0)=\bigcap_{(l,j,j_0)\in\mathbb{Z}^{d}\times(\mathbb{S}_0^c)^2\atop\underset{(l,j)\neq(0,j_0)}{\langle l,j-j_0\rangle\leqslant N_n}}\Big\{(b,\omega)\in\mathcal{O}_{\infty,n}^{\gamma,\tau_1}(i_0)\quad\textnormal{s.t.}\quad\big|\omega\cdot l+\mu_{j}^{\infty}(b,\omega)-\mu_{j_0}^{\infty}(b,\omega)\big|>\tfrac{2\gamma\langle j-j_0\rangle}{\langle l\rangle^{\tau_2}}\Big\}$$
the following decomposition holds
$$\Phi_{\infty}^{-1}\mathscr{L}_0\Phi_{\infty}=\omega\cdot\partial_{\varphi}+\mathscr{D}_{\infty}+\mathtt{E}_n^2:=\mathscr{L}_{\infty}+\mathtt{E}_n^2.$$
Now, we can invert the operator $\mathscr{L}_{\infty}$ when  the parameters are restricted to the Cantor set
$$\Lambda_{\infty,n}^{\gamma,\tau_1}(i_0)=\bigcap_{(l,j)\in\mathbb{Z}^{d }\times\mathbb{S}_{0}^{c}\atop |l|\leqslant N_{n}}\Big\{(b,\omega)\in\mathcal{O}\quad\textnormal{s.t.}\quad\big|\omega\cdot l+\mu_{j}^{\infty}(b,\omega)\big|>\tfrac{\gamma\langle j\rangle}{\langle l\rangle^{\tau_1}}\Big\}.$$
Therefore, we are able to construct an approximate right inverse of $\widehat{\mathcal{L}}_{\omega}$ in  the Cantor set
$$\mathcal{G}_{n}^{\gamma}(i_0)=\mathcal{O}_{\infty,n}^{\gamma,\tau_1}(i_0)\cap\mathscr{O}_{\infty,n}^{\gamma,\tau_1,\tau_2}(i_0)\cap\Lambda_{\infty,n}^{\gamma,\tau_1}(i_0).$$
We refer to Proposition \ref{inversion of the linearized operator in the normal directions} for more details. Now we can implement a Nash-Moser scheme in a similar way to \cite{BM18,HHM21,HR21} to find a solution $(b,\omega)\mapsto(i_{\infty}(b,\omega),\alpha_{\infty}(b,\omega))$ to the equation $\mathcal{F}(i,\alpha,b,\omega,\varepsilon)=0$ provided that the parameters $(b,\omega)$ are selected among a Cantor set $\mathcal{G}_{\infty}^{\gamma}$ which is constructed as the intersection of all the Cantor sets appearing in the scheme to invert at each step the linearized operator. To find a  solution to the original problem we construct a frequency curve $b\mapsto\omega(b,\varepsilon)$ implicitly defined by solving the equation 
$$\alpha_{\infty}(b,\omega(b,\varepsilon))=-\omega_{\textnormal{Eq}}(b).$$
Hence, we obtain the desired result for any value of $b$ in the Cantor set
$$\mathcal{C}_{\infty}^{\varepsilon}=\Big\{b\in(b_0,b_1)\quad\textnormal{s.t.}\quad(b,\omega(b,\varepsilon))\in\mathcal{G}_{\infty}^{\gamma}\Big\}.$$
Then, it remains to check that this set is non-trivial. This is done by estimating its measure using perturbed Rüssemann conditions from the equilibrium. In Proposition \ref{lem-meas-es1}, we find a lower bound for the measure of $\mathcal{C}_{\infty}^{\varepsilon},$ namely
$$|\mathcal{C}_{\infty}^{\varepsilon}|\geqslant(b_1-b_0)-C\varepsilon^{\delta}\quad\textnormal{for some }\delta=\delta(q_0,d,\tau_1,\tau_2)>0.$$

\paragraph{Notations.}
Along this paper we shall make use of the following parameters and sets.
	\begin{enumerate}[label=\textbullet]
	\item We denote by 
	$$\mathbb{N}:=\{0,1,\cdots\},\qquad \mathbb{Z}:=\{\cdots,-1,0,1,\cdots\}$$ the set of natural numbers and the set of integers, respectively, and   we set 
	$$\mathbb{N}^*:=\mathbb{N}\backslash\{0\},\qquad \mathbb{Z}^*:=\mathbb{Z}\backslash\{0\}.$$
	
	 \item The integer $d$ is the number of excited frequencies that will generate the  quasi-periodic solutions. This is the dimension of the space where lies the  frequency vector  $\omega\in\mathbb{R}^{d},$ that will be a perturbation of the equilibrium frequency ${\omega}_{\textnormal{Eq}}(b)$, defined by \eqref{def freq vec eqb}.
	   \item The real numbers $b_0$ and $b_1$ are fixed such that
        $$0<b_0<b_1<1.$$
        The parameter $b$ is the radius of the disc corresponding the the equilibrium state and  lies in the interval $(b_{0},b_{1}).$ However at the end it  will belong to a Cantor set for which invariant torus can be constructed.
        \item Since the application $b\mapsto\omega_{\textnormal{Eq}}(b)$ is continuous then $\omega_{\textnormal{Eq}}\left([b_{0},b_{1}]\right)$ is a compact subset of $\mathbb{R}^{d}.$ In particular,  there exists  $R_{0}>0$ such that
$$\omega_{\textnormal{Eq}}\left((b_{0},b_{1})\right)\subset\mathscr{U}:=B(0,R_{0}).$$
We also consider $\mathcal{O}$ the open bounded subset of $\mathbb{R}^{d+1}$ defined by
\begin{equation}\label{def initial parameters setb}
	\mathcal{O}:=(b_{0},b_{1})\times \mathscr{U}.
\end{equation}
\item  The integer  $q$ is the index of regularity of our functions/operators with respect to the parameters $b$ and $\omega$.  It is  chosen  as
$$q:=q_0+1,$$
with $q_0$ being the non-degeneracy index provided by Lemma \ref{lemma transversalityb}, which only depends on the
linear unperturbed frequencies. 
\item  The real parameters $\gamma$, $\tau_{1}$ and $\tau_{2}$ satisfy
		\begin{equation}\label{setting tau1 and tau2}
	0<\gamma<1,\quad \tau_{2}>\tau_{1}>d
\end{equation}
 and   are linked to  different Diophantine conditions, see for instance Propositions \ref{reduction of the transport part} and \ref{reduction of the remainder term}. The choice of $\tau_{1}$ and $\tau_{2}$ will be finally fixed in \eqref{choice tau 1 tau2 upsilon}. We point out that the parameter  $\gamma$ appears in the weighted Sobolev spaces and will be fixed in Proposition \ref{Nash-Moser} with respect to  the rescaling parameter $\varepsilon$ giving  the smallness  condition of the solutions around the equilibrium.	
		\item The real number $s$ is the Sobolev index  regularity of the functions in the variables $\varphi$ and $\theta.$ The index $s$ will vary between $s_{0}$ and $S$,
		\begin{equation}\label{initial parameter condition}
	 S\geqslant s\geqslant s_{0}>\tfrac{d+1}{2}+q+2,
\end{equation}
where $S$ is a fixed large number. 
\item  For a given continuous complex function $f : \mathbb{T}^{n}\to \mathbb{C}$, $n\geqslant 1$, $\mathbb{T}:=\mathbb{R}/2\pi\mathbb{Z}$, we denote by
 \begin{equation}\label{Convention1}
\int_{\mathbb{T}^{n}}f(x)dx:=\frac{1}{(2\pi)^{n}}\bigintssss_{[0,2\pi]^{n}}f(x)dx.
\end{equation}
\item We denote by  $(\mathbf{e}_{l,j})_{(l,j)\in\mathbb{Z}^{d}\times\mathbb{Z}}$ the Hilbert basis of the  $L^{2}(\mathbb{T}^{d+1},\mathbb{C})$,
$$\mathbf{e}_{l,j}(\varphi,\theta):=e^{\ii(l\cdot\varphi+j\theta)},$$
and we endow this space   with the Hermitian inner  product
\begin{equation}\label{inter-product}
\begin{aligned}
	\big\langle\rho_{1},\rho_{2}\big\rangle_{L^{2}(\mathbb{T}^{d+1},\mathbb{C})}&=\bigintssss_{\mathbb{T}^{d+1}}\rho_{1}(\varphi,\theta)\overline{\rho_{2}(\varphi,\theta)}d\varphi d\theta. 
\end{aligned}
\end{equation}
\end{enumerate}

\section{Hamiltonian reformulation}
In this section, we shall write down the equation governing the boundary dynamics. For that purpose, we shall consider a polar parametrization of the boundary and see that the radial deformation in there is subject to a non-linear and non-local Hamiltonian equation of transport type.
\subsection{Equation satisfied by the radial deformation of the patch}
Given $b\in (0,1),$ consider a vortex patch $t\mapsto \mathbf{1}_{D_t}$,  near the Rankine vortex   $\mathbf{1}_{b\mathbb{D}}$ with a smooth boundary whose polar parametrization is given by
\begin{equation}\label{parametrization}
z(t,\theta)=\left(b^{2}+2r(t,\theta)\right)^{\frac{1}{2}}e^{\ii\theta},
\end{equation} 
where $r$ is the radial deformation assumed to be small, namely $|r(t,\theta)|\ll 1.$
 In the sequel,  we shall frequently  use the following notations
\begin{align}
        R(b,t,\theta)&:=\left(b^{2}+2r(t,\theta)\right)^{\frac{1}{2}}, \label{definition of R}\\
	A_r(b,t,\theta,\eta)&:=\left|R(b,t,\theta)e^{\ii\theta}-R(b,t,\eta)e^{\ii\eta}\right|,\label{Ar}\\
	B_r(b,t,\theta,\eta)&:=\left|1-R(b,t,\theta)R(b,t,\eta)e^{\ii(\eta-\theta)}\right|.\label{Br}
\end{align}
The equation satisfied by  $r$ is given by the following lemma.
\begin{lem}\label{lem eq ED r}
	For short time $T>0,$ the radial deformation $r$, defined through \eqref{definition of R}, satisfies the following nonlinear and nonlocal transport PDE:
	\begin{equation}\label{ED eq r}
		\forall (t,\theta)\in[0,T]\times\mathbb{T},\quad\partial_{t}r(t,\theta)+F_{b}[r](t,\theta)=0,
	\end{equation}
where
\begin{align}
        F_b[r]&:=-F_b^0[r]-F_b^1[r]+F_b^2[r], \label{Fb}\\
	F_b^0[r]&:=\tfrac{1}{2}\partial_{\theta}r(t,\theta)\int_{\mathbb{T}}\tfrac{R^2(b,t,\eta)}{R^2(b,t,\theta)}d\eta,\label{Fb0}\\
	F_b^1[r]&:=\int_{\mathbb{T}}\log\big(A_r(b,t,\theta,\eta)\big)\partial_{\theta\eta}^2\Big(R(b,t,\theta)R(b,t,\eta)\sin(\eta-\theta)\Big)d\eta,\label{Fb1}\\
	F_b^2[r]&:=\int_{\mathbb{T}}\log\big(B_r(b,t,\theta,\eta)\big)\partial_{\theta\eta}^2\Big(\tfrac{R(b,t,\eta)}{R(b,t,\theta)}\sin(\eta-\theta)\Big)d\eta,\label{Fb2}
\end{align}
where  $R(b,t,\theta)$, $A_r(b,t,\theta,\eta)$ and $B_r(b,t,\theta,\eta)$ are given by \eqref{definition of R}-\eqref{Br}.
\end{lem}
\begin{proof}
We start with the vortex patch equation. Denoting $\mathbf{n}$ the outward normal vector to the boundary of the patch, the evolution equation of the boundary can be written as
  $$\partial_{t}z(t,\theta)\cdot \mathbf{n}(t,z(t,\theta))=-\partial_{\theta}\boldsymbol{\Psi}(t,z(t,\theta)).$$
  For a detailed proof  see for instance \cite[p.174]{HMV13}.
We shall identify $\mathbb{C}$ with $\mathbb{R}^{2}.$ In particular, the Euclidean structure of $\mathbb{R}^{2}$ is seen in the complex sense through the usual inner  product defined for all $z_{1}=a_{1}+\ii\, b_{1}\in\mathbb{C}$ and $z_{2}=a_{2}+\ii\, b_{2}\in\mathbb{C}$ by
\begin{equation}\label{definition of the scalar product on R2}
	z_{1}\cdot z_{2}:=\langle z_{1},z_{2}\rangle_{\mathbb{R}^{2}}=\mbox{Re}\left(z_{1}\overline{z_{2}}\right)=a_{1}a_{2}+b_{1}b_{2}.
\end{equation}
Since $\mathbf{n}(t,z(t,\theta))=-\ii\, \partial_{\theta}z(t,\theta)$ (up to a real constant of renormalization) then the complex formulation of the vortex patch equation is given by 
$$
	\mbox{Im}\left(\partial_{t}z(t,\theta)\overline{\partial_{\theta}z(t,\theta)}\right)=\partial_{\theta}\boldsymbol{\Psi}(t,z(t,\theta)).
$$
Using the parametrization \eqref{parametrization}, one easily checks that
$$\mbox{Im}\left(\partial_{t}z(t,\theta)\overline{\partial_{\theta}z(t,\theta)}\right)=-\partial_{t}r(t,\theta).$$
Thus, the vortex patch equation writes  in the following way
\begin{equation}\label{ED eq r-Psi}
	\partial_{t}r(t,\theta)+\partial_{\theta}\boldsymbol{\Psi}(t,z(t,\theta))=0.
\end{equation}
Now we shall compute $\partial_{\theta}\boldsymbol{\Psi}(t,z(t,\theta))$. Using complex notations, we have 
\begin{align}\label{partial-theta-psi}
	\partial_{\theta}\boldsymbol{\Psi}(t,z(t,\theta))&=\nabla\boldsymbol{\Psi}(t,z(t,\theta))\cdot\partial_{\theta}z(t,\theta)=2\textnormal{Re}\left(\partial_{\overline{w}}\mathbf{\Psi}(t,z(t,\theta))\overline{\partial_{\theta}z(t,\theta)}\right).
\end{align}
Recall, from \eqref{expression of the stream function}, that the stream function $\boldsymbol{\Psi}$ writes
\begin{equation*}
\forall w\in \mathbb{C},\quad 	\boldsymbol{\Psi}(t,w)=\tfrac{1}{4\pi}\int_{D_t}\log\big(|w-\xi|^2\big)dA(\xi)-\tfrac{1}{4\pi}\int_{D_t}\log\big(|\overline{\xi}w-1|^{2}\big)dA(\xi).
	\end{equation*}
Let $\epsilon>0$. We set
$$
f_\epsilon(\xi,\overline{\xi}):=(\overline{\xi}-\overline{w})\Big[\log\left(|\xi-w|^{2}+\epsilon\right)-1\Big]-\left(\overline{\xi}-\tfrac{1}{w}\right)\Big[\log\left(|1-w\overline{\xi}|^{2}\right)-1\Big].
$$
Then	
$$
\partial_{\overline{\xi}} f_\epsilon(\xi,\overline{\xi})=\log\big(|w-\xi|^2+\epsilon\big)-\frac{\epsilon}{|w-\xi|^2+\epsilon}-\log\big(|\overline{\xi}w-1|^{2}\big).
$$
Using  the complex version of Stokes' Theorem, 
$$
	2\ii\int_{D}\partial_{\overline{\xi}}f_\epsilon(\xi,\overline{\xi})dA(\xi)=\int_{\partial D}f_\epsilon(\xi,\overline{\xi})d\xi,
$$
then passing to the limit as  $\epsilon$ goes to $0$ we obtain
$$
\boldsymbol{\Psi}(t,w)=\tfrac{1}{8\ii\pi}\int_{\partial D_{t}}(\overline{\xi}-\overline{w})\Big[\log\left(|\xi-w|^{2}\right)-1\Big]d\xi-\tfrac{1}{8\ii\pi}\int_{\partial D_{t}}\left(\overline{\xi}-\tfrac{1}{w}\right)\Big[\log\left(|1-w\overline{\xi}|^{2}\right)-1\Big]d\xi.
$$
Performing the change of variables $\xi=z(t,\eta)$, given by \eqref{parametrization}, and using the notation \eqref{Convention1}  we can write
\begin{equation*}
\begin{aligned}
	\mathbf{\Psi}(t,w) & = \tfrac{1}{4\ii}\int_{\mathbb{T}}(\overline{z}(t,\eta)-\overline{w})\Big[\log\left(|z(t,\eta)-w|^{2}\right)-1\Big]\partial_{\eta}z(t,\eta)d\eta\\
	&  \quad -\tfrac{1}{4\ii}\int_{\mathbb{T}}\left(\overline{z}(t,\eta)-\tfrac{1}{w}\right)\Big[\log\left(|1-w\overline{z}(t,\eta)|^{2}\right)-1\Big]\partial_{\eta}z(t,\eta)d\eta.
\end{aligned}
\end{equation*}
It follows that
\begin{equation}\label{partial-w-psi}
\begin{aligned}
	\partial_{\overline{w}}\mathbf{\Psi}(t,w) & = -\tfrac{1}{4\ii}\int_{\mathbb{T}}\log\left(|z(t,\eta)-w|^{2}\right)\partial_{\eta}z(t,\eta)d\eta \\ &\quad -\tfrac{1}{4\ii}\int_{\mathbb{T}}\big(\overline{z}(t,\eta)-\frac{1}{w}\big)\Big[\frac{\frac{1}{\overline{w}^2}}{{z}(t,\eta)-\frac{1}{\overline{w}}}+\frac{1}{\overline{w}}\Big]\partial_{\eta}z(t,\eta)d\eta.
\end{aligned}
\end{equation}
Direct computations lead to 
\begin{align*}
\Big[\frac{\overline{z}(t,\eta)-\frac{1}{w}}{{z}(t,\eta)-\frac{1}{\overline{w}}}\Big]\partial_{\eta}z(t,\eta)=\partial_\eta\Big[\log\big( \big|{z}(t,\eta)-\frac{1}{\overline{w}}\big|^2\big)+\log\big( |{w}|^2\big)\Big]\Big(\overline{z}(t,\eta)-\frac{1}{w}\Big)-\partial_\eta \overline{z}(t,\eta).
\end{align*}
Inserting this identity into \eqref{partial-w-psi},  integrating by parts, using the morphism property of the logarithm and the periodicity  imply
\begin{equation}\label{dw-psi00}
\begin{aligned}
\partial_{\overline{w}}\mathbf{\Psi}(t,w) & = -\tfrac{1}{4\ii}\int_{\mathbb{T}}\log\left(|z(t,\eta)-w|^{2}\right)\partial_{\eta}z(t,\eta)d\eta
\\ &\quad +\tfrac{1}{4\ii}\int_{\mathbb{T}}\log\left(|1-\overline{w}{z}(t,\eta)|^2\right)\partial_\eta\overline{z}(t,\eta)\frac{1}{\overline{w}^2}d\eta
\\ &\quad-\tfrac{1}{4\ii}\int_{\mathbb{T}}\overline{z}(t,\eta)\partial_\eta {z}(t,\eta)\frac{1}{\overline{w}}d\eta.
\end{aligned}
\end{equation}
As a consequence,  one gets
\begin{align*}
\nonumber 2\textnormal{Re}\Big(\partial_{\overline{w}}\mathbf{\Psi}(t,z(t,\theta)) \partial_\theta\overline{z}(t,\theta)\Big)& = -\tfrac{1}{2}\int_{\mathbb{T}}\log\left(|z(t,\eta)-z(t,\theta)|^{2}\right)\textnormal{Im}\big(\partial_{\eta} z(t,\eta) \partial_{\theta}\overline{z}(t,\theta)\big)d\eta 
\\ &\quad
 +\tfrac{1}{2}\int_{\mathbb{T}}\log\left(|1-\overline{z}(t,\theta){z}(t,\eta)|^2\right) \textnormal{Im}\bigg(\partial_\eta\overline{z}(t,\eta)\frac{\partial_\theta \overline{z}(t,\theta)}{\overline{z}(t,\theta)^2}\bigg)d\eta
\\ &\quad-\tfrac{1}{2}\int_{\mathbb{T}}\textnormal{Im}\bigg(\overline{z}(t,\eta)\partial_\eta{z}(t,\eta) \frac{\partial_\theta\overline{z}(t,\theta)}{\overline{z}(t,\theta)}\bigg)d\eta.
\end{align*}
That is, by \eqref{partial-theta-psi},
\begin{equation*}
\begin{aligned}
\nonumber \partial_{\theta}\boldsymbol{\Psi}(t,z(t,\theta))& = -\tfrac{1}{2}\int_{\mathbb{T}}\log\left(|z(t,\eta)-z(t,\theta)|^{2}\right)\partial^2_{\theta\eta}\textnormal{Im}\big(z(t,\eta) \overline{z}(t,\theta)\big)d\eta 
\\ &\quad
 +\tfrac{1}{2}\int_{\mathbb{T}}\log\left(|1-\overline{z}(t,\theta){z}(t,\eta)|^2\right)\partial^2_{\theta\eta} \textnormal{Im}\bigg(\frac{{z}(t,\eta)}{{z}(t,\theta)}\bigg)d\eta
\\ &\quad-\tfrac{1}{2}\int_{\mathbb{T}}\textnormal{Im}\bigg(\overline{z}(t,\eta)\partial_\eta {z}(t,\eta) \frac{\partial_\theta\overline{z}(t,\theta)}{\overline{z}(t,\theta)}\bigg)d\eta.
\end{aligned}
\end{equation*}
From \eqref{parametrization} we immediately get
\begin{align*}
\textnormal{Im}\big(z(t,\eta) \overline{z}(t,\theta)\big)&=R(b,t,\theta)R(b,t,\eta)\sin(\eta-\theta), \\
 \textnormal{Im}\bigg(\frac{{z}(t,\eta)}{{z}(t,\theta)}\bigg)&=\frac{R(b,t,\theta)}{R(b,t,\eta)}\sin(\eta-\theta),
\\
\textnormal{Im}\bigg(\overline{z}(t,\eta)\partial_\eta {z}(t,\eta) \frac{\partial_\theta\overline{z}(t,\theta)}{\overline{z}(t,\theta)}\bigg)&=\frac{R^2(b,t,\eta)}{R^2(b,t,\theta)}\partial_\theta r(t,\theta)-\partial_\eta r(t,\eta).
\end{align*}
Combining the  last four  identities with \eqref{ED eq r-Psi} and using the notations \eqref{parametrization}-\eqref{Br} we conclude the desired result.
\end{proof}

We look for time quasi-periodic solutions of \eqref{ED eq r}; that are functions in the form 
$$\widehat{r}(t,\theta)=r(\omega t,\theta),$$
where  $r=r(\varphi,\theta):\mathbb{T}^{d+1}\rightarrow\mathbb{R},$ $\omega\in\mathbb{R}^{d}$, $d\in\mathbb{N}^{*}.$ With this ansatz, the equation \eqref{ED eq r} becomes 
\begin{equation}\label{equation with quasi-periodic ansatz}
\omega\cdot\partial_{\varphi}r(\varphi,\theta)+F_{b}[r](\varphi,\theta)=0.
\end{equation}

\subsection{Hamiltonian structure}
In this section, we show that the contour dynamics equation \eqref{ED eq r} has a Hamiltonian structure related to the kinetic energy
\begin{equation}\label{def:energy}
E(r)(t):=\frac{1}{2\pi}\int_{D_{t}}\mathbf{\Psi}(t,z)dA(z),
\end{equation}
 which is a conserved quantity for \eqref{Euler velocity-vorticity}. It is well-known that the bidimensional Euler equations admits a Hamiltonian structure and we shall see here that such structure still persists at the level of the boundary equation, which is a stronger formulation.

\begin{prop}\label{prop Ham eq r}
The equation \eqref{ED eq r} is a Hamiltonian equation in the form 
\begin{equation}\label{Ham eq ED r}
\partial_{t}r=\partial_{\theta}\nabla H(r),\quad \textnormal{where}\quad H(r)=-\tfrac{1}{2}E(r),
\end{equation}
and $\nabla$ is the $L_{\theta}^{2}(\mathbb{T})$-gradient associated with the $L_\theta^{2}(\mathbb{T})$ normalized inner product 
$$\big\langle \rho_{1}, \rho_{2}\big\rangle_{L^{2}(\mathbb{T})}=\int_{\mathbb{T}}\rho_{1}(\theta)\rho_{2}(\theta)d\theta.$$
\end{prop}
\begin{proof}
In polar coordinates, the stream function, given by \eqref{expression of the stream function}, at some point $w\in\mathbb{C}$  
writes 
\begin{align}\label{psi-polar}
	\mathbf{\Psi}(t,w)&=\int_{\mathbb{T}}\int_{0}^{R(b,t, \eta)}G\left(w,\ell_2e^{\ii \eta}\right)\ell_2d\ell_2  d\eta \quad{\rm with}\quad G\left(w,\xi\right):=\log\left(\left|\frac{w-\xi}{1-w\overline{\xi}}\right|\right)
\end{align}
and  kinetic energy $E$, in \eqref{def:energy}, reads
\begin{align*}
	E(r)(t)&=\int_{\mathbb{T}}\int_{\mathbb{T}}\int_{0}^{R(b,t, \theta)}\left(\int_{0}^{R(b,t, \eta)}G\left(\ell_1e^{\ii \theta},\ell_2e^{\ii \eta}\right)\ell_2d\ell_2\right)\ell_1 d\ell_1 d\theta d\eta.
\end{align*}
Differentiating with respect to  $r$ in the direction $\rho$ and using the symmetry of the kernel
$
G(w,\xi)=G(\xi,w)
$
 yields
\begin{align*}
	d_r E(r)[\rho](t)&=2\int_{\mathbb{T}}\rho(t,\theta)\left(\int_{\mathbb{T}}\int_{0}^{R(b,t, \eta)}G\left(R(b,t, \theta)e^{\ii \theta},\ell_2e^{\ii \eta}\right)\ell_2 d\ell_2 d\eta \right)d\theta\\ &=2\int_{\mathbb{T}}\rho(t,\theta)\mathbf{\Psi}\big(t,R(b,t, \theta) e^{\ii \theta} \big)d\theta.
\end{align*}
Since $d_r E(r)[\rho]=\langle\nabla E,\rho\rangle_{L^{2}(\mathbb{T})}$ then 
\begin{align}\label{nablaE=nablapsi}
	\nabla E(r)(t,\theta)&=2\mathbf{\Psi}\big(t,R(b,t, \theta) e^{\ii \theta} \big).
\end{align}
Finally, using \eqref{nablaE=nablapsi} and comparing \eqref{Ham eq ED r} with \eqref{ED eq r-Psi}  we conclude the desired result. This achieves the proof of Proposition \ref{prop Ham eq r}.
\end{proof}
Now, we shall present the symplectic structure associated with the Hamiltonian equation \eqref{Ham eq ED r}. This will be relevant later in Section \ref{subsec act-angl} when introducing the action-angle variables. We shall also explore some symmetry property for \eqref{Ham eq ED r}. Observe that this latter equation  implies
$$\frac{d}{dt}\int_{\mathbb{T}}r(t,\theta)d\theta=0.$$
Therefore, we will consider the phase space with zero average in the space variable, that is
$$L_0^2(\mathbb{T}):=\Big\{r=\sum_{j\in\mathbb{Z}^*}r_{j}e_j\quad\textnormal{s.t.}\quad r_{-j}=\overline{r_j}\quad\textnormal{and}\quad\|r\|_{L^2}^2:=\sum_{j\in\mathbb{Z}^*}|r_j|^2<+\infty\Big\},\qquad e_{j}(\theta):=e^{\ii j\theta}.$$
The equation \eqref{Ham eq ED r} induces on the phase space $L_{0}^{2}(\mathbb{T})$ a symplectic structure induced by the symplectic $2$-form
\begin{equation}\label{def symp-form}
	\mathcal{W}(r,h)=\int_{\mathbb{T}}\partial_{\theta}^{-1}r(\theta)h(\theta)d\theta\quad\mbox{ where }\quad\partial_{\theta}^{-1}r(\theta)=\sum_{j\in\mathbb{Z}^{*}}\frac{r_{j}}{\ii j}e^{\ii j\theta}.
\end{equation}
The  Hamiltonian vector field is $X_{H}(r)=\partial_{\theta}\nabla H(r)$ 
associated to the Hamiltonian $H$ is defined as the symplectic gradient of the Hamiltonian $H$ with respect to the symplectic $2$-form $\mathcal{W}$, namely
$$dH(r)[\cdot]=\mathcal{W}(X_{H}(r),\cdot).$$
Decomposing  into Fourier series
$$r=\sum_{j\in\mathbb{Z}^{*}}r_{j}e_j\quad\textnormal{with}\quad r_{-j}=\overline{r_{j}},$$
the symplectic form $\mathcal{W}$ becomes
$$\mathcal{W}(r,h)=\sum_{j\in\mathbb{Z}^{*}}\frac{1}{\ii j}r_{j}h_{-j}=\sum_{j\in\mathbb{Z}^{*}}\frac{1}{\ii j}r_{j}\overline{h_{j}},$$
that is
\begin{equation}\label{sympl ref}
	\mathcal{W}=\frac{1}{2}\sum_{j\in\mathbb{Z}^{*}}\frac{1}{\ii j}dr_{j}\wedge dr_{-j}=\sum_{j\in\mathbb{N}^{*}}\frac{1}{\ii j}dr_{j}\wedge dr_{-j},
\end{equation}
where for any $j\in\mathbb{Z}^*$ the exterior product $dr_j\wedge dr_{-j}$ is defined by
$$dr_j\wedge dr_{-j}(r,h)=r_jh_{-j}-r_{-j}h_j.$$
We shall now look at the reversibility  property of the equation \eqref{Ham eq ED r}.
We consider the involution $\mathscr{S}$ defined on the phase space $L_{0}^{2}(\mathbb{T})$ by 
\begin{equation}\label{definition of the involution mathcal S}
	(\mathscr{S}r)(\theta):=r(-\theta),
\end{equation}
which satisfies
\begin{equation}\label{properties of the involution mathcal S}
\mathscr{S}^{2}=\textnormal{Id}\quad \mbox{ and }\quad \partial_{\theta}\circ\mathscr{S}=-\mathscr{S}\circ\partial_{\theta}.
\end{equation}
Using the change of variables $\eta\mapsto-\eta$ and parity arguments, one gets 
$$F_{b}\circ\mathscr{S}=-\mathscr{S}\circ F_{b},$$
where $F_{b}$ is given by \eqref{Fb}.
Then we conclude by Lemma \ref{lem eq ED r}, \eqref{Ham eq ED r} and \eqref{properties of the involution mathcal S} that
the Hamiltonian vector field $X_H$ satisfies
$$X_H\circ\mathscr{S}=-\mathscr{S}\circ X_H.$$
Thus, we will look for quasi-periodic solutions satisfying the reversibility condition
$$r(-t,-\theta)=r(t,\theta).$$
\section{Linearization and structure of the equilibrium frequencies}
In the current section, we linearize the equation \eqref{ED eq r} at a given
small state $r$ close to the equilibrium. At this latter, we shall see that the linear operator is a Fourier multiplier with polynomial linear frequencies with respect to the radius of the Rankine patch $b\mathbb{D}$. At the end of this section, we also check the transversality conditions for the unperturbed frequency vector.
\subsection{Linearized operator}
We shall first prove that the linearized operator at a general small state $r$ can be decomposed into the sum of a variable coefficients transport operator, a non-local operator of order $0$ and a smoothing non-local operator in the variable $\theta.$ More precisely, we have the following lemma.
\begin{lem}\label{lemma general form of the linearized operator}
	The linearized Hamiltonian equation of \eqref{Ham eq ED r} at a state $r$ is the time-dependent Hamiltonian system 
		$$\partial_{t}\rho(t,\theta)=-\partial_{\theta}\Big(V_{r}(b,t,\theta)\rho(t,\theta)+\mathbf{L}_{r}(\rho)(b,t,\theta)-\mathbf{S}_{r}(\rho)(b,t,\theta)\Big),$$
		where the function $V_{r}$ is  defined by
		\begin{align}\label{Vr}
			V_{r}(b,t,\theta)&=-\tfrac{1}{2}\int_{\mathbb{T}}\tfrac{R^{2}(b,t,\eta)}{R^{2}(b,t,\theta)}d\eta\\
			&\quad-\tfrac{1}{R(b,t,\theta)}\int_{\mathbb{T}}\log\big(A_{r}(b,t,\theta,\eta)\big)\partial_{\eta}\big(R(b,t,\eta)\sin(\eta-\theta)\big)d\eta\nonumber\\
			&\quad-\tfrac{1}{R^{3}(b,t,\theta)}\int_{\mathbb{T}}\log\big(B_{r}(b,t,\theta,\eta)\big)\partial_{\eta}\big(R(b,t,\eta)\sin(\eta-\theta)\big)d\eta,\nonumber
		\end{align} 
		$\mathbf{L}_{r}$ is a non-local operator in the form
		\begin{equation}\label{mathbfLr}
			\mathbf{L}_{r}(\rho)(b,t,\theta)=\int_{\mathbb{T}}\rho(t,\eta)\log\left(A_{r}(t,\theta,\eta)\right)d\eta
		\end{equation}
		and $\mathbf{S}_{r}$ is a smoothing non-local operator in the form
		\begin{align}\label{mathbfSr}
			\mathbf{S}_{r}(\rho)(b,t,\theta)&=\int_{\mathbb{T}}\rho(t,\eta)\log\left(B_{r}(t,\theta,\eta)\right)d\eta.
		\end{align}
		We recall that $A_{r}$, $B_r$ and $R$ are defined by \eqref{Ar}, \eqref{Br} and \eqref{definition of R}, respectively.\\
		Moreover, if $r(-t,-\theta)=r(t,\theta)$, then
		\begin{equation}\label{symVr}
			V_{r}(b,-t,-\theta)=V_{r}(b,t,\theta).
	\end{equation}
\end{lem}
\begin{proof} In all the proof, we shall omit the dependence of our quantities with respect to the variables $b$ and $t$. Notice that linearizing \eqref{ED eq r-Psi}  amounts to compute the Gâteaux derivative of the stream function $\boldsymbol{\Psi}(r,z(\theta)):=\boldsymbol{\Psi}(z(\theta))$ given by \eqref{expression of the stream function} at point $r$ in the direction $\rho$ (real-valued). All the computations are done at a formal level, but can be rigorously justified in a classical way in the functional context introduced in Section \ref{sec funct set}. Applying the chain rule gives 
\begin{align}\label{partial-w-psi2}
d_r \big(\boldsymbol{\Psi}\big(r,z(\theta)\big)\big)[\rho]=\big(d_r \boldsymbol{\Psi}(r,w)[\rho]\big)_{|w=z(\theta)}+
2\textnormal{Re}\left(\big(\partial_{\overline{w}} \boldsymbol{\Psi}(r,w)\big)_{|w=z(\theta)} d_{r}\overline{z}(\theta)[\rho]\right).
\end{align}
Differentiating \eqref{psi-polar} gives
\begin{align}\label{partial-w-psi1}
 d_r\boldsymbol{\Psi}(r, w)[\rho]=\int_{\mathbb{T}}\log\left(\left|\frac{w-R(\eta)e^{\ii \eta}}{1-R(\eta)e^{-\ii \eta} w}\right|\right)\rho(\eta)  d\eta. 
\end{align}
On the other hand, from \eqref{dw-psi00} 
and the identity
\begin{equation*}
d_{r}z(\theta)[\rho](\theta)=\tfrac{\rho(\theta)}{R(\theta)}e^{\ii\theta},
\end{equation*}
we obtain
\begin{align}\label{partial-w-psi3}
\nonumber	2\textnormal{Re}\left(\big(\partial_{\overline{w}} \boldsymbol{\Psi}(r,w)\big)_{|w=z(\theta)} d_{r}\overline{z}(\theta)[\rho]\right) & = -\tfrac{\rho(\theta)}{R(\theta)}\tfrac{1}{2}\int_{\mathbb{T}}\log\left(|z(\eta)-z(\theta)|^{2}\right)\partial_{\eta}\textnormal{Im}\left(z(\eta)e^{-\ii\theta}\right)d\eta\\ &\quad \nonumber -\tfrac{\rho(\theta)}{R^3(\theta)}\tfrac{1}{2}\int_{\mathbb{T}}\log\left(|1-\overline{z(\theta)}{z}(\eta)|^2\right)\partial_\eta \textnormal{Im}\left(\overline{z}(\eta)e^{\ii\theta}\right)d\eta\\ &\quad+\tfrac{\rho(\theta)}{R^2(\theta)}\tfrac{1}{2}\int_{\mathbb{T}}\textnormal{Im}\big(\partial_\eta \overline{z}(\eta){z}(\eta)\big)d\eta.
\end{align}
Putting together  \eqref{partial-w-psi1}, \eqref{partial-w-psi2}, \eqref{partial-w-psi3} and using the identities 
$$
\textnormal{Im}\left(z(\eta)e^{-\ii\theta}\right)=R(\eta)\sin(\eta-\theta),\quad \textnormal{Im}\big(\partial_\eta \overline{z}(\eta){z}(\eta)\big)=-R^2(\eta),
$$
 we conclude the desired result. The symmetry property \eqref{symVr} is an immediate consequence of \eqref{Vr} with the change of variables $\eta\mapsto-\eta.$ This achieves the proof of Lemma \ref{lemma general form of the linearized operator}.
\end{proof}
The following lemma shows that the linearized operator at the equilibrium state is a Fourier multiplier. 
 This provides an integrable Hamiltonian equation from which we shall generate, in Proposition \ref{lemma sol Eq}, quasi-periodic solutions. 
\begin{lem}\label{lemma linearized operator at equilibrium}
\begin{enumerate}
\item The linearized equation of \eqref{Ham eq ED r} at the equilibrium state $(r=0)$ writes 
\begin{equation}\label{Ham eq 0}
\partial_{t}\rho=\partial_{\theta}\mathrm{L}(b)\rho=\partial_{\theta}\nabla H_{\mathrm{L}}(\rho),
\end{equation}
where $\mathrm{L}(b)$ is the self-adjoint operator on $L_0^{2}(\mathbb{T})$ defined by 
\begin{equation}\label{formula Lb}
\mathrm{L}(b):=-\frac{1}{2}-\mathcal{K}_{b}\ast\cdot
\end{equation} with
\begin{align}\label{mathcalKb}
\mathcal{K}_{b}&:=\mathcal{K}_{1,b}-\mathcal{K}_{2,b},\\
\mathcal{K}_{1,b}(\theta)&:=\tfrac{1}{2}\log\left(\sin^{2}\left(\tfrac{\theta}{2}\right)\right),\label{mathcalK1b}\\
\mathcal{K}_{2,b}(\theta)&:=\log\left(|1-b^2e^{\ii\theta}|\right).\label{mathcalK2b}
\end{align}
It is generated by the quadratic Hamiltonian
\begin{equation}\label{defLHL}
	H_{\mathrm{L}}(\rho):=\frac{1}{2}\langle\mathrm{L}(b)\rho,\rho\rangle_{L^{2}(\mathbb{T})}.
\end{equation}
\item From Fourier point of view, if we write $\rho(t,\theta)=\displaystyle\sum_{j\in\mathbb{Z}^{*}}\rho_{j}(t)e^{\ii j\theta}$ with $\rho_{-j}(t)=\overline{\rho_{j}}(t)$,   then the self-adjoint   operator  $\mathrm{L}(b)$ and the Hamiltonian $H_{\mathrm{L}}$  write 
\begin{align}\label{Ham-Fourier}
	\mathrm{L}(b)\rho(\theta)=-\sum_{j\in\mathbb{Z}^{*}}\tfrac{\Omega_{j}(b)}{j}\rho_{j}e^{\ii j\theta}\quad\mbox{ and }\quad H_{\mathrm{L}}(\rho)=-\sum_{j\in\mathbb{Z}^{*}}\tfrac{\Omega_{j}(b)}{2j}|\rho_{j}|^{2},
\end{align}
where $\big(\Omega_{j}(b)\big)_{j\in\mathbb{Z}^*}$ is defined by
\begin{equation}\label{Omegajb}
	\forall j\in\mathbb{N}^*,\quad\Omega_{j}(b)=\frac{j-1+b^{2j}}{2}\quad\textnormal{and}\quad\Omega_{-j}(b)=-\Omega_{j}(b).
\end{equation}
Moreover, the reversible  solutions of the equation \eqref{Ham eq 0} take the form
\begin{equation}\label{solution of the linear system}
\rho(t,\theta)=\sum_{j\in\mathbb{Z}^*}\rho_{j}\cos{(j\theta-\Omega_{j}(b)t)}, \quad \rho_j\in\mathbb{R}.
\end{equation}
\end{enumerate}
\end{lem}
\begin{proof}
\textbf{1.} Notice that we can rewrite $A_r$ and $B_r$, defined in \eqref{Ar}, \eqref{Br}, as
\begin{align}\label{formula}
	A_r(b,t,\theta,\eta)&=\Big(R^2(b,t,\theta)+R^{2}(b,t,\eta)-2R(b,t,\theta)R(b,t,\eta)\cos(\eta-\theta)\Big)^{\frac{1}{2}}\nonumber\\
	&=\Big(\big(R(b,t,\theta)-R(b,t,\eta)\big)^2+4R(b,t,\theta)R(b,t,\eta)\sin^2(\eta-\theta)\Big)^{\frac{1}{2}}
\end{align}
and
\begin{align}\label{formulb}
	B_r(b,t,\theta,\eta)&=\Big(R^2(b,t,\theta)R^2(b,t,\eta)-2R(b,t,\theta)R(b,t,\eta)\cos(\eta-\theta)+1\Big)^{\frac{1}{2}}.
\end{align}
Taking $r=0$ in \eqref{formula},  \eqref{Br} and \eqref{definition of R}  gives
\begin{equation}\label{A0 R0}
	A_{0}(b,t,\theta,\eta)=2b\left|\sin\left(\tfrac{\eta-\theta}{2}\right)\right|,\quad B_0(b,t,\theta,\eta)=|1-b^2e^{\ii(\eta-\theta)}|\quad\textnormal{and}\quad R(b,t,\theta)=b.
\end{equation}
According to \eqref{Vr}, \eqref{mathbfLr} and \eqref{mathbfSr} we obtain, after straightforward simplifications using \eqref{A0 R0}, 
	\begin{align*}
		V_{0}(b,t,\theta)&=-\tfrac{1}{2}-\tfrac{1}{2}\int_{\mathbb{T}}\log\left(4b^{2}\sin^{2}\left(\tfrac{\eta}{2}\right)\right)\cos(\eta)d\eta-\frac{1}{b^{2}}\int_{\mathbb{T}}\log\big(|1-b^2e^{\ii\eta}|\big)\cos(\eta)d\eta,\\
		\mathbf{L}_0(\rho)(b,t,\theta)&=\int_{\mathbb{T}}\log\left(2b\left|\sin\left(\tfrac{\eta-\theta}{2}\right)\right|\right)\rho(t,\eta)d\eta,\\
		\mathbf{S}_0(\rho)(b,t,\theta)&=\int_{\mathbb{T}}\log\left(\big|1-b^2e^{\ii(\eta-\theta)}\big|\right)\rho(t,\eta)d\eta.
	\end{align*}
We then see that $\mathbf{L}_0$ and $\mathbf{S}_0$ are convolution operators given by
\begin{align*}
\mathbf{L}_0&=\mathcal{K}_{1,b}\ast\cdot\quad\textnormal{with}\quad \mathcal{K}_{1,b}(\theta):=\tfrac{1}{2}\log\left(\sin^{2}\left(\tfrac{\theta}{2}\right)\right),\\
\mathbf{S}_0&=\mathcal{K}_{2,b}\ast\cdot\quad\textnormal{with}\quad \mathcal{K}_{2,b}(\theta):=\log\big(|1-b^2e^{\ii\theta}|\big).
\end{align*}
\textbf{2.}  To describe the operators above, it suffices to look for their actions on the Fourier basis $(e_j)_{j\in\mathbb{Z}^{*}}$ of $L^2_0(\mathbb{T})$. 
 We first study the operator $\mathbf{L}_0.$ Recall the following formula which can be found in \cite[Lem.~A.3]{CCG16}
\begin{equation}\label{diego formula}
	\forall j\in\mathbb{Z}^{*},\quad \int_{\mathbb{T}}\log\left(\sin^{2}\left(\tfrac{\eta}{2}\right)\right)\cos(j\eta)d\eta=-\frac{1}{|j|}.
\end{equation}
Using \eqref{diego formula} together with symmetry arguments, one obtains
\begin{align}
	\forall j\in\mathbb{Z}^*,\quad\mathcal{K}_{1,b}\ast e_{j}(\theta)&=\tfrac{1}{2}\int_{\mathbb{T}}\log\left(\sin^{2}\left(\tfrac{\eta}{2}\right)\right)e^{\ii j(\theta-\eta)}d\eta\nonumber\\
	&=\frac{e_{j}(\theta)}{2}\int_{\mathbb{T}}\log\left(\sin^{2}\left(\tfrac{\eta}{2}\right)\right)\cos(j\eta)d\eta\\
	&=-\frac{e_{j}(\theta)}{2|j|}.\label{fourierk1b}
\end{align}
We now turn to the study of the operator $\mathbf{S}_0.$ Using the following identity proved in \cite[Lem. 3.2]{R21}
\begin{equation}\label{emeric formula}
	\forall j\in\mathbb{Z}^{*},\quad \int_{\mathbb{T}}\log\left(|1-b^2e^{\ii\eta}|\right)\cos(j\eta)d\eta=-\frac{b^{2|j|}}{2|j|},
\end{equation} 
we obtain 
\begin{align}
	\forall j\in\mathbb{Z}^*,\quad\mathcal{K}_{2,b}\ast e_j(\theta)
	&=e_{j}(\theta)\int_{\mathbb{T}}\log\left(|1-b^2e^{\ii\eta}|\right)\cos(j\eta)d\eta\nonumber\\
	&=-\frac{b^{2|j|}e_{j}(\theta)}{2|j|}.\label{fourierk2b}
\end{align}
In view of the expression of $V_{0}$ and using formulae \eqref{diego formula} and \eqref{emeric formula}  we find
\begin{equation}\label{V0}
	V_{0}(b,t,\theta)=\frac{1}{2}.
\end{equation}
Notice that, the kernels $\mathcal{K}_{1,b}$ and $\mathcal{K}_{2,b}$ being even, the operator $\mathrm{L}(b)$ is self-adjoint. The identities in \eqref{Ham-Fourier} follows immediately from \eqref{formula Lb},  \eqref{fourierk1b}, \eqref{fourierk2b} and  \eqref{V0}.
Then, according to \eqref{Ham-Fourier},  a real function $\rho$ with  Fourier representation $\rho(t,\theta)=\displaystyle\sum_{j\in\mathbb{Z}^*}\rho(t)e^{\ii j\theta}$   is a solution to \eqref{Ham eq 0} if and only if
$$\forall j\in\mathbb{Z}^*,\quad \dot{\rho_{j}}=-\ii\,\Omega_{j}(b)\rho_{j},$$
where $\Omega_{j}(b)$ is defined by \eqref{Omegajb}.
Solving the previous ODE gives
\begin{equation*}
\rho(t,\theta)=\sum_{j\in\mathbb{Z}^*}\rho_{j}(0)e^{\ii(j\theta-\Omega_{j}(b)t)}.
\end{equation*}
Therefore, every real-valued reversible solution to \eqref{Ham eq 0} has the form \eqref{solution of the linear system}. This ends the proof of Lemma \ref{lemma linearized operator at equilibrium}.
\end{proof}
\subsection{Properties  of the equilibrium frequencies}
The goal of this section is to explore some important properties of the equilibrium frequencies.   We shall first show some bounds on these frequencies then discuss their non-degeneracy  through the transversality  conditions. Such conditions are crucial in the measure estimates of the final Cantor set giving rise to quasi-periodic solutions for the linear and the nonlinear problems. 
\begin{lem}\label{properties omegajb}
	\begin{enumerate}[label=(\roman*)]
		\item For all $b\in(0,1),$ the sequence $\big(\frac{\Omega_{j}(b)}{j}\big)_{j\in\mathbb{N}^{*}}$ is strictly increasing.
		\item For all $j\in\mathbb{Z}^{*},$ we have
		$$\forall\,  0<b_0\leqslant b< 1,\quad |\Omega_{j}(b)|\geqslant\frac{b_{0}^{2}}{2}|j|.$$
		\item For all $j,j'\in\mathbb{Z}^{*}$, we have 
		$$\forall\, 0<b_0\leqslant b<1,\quad |\Omega_{j}(b)\pm\Omega_{j'}(b)|\geqslant \frac{b_{0}^{2}}{6}{|j\pm j'|}.$$
		\item Given $0<b_0<b_1<1$ and $q_0\in\mathbb{N},$ there exists $C_0>0$ such that
		$$\forall j,j'\in\mathbb{Z}^{*},\quad \max_{q\in\llbracket 0,q_0\rrbracket}\sup_{b\in[b_0,b_1]}\left|\partial_{b}^{q}\big(\Omega_{j}(b)-\Omega_{j'}(b)\big)\right|\leqslant C_0|j-j'|.$$
	\end{enumerate}
\end{lem}
\begin{proof}
	\textbf{(i)} This point was proved in the proof of \cite[Prop. 2]{HHHM15}.\\
	\textbf{(ii)} By symmetry \eqref{Omegajb}, it suffices to show the inequality for $j\in\mathbb{N}^{*}.$ From $(i)$ we have
	$$\frac{\Omega_{j}(b)}{j}\geqslant \Omega_{1}(b)=\frac{b^{2}}{2}\geqslant\frac{b_0^{2}}{2}\cdot$$ 
	\textbf{(iii)} In view of the symmetry \eqref{Omegajb}, it suffices to check the property for $j,j'\in\mathbb{N}^*.$ By symmetry in $j,j'$  we may assume that $j\geqslant j'.$ For $j=j'=1$ one has
	$$\Omega_{1}(b)+\Omega_{1}(b)=b^{2}\geqslant b_0^2\cdot$$
	In the case where $j\geqslant 2$ and $j'\geqslant 1$ we get
	$$\Omega_{j}(b)+\Omega_{j'}(b)=\frac{j+j'-2}{2}+\frac{b^{2j}+b^{2j'}}{2}\geqslant (j+ j')\frac{j+ j'-2}{2(j+ j')}\geqslant\frac{j+ j'}{6}\cdot$$
	Now we shall move to the difference. Using Taylor formula we obtain, for all $j> j'\geqslant 1$,
	\begin{align*}
	\Omega_{j}(b)-\Omega_{j'}(b)&=\frac{j- j'}{2}+\frac{b^{2j}-b^{2j'}}{2}\\
	&=\frac{j- j'}{2}+\log (b)\int_{j'}^{j}b^{2x}dx\\
	&\geqslant\frac{j- j'}{2}(1+2\log(b)b^{2j})\geqslant\frac{j- j'}{4}\cdot
	\end{align*}
		\textbf{(iv)} The case $j=j'$ is trivial, then from the symmetry \eqref{Omegajb} and without loss of generality we shall assume that $j>j'\geqslant 1.$ First, remark that
	$$\forall b\in(0,1),\quad|\Omega_{j}(b)\pm\Omega_{j'}(b)|\leqslant\frac{(j-1)\pm (j'-1)}{2}+\frac{b^{2j'}\pm b^{2j}}{2}\leqslant j\pm j'.$$
	Now, for  all $q\in\mathbb{N}^*,$ one has
	$$\partial_{b}^{q}\Big(\Omega_{j}(b)\pm \Omega_{j'}(b)\Big)=\frac{1}{2}\partial_{b}^{q}\Big(b^{2j}\pm b^{2j'}\Big).$$
	Moreover, for all $q\in\llbracket 1,q_0\rrbracket$ and $n\in\mathbb{N}^*,$ 
	\begin{align*}
		0\leqslant\partial_{b}^{q}(b^n)\leqslant q!\binom{n}{q}b^{n-q}\leqslant\frac{\,n^{q_0}b_1^{n}}{b_0^{q_0}}\cdot
	\end{align*}
Since $b_{1}\in(0,1)$ then the sequence $(n^{q_0}b_{1}^{n})_{n\in\mathbb{N}}$  is bounded. Therefore, there exists $C_0:=C_0(q_0,b_0,b_1)>0$ such that
\begin{align}\label{unif bnd dbn}
	\forall n\in\mathbb{N},\quad 0\leqslant\partial_{b}^{q}(b^n)\leqslant C_0.
\end{align} 
We deduce that for all $q\in\llbracket 1,q_0\rrbracket$,
$$\left|\partial_{b}^{q}\Big(\Omega_{j}(b)\pm\Omega_{j'}(b)\Big)\right|\leqslant C_0\leqslant C_0(j\pm j').$$
This concludes the proof of Lemma \ref{properties omegajb}.
\end{proof}

Let us consider finitely many Fourier modes, called tangential sites, gathered in the tangential  set $\mathbb{S}$ defined by 
\begin{equation}\label{def bbS}
	\quad{\mathbb S} := \{ j_1, \ldots, j_d \}\subset\mathbb{N}^* \quad\textnormal{with}\quad 
	1 \leqslant j_1 < j_2 < \ldots < j_d.
\end{equation}
Now, we define the equilibrium frequency vector by
\begin{equation}\label{def freq vec eqb}
	\omega_{\textnormal{Eq}}(b)=(\Omega_{j}(b))_{j\in\mathbb{S}},
\end{equation}
where $\Omega_{j}(b)$ is defined by \eqref{Omegajb}. We shall now investigate the  non-degeneracy and the transversality  properties satisfied by $\omega_{\textnormal{Eq}}.$ Let us first start with the non-degeneracy, for which we recall the definition.

\begin{defin}\label{def-degenerate} 
	Given two numbers $b_0<b_1$ and $d\in\mathbb{N}^*$. A vector-valued function $f = (f_1, ..., f_d ) : [b_0,b_1] \to \mathbb{R}^d$ is called non-degenerate if, for any vector $c = (c_1,...,c_d) \in  \mathbb{R}^d \setminus \{0\}$, the function $f \cdot c = f_1c_1 + ... + f_dc_d$ is not identically zero on the whole interval $[b_0,b_1]$. In other words,  the curve of $f$ is not contained in an hyperplane.
\end{defin}
We have the following result.
\begin{lem}\label{Non-degeneracy2}
	The equilibrium frequency vector $\omega_{\textnormal{Eq}}$ and the vector-valued function $(\omega_{\textnormal{Eq}},1)$ are non-degenerate on $[b_0,b_1]$ in the sense of Definition \ref{def-degenerate}.
\end{lem}
\begin{proof}
	 $\blacktriangleright$ We shall first prove that the equilibrium frequency vector $\omega_{\textnormal{Eq}}$ is non-degenerate on $[b_0,b_1]$. Arguing  by contradiction, suppose that there exists $c:=(c_1,\ldots,c_d)\in \mathbb{R}^d\backslash\{0\}$ such that
	\begin{equation}\label{rel Omgjk}
		\forall b\in[b_0,b_1],\quad \sum_{k=1}^{d}c_{k}\Omega_{j_k}(b)=0.
	\end{equation}
	Since $\Omega_{j}(b)$ is polynomial in $b$ then, from \eqref{Omegajb},  one has
	\begin{equation}\label{rel Omgjk-1}
		\forall b\in\mathbb{R},\quad \sum_{k=1}^{d}c_{k}(j_{k}-1+b^{2j_k})=0.
	\end{equation}
Taking the limit $b\rightarrow0$ in \eqref{rel Omgjk-1} gives the relation
$\displaystyle\sum_{k=1}^{d}c_{k}(j_{k}-1)=0,$
which, inserted into \eqref{rel Omgjk-1}, implies
$$\forall b\in\mathbb{R},\quad\sum_{k=1}^{d}c_{k}b^{2j_{k}}=0.$$
Since $j_{1}<j_{2}<\ldots<j_d$, then 
$$\forall k\in\llbracket 1,d\rrbracket,\quad c_k=0,$$
which contradicts the assumption.

\noindent $\blacktriangleright$ Next, we shall check that  the  function $(\omega_{\textnormal{Eq}},1)$ is non-degenerate on $[b_0,b_1]$. 
 Suppose, by contradiction,  that there exists $c:=(c_1,\ldots,c_d,c_{d+1})\in \mathbb{R}^{d+1}\backslash\{0\}$ such that
\begin{equation}\label{rel Omgjk-bis}
	\forall b\in[b_0,b_1],\quad c_{d+1}+\sum_{k=1}^{d}c_{k}\Omega_{j_k}(b)=0.
\end{equation}
Since $\Omega_{j}(b)$ is polynomial in $b$ then, from \eqref{Omegajb},  one may writes 
\begin{equation}\label{rel Omgjk-2}
	\forall b\in\mathbb{R},\quad c_{d+1}+\tfrac{1}{2}\sum_{k=1}^{d}c_{k}(j_{k}-1+b^{2j_k})=0.
\end{equation}
Taking the limit $b\rightarrow0$ in \eqref{rel Omgjk-2} yields
$$c_{d+1}+\tfrac{1}{2}\sum_{k=1}^{d}c_{k}(j_{k}-1)=0.$$
Inserting this relation into \eqref{rel Omgjk-2} gives
$$\forall b\in\mathbb{R},\quad \sum_{k=1}^{d}c_{k}b^{2j_{k}}=0.$$
Reasoning as in the previous point, we obtain 
$$\forall k\in\llbracket 1,d\rrbracket,\,c_{k}=0$$
and then $c_{d+1}=0$, by coming back to \eqref{rel Omgjk-2}, contradicting the assumption.
\end{proof}
We shall now state the transversality conditions satisfied by the unperturbed frequencies.
\begin{lem}{\textnormal{[Transversality]}}\label{lemma transversalityb}
	{Let $0<b_0<b_1<1.$ Set $q_0=2j_d+2.$ Then, there exists $\rho_{0}>0$ such that the following results hold true. Recall that $\omega_{\textnormal{Eq}}$ and $\Omega_j$ are defined in \eqref{def freq vec eqb} and \eqref{Omegajb}, respectively.
		\begin{enumerate}[label=(\roman*)]
			\item For all $l\in\mathbb{Z}^{d }\setminus\{0\},$ we have
			$$
			\inf_{b\in[b_{0},b_{1}]}\max_{q\in\llbracket 0, q_{0}\rrbracket}|\partial_{b}^{q}\omega_{\textnormal{Eq}}(b)\cdot l|\geqslant\rho_{0}\langle l\rangle.
			$$
			\item  For all $ (l,j)\in\mathbb{Z}^{d }\times (\mathbb{N}^*\setminus\mathbb{S})$ 
			$$
			\quad\inf_{b\in[b_{0},b_{1}]}\max_{q\in\llbracket 0, q_{0}\rrbracket}\big|\partial_{b}^{q}\big(\omega_{\textnormal{Eq}}(b)\cdot l\pm\tfrac{j}{2}\big)\big|\geqslant\rho_{0}\langle l\rangle.
			$$
			\item  For all $ (l,j)\in\mathbb{Z}^{d }\times (\mathbb{N}^*\setminus\mathbb{S})$ 
			$$
			\quad\inf_{b\in[b_{0},b_{1}]}\max_{q\in\llbracket 0, q_{0}\rrbracket}\big|\partial_{b}^{q}\big(\omega_{\textnormal{Eq}}(b)\cdot l\pm\Omega_{j}(b)\big)\big|\geqslant\rho_{0}\langle l\rangle.
			$$
			\item For all $ l\in\mathbb{Z}^{d }, j,j^\prime\in\mathbb{N}^*\setminus\mathbb{S}$  with $(l,j)\neq(0,j^\prime),$ we have
			$$\,\quad\inf_{b\in[b_{0},b_{1}]}\max_{q\in\llbracket 0, q_{0}\rrbracket}\big|\partial_{b}^{q}\big(\omega_{\textnormal{Eq}}(b)\cdot l+\Omega_{j}(b)\pm\Omega_{j^\prime}(b)\big)\big|\geqslant\rho_{0}\langle l\rangle.$$	
	\end{enumerate}}
\end{lem}
\begin{proof}
	\textbf{(i)} Assume by contradiction that for all $\rho_{0}>0$, there exist $l\in\mathbb{Z}^{d }\setminus\{0\}$ and $b\in[b_{0},b_{1}]$ such that 
	$$
	\max_{q\in\llbracket 0, q_{0}\rrbracket}|\partial_{b}^{q}\omega_{\textnormal{Eq}}(b)\cdot l|<\rho_{0}\langle l\rangle.
	$$
	In particular, for the choice $\rho_{0}=\frac{1}{m+1}$, we can construct sequences $l_{m}\in\mathbb{Z}^{d}\setminus\{0\}$ and $b_{m}\in[b_{0},b_{1}]$ such that 
	\begin{equation}\label{Rossemann 0b}
		\forall q\in\llbracket 0, q_0\rrbracket,\quad\big|\partial_{b}^{q}\omega_{\textnormal{Eq}}(b_{m})\cdot \tfrac{l_{m}}{\langle l_m\rangle}\big|<\tfrac{1}{m+1}\cdot
	\end{equation}
	Since the sequences $\left(\frac{l_{m}}{\langle l_m\rangle}\right)_{m}$ and $(b_{m})_{m}$ are bounded, then  by compactness arguments and, up to an extraction, we can assume that
	$$\lim_{m\to\infty}\tfrac{l_{m}}{\langle l_m\rangle}=\bar{c}\neq 0\quad\hbox{and}\quad \lim_{m\to\infty}b_{m}=\bar{b}.
	$$
	Therefore, denoting
	$$P_0:=\omega_{\textnormal{Eq}}(X)\cdot\bar{c}\,\in\mathbb{R}_{2j_d}[X]$$
	then passing to the limit  in \eqref{Rossemann 0b}  as  $m\rightarrow\infty$ leads to
	$$\forall q\in\llbracket 0, q_0\rrbracket,\quad P_0^{(q)}(\bar{b})=0.$$
	Hence, using the particular choice of $q_0$, we conclude that the polynomial $(X-\bar{b})^{2j_d+3}$ divides $P_0$,  
	$$(X-\bar{b})^{2j_d+3}|P_0.$$
	Since $\deg(P_0)\leqslant 2j_d,$ we conclude that $P_0$ is identically zero. This contradicts the non-degeneracy of the equilibrium frequency vector $\omega_{\textnormal{Eq}}$ stated in Lemma \ref{Non-degeneracy2}.
	\\
	\textbf{(ii)}  
The case $l=0, j\in\mathbb{N}^*$ is trivially satisfied. Thus, we shall consider the case  $j\in\mathbb{N}$, $l\in\mathbb{Z}^d\setminus\{0\}$.
By the triangle inequality combined with the boundedness of  ${\omega}_{\textnormal{Eq}}$  we find
$$\big|{\omega}_{\textnormal{Eq}}(b)\cdot l+\tfrac{j}{2}\big|\geqslant\tfrac12|j|-|{\omega}_{\textnormal{Eq}}(b)\cdot l|\geqslant \tfrac12|j|-C|l|\geqslant |l|$$
provided that  $|j|\geqslant C_{0}|l|$ for some $C_{0}>0.$ Thus, we shall restrict the proof to  indices  $j$ and $l$ with
\begin{equation}\label{parameter condition 10}
 |j|\leqslant C_{0}|l|, \quad j\in\mathbb{N}, \quad l\in\mathbb{Z}^d\setminus\{0\}.
\end{equation}
Arguing by contradiction as in the previous case, we may assume the existence of sequences $l_{m}\in \mathbb{Z}^d\backslash\{0\}$, $j_m \in \mathbb{N}$ satisfying \eqref{parameter condition 10} and $b_{m}\in[b_0,b_1]$ such that 
\begin{equation}\label{Rossemann 00b}
\forall q\in\llbracket 0, q_0\rrbracket,\quad \left|\partial_{b}^{q}\left({\omega}_{\textnormal{Eq}}(b_{m})\cdot\tfrac{l_{m}}{\langle l_{m}\rangle}+\tfrac{j_{m}}{2\langle l_{m}\rangle}\right)\right|<\tfrac{1}{1+m}.
\end{equation}
Since the sequences $(b_m)_m$, $\big(\frac{j_m}{2\langle l_m\rangle}\big)$  and $ \big(\frac{l_m}{\langle l_m\rangle}\big)$ are bounded, then up to an extraction we can assume that
$$
\lim_{m\to\infty}b_{m}=\bar{b},\quad \lim_{m\to\infty}\tfrac{j_m}{2\langle l_m\rangle}=\bar{d}\neq 0\quad\hbox{and}\quad \lim_{m\to\infty}\tfrac{l_m}{\langle l_m\rangle}=\bar{c}\neq 0.
$$
Denoting
	$$Q_{0}:=\omega_{\textnormal{Eq}}(X)\cdot\bar{c}+\bar{d}\,\in\mathbb{R}_{2j_d}[X]$$
	and letting  $m\rightarrow\infty$ in  \eqref{Rossemann 00b}  we obtain
	$$\forall q\in\llbracket 0, q_0\rrbracket,\quad Q_0^{(q)}(\bar{b})=0.$$
	Consequently, using the particular choice of $q_0$, we get
	$$(X-\bar{b})^{2j_d+3}|Q_0.$$
	Since $\deg(Q_0)\leqslant 2j_d,$ we conclude that $Q_0$ is identically zero. This contradicts Lemma  \ref{Non-degeneracy2}.
	\\
	\textbf{(iii)} Consider $(l,j)\in\mathbb{Z}^{d }\times (\mathbb{N}^*\setminus\mathbb{S})$. Then applying  the  triangle inequality and Lemma \ref{properties omegajb}-(ii), yields
	\begin{align*}
		|\omega_{\textnormal{Eq}}(b)\cdot l\pm\Omega_{j}(b)|&\geqslant|\Omega_{j}(b)|-|\omega_{\textnormal{Eq}}(b)\cdot l|\\
		&\geqslant \tfrac{b_0^2}{2}j-C|l|\geqslant \langle l\rangle
	\end{align*}
	provided $j\geqslant C_{0} \langle l\rangle$ for some $C_{0}>0.$ Thus  as before we shall restrict the proof to indices $j$ and $l$ with 
	\begin{equation}\label{parameter condition 1b}
		0\leqslant j< C_{0} \langle l\rangle,\quad j\in\mathbb{N}^*\setminus\mathbb{S}\quad\hbox{and}\quad l\in\mathbb{Z}^d\backslash\{0\}.
	\end{equation}
	Proceeding   by contradiction, we may assume the existence of sequences  $l_{m}\in\mathbb{Z}^{d }\setminus\{0\}$, $j_{m}\in\mathbb{N}\setminus\mathbb{S}$ satisfying \eqref{parameter condition 1b} and $b_{m}\in[b_{0},b_{1}]$ such that 
	\begin{equation}\label{Rossemann 1b}
		\forall q\in\llbracket 0, q_0\rrbracket,\quad \left|\partial_{b}^{q}\left(\omega_{\textnormal{Eq}}(b)\cdot\tfrac{l_{m}}{| l_{m}|}\pm\tfrac{\Omega_{j_{m}}(b)}{| l_{m}|}\right)_{|b=b_m}\right|<\tfrac{1}{m+1}\cdot
	\end{equation}
	Since  the sequences $\left(\tfrac{l_{m}}{|l_{m}|}\right)_{m}$ and $(b_{m})_{m}$ are bounded, then up to an extraction we can assume  that 
	$$\lim_{m\to\infty}\tfrac{l_{m}}{| l_{m}|}=\bar{c}\neq 0\quad\hbox{and}\quad \lim_{m\to\infty}b_{m}=\bar{b}.
	$$
	Now we shall distinguish  two cases.\\
	$\blacktriangleright$ Case \ding{182} : $(l_{m})_{m}$ is bounded. In this case, by \eqref{parameter condition 1b} we  find  that $(j_{m})_{m}$ is bounded too and thus up to to an extraction we may assume  $ \displaystyle \lim_{m\to\infty}l_{m}=\bar{l}$ and $  \displaystyle \lim_{m\to\infty}j_{m}=\bar{\jmath}.$
	Since $(j_{m})_{m}$ and $(|l_{m}|)_{m}$ are sequences of integers, then they are necessary stationary. In particular, the condition \eqref{parameter condition 1b} implies $\bar{l}\neq 0$ and $\bar{\jmath}\in\mathbb{N}\setminus\mathbb{S}.$ Hence, denoting 
	$$P_{0,\bar{\jmath}}:=\omega_{\textnormal{Eq}}(X)\cdot\bar{l}\pm\Omega_{\bar{\jmath}}(X)\in\mathbb{R}_{\max(2j_d,2\bar{\jmath})}[X],$$ 
	then taking the limit $m\rightarrow\infty$ in \eqref{Rossemann 1b}, yields
	$$\forall q\in\llbracket 0,q_0\rrbracket,\quad P_{0,\bar{\jmath}}^{(q)}(\bar{b})=0.$$
	If $\bar{\jmath}<  j_d,$ then in a similar way to the point (i), we find that $P_{0,\bar{\jmath}}=0$ which contradicts Lemma \ref{Non-degeneracy2}, applied with $\big(\omega_{\textnormal{Eq}},\Omega_{\bar{\jmath}}\big)$ in place of $\omega_{\textnormal{Eq}}.$ Hence, we shall restrict the discussion to the case $\bar{\jmath}>j_d.$ Since $\omega_{\textnormal{Eq}}(X)\cdot\bar{l}$ is of degree $2j_d,$ then we obtain in view of our choice of $q_0$ that
	$$\tfrac{1}{2}q!\binom{2\bar{\jmath}}{q}b^{2\bar{\jmath}-2j_d-1}=\partial_{b}^{2j_d+1}\Omega_{\bar{\jmath}}(\bar{b})=0.$$ 
	This implies that $\bar{b}=0$ which contradicts the fact that $\bar{b}\in[b_0,b_1]\subset(0,1).$
	\\
	$\blacktriangleright$ Case \ding{183} : $(l_{m})_{m}$ is unbounded. Up to an extraction  we can assume that $\displaystyle \lim_{m\to\infty}|l_{m}|=\infty.$
	We have two sub-cases.\\
	$\bullet$ Sub-case \ding{172} : $(j_{m})_{m}$ is bounded. In this case and up to an extraction   we can assume that it converges. Then, taking the limit $m\rightarrow\infty$ in \eqref{Rossemann 1b}, we find
	$$\forall q\in\llbracket 0,q_0\rrbracket,\quad \partial_{b}^{q}\omega_{\textnormal{Eq}}(\bar{b})\cdot\bar{c}=0.$$
	Therefore, we obtain a contradiction in a similar way to the point (i).\\
	$\bullet$ Sub-case \ding{173} : $(j_{m})_{m}$ is unbounded. Then  up to an extraction we can assume that $  \displaystyle \lim_{m\to\infty}j_{m}=\infty$. We write  according to \eqref{Omegajb}
	\begin{align}\label{FHK1b}
		\frac{\Omega_{j_{m}}(b)}{| l_{m}|}=\frac{j_{m}}{2| l_{m}|}-\frac{1}{2|l_m|}+\frac{b^{2j_m}}{2|l_m|}.
	\end{align}
	By \eqref{parameter condition 1b}, the sequence $\left(\frac{j_{m}}{2| l_{m}|}\right)_{n}$ is bounded, thus  up to an extraction  we can assume that it converges to $\bar{d}.$ Moreover, since $\displaystyle \lim_{m\to\infty}j_{m}=\displaystyle \lim_{m\to\infty}|l_{m}|=\infty$ and $b_m\in(b_0,b_1)$, then taking the limit in \eqref{FHK1b}, one obtains from \eqref{unif bnd dbn},
	$$\lim_{m\to\infty}\tfrac{\partial_b^q\Omega_{j_{m}}(b)_{|b=b_m}}{| l_{m}|}=\left\lbrace\begin{array}{ll}
		\bar{d} & \textnormal{if }q=0\\
		0 & \textnormal{else.}	
		\end{array}\right.$$
	Consequently, taking the limit $m\rightarrow\infty$ in \eqref{Rossemann 1b}, we have 
	$$\forall q\in\llbracket 0,q_0\rrbracket,\quad \partial_{b}^{q}\left(\omega_{\textnormal{Eq}}({b})\cdot\bar{c}\pm \bar{d}\right)_{|b=\bar{b}}=0.$$
	Then, in a similar way the the point (ii), we deduce that the polynomial $\omega_{\textnormal{Eq}}(X)\cdot\bar{c}+\bar{d}$ is identically zero, which is in contradiction with Lemma \ref{Non-degeneracy2}.\\
	\textbf{(iv)} Consider $ l\in\mathbb{Z}^{d }, j,j^\prime\in\mathbb{N}^*\setminus\mathbb{S}$  with $(l,j)\neq(0,j^\prime).$ Then applying the  triangle inequality combined with Lemma \ref{properties omegajb}-(iii), we infer that
	$$|\omega_{\textnormal{Eq}}(b)\cdot l+\Omega_{j}(b)\pm\Omega_{j'}(b)|\geqslant|\Omega_{j}(b)\pm\Omega_{j'}(b)|-|\omega_{\textnormal{Eq}}(b)\cdot l|\geqslant \tfrac{b_0^2}{6}|j\pm j'|-C|l|\geqslant \langle l\rangle$$
	provided that $|j\pm j'|\geqslant c_{0}\langle l\rangle$ for some $c_{0}>0.$ Then it remains to check the proof for  indices  satisfying 
	\begin{equation}\label{parameter condition 2b}
		|j\pm j'|< c_{0}\langle l\rangle,\quad  l\in\mathbb{Z}^{d }\backslash\{0\},\quad  j,j^\prime\in\mathbb{N}^*\setminus\mathbb{S}.
	\end{equation}
	Reasoning by contradiction as in the previous cases, we get   for all $m\in\mathbb{N}$, real numbers  $l_{m}\in\mathbb{Z}^{d }\setminus\{0\}$, $j_{m},j'_{m}\in\mathbb{N}^*\setminus\mathbb{S}$ satisfying \eqref{parameter condition 2b} and $b_{m}\in[b_{0},b_{1}]$ such that 
	\begin{equation}\label{Rossemann 2b}
		\forall q\in\llbracket 0,q_0\rrbracket,\quad \left|\partial_{b}^{q}\left(\omega_{\textnormal{Eq}}(b)\cdot\tfrac{l_{m}}{| l_{m}|}+\tfrac{\Omega_{j_{m}}(b)\pm \Omega_{j'_{m}}(b)}{| l_{m}|}\right)_{|b=b_m}\right|<\tfrac{1}{m+1}\cdot
	\end{equation}
	Up to an extraction we can assume that $\displaystyle \lim_{m\to\infty}\tfrac{l_{m}}{| l_{m}|}=\bar{c}\neq 0$ and $\displaystyle  \lim_{m\to\infty}b_{m}=\bar{b}.$ \\
	As before we shall distinguish two cases.\\
	$\blacktriangleright$ Case \ding{182} : $(l_{m})_{m}$ is bounded. Up to an extraction we may assume  that  $\displaystyle \lim_{m\to\infty}l_{m}=\bar{l}\neq 0.$ Now according  to  \eqref{parameter condition 2b} we have two sub-cases to discuss depending whether the sequences $(j_{m})_{m}$ and $(j'_{m})_{m}$ are simultaneously bounded  or unbounded.\\
	$\bullet$ Sub-case \ding{172} : $(j_{m})_{m}$ and $(j'_{m})_{m}$ are bounded. In this case, up to a extraction we may assume that these sequences are stationary  $j_{m}=\bar{\jmath}$ and $j'_{m}=\bar{\jmath}'$ with $ \bar{\jmath},\bar{\jmath}'\in\mathbb{N}^*\setminus\mathbb{S}.$
	Hence, denoting
	$$P_{0,\bar{\jmath},\bar{\jmath}'}:=\omega_{\textnormal{Eq}}(X)\cdot\bar{l}+\Omega_{\bar{\jmath}}(X)\pm \Omega_{\bar{\jmath}'}(X)\in\mathbb{R}_{\max(2j_d,2\bar{\jmath},2\bar{\jmath}')}[X],$$ 
	then, taking the limit $m\rightarrow\infty$ in \eqref{Rossemann 1b}, we have 
	$$\forall q\in\llbracket 0,q_0\rrbracket,\quad P_{0,\bar{\jmath},\bar{\jmath}'}^{(q)}(\bar{b})=0.$$
	If $\max(\bar{\jmath},\bar{\jmath}')< j_d,$ then, we deduce that $P_{0,\bar{\jmath},\bar{\jmath}'}=0$ which gives a contradiction as the previous cases, up to replacing $\omega_{\textnormal{Eq}}$ by $\big(\omega_{\textnormal{Eq}},\Omega_{\bar{\jmath}},\Omega_{\bar{\jmath}'}\big).$ Therefore, we are left to study the case $\max(\bar{\jmath},\bar{\jmath}')>j_d.$ Notice that the cases $\bar{\jmath}=\bar{\jmath}'$ and $\min(\bar{\jmath},\bar{\jmath}')>j_d$ are byproducts of point (i) and (iii). Without loss of generality, we may assume that $\bar{\jmath}>\bar{\jmath}'\geqslant j_d+1.$ In particular, since $\omega_{\textnormal{Eq}}(X)\cdot\bar{l}$ is of degree $2j_d$, then, according to our choice of $q_0$, we obtain
	\begin{equation}\label{syst b}
		\left\lbrace\begin{array}{l}
			C_1\bar{b}^{\alpha}\pm C_2\bar{b}^{\beta}=0\\
			C_1\alpha\bar{b}^{\alpha}\pm C_2\beta\bar{b}^{\beta}=0,
		\end{array}\right.
	\end{equation}
	with
	$$\alpha:=2\bar{\jmath}-2j_d-1,\quad\beta:=2\bar{\jmath}'-2j_d-1, \quad C_1:=q_0!\binom{2\bar{\jmath}}{q_0}\quad\textnormal{and}\quad C_2:=q_0!\binom{2\bar{\jmath}'}{q_0}.$$
	Since $C_1$ and $C_2$ are positive, we immediately get from the first equation in \eqref{syst b} that 
	$$C_1\bar{b}^{\alpha}+ C_2\bar{b}^{\beta}=0\quad\Rightarrow\quad \bar{b}=0.$$
	This  contradicts the fact that $\bar{b}\in[b_0,b_1]\subset(0,1).$ In the case where we have the difference, the system \eqref{syst b} gives
	$$\frac{C_2}{C_1}=\frac{C_2\beta}{C_1\alpha},$$
	which implies in turn that $\alpha=\beta,$ that is $\bar{\jmath}=\bar{\jmath}'$ which is excluded by hypothesis.\\
	$\bullet$ Sub-case \ding{173} : $(j_{m})_{m}$ and $(j'_{m})_{m}$ are both unbounded and without loss of generality we can assume that $\displaystyle \lim_{m\to\infty}j_{m}= \lim_{m\to\infty}j'_{m}=\infty$. Coming back to  \eqref{Omegajb} we get the splitting 
	\begin{align}\label{Appla1}
		\nonumber\frac{\Omega_{j_m}(b)\pm\Omega_{j_{m}^\prime}(b)}{|l_m|}=&\frac{j_m\pm j^\prime_{m}}{2|l_m|}+\frac{b^{2j_m}\pm b^{2j'_m}}{2|l_m|}.
	\end{align}
Using once again \eqref{parameter condition 2b} and up to an extraction we have   $\displaystyle \lim_{m\to\infty}\tfrac{j_{m}\pm j'_{m}}{|l_m|}=\bar{d}.$
	Thus
	$$
	\lim_{m\to\infty}|l_m|^{-1}\partial_{b}^{q}\left({\Omega_{j_{m}}(b)\pm\Omega_{j'_{m}}(b)}\right)_{|b=b_m}=\left\lbrace\begin{array}{ll}
		\bar{d} & \textnormal{if }q=0\\
		0 & \textnormal{if }q\neq 0.
	\end{array}\right.$$
	By taking the limit as $m\rightarrow\infty$ in \eqref{Rossemann 2b}, we find				$$
	\forall q\in\llbracket 0,q_0\rrbracket,\quad \partial_{b}^{q}\left({\omega}_{\textnormal{Eq}}({b})\cdot\bar{c}+\bar{d}\right)_{|b=\overline b}=0.
	$$
	This leads to a contradiction as in the point (ii).\\
	$\blacktriangleright$  Case \ding{183} : $(l_{m})_{m}$ is unbounded. Up to an extraction  we can assume that $\displaystyle \lim_{m\to\infty}|l_{m}|=\infty.$\\
	We shall distinguish three sub-cases.\\
	$\bullet$ Sub-case \ding{172}. The sequences  $(j_{m})_{m}$ and $(j'_{m})_{m}$ are bounded. In this case and  up to an extraction  they will  converge and then taking the limit in \eqref{Rossemann 2b} yields,
	$$\forall q\in\llbracket 0,q_0\rrbracket,\quad\partial_{b}^{q}{\omega}_{\textnormal{Eq}}(\bar{b})\cdot\bar{c}=0.
	$$
	which leads to a contradiction as before. \\
	$\bullet$ Sub-case \ding{173}. The sequences  $(j_{m})_{m}$ and $(j'_{m})_{m}$ are both unbounded. This is similar to  the sub-case \ding{173} of the case \ding{182}.\\
	$\bullet$ Sub-case \ding{174}. The sequence $(j_{m})_{m}$ is unbounded and $(j'_{m})_{m}$ is bounded (the symmetric case is similar).  Without loss of generality we can assume that $\displaystyle \lim_{m\to\infty}j_m=\infty$ and $  j_{m}^\prime=\bar{\jmath}'.$ By \eqref{parameter condition 2b} and  up to an extraction one gets  $\displaystyle \lim_{m\to\infty}\tfrac{j_{m}\pm j'_{m}}{| l_{m}|}=\bar{d}.$ Once again, we have
	$$
	\lim_{m\to\infty}|l_m|^{-1}\partial_{b}^{q}\left({\Omega_{j_{m}}(b)\pm\Omega_{j'_{m}}(b)}\right)_{|b=b_m}=\left\lbrace\begin{array}{ll}
		\bar{d} & \textnormal{if }q=0\\
		0 & \textnormal{if }q\neq 0.
	\end{array}\right.$$
	Hence, taking the limit  in \eqref{Rossemann 2b} implies 
	$$\forall q\in\llbracket 0,q_0\rrbracket,\quad \partial_{b}^{q}\left({\omega}_{\textnormal{Eq}}({b})\cdot\bar{c}+\bar{d}\right)_{b=\overline b}=0,$$
	which also  gives a contradiction as the previous cases.
	This completes the proof of Lemma \ref{lemma transversalityb}.
\end{proof}
Notice that 
by selecting  only  a finite number of frequencies, the sum  in \eqref{solution of the linear system}  give rise to  quasi-periodic solutions of the linearized equation \eqref{Ham eq 0}, up to selecting the parameter $b$ in a Cantor-like set of full measure.
We have the following result.
\begin{prop}\label{lemma sol Eq}
	Let $0<b_0<b_1<1,$ $d\in\mathbb{N}^{*}$ and $\mathbb{S}\subset\mathbb{N}^*$ with $|\mathbb{S}|=d.$ Then, there exists a Cantor-like set $\mathcal{C}\subset[b_0,b_1]$ satisfying $|\mathcal{C}|=b_1-b_0$ and such that for all $\lambda\in\mathcal{C}$, every function in the form
	\begin{equation}\label{rev sol eq}
		\rho(t,\theta)=\sum_{j\in\mathbb{S}}\rho_{j}\cos(j\theta-\Omega_{j}(b)t),\quad\rho_{j}\in\mathbb{R}^*
	\end{equation} 
	is a time quasi-periodic reversible solution to the equation \eqref{Ham eq ED r} with the vector frequency 
	$$\omega_{\textnormal{Eq}}(b)=(\Omega_{j}(b))_{j\in\mathbb{S}}.$$
\end{prop}
The proof of this proposition follows in a similar way to \cite[Lem 3.3]{HHM21} or \cite[Prop. 3.1]{HR21} where the main ingredient 
is  the following  Rüssmann Lemma \cite[Thm. 17.1]{R01}. This latter is also needed to prove Proposition~\ref{lem-meas-es1} . 
\begin{lem}\label{lemma useful for measure estimates}
	Let $q_{0}\in\mathbb{N}^{*}$ and $\alpha,\beta\in\mathbb{R}_{+}.$ Let $f\in C^{q_{0}}([a,b],\mathbb{R})$ such that
	$$\inf_{x\in[a,b]}\max_{k\in\llbracket 0,q_{0}\rrbracket}|f^{(k)}(x)|\geqslant\beta.$$
	Then, there exists $C=C(a,b,q_{0},\| f\|_{C^{q_{0}}([a,b],\mathbb{R})})>0$ such that 
	$$\Big|\left\lbrace x\in[a,b]\quad\textnormal{s.t.}\quad |f(x)|\leqslant\alpha\right\rbrace\Big|\leqslant C\tfrac{\alpha^{\frac{1}{q_{0}}}}{\beta^{1+\frac{1}{q_{0}}}},$$
	where the notation  $|A|$  corresponds to the Lebesgue measure of a given measurable set $A.$
\end{lem}

\section{Topological and algebraic aspects for functions and operators}\label{sec funct set}
In this section we present  the general topological framework for both functions and operators classes. We also recall basic   definitions and gather some technical results that will be used along the paper.

\subsection{Function spaces}\label{Functionspsaces}
Recall  the classical complex Sobolev space $H^{s}(\mathbb{T}^{d +1},\mathbb{C})$ which is the set
of all complex periodic functions   with the Fourier expansion 
$$\rho=\sum_{(l,j)\in\mathbb{Z}^{d+1 }}\rho_{l,j}\,\mathbf{e}_{l,j}\quad \mbox{ where }\quad \rho_{l,j}=\big\langle\rho,\mathbf{e}_{l,j}\big\rangle_{L^{2}(\mathbb{T}^{d+1},\mathbb{C})}$$
such that
$$\| \rho\|_{H^{s}}^{2}:=\sum_{(l,j)\in\mathbb{Z}^{d+1 }}\langle l,j\rangle^{2s}|\rho_{l,j}|^{2}<\infty\quad\textnormal{where}\quad \langle l,j\rangle:=\max(1,|l|,|j|),$$
 with $|\cdot|$ denoting either the $\ell^{1}$ norm in $\mathbb{R}^{d }$ or the absolute value in $\mathbb{R}.$
The real Sobolev spaces can be viewed as closed sub-spaces of the preceding one, 
\begin{align*}
	\nonumber H^{s}=H^{s}(\mathbb{T}^{d+1},\mathbb{R})	&=\Big\{\rho\in H^{s}(\mathbb{T}^{d +1},\mathbb{C})\quad\textnormal{s.t.}\quad\forall \,(l,j)\in\mathbb{Z}^{d+1},\,\rho_{-l,-j}=\overline{\rho_{l,j}}\Big\}.
\end{align*}
For $N\in\mathbb{N}^{*},$ we define the cut-off frequency projectors on  $H^{s}(\mathbb{T}^{d +1},\mathbb{C})$ as follows 
\begin{equation}\label{definition of projections for functions}
	\Pi_{N}\rho=\sum_{\underset{\langle l,j\rangle\leqslant N}{(l,j)\in\mathbb{Z}^{d+1 }}}\rho_{l,j}\mathbf{e}_{l,j}\quad \mbox{ and }\quad \Pi^{\perp}_{N}=\textnormal{Id}-\Pi_{N}.
\end{equation}
We shall also make use of the following mixed weighted Sobolev spaces with respect to the parameter  $\gamma$. 
Let $\mathcal{O}$ be an open bounded set of $\mathbb{R}^{d+1}$ and define   
\begin{align*}
	W^{q,\infty,\gamma}(\mathcal{O},H^{s})&=\Big\lbrace \rho:\mathcal{O}\rightarrow H^{s}\quad\textnormal{s.t.}\quad\|\rho\|_{q,s}^{\gamma,\mathcal{O}}<\infty\Big\rbrace,\\
	W^{q,\infty,\gamma}(\mathcal{O},\mathbb{C})&=\Big\lbrace\rho:\mathcal{O}\rightarrow\mathbb{C}\quad\textnormal{s.t.}\quad\|\rho\|_{q}^{\gamma,\mathcal{O}}<\infty\Big\rbrace,
\end{align*}
where for $\mu\in\mathcal{O}\mapsto \rho(\mu)\in H^s$ or $\mu\in\mathcal{O}\mapsto \rho(\mu)\in \mathbb{C}$ the norm is defined by
\begin{align}\label{Norm-def}
	\|\rho\|_{q,s}^{\gamma,\mathcal{O}}&=\sum_{\underset{|\alpha|\leqslant q}{\alpha\in\mathbb{N}^{d+1}}}\gamma^{|\alpha|}\sup_{\mu\in{\mathcal{O}}}\|\partial_{\mu}^{\alpha}\rho(\mu,\cdot)\|_{H^{s-|\alpha|}},\nonumber\\
	\|\rho\|_{q}^{\gamma,\mathcal{O}}&=\sum_{\underset{|\alpha|\leqslant q}{\alpha\in\mathbb{N}^{d+1}}}\gamma^{|\alpha|}\sup_{\mu\in{\mathcal{O}}}|\partial_{\mu}^{\alpha}\rho(\mu)|.
\end{align}
\begin{remark}
	\begin{enumerate}[label=\textbullet]
		\item From  Sobolev embeddings we obtain,
		$$ W^{q,\infty,\gamma}(\mathcal{O},\mathbb{C})\hookrightarrow C^{q-1}(\mathcal{O},\mathbb{C}).$$
		\item The spaces  $\left(W^{q,\infty,\gamma}(\mathcal{O},H^{s}),\|\cdot\|_{q,s}^{\gamma,\mathcal{O}}\right)$ and $\left(W^{q,\infty,\gamma}(\mathcal{O},\mathbb{C}),\|\cdot\|_{q}^{\gamma,\mathcal{O}}\right)$ are complete.
		\item For needs related to the use of the kernels of integral operators, we will have to duplicate the variable $\theta$. Thus we may define the weighted Sobolev space $W^{q,\infty,\gamma}(\mathcal{O},H_{\varphi,\theta,\eta}^s)$ similarly as above and denote the corresponding norm by		
		$\|\cdot\|_{q,H_{\varphi,\theta,\eta}^s}^{\gamma,\mathcal{O}}.$
	\end{enumerate}
\end{remark}
The next lemma gathers  some useful classical results related to various operations in weighted Sobolev spaces. The proofs are standard and can be found for instance in \cite{BFM21,BFM21-1,BM18}.
\begin{lem}\label{Lem-lawprod}
	{Let $(\gamma,q,d,s_{0},s)$ satisfy \eqref{initial parameter condition}, then the following assertions hold true.
		\begin{enumerate}[label=(\roman*)]
			\item Space translation invariance:  Let $\rho\in W^{q,\infty,\gamma}(\mathcal{O},H^{s}),$ then for all $\eta\in\mathbb{T},$ the function $(\varphi,\theta)\mapsto\rho(\varphi,\eta+\theta)$ belongs to $W^{q,\infty,\gamma}(\mathcal{O},H^{s})$, and satisfies
			$$\|\rho(\cdot,\eta+\cdot)\|_{q,s}^{\gamma,\mathcal{O}}=\|\rho\|_{q,s}^{\gamma,\mathcal{O}}.$$
			\item Projectors properties: Let $\rho\in W^{q,\infty,\gamma}(\mathcal{O},H^{s}),$ then for all $N\in\mathbb{N}^{*}$ and for all $t\in\mathbb{R}_{+}^{*},$
			$$\|\Pi_{N}\rho\|_{q,s+t}^{\gamma,\mathcal{O}}\leqslant N^{t}\|\rho\|_{q,s}^{\gamma,\mathcal{O}}\quad\mbox{ and }\quad\|\Pi_{N}^{\perp}\rho\|_{q,s}^{\gamma,\mathcal{O}}\leqslant N^{-t}\|\rho\|_{q,s+t}^{\gamma,\mathcal{O}},
			$$
			where the projectors are defined in \eqref{definition of projections for functions}.
			\item Interpolation inequality: 
			Let $q<s_{1}\leqslant s_{3}\leqslant s_{2}$ and $\overline{\theta}\in[0,1],$ with  $s_{3}=\overline{\theta} s_{1}+(1-\overline{\theta})s_{2}.$\\
			If $\rho\in W^{q,\infty,\gamma}(\mathcal{O},H^{s_{2}})$, then  $\rho\in W^{q,\infty,\gamma}(\mathcal{O},H^{s_{3}})$ and
			$$\|\rho\|_{q,s_{3}}^{\gamma,\mathcal{O}}\lesssim\left(\|\rho\|_{q,s_{1}}^{\gamma,\mathcal{O}}\right)^{\overline{\theta}}\left(\|\rho\|_{q,s_{2}}^{\gamma,\mathcal{O}}\right)^{1-\overline{\theta}}.$$
					\item Law products:
			\begin{enumerate}[label=(\alph*)]
				\item Let $\rho_{1},\rho_{2}\in W^{q,\infty,\gamma}(\mathcal{O},H^{s}).$ Then $\rho_{1}\rho_{2}\in W^{q,\infty,\gamma}(\mathcal{O},H^{s})$ and 
				$$\| \rho_{1}\rho_{2}\|_{q,s}^{\gamma,\mathcal{O}}\lesssim\| \rho_{1}\|_{q,s_{0}}^{\gamma,\mathcal{O}}\| \rho_{2}\|_{q,s}^{\gamma,\mathcal{O}}+\| \rho_{1}\|_{q,s}^{\gamma,\mathcal{O}}\| \rho_{2}\|_{q,s_{0}}^{\gamma,\mathcal{O}}.$$
				\item Let $\rho_{1},\rho_{2}\in W^{q,\infty,\gamma}(\mathcal{O},\mathbb{C}).$ Then $\rho_{1}\rho_{2}\in W^{q,\infty,\gamma}(\mathcal{O},\mathbb{C})$ and $$\| \rho_{1}\rho_{2}\|_{q}^{\gamma,\mathcal{O}}\lesssim\| \rho_{1}\|_{q}^{\gamma,\mathcal{O}}\| \rho_{2}\|_{q}^{\gamma,\mathcal{O}}.$$
				\item Let $(\rho_{1},\rho_{2})\in W^{q,\infty,\gamma}(\mathcal{O},\mathbb{C})\times W^{q,\infty,\gamma}(\mathcal{O},H^{s}).$ Then $\rho_{1}\rho_{2}\in W^{q,\infty,\gamma}(\mathcal{O},H^{s})$ and 
				$$\| \rho_{1}\rho_{2}\|_{q,s}^{\gamma,\mathcal{O}}\lesssim\| \rho_{1}\|_{q}^{\gamma,\mathcal{O}}\| \rho_{2}\|_{q,s}^{\gamma,\mathcal{O}}.$$
			\end{enumerate}
			\item Composition law 1: Let $f\in C^{\infty}(\mathcal{O}\times\mathbb{R},\mathbb{R})$ and  $\rho_{1},\rho_{2}\in W^{q,\infty,\gamma}(\mathcal{O},H^{s})$  such that $$\| \rho_{1}\|_{q,s}^{\gamma,\mathcal{O}},\|\rho_{2}\|_{q,s}^{\gamma,\mathcal{O}}\leqslant C_{0}$$ for an  arbitrary  constant  $C_{0}>0$ and define the pointwise composition $$\forall (\mu,\varphi,\theta)\in \mathcal{O}\times\mathbb{T}^{d+1},\quad f(\rho)(\mu,\varphi,\theta):= f(\mu,\rho(\mu,\varphi,\theta)).$$
			Then $f(\rho_{1})-f(\rho_{2})\in W^{q,\infty,\gamma}(\mathcal{O},H^{s})$ with 
			$$\| f(\rho_{1})-f(\rho_{2})\|_{q,s}^{\gamma,\mathcal{O}}\leqslant C(s,d,q,f,C_{0})\| \rho_{1}-\rho_{2}\|_{q,s}^{\gamma,\mathcal{O}}.$$
			\item Composition law 2: Let $f\in C^{\infty}(\mathbb{R},\mathbb{R})$ with bounded derivatives. Let $\rho\in W^{q,\infty,\gamma}(\mathcal{O},\mathbb{C}).$ Then 
			$$\|f(\rho)-f(0)\|_{q}^{\gamma,\mathcal{O}}\leqslant C(q,d,f)\|\rho\|_{q}^{\gamma,\mathcal{O}}\left(1+\|\rho\|_{L^{\infty}(\mathcal{O})}^{q-1}\right).
			$$
			This estimate is also true for $\gamma=1$, corresponding to the classical Sobolev space $W^{q,\infty}(\mathcal{O},\mathbb{C}).$
	\end{enumerate}}
\end{lem}
The next result will be useful later in the study of the linearized operator.
\begin{lem}\label{cheater lemma}
	Let $(\gamma,q,d,s_{0},s)$ satifying \eqref{initial parameter condition} and  $f\in W^{q,\infty,\gamma}(\mathcal{O},H^{s}).$\\
	We consider the function $g:\mathcal{O}\times\mathbb{T}_{\varphi}^{d}\times\mathbb{T}_{\theta}\times\mathbb{T}_{\eta}\rightarrow\mathbb{C}$ defined by
	$$g(\mu,\varphi,\theta,\eta)=\left\lbrace\begin{array}{ll}
		\frac{f(\mu,\varphi,\eta)-f(\mu,\varphi,\theta)}{\sin\left(\tfrac{\eta-\theta}{2}\right)} & \textnormal{if }\theta\neq \eta\\
		2\partial_{\theta}f(\mu,\varphi,\theta)& \textnormal{if }\theta=\eta.
	\end{array}\right.$$
	Then 
	\begin{enumerate}[label=(\roman*)]
		\item $\forall k\in\mathbb{N},\quad\displaystyle\sup_{\eta\in\mathbb{T}}\|(\partial_{\theta}^kg)(\ast,\cdot,\centerdot,\eta+\centerdot)\|_{q,s}^{\gamma,\mathcal{O}}\lesssim\|\partial_{\theta}f\|_{q,s+k}^{\gamma,\mathcal{O}}\lesssim\|f\|_{q,s+k+1}^{\gamma,\mathcal{O}}.$
		\item $\|g\|_{q,H_{\varphi,\theta,\eta}^{s}}^{\gamma,\mathcal{O}}\lesssim\|f\|_{q,s+1}^{\gamma,\mathcal{O}}.$
	\end{enumerate}
\end{lem}
\begin{proof}
	\textbf{(i)} This result has been proved in \cite[Lem. 4.2]{HR21}.\\
	\textbf{(ii)} It suffices to prove the case $q=0.$ Recall the following classical norm estimate
	\begin{equation}\label{decomp norm}
		\|g\|_{H_{\varphi,\theta,\eta}^{s}}\lesssim\|g\|_{H_{\varphi,\theta}^{s}L_{\eta}^{2}}+\|g\|_{L_{\theta}^{2}H_{\varphi,\eta}^{s}}. 
			\end{equation}
By the translation invariance property  
\begin{align*}\|g\|_{L_{\theta}^{2}H_{\varphi,\eta}^{s}}^2&= \int_0^{2\pi}\|g(\cdot,\theta+\centerdot,\centerdot)\|_{H_{\varphi,\eta}^{s}}^2d\theta\\
&\lesssim\; \sup_{\theta\in\mathbb{T}}\| g(\ast,\cdot,\theta+\centerdot,\centerdot)\|_{H_{\varphi,\eta}^{s}}.
\end{align*}
Using the first point and the symmetry $g$ in $(\eta, \theta)$ we obtain
$$\|g\|_{L_{\theta}^{\infty}H_{\varphi,\eta}^{s}}\lesssim\|f\|_{s+1}.$$
Introducing the Bessel potential $J^s$ defined in Fourier by
\begin{equation}\label{Bessel potential}
	\forall j\in\mathbb{Z}^{d},\quad(J^su)_j=\max(1,|j|)^su_j,
\end{equation}
a use of Fubini's Theorem implies
\begin{align*}
	\|g\|_{H_{\varphi,\theta}^{s}L_{\eta}^{2}}=\|J_{\varphi,\theta}^{s}\,g\|_{L_{\varphi}^{2}L_{\theta}^{2}L_{\eta}^{2}}=\|J_{\varphi,\theta}^{s}\,g\|_{L_{\eta}^{2}L_{\varphi}^{2}L_{\theta}^{2}}=\|g\|_{L_{\eta}^{2}H_{\varphi,\theta}^{s}}.
\end{align*}
Since $g$ is symmetric in the variables $\theta$ and $\eta,$ we get
$$\|g\|_{L_{\eta}^{2}H_{\varphi,\theta}^{s}}=\|g\|_{L_{\theta}^{2}H_{\varphi,\eta}^{s}}.$$
Combining the foregoing estimates leads to
$$\|g\|_{H_{\varphi,\theta,\eta}^{s}}\lesssim\|f\|_{s+1}.$$
This ends the proof of Lemma \ref{cheater lemma}.
\end{proof}
We now turn to the presentation of quasi-periodic symplectic change of variables needed for the reduction of the transport part of the linearized operator in the construction of the approximate inverse in the normal directions.
Let $\beta: \mathcal{O}\times \T^{d+1}\to \mathbb{T}$ be a smooth function such that $\displaystyle\sup_{\mu\in \mathcal{O}}\|\beta(\mu,\cdot,\centerdot)\|_{\textnormal{Lip}}<1$ 
then  the map
$$(\varphi,\theta)\in\T^{d+1}\mapsto (\varphi, \theta+\beta(\mu,\varphi,\theta))\in\T^{d+1}$$
is a diffeomorphism with inverse having the form  
$$(\varphi,\theta)\in\T^{d+1}\mapsto (\varphi, \theta+\widehat{\beta}(\mu,\varphi,\theta))\in\T^{d+1}.$$
Moreover, one has the relation
\begin{equation}\label{def betahat}
	y=\theta+\beta(\mu,\varphi,\theta)\Longleftrightarrow \theta=y+\widehat\beta(\mu,\varphi,y).
\end{equation}
Define the operators
\begin{equation}\label{definition symplectic change of variables}
	\mathscr{B}=(1+\partial_{\theta}\beta)\mathcal{B},
\end{equation}
with
\begin{equation}\label{definition change of variables}
	\mathcal{B}\rho(\mu,\varphi,\theta)=\rho\big(\mu,\varphi,\theta+\beta(\mu,\varphi,\theta)\big).
\end{equation}
By straightforward  computations we obtain
\begin{equation}\label{mathscrB1}
	\mathscr{B}^{-1}\rho(\mu,\varphi,y)=\Big(1+\partial_y\widehat{\beta}(\mu,\varphi,y)\Big) \rho\big(\mu,\varphi,y+\widehat{\beta}(\mu,\varphi,y)\big)
\end{equation}
and
$$
\mathcal{B}^{-1} \rho(\mu,\varphi,y)=\rho\big(\mu,\varphi,y+\widehat{\beta}(\mu,\varphi,y)\big).
$$
The following lemma gives some elementary algebraic properties for $\mathcal{B}^{\pm 1}$ and $\mathscr{B}^{\pm 1}$.
\begin{lem}\label{algeb1}
	The following assertions hold true.
	\begin{enumerate}[label=(\roman*)]
		\item The action of $\mathscr{B}^{-1}$ on the derivative is given by
		\begin{equation*}
			\mathscr{B}^{-1}\partial_{\theta}=\partial_{\theta}\mathcal{B}^{-1}.
		\end{equation*}
		\item Denote by $\mathscr{B}^\star$ the $L^2_\theta(\T)$-adjoint of $\mathscr{B}$, then
		$$
		\mathscr{B}^\star=\mathcal{B}^{-1}\quad\hbox{and}\quad \mathcal{B}^\star=\mathscr{B}^{-1}.
		$$
		
	\end{enumerate}
\end{lem}
Now we shall state the following result proved in \cite[Lem. 6.2]{HR21}.
\begin{lem}\label{Compos1-lemm}
	Let $(q,d\gamma,s_0)$ as in \eqref{initial parameter condition}.
	Let $\beta\in W^{q,\infty,\gamma}\big(\mathcal{O},H^{\infty}(\T^{d+1})\big) $ such that 
	\begin{equation}\label{small beta lem}
		\|\beta \|_{q,2s_0}^{\gamma,\mathcal{O}}\leqslant \varepsilon_0,
	\end{equation}
	with $\varepsilon_{0}$ small enough. Then the following assertions hold true.
	\begin{enumerate}[label=(\roman*)]
		\item The linear operators $\mathcal{B},\mathscr{B}:W^{q,\infty,\gamma}\big(\mathcal{O},H^{s}(\T^{d+1})\big)\to W^{q,\infty,\gamma}\big(\mathcal{O},H^{s}(\T^{d+1})\big)$ are continuous and invertible, with 
		\begin{equation}\label{tame comp}
			\forall s\geqslant s_0,\quad \|\mathcal{B}^{\pm1}\rho\|_{q,s}^{\gamma,\mathcal{O}}\leqslant \|\rho\|_{q,s}^{\gamma,\mathcal{O}}\left(1+C\|\beta\|_{q,s_{0}}^{\gamma,\mathcal{O}}\right)+C\|\beta\|_{q,s}^{\gamma,\mathcal{O}}\|\rho\|_{q,s_{0}}^{\gamma,\mathcal{O}}
		\end{equation}
		and 
		\begin{equation}\label{tame comp symp}
			\forall s\geqslant s_0,\quad\|\mathscr{B}^{\pm1}\rho\|_{q,s}^{\gamma,\mathcal{O}}\leqslant \|\rho\|_{q,s}^{\gamma,\mathcal{O}}\left(1+C\|\beta\|_{q,s_{0}}^{\gamma,\mathcal{O}}\right)+C\|\beta\|_{q,s+1}^{\gamma,\mathcal{O}}\|\rho\|_{q,s_{0}}^{\gamma,\mathcal{O}}.
		\end{equation}
		\item The functions $\beta$ and $\widehat{\beta}$ are linked through
		\begin{equation}\label{link betah and beta}
			\forall s\geqslant s_0,\quad\|\widehat{\beta}\|_{q,s}^{\gamma,\mathcal{O}}\leqslant C\|\beta\|_{q,s}^{\gamma,\mathcal{O}}.
		\end{equation}
		\item Let $\beta_{1},\beta_{2}\in W^{q,\infty,\gamma}(\mathcal{O},H^{\infty}(\mathbb{T}^{d+1}))$ satisfying \eqref{small beta lem}. If we denote 
		$$\Delta_{12}\beta=\beta_{1}-\beta_{2}\quad\textnormal{ and }\quad\Delta_{12}\widehat{\beta}=\widehat{\beta}_{1}-\widehat{\beta}_{2},$$
		then they are linked through
		\begin{equation}\label{link diff beta hat and diff beta}
			\forall s\geqslant s_0,\quad\|\Delta_{12}\widehat{\beta}\|_{q,s}^{\gamma,\mathcal{O}}\leqslant C\left(\|\Delta_{12}\beta\|_{q,s}^{\gamma,\mathcal{O}}+\|\Delta_{12}\beta\|_{q,s_{0}}^{\gamma,\mathcal{O}}\max_{j\in\{1,2\}}\|\beta_{j}\|_{q,s+1}^{\gamma,\mathcal{O}}\right).
		\end{equation}
	\end{enumerate}
\end{lem}

\subsection{Operators}\label{section-ope}
In this section we shall collect some basic definitions  and properties related to suitable  operators class.
Consider a smooth family of bounded operators on Sobolev spaces $H^s(\T^{d+1})$, 
$$T: \mu=(b,\omega)\in \mathcal{O}\mapsto T(\mu)\in  \mathcal{L}(H^s(\T^{d+1},\mathbb{C})). $$ 
The linear operator   $T(\mu)$ can be represented by  the infinite dimensional matrix $\left(T_{l_{0},j_{0}}^{l,j}(\mu)\right)_{\underset{(j,j_{0})\in\mathbb{Z}^{2}}{(l,l_{0})\in(\mathbb{Z}^{d })^{2}}}$ with 
$$T(\mu)\mathbf{e}_{l_{0},j_{0}}=\sum_{(l,j)\in\mathbb{Z}^{d+1}}T_{l_{0},j_{0}}^{l,j}(\mu)\mathbf{e}_{l,j}\quad\textnormal{where}\quad T_{l_{0},j_{0}}^{l,j}(\mu)=\big\langle T(\mu)\mathbf{e}_{l_{0},j_{0}},\mathbf{e}_{l,j}\big\rangle_{L^{2}(\mathbb{T}^{d+1})}.$$
Along this paper the operators and the test functions may depend on the same parameter $\mu$ and thus 
the action of the operator $T(\mu)$ on a scalar function $\rho\in W^{q,\infty,\gamma}(\mathcal{O},H^{s}(\mathbb{T}^{d+1},\mathbb{C}))$  is by convention
defined through
$$(T\rho)(\mu,\varphi,\theta):=T(\mu)\rho(\mu,\varphi,\theta).$$
\subsubsection{Toeplitz in time operators}\label{Top-Sec-11}
In this paper we always consider Toeplitz in time operators, whose definition is the following.
\begin{defin}
	An operator $T(\mu)$ is said  Toeplitz in time (actually in the variable $\varphi$) if its Fourier coefficients satisfy,
	$$T_{l_{0},j_{0}}^{l,j}(\mu)=T_{0,j_{0}}^{l-l_0,j}(\mu).$$
	Or equivalently, 
	$$T_{l_{0},j_{0}}^{l,j}(\mu)=T_{j_{0}}^{j}(\mu,l-l_{0})\quad\textnormal{with}\quad T_{j_{0}}^{j}(\mu,l):=T_{0,j_{0}}^{l,j}(\mu).$$
\end{defin}
\noindent  Notice that the action of a Toeplitz operator  $T(\mu)$ on a function $\rho=\displaystyle\sum_{(l_{0},j_{0})\in\mathbb{Z}^{d+1}}\rho_{l_{0},j_{0}}\mathbf{e}_{l_{0},j_{0}}$ is  given~by
\begin{equation}\label{action of Toeplitz in time operators}
	T(\mu)\rho=\sum_{(l,l_{0})\in(\mathbb{Z}^{d})^{2}\\\atop (j,j_{0})\in\mathbb{Z}^{2}}T_{j_{0}}^{j}(\mu,l-l_{0})\rho_{l_{0},j_{0}}\mathbf{e}_{l,j}.
\end{equation}
For Toeplitz operators, we can define the off-diagonal norm as
\begin{align}\label{Top-NormX}
	\| T\|_{\textnormal{\tiny{O-d}},q,s}^{\gamma,\mathcal{O}}=\sum_{\underset{|\alpha|\leqslant q}{\alpha\in\mathbb{N}^{d+1}}}\gamma^{|\alpha|}\sup_{(b,\omega)\in{\mathcal{O}}}\|\partial_{\mu}^{\alpha}(T)(\mu)\|_{\textnormal{\tiny{O-d}},s-|\alpha|},
\end{align}
where 
$$\| T\|_{\textnormal{\tiny{O-d}},s}^{2}=\sum_{(l,m)\in\mathbb{Z}^{d+1}}\langle l,m\rangle^{2s}\sup_{j-k=m}|T_{j}^{k}(l)|^{2}.$$
The cut-off projectors $(P_N)_ {N\in\mathbb{N}^{*}}$ are defined as follows 
\begin{equation}\label{definition of projections for operators}
	\left(P_{N}T(\mu)\right)\mathbf{e}_{l_{0},j_{0}}=\sum_{\underset{|l-l_{0}|,|j-j_{0}|\leqslant N}{(l,j)\in\mathbb{Z}^{d+1}}}T_{l_{0},j_{0}}^{l,j}(\mu)\mathbf{e}_{l,j}\quad\mbox{and}\quad P_{N}^{\perp}T=T-P_{N}T.
\end{equation}
The next lemma lists  some classical results. The proofs are very similar to those in \cite{BM18}.
\begin{lem}\label{properties of Toeplitz in time operators}
	Let $(\gamma,q,d,s_{0},s)$ satisfying \eqref{initial parameter condition}.  Let $T,$ $T_{1}$ and $T_{2}$ be Toeplitz in time operators.
	\begin{enumerate}[label=(\roman*)]
		\item Projectors properties : Let $N\in\mathbb{N}^{*}.$ Let $t\in\mathbb{R}_{+}.$ Then
		$$\| P_{N}T\rho\|_{\textnormal{\tiny{O-d}},q,s+t}^{\gamma,\mathcal{O}}\leqslant N^{t}\| T\rho\|_{\textnormal{\tiny{O-d}},q,s}^{\gamma,\mathcal{O}}\quad\mbox{and}\quad\| P_{N}^{\perp}T\rho\|_{\textnormal{\tiny{O-d}},q,s}^{\gamma,\mathcal{O}}\leqslant N^{-t}\| T\rho\|_{\textnormal{\tiny{O-d}},q,s+t}^{\gamma,\mathcal{O}}.$$
		\item Interpolation inequality : Let $q<s_{1}\leqslant s_{3}\leqslant s_{2}, \,\overline{\theta}\in[0,1]$ with $s_{3}=\overline{\theta}s_{1}+(1-\overline{\theta})s_{2}.$  Then 
		$$\| T\|_{\textnormal{\tiny{O-d}},q,s_{3}}^{\gamma,\mathcal{O}}\lesssim\left(\| T\|_{\textnormal{\tiny{O-d}},q,s_{1}}^{\gamma,\mathcal{O}}\right)^{\overline{\theta}}\left(\| T\|_{\textnormal{\tiny{O-d}},q,s_{2}}^{\gamma,\mathcal{O}}\right)^{1-\overline{\theta}}.$$
		\item Composition law :
		$$\| T_{1}T_{2}\|_{\textnormal{\tiny{O-d}},q,s}^{\gamma,\mathcal{O}}\lesssim\| T_{1}\|_{\textnormal{\tiny{O-d}},q,s}^{\gamma,\mathcal{O}}\| T_{2}\|_{\textnormal{\tiny{O-d}},q,s_{0}}^{\gamma,\mathcal{O}}+\| T_{1}\|_{\textnormal{\tiny{O-d}},q,s_{0}}^{\gamma,\mathcal{O}}\| T_{2}\|_{\textnormal{\tiny{O-d}},q,s}^{\gamma,\mathcal{O}}.$$ 
		\item Link between operators and off-diagonal norms :
		$$\| T\rho\|_{q,s}^{\gamma,\mathcal{O}}\lesssim\| T\|_{\textnormal{\tiny{O-d}},q,s_{0}}^{\gamma,\mathcal{O}}\|\rho\|_{q,s}^{\gamma,\mathcal{O}}+\| T\|_{\textnormal{\tiny{O-d}},q,s}^{\gamma,\mathcal{O}}\|\rho\|_{q,s_{0}}^{\gamma,\mathcal{O}}.$$
		In particular
		$$\| T \rho\|_{q,s}^{\gamma,\mathcal{O}}\lesssim\| T\|_{\textnormal{\tiny{O-d}},q,s}^{\gamma,\mathcal{O}}\|\rho\|_{q,s}^{\gamma,\mathcal{O}}.$$
	\end{enumerate}
\end{lem}
\subsubsection{Reversible and reversibility preserving operators}
We recall the following definitions of real reversible and reversibility preserving operators, see for instance \cite[Def. 2.2]{BBM14}.
\begin{defin}\label{Def-Rev}
	We define the following involution
	\begin{equation}\label{definition involution mathcalS2}
		(\mathscr{S}_{2}\rho)(\varphi,\theta)=\rho(-\varphi,-\theta).
	\end{equation}
	We say that an operator $T(\mu)$ is 
	\begin{enumerate}[label=\textbullet]
		\item real if for all $\rho\in L^{2}(\mathbb{T}^{d+1},\mathbb{C}),$ we have 
		$$\overline{\rho}=\rho\quad\Longrightarrow\quad\overline{T\rho}=T\rho.$$
		\item reversible if
		$$T(\mu)\circ\mathscr{S}_{2}=-\mathscr{S}_{2}\circ T(\mu).$$
		\item reversibility preserving if
		$$T(\mu)\circ\mathscr{S}_{2}=\mathscr{S}_{2}\circ T(\mu).$$
	\end{enumerate}
\end{defin}
The following lemma gives  characterizations of the reversibility notion in terms of Fourier coefficients. A similar result is stated in \cite[Lem. 2.6]{BBM14}.
\begin{lem}\label{characterization of real operator by its Fourier coefficients}
	Let $T$ be an operator. Then $T$ is 
	\begin{enumerate}[label=\textbullet]
		\item real if and only if
		$$\forall(l,l_{0},j,j_{0})\in(\mathbb{Z}^{d})^{2}\times\mathbb{Z}^{2},\quad T_{-l_{0},-j_{0}}^{-l,-j}=\overline{T_{l_{0},j_{0}}^{l,j}}.$$
		\item reversible if and only if
		$$\forall(l,l_{0},j,j_{0})\in(\mathbb{Z}^{d})^{2}\times\mathbb{Z}^{2},\quad T_{-l_{0},-j_{0}}^{-l,-j}=-T_{l_{0},j_{0}}^{l,j}.$$
		\item reversibility-preserving if and only if
		$$\forall(l,l_{0},j,j_{0})\in(\mathbb{Z}^{d})^{2}\times\mathbb{Z}^{2},\quad T_{-l_{0},-j_{0}}^{-l,-j}=T_{l_{0},j_{0}}^{l,j}.$$
	\end{enumerate}
\end{lem}
Throughout this paper we shall constantly make use of two kinds  of  operators : multiplication and integral operators.
\begin{defin}
	Let $T$ be an operator as in Section $\ref{section-ope}.$
	We say that 
	\begin{enumerate}[label=\textbullet]
		\item $T$ is a multiplication operator if there exists a function $M:(\mu,\varphi,\theta)\mapsto M(\mu,\varphi,\theta)$ such that
		$$(T\rho)(\mu,\varphi,\theta)=M(\mu,\varphi,\theta)\rho(\mu,\varphi,\theta).$$
		\item $T$ is an integral operator if there exists a function (called the kernel) $K:(\mu,\varphi,\theta,\eta)\mapsto K(\mu,\varphi,\theta,\eta)$ such that
		$$(T\rho)(\mu,\varphi,\theta)=\int_{\mathbb{T}}\rho(\mu,\varphi,\eta)K(\mu,\varphi,\theta,\eta)d\eta.$$
		
	\end{enumerate}
\end{defin}
We shall need the following  lemma whose proof can be found in \cite[Lem. 4.4]{HR21}.
\begin{lem}\label{lemma symmetry and reversibility}
	Let $(\gamma,q,d,s_{0},s)$ satisfy \eqref{initial parameter condition}, then the following assertions hold true.
		\begin{enumerate}[label=(\roman*)]
			\item Let $T$ be a multiplication operator by a real-valued function $M$, then the following holds true.
			\begin{enumerate}[label=\textbullet]
				\item If $M(\mu,-\varphi,-\theta)=M(\mu,\varphi,\theta)$, then $T$ is  real and reversibility preserving Toeplitz in time and space operator.
				\item If $M(\mu,-\varphi,-\theta)=-M(\mu,\varphi,\theta)$, then $T$ is  real and reversible Toeplitz in time and space operator.
			\end{enumerate}
			Moreover,
			$$\| T\|_{\textnormal{\tiny{O-d}},q,s}^{\gamma,\mathcal{O}}\lesssim\| M\|_{q,s+s_{0}}^{\gamma,\mathcal{O}}.$$
			\item Let $T$ be an integral operator with a real-valued kernel $K$.
			\begin{enumerate}[label=\textbullet]
				\item If $K(\mu,-\varphi,-\theta,-\eta)=K(\mu,\varphi,\theta,\eta)$, then $T$ is  a real and reversibility preserving Toeplitz in time operator.
				\item If $K(\mu,-\varphi,-\theta,-\eta)=-K(\mu,\varphi,\theta,\eta)$, then $T$ is  a real and reversible Toeplitz in time operator.
			\end{enumerate}
			In addition,
			\begin{align*}
			\| T\|_{\textnormal{\tiny{O-d}},q,s}^{\gamma,\mathcal{O}}&\lesssim \int_{\T}\|K(\ast,\cdot,\centerdot,\eta+\centerdot)\|_{q,s+s_{0}}^{\gamma,\mathcal{O}} d\eta\lesssim \|K\|_{q,H^{s+s_{0}}_{\varphi,\theta,\eta}}^{\gamma,\mathcal{O}} 
			\end{align*}
			and
			\begin{align*}
			\| T\rho\|_{q,s}^{\gamma,\mathcal{O}}&\lesssim \|\rho\|_{q,s_0}^{\gamma,\mathcal{O}} \int_{\T}\|K(\ast,\cdot,\centerdot,\eta+\centerdot)\|_{q,s}^{\gamma,\mathcal{O}} d\eta+\|\rho\|_{q,s}^{\gamma,\mathcal{O}} \int_{\T}\|K(\ast,\cdot,\centerdot,\eta+\centerdot)\|_{q,s_0}^{\gamma,\mathcal{O}} d\eta\\
			&\lesssim \|\rho\|_{q,s_0}^{\gamma,\mathcal{O}} \|K\|_{q,H^s_{\varphi,\theta,\eta}}^{\gamma,\mathcal{O}} +\|\rho\|_{q,s}^{\gamma,\mathcal{O}} \|K\|_{q,H^{s_0}_{\varphi,\theta,\eta}}^{\gamma,\mathcal{O}}
			\end{align*}
			where the notation $\ast,\cdot,\centerdot$ denote $\mu,\varphi,\theta$, respectively.
	\end{enumerate}
\end{lem}
In the following lemma we shall study the action of a change of variables as in  \eqref{definition change of variables} on an integral operator. More precisely, we shall need two partial change of coordinates $\mathcal{B}^1$ and $\mathcal{B}^2$ acting respectively on the variables $\theta$ and $\eta$ and defined through
\begin{align}\label{partial change of variables}
	(\mathcal{B}^1\rho)(\mu,\varphi,\theta,\eta)&=\rho\big(\mu,\varphi,\theta+\beta_1(\mu,\varphi,\theta),\eta\big),\\
	(\mathcal{B}^2\rho)(\mu,\varphi,\theta,\eta)&=\rho\big(\mu,\varphi,\theta,\eta+\beta_2(\mu,\varphi,\eta)\big),\nonumber
\end{align}
with $\beta_{1},\beta_2$ two smooth functions satisfying \eqref{small beta lem}. 
A similar result is proved in \cite[Lem. 2.34]{BM18} for pseudo-differential integral operators, so we omit the proof here. We also include the difference estimate which is useful to study the stability of the Cantor sets in Section \ref{Section measure of the final Cantor set}. The proof of the difference estimate is standard and we shall also skip it here.
\begin{lem}\label{lem CVAR kernel}
	Let $(\gamma,q,d,s_0,s)$ satisfy \eqref{initial parameter condition}. Given $r\in W^{q,\infty,\gamma}(\mathcal{O},H^{s})$, we consider a $C^{\infty}$ function in the form
	$$K:(\mu,\varphi,\theta,\eta)\mapsto K(\mu,\varphi,\theta,\eta).$$
We consider the integral operator associated to $K$, namely
$$T\rho(\mu,\varphi,\theta)=\int_{\mathbb{T}}\rho(\varphi,\eta)K(\mu,\varphi,\theta,\eta)d\eta.$$
	Then the following assertions hold true.
	\begin{enumerate}[label=(\roman*)]
	\item Let $\mathcal{B}^{1}$ and $\mathcal{B}^{2}$ as in \eqref{partial change of variables} associated to $\beta_{1}$ and $\beta_{2},$ respectively and enjoying the smallness condition \eqref{small beta lem}. Then,
	\begin{equation}\label{e-B1B2Kr}
		\|\mathcal{B}^1\mathcal{B}^{2}K\|_{q,H_{\varphi,\theta,\eta}^{s}}^{\gamma,\mathcal{O}}\lesssim\|K\|_{q,H_{\varphi,\theta,\eta}^{s}}^{\gamma,\mathcal{O}}+\Big(\max_{i\in\{1,2\}}\|\beta_{i}\|_{q,s}^{\gamma,\mathcal{O}}\Big)\|K\|_{q,H_{\varphi,\theta,\eta}^{s_0}}^{\gamma,\mathcal{O}}.
	\end{equation}
	Now, assume that $\beta_{1}=\beta_2=\beta$ satisfies the following symmetry conditions
	\begin{align}\label{assumption sym beta1}
		\beta(\mu,-\varphi,-\theta)=-\beta(\mu,\varphi,\theta).
	\end{align}
	Consider   $\mathscr{B}$, $\mathcal{B}$ be quasi-periodic changes of variables as in  \eqref{definition symplectic change of variables}-\eqref{definition change of variables}, 
then
\begin{enumerate}[label=\textbullet]
	\item if $K(\mu,-\varphi,-\theta,-\eta)=K(\mu,\varphi,\theta,\eta)$, then $\mathcal{B}^{-1}T\mathscr{B}$ is a real and reversibility preserving Toeplitz in time integral operator.
	\item if $K(\mu,-\varphi,-\theta,-\eta)=-K(\mu,\varphi,\theta,\eta)$, then $\mathcal{B}^{-1}T\mathscr{B}$ is a real and reversible Toeplitz in time inegral operator.
\end{enumerate}
In this case, for any $k\in\mathbb{N},$
	\begin{equation}\label{e-odsBtB}
		\|\partial_{\theta}^{k}\mathcal{B}^{-1}T\mathscr{B}\|_{\textnormal{\tiny{O-d}},q,s}^{\gamma,\mathcal{O}}\lesssim\|K\|_{q,H_{\varphi,\theta,\eta}^{s+s_0+k}}^{\gamma,\mathcal{O}}+\|\beta\|_{q,s+s_0+k}^{\gamma,\mathcal{O}}\|K\|_{q,H_{\varphi,\theta,\eta}^{s_0}}^{\gamma,\mathcal{O}}.
	\end{equation}
\item Introduce $\mathcal{B}_{r}$ a quasi-periodic change of variables as in \eqref{definition change of variables} associated to $\beta_{r}$ (linked to $r$) Consider $r_1,\,r_2\in W^{q,\infty,\gamma}(\mathcal{O},H^s).$ Denote
$$\Delta_{12}r=r_1-r_2,\quad\Delta_{12}f_r=f_{r_1}-f_{r_2}$$ 
for any quantity $f_r$ depending on $r$ and assume that there exist $\varepsilon_0>0$ small enough such that 
\begin{equation}\label{hyp Kr-r}
	\forall i\in\{1,2\},\quad \|\beta_{r_i}\|_{q,2s_0}^{\gamma,\mathcal{O}}+\|K_{r_i}\|_{q,H_{\varphi,\theta,\eta}^{s_0+1}}^{\gamma,\mathcal{O}}\leqslant\varepsilon_0.
\end{equation}
Then, for any $k\in\mathbb{N},$ the following estimate holds
\begin{align}\label{e-d12odsBtB}
	\|\Delta_{12}\partial_{\theta}^{k}\mathcal{B}_{r}^{-1}T_r\mathscr{B}_{r}&\|_{\textnormal{\tiny{O-d}},q,s+s_0+k}^{\gamma,\mathcal{O}}\lesssim\|\Delta_{12}K_r\|_{q,H_{\varphi,\theta,\eta}^{s+s_0+k}}^{\gamma,\mathcal{O}}+\|\Delta_{12}\beta_{r}\|_{q,s+s_0+k}^{\gamma,\mathcal{O}}\\
	&\quad+\Big(\max_{i\in\{1,2\}}\|\beta_{r_i}\|_{q,s+s_0+k}^{\gamma,\mathcal{O}}\Big)\|\Delta_{12}K_r\|_{q,H_{\varphi,\theta,\eta}^{s_0}}^{\gamma,\mathcal{O}}\nonumber\\
	&\quad+\Big(\max_{i\in\{1,2\}}\|K_{r_i}\|_{q,H_{\varphi,\theta,\eta}^{s+s_0+k+1}}^{\gamma,\mathcal{O}}+\max_{i\in\{1,2\}}\|\beta_{r_i}\|_{q,s+s_0+k+1}^{\gamma,\mathcal{O}}\Big)\|\Delta_{12}\beta_r\|_{q,s_0}^{\gamma,\mathcal{O}}.\nonumber
\end{align}
\end{enumerate}
\end{lem}
\section{Functional of interest and regularity aspects}
The main  goal of this section  is to reformulate the problem in a dynamical system language more adapted to KAM techniques.  
More precisely, we shall write   the equation \eqref{Ham eq ED r}  as a Hamiltonian perturbation of an integrable system, given by the linear dynamics at the equilibrium state. Then, by 
 selecting finitely-many tangential sites  and decomposing the phase space into tangential
and normal subspaces  we can introduce action-angle variables on the tangential part allowing to reformulate the problem in terms of embedded tori. This reduces the problem into the  search for zeros of a functional $\mathcal{F}$ to which the Nash-Moser implicit function theorem will be applied. We shall also study in this section some regularity aspects for the perturbed Hamiltonian vector field appearing in $\mathcal{F}$ and
needed during the Nash-Moser scheme. This approach has been intensively used
before, for instance  in \cite{BBMH18,BBM16,BFM21,BFM21-1,BM18}.\\
Notice that, according to  Lemmata \ref{lem eq ED r} and \ref{lemma linearized operator at equilibrium}, the equation \eqref{Ham eq ED r}, that is also \eqref{ED eq r}, can be written in the form
\begin{equation}\label{def XP}\partial_{t}r=\partial_{\theta}\mathrm{L}(b)(r)+X_{P}(r)\quad \textnormal{with}\quad X_{P}(r):=\tfrac{1}{2}\partial_{\theta}r+\partial_{\theta}\mathcal{K}_{b}\ast r-F_{b}[r],
\end{equation}
where the nonlinear functional $F_{b}[r]$ is introduced in  \eqref{Fb} and the convolution kernel is given by \eqref{mathcalKb}. Since we shall  look for small amplitude quasi-periodic solutions then it is  more convenient  to 
 rescale the solution as follows  $r\mapsto\varepsilon r$ with $r$ bounded. Hence, the Hamiltonian equation \eqref{Ham eq ED r} takes the form
\begin{equation}\label{perturbed hamiltonian}
	\partial_{t}r=\partial_{\theta}\mathrm{L}(b)(r)+\varepsilon X_{P_{\varepsilon}}(r),
\end{equation}
where $X_{P_{\varepsilon}}$ is the Hamiltonian vector field defined by
$X_{P_{\varepsilon}}(r):=\varepsilon^{-2}X_{P}(\varepsilon r).$
Notice  that \eqref{perturbed hamiltonian} is the Hamiltonian system generated by
 the rescaled Hamiltonian  
\begin{align}\label{HEE}
	\nonumber \mathcal{H}_{\varepsilon}(r)&=\varepsilon^{-2}H(\varepsilon r)\\
	&:=H_{\mathrm{L}}(r)+\varepsilon P_{\varepsilon}(r),
\end{align}
with  $H_{\mathrm{L}}$ the quadratic Hamiltonian defined in Lemma \ref{lemma linearized operator at equilibrium} and 
$\varepsilon P_{\varepsilon}(r)$ containing terms
of higher order more than cubic.
\subsection{Reformulation with  the action-angle and normal variables}\label{subsec act-angl}
Recall from \eqref{def bbS} that the tangential sites are defined by 
$$\quad{\mathbb S} := \{ j_1, \ldots, j_d \}\subset\mathbb{N}^* \quad\textnormal{with}\quad 
1 \leqslant j_1 < j_2 < \ldots < j_d.$$
We now define the symmetrized tangential sets $\overline{\mathbb{S}}$ and $\mathbb{S}_{0}$ by
\begin{align}\label{tangent-set}
	\overline{\mathbb{S}}:={\mathbb S}\cup (-{\mathbb S})= \{ \pm j,\,\,j\in {\mathbb S}\}\quad\textnormal{and}\quad{\mathbb S}_0=\overline{\mathbb{S}}\cup\{0\}.
\end{align}
Then we decompose the phase space $L^2_0 (\mathbb{T})$  into the following $L^2(\mathbb{T})$-orthogonal direct sum 
\begin{equation}\label{decoacca}
	L^2_0 (\mathbb{T}) = 
	L_{\overline{\mathbb{S}}}
	\overset{\perp}{\oplus} {L}^2_{\bot},\quad L_{\overline{\mathbb{S}}} := \Big\{ \sum_{ j\in\overline{\mathbb{S}}} r_j e_j,\,\, \overline{r_j}=r_{-j} \Big\}  , \quad
	L_{\bot}^2:= \Big\{ z = 
	\sum_{ j\in \mathbb{Z}\setminus\mathbb{S}_0} z_j e_j  \in L^2_0 (\mathbb{T})\Big\} \, ,
\end{equation}
where we denote $e_{j}(\theta)=e^{\ii j\theta}.$ The associated orthogonal projectors $\Pi_{\overline{\mathbb{S}}}, \Pi^\bot_{\mathbb S_0}$ are defined by
\begin{equation}\label{def proj}
	r=\sum_{ j\in \mathbb{Z}^*} r_j e_j =v+z,\quad v:=\Pi_{\overline{\mathbb{S}}}r:=\sum_{ j\in\overline{\mathbb{S}}} r_j e_j,\quad z:=\Pi^\bot_{\mathbb S_0}r:=\sum_{ j\in \mathbb{Z}\setminus\mathbb{S}_0} r_j e_j,
\end{equation}
where $v$ and $z$ are respectively called the tangential and normal variables. For fixed  small amplitudes $(\mathtt{a}_{j})_{j\in\overline{\mathbb{S}}}\in(\mathbb{R}_{+}^{*})^{d}$ satisfying  $\mathtt{a}_{-j}=\mathtt{a}_{j}$, we introduce the  action-angle variables on the tangential set $L_{\overline{\mathbb{S}}}$ by making the following symplectic polar change of coordinates
\begin{equation}\label{ham-syst-k}
	\forall j\in\overline{\mathbb{S}},\quad r_j  =
	\sqrt{\mathtt{a}_{j}^{2}+{|j|}I_j}\,  e^{\ii \vartheta_j},
\end{equation}
where
\begin{equation}\label{sym I-vartheta}
	\forall j\in\overline{\mathbb{S}},\quad I_{-j}=I_j\in\mathbb{R}\quad \textnormal{and}\quad \vartheta_{-j}=-\vartheta_j\in \mathbb{T}.
\end{equation}
Thus, any function $r$
of the phase space $L_0^2$ can be represented as
\begin{equation}\label{definition of A action-angle-normal}
	r= A( \vartheta,I,z) :=  
	v (\vartheta,I)+ z  \quad
	\textnormal{where} \quad  
	v (\vartheta, I) := \sum_{j \in\overline{\mathbb{S}}}    
	\sqrt{\mathtt{a}_{j}^{2}+{|j|}I_j}\,  e^{\ii \vartheta_j}e_j \, . 
\end{equation}
Observe that the function  $v (-{\omega}_{\textnormal{Eq}} (b)t,0)$, where ${\omega}_{\textnormal{Eq}}$ is defined in \eqref{def freq vec eqb}, corresponds to the solution of the linear system \eqref{Ham eq 0} described by  \eqref{rev sol eq}. In these new coordinates, the involution $ {\mathscr S}$ defined in \eqref{definition of the involution mathcal S} reads 
\begin{equation}\label{rev_aa}
	{\mathfrak S} : (   \vartheta,I, z)\mapsto ( -\vartheta,I, {\mathscr S} z )
\end{equation}
and the symplectic  $2$-form in \eqref{def symp-form} becomes, after straightforward computations using \eqref{ham-syst-k} and \eqref{sym I-vartheta},
\begin{equation}\label{sympl_form}
	{\mathcal W} =  
	\sum_{j \in \mathbb{S}} d\vartheta_j  \wedge  d I_j   +\frac12\sum_{j \in \mathbb{Z} \setminus \mathbb{S}_0}\frac{1}{\ii j}dr_j\wedge dr_{-j}=
	\Big(\sum_{j \in\mathbb{S}}d\vartheta_j  \wedge  d I_j   \Big)  \oplus {\mathcal W}_{|L_\bot^2},  
\end{equation}
where ${\mathcal W}_{|L^2_{\bot}}$ denotes the restriction of $\mathcal{W}$ to $L^2_{\bot}$.
This proves that the transformation $A$ defined in \eqref{definition of A action-angle-normal} is symplectic and in the action-angle and normal coordinates $ (\vartheta,I,z)\in  \mathbb{T}^d\times \mathbb{R}^d\times {L}^2_{\bot}$,  the Hamiltonian system generated by $ {\mathcal H}_\varepsilon   $ in \eqref{HEE} transforms into the one generated
by the Hamiltonian 
\begin{equation}\label{Hepsilon}
	H_{\varepsilon} =\mathcal{H}_{\varepsilon} \circ  A.
\end{equation}
Since $ \mathrm{L}(b) $ in Lemma \ref{lemma linearized operator at equilibrium} preserves the subspaces $L_{\overline{\mathbb{S}}}$ and ${L}^2_{\bot}$ then the quadratic Hamiltonian $ H_{\rm L}$ in \eqref{defLHL} (see \eqref{Ham-Fourier})
in the variables $ (\vartheta, I, z) $  reads, up to an additive constant,
 \begin{align}\label{QHAM}
	H_\mathrm{L} \circ A =  -\sum_{j\in\mathbb{S}} \, \Omega_j(b)I_j+ \frac12  \langle  \mathrm{L}(b)\, z, z \rangle_{L^2(\mathbb{T})} =   -{\omega}_{\textnormal{Eq}}(b)\cdot I
	+ \frac12  \langle  \mathrm{L}(b)\, z, z \rangle_{L^2(\mathbb{T})},  
\end{align}
where $ {\omega}_{\textnormal{Eq}} \in \mathbb{R}^d $ is the unperturbed 
tangential frequency vector defined by \eqref{def freq vec eqb}.
By \eqref{HEE} and \eqref{QHAM}, 
the Hamiltonian $H_{\varepsilon} $ in \eqref{Hepsilon} reads
\begin{equation}\label{cNP}
	H_{\varepsilon} = 
	{\mathcal N} + \varepsilon \mathcal{ P}_{\varepsilon}  \quad {\rm with} \quad{\mathcal N} :=   -{\omega}_{\textnormal{Eq}}(b)\cdot I  + \frac12  \langle  \mathrm{L}(b)\, z, z \rangle_{L^2(\mathbb{T})}  
	\quad \textnormal{and}
	\quad \mathcal{ P}_{\varepsilon} :=   P_\varepsilon \circ A .  
\end{equation}
We look for an embedded invariant torus
\begin{equation}\label{rev-torus}
	i:\begin{array}[t]{rcl} \mathbb{T}^d &\rightarrow&
		\mathbb{R}^d \times \mathbb{R}^d \times {L}^2_{\bot} \\
		\varphi &\mapsto& i(\varphi):= (  \vartheta(\varphi), I(\varphi), z(\varphi))
	\end{array}  
\end{equation}
of the Hamiltonian vector field 
\begin{equation}\label{hamiltonian vector filed associated with Hepsilon}
	X_{H_{\varepsilon}}:= 
	(\partial_I H_{\varepsilon} , -\partial_\vartheta H_{\varepsilon} , \Pi_{\mathbb{S}_0}^\bot
	\partial_\theta \nabla_{z} H_{\varepsilon} ) 
\end{equation}
filled by quasi-periodic solutions with Diophantine frequency 
vector $\omega$. 
We point out that for the value   $\varepsilon=0$ the Hamiltonian system 
$$\omega\cdot\partial_\varphi i (\varphi) = X_{H_0} ( i (\varphi))$$   possesses, for any value of the parameter $b\in (b_0,b_1)$, the invariant torus
\begin{equation}\label{def iflat}
	i_{\textnormal{\tiny{flat}}}(\varphi):=(\varphi,0,0).
\end{equation}
Now we consider the family  of Hamiltonians,
\begin{equation}\label{H alpha}
	\begin{aligned}
		H_\varepsilon^\alpha := {\mathcal N}_\alpha +\varepsilon  {\mathcal P}_{\varepsilon} \quad  \textnormal{where}\quad{\mathcal N}_\alpha :=  
		\alpha \cdot I 
		+ \frac12 \langle \mathrm{L}(b)\, z, z\rangle_{L^2(\mathbb{T})},
	\end{aligned}
\end{equation}
which depends on the constant vector $ \alpha \in \mathbb{R}^d $. For the value $\alpha=-{\omega}_{\textnormal{Eq}}(b)$ we have $H_\varepsilon^\alpha= H_\varepsilon$.
The parameter $\alpha$ is introduced in order to ensure the validity of some compatibility conditions  during the approximate inverse process.
We look for zeros of the nonlinear operator
\begin{equation}\label{main function}
	\begin{array}{l}
		\mathcal{F}(i,\alpha,(b,\omega),\varepsilon):=\omega\cdot\partial_{\varphi}i(\varphi)-X_{H_{\varepsilon}^{\alpha}}(i(\varphi))\\
		\mbox{\hspace{2cm}}=\left(\begin{array}{c}
			\omega\cdot\partial_{\varphi}\vartheta(\varphi)-\alpha-\varepsilon\partial_{I}\mathcal{P}_{\varepsilon}(i(\varphi))\\
			\omega\cdot\partial_{\varphi}I(\varphi)+\varepsilon\partial_{\vartheta}\mathcal{P}_{\varepsilon}(i(\varphi))\\
			\omega\cdot\partial_{\varphi}z(\varphi)-\partial_{\theta}\big[\mathrm{L}(b)z(\varphi)+\varepsilon\nabla_{z}\mathcal{P}_{\varepsilon}\big(i(\varphi)\big)\big]
		\end{array}\right),
	\end{array}
\end{equation}
where $\mathcal{P}_{\varepsilon}$ is defined in \eqref{HEE}. For any $ \alpha \in \mathbb{R}^d $,
the Hamiltonian $H_{\varepsilon}^{\alpha}$  is invariant under the involution $\mathfrak{S}$  defined in \eqref{rev_aa}, 
$$H_{\varepsilon}^{\alpha}\circ\mathfrak{S}=H_{\varepsilon}^{\alpha}.$$
Thus, we look for reversible solutions of $\mathcal{F}(i,\alpha,(b,\omega),\varepsilon)=0,$ namely satisfying
\begin{equation}\label{reversibility condition in the variables theta I and z}
	\begin{array}{ccc}
		\vartheta(-\varphi)=-\vartheta(\varphi), & I(-\varphi)=I(\varphi), & z(-\varphi)=(\mathscr{S}z)(\varphi).
	\end{array}
\end{equation}
We define the periodic component $\mathfrak{I}$ of the torus $i$ by
$$\mathfrak{I}(\varphi):=i(\varphi)-(\varphi,0,0)=(\Theta(\varphi),I(\varphi),z(\varphi))\quad \mbox{ with }\quad \Theta(\varphi)=\vartheta(\varphi)-\varphi,$$
and the weighted Sobolev norm of $\mathfrak{I}$ as 
$$\|\mathfrak{I}\|_{q,s}^{\gamma,\mathcal{O}}:=\|\Theta\|_{q,s}^{\gamma,\mathcal{O}}+\| I\|_{q,s}^{\gamma,\mathcal{O}}+\| z\|_{q,s}^{\gamma,\mathcal{O}}.$$
\subsection{Regularity of the perturbed Hamiltonian vector field}
This section is devoted to some regularity aspects of the Hamiltonian involved in the equation \eqref{Ham eq ED r}.  
We shall need the following lemma.
\begin{lem}\label{lemma estimates L S}
Let $(\gamma,q,s_{0},s)$ satisfy \eqref{initial parameter condition}. There exists $\varepsilon_{0}\in(0,1]$ such that if
		$$\| r\|_{q,s_{0}+2}^{\gamma,\mathcal{O}}\leqslant\varepsilon_{0},$$
		then the operators $\partial_{\theta}\mathbf{L}_{r}$ and $\partial_{\theta}\mathbf{S}_{r}$, defined in \eqref{mathbfLr} and \eqref{mathbfSr} write 
	\begin{align}
		\partial_{\theta}\mathbf{L}_{r}&=\partial_{\theta}\mathcal{K}_{1,b}\ast\cdot+\partial_{\theta}\mathbf{L}_{r,1}\quad {\rm with}\quad \mathbf{L}_{r,1}(\rho)(b,\varphi,\theta):=\int_{\mathbb{T}}\rho(\varphi,\eta)\mathbb{K}_{r,1}(b,\varphi,\theta,\eta)d\eta, \label{dcp Lr}\\
	\partial_{\theta}\mathbf{S}_{r}&=\partial_{\theta}\mathcal{K}_{2,b}\ast\cdot+\partial_{\theta}\mathbf{S}_{r,1} \;\;\;\quad\textnormal{with}\quad\mathbf{S}_{r,1}(\rho)(b,\varphi,\theta)=\int_{\mathbb{T}}\rho(\varphi,\eta)\mathscr{K}_{r,1}(b,\varphi,\theta,\eta)d\eta\label{dcp Sr}
	\end{align}
where $\mathcal{K}_{1,b}$, $\mathcal{K}_{2,b}$ are given by \eqref{mathcalK1b}-\eqref{mathcalK2b} and  the kernels $\mathbb{K}_{r,1}(b,\varphi,\theta,\eta), \mathscr{K}_{r,1}(b,\varphi,\theta,\eta)\in\mathbb{R}$ satisfy the following symmetry property:  if $r(-\varphi,-\theta)=r(\varphi,\theta)$ then
	\begin{align}\label{symmetry kernel K1}
		\mathbb{K}_{r,1}(b,-\varphi,-\theta,-\eta)&=\mathbb{K}_{r,1}(b,\varphi,\theta,\eta),\\
		\label{sym scrK}
	\mathscr{K}_{r,1}(b,-\varphi,-\theta,-\eta)&=\mathscr{K}_{r,1}(b,\varphi,\theta,\eta)
	\end{align}
	and the following estimates
	\begin{align}
		\|\mathbb{K}_{r,1}\|_{q,H_{\varphi,\theta,\eta}^{s}}^{\gamma,\mathcal{O}}&\lesssim\| r\|_{q,s+1}^{\gamma,\mathcal{O}}, \label{estimate kernel mathbbK1}\\
		\|\mathscr{K}_{r,1}\|_{q,H_{\varphi,\theta,\eta}^s}^{\gamma,\mathcal{O}}&\lesssim\|r\|_{q,s}^{\gamma,\mathcal{O}}.\label{e-scrK}
	\end{align}
	Moreover, 
\begin{align}
		\|\partial_{\theta}\mathcal{K}_{1,b}\ast \rho\|_{q,s}^{\gamma,\mathcal{O}}&\lesssim\|\rho\|_{q,s}^{\gamma,\mathcal{O}},\label{e-dKb}\\
		\|\partial_{\theta}\mathcal{K}_{2,b}\ast \rho\|_{q,s}^{\gamma,\mathcal{O}}&\lesssim\|\rho\|_{q,s}^{\gamma,\mathcal{O}},\label{e-dKb2}\\
	\|\partial_{\theta}\mathbf{L}_{r,1}\rho\|_{q,s}^{\gamma,\mathcal{O}}&\lesssim \|r\|_{q,s_0+2}^{\gamma,\mathcal{O}}\|\rho\|_{q,s}^{\gamma,\mathcal{O}}+\|r\|_{q,s+2}^{\gamma,\mathcal{O}}\|\rho\|_{q,s_{0}}^{\gamma,\mathcal{O}},\label{e-dLr1}\\
	\|\partial_{\theta}\mathbf{S}_{r,1}\rho\|_{q,s}^{\gamma,\mathcal{O}}&\lesssim\|r\|_{q,s_0+1}^{\gamma,\mathcal{O}} \|\rho\|_{q,s}^{\gamma,\mathcal{O}}+\|r\|_{q,s+1}^{\gamma,\mathcal{O}}\|\rho\|_{q,s_{0}}^{\gamma,\mathcal{O}}.\label{e-dSr1}
\end{align}
\end{lem}
\begin{proof}
 According to \eqref{formula} we may write
	\begin{align}\label{expression of Ar with vr1}
		\nonumber A_{r}(\varphi,\theta,\eta)&=2b\left|\sin\left(\tfrac{\eta-\theta}{2}\right)\right|\bigg(\bigg(\frac{R(b,\varphi,\eta)-R(b,\varphi,\theta)}{2b\sin\big(\tfrac{\eta-\theta}{2}\big)}\bigg)^{2}+\frac{1}{b^2} R(b,\varphi,\eta)R(b,\varphi,\theta)\bigg)^{\frac{1}{2}}\\
		&:=2b\left|\sin\left(\tfrac{\eta-\theta}{2}\right)\right|v_{r,1}(b,\varphi,\theta,\eta).
	\end{align}
	Notice that $v_{r,1}$ is smooth when $r$ is smooth and small enough, and $v_{0,1}=1.$ More precisely, by using  Lemma \ref{Lem-lawprod}-$(iv)$-$(v)$ combined with Lemma \ref{cheater lemma}-$(ii)$ and the smallness condition on $r$, we get
	\begin{align}\label{estimate vr1}
		\| v_{r,1}-1\|_{q,H_{\varphi,\theta,\eta}^s}^{\gamma,\mathcal{O}}&\lesssim \| r\|_{q,s+1}^{\gamma,\mathcal{O}}.
	\end{align}
Using the morphism property of the logarithm, we can write
\begin{align}
	\log(A_{r}(b,\varphi,\theta,\eta))&=\log\big(2b\big)+\tfrac{1}{2}\log\left(\sin^{2}\left(\tfrac{\eta-\theta}{2}\right)\right)+\log\left(v_{r,1}(b,\varphi,\theta,\eta)\right)\nonumber\\
	&:=\log\big(2b\big)+\mathcal{K}_{1,b}(\eta-\theta)+\mathbb{K}_{r,1}(b,\varphi,\theta,\eta)\label{dcp kLr}
\end{align}
and \eqref{symmetry kernel K1} immediately follows. Moreover,    \eqref{mathbfLr} and \eqref{dcp kLr} give  \eqref{dcp Lr}. 
Applying Lemma \ref{Lem-lawprod}-(v) together with \eqref{estimate vr1} and the smallness condition on $r$, we obtain \eqref{estimate kernel mathbbK1}. Using \eqref{estimate kernel mathbbK1}, Lemma \ref{lemma symmetry and reversibility}-$(ii)$ and the smallness property on $r$, we get \eqref{e-dLr1}.
	Similarly, from \eqref{formulb} we can link $B_r^2$ to $B_0^2$ by
\begin{align*}
B_r^{2}(b,\varphi,\theta,\eta)&=B_0^{2}(b,\varphi,\theta,\eta)+\Big(R^2(b,\varphi,\theta)R^2(b,\varphi,\eta)-b^4\Big)-2\Big(R(b,\varphi,\theta)R(b,\varphi,\eta)-b^2\Big)\cos(\eta-\theta)\\
&=B_0^{2}(b,\varphi,\theta,\eta)\big(1+P_r(b,\varphi,\theta,\eta)\big)
\end{align*}
with $$P_r(b,\varphi,\theta,\eta):=\tfrac{\big(R^2(b,\varphi,\theta)R^2(b,\varphi,\eta)-b^4\big)-2\big(R(b,\varphi,\theta)R(b,\varphi,\eta)-b^2\big)\cos(\eta-\theta)}{1+b^4-2b^2\cos(\eta-\theta)}.$$
so that we can write
\begin{align}\label{scrK}
\log\big(B_{r}(b,\varphi,\theta,\eta)\big)&=\log\big(B_0(b,\varphi,\theta,\eta)\big)+\tfrac12\log\big(1+P_r(b,\varphi,\theta,\eta)\big) \nonumber\\ 
	&:=\mathcal{K}_{2,b}(\eta-\theta)+\mathscr{K}_{r,1}(b,\varphi,\theta,\eta)
\end{align}
and \eqref{sym scrK} immediately follows. Moreover, \eqref{mathbfSr} and \eqref{scrK} give \eqref{dcp Sr}. 
Notice that that $P_r$ is smooth with respect to each variable and with respect to $r$ with $P_0=0$.  We conclude by Lemma \ref{Lem-lawprod}-(iv)-(v) and the smallness property on $r$ that
$$\|P_r\|_{q,H_{\varphi,\theta,\eta}^s}^{\gamma,\mathcal{O}}\lesssim\|r\|_{q,s}^{\gamma,\mathcal{O}}.$$
As a consequence, composition laws in Lemma \ref{Lem-lawprod} together with the smallness property on $r$ imply \eqref{e-scrK}.
Then, using \eqref{e-scrK}, Lemma \ref{lemma symmetry and reversibility}-$(ii)$ and the smallness property on $r$, we get \eqref{e-dSr1}. The estimates \eqref{e-dKb}-\eqref{e-dKb2} can be obtained using \eqref{fourierk1b}, \eqref{fourierk2b} and Leibniz rule combined with the following estimate
	$$\sup_{n\in\mathbb{N}}\|b\mapsto b^n\|_{q}^{\gamma,\mathcal{O}}\lesssim 1.$$ 
This ends the proof of Lemma \ref{lemma estimates L S}.
\end{proof}

We now provide tame estimates for  the vector field $X_P$  defined in \eqref{def XP}.
\begin{lem}\label{lemma estimates vector field XP}
	{Let $(\gamma,q,s_{0},s)$ satisfy \eqref{initial parameter condition}. There exists $\varepsilon_{0}\in(0,1]$ such that if
		$$\| r\|_{q,s_{0}+2}^{\gamma,\mathcal{O}}\leqslant\varepsilon_{0},$$
		then the vector field $X_{P}$, defined in \eqref{def XP} satisfies the following estimates
		\begin{enumerate}[label=(\roman*)]
			\item $\| X_{P}(r)\|_{q,s}^{\gamma,\mathcal{O}}\lesssim \| r\|_{q,s+2}^{\gamma,\mathcal{O}}\| r\|_{q,s_0+1}^{\gamma,\mathcal{O}}.$
			\item $\| d_{r}X_{P}(r)[\rho]\|_{q,s}^{\gamma,\mathcal{O}}\lesssim\|\rho\|_{q,s+2}^{\gamma,\mathcal{O}}\| r\|_{q,s_0+1}^{\gamma,\mathcal{O}}+\| r\|_{q,s+2}^{\gamma,\mathcal{O}}\|\rho\|_{q,s_{0}+1}^{\gamma,\mathcal{O}}.$
			\item 
			$\| d_r^{2}X_{P}(r)[\rho_{1},\rho_{2}]\|_{q,s}^{\gamma,\mathcal{O}}\lesssim\|\rho_{1}\|_{q,s_{0}+1}^{\gamma,\mathcal{O}}\|\rho_{2}\|_{q,s+2}^{\gamma,\mathcal{O}}+\big(\|\rho_{1}\|_{q,s+2}^{\gamma,\mathcal{O}}+\| r\|_{q,s+2}^{\gamma,\mathcal{O}}\|\rho_{1}\|_{q,s_{0}+1}^{\gamma,\mathcal{O}}\big)\|\rho_{2}\|_{q,s_{0}+1}^{\gamma,\mathcal{O}}$.
	\end{enumerate}}
\end{lem}
\begin{proof}
	We shall follow the proof developed in \cite[Lem 10.2]{BHM21}. We first prove the estimate $(iii)$ and
the estimates $(ii)$ and $(i)$ then follow by Taylor formula since $d_{r}X_{P}(0)=0$ and $X_{P}(0)=0$.
Recall from Lemma \ref{lemma general form of the linearized operator}, \eqref{dcp Lr} and \eqref{dcp Sr}  that
	$$
	d_rX_H(r)[\rho]=-d_{r}F_{b}(r)[\rho]=-\partial_{\theta}\left(V_{r}\rho\right)-\partial_{\theta}\mathcal{K}_{b}\ast \rho-\partial_{\theta}\mathbf{L}_{r,1}\rho+\partial_{\theta}\mathbf{S}_{r,1}\rho.
	$$
According to \eqref{def XP},  $P$ is is the Hamiltonian  generated by  higher order more than cubic  terms $H_{\geqslant 3}$. Then differentiating with respect to $r$ the last expression we obtain
\begin{equation}\label{Second-diff}
		d_{r}^{2}X_{P}(r)[\rho_{1},\rho_{2}]=-\partial_{\theta}\left(\big(d_{r}V_{r}[\rho_{2}]\big)\rho_{1}\right)-\partial_{\theta}\left(d_{r}\mathbf{L}_{r,1}[\rho_{2}]\rho_{1}\right)+\partial_{\theta}\left(d_{r}\mathbf{S}_{r,1}[\rho_{2}]\rho_{1}\right).
	\end{equation}
	Recall, from \eqref{dcp Lr} and \eqref{dcp kLr}, that 
	\begin{equation}\label{id-L1r}
		\mathbf{L}_{r,1}(\rho)(b,\varphi,\theta)=\int_{\mathbb{T}}\rho(\varphi,\eta)\log(v_{r,1}(b,\varphi,\theta,\eta))d\eta.
	\end{equation} 
	Hence by differentiation  we obtain
	\begin{align}\label{definition of mathbbK1M}
		d_{r}\mathbf{L}_{r,1}(r)[\rho_{2}]\rho_{1}(b,\varphi,\theta)
		&=\frac12\int_{\mathbb{T}}\rho_1(\varphi,\eta)\tfrac{\big(d_{r}v_{r,1}^2\big)[\rho_{2}](b,\varphi,\theta,\eta)}{v_{r,1}^2(b,\varphi,\theta,\eta)}d\eta.
	\end{align} 
	Coming back to \eqref{expression of Ar with vr1} it is obvious that the dependance in $r$ of the functional $v_{r,1}^2$ is smooth. Straightforward calculus leads to   
	\begin{align*}
		\frac12 d_{r}v_{r,1}^2(r)[\rho_2](b,\varphi,\theta,\eta)&=\tfrac{R(b,\varphi,\theta)-R(b,\varphi,\eta)}{\sin^2\left(\frac{\eta-\theta}{2}\right)}\left(\tfrac{\rho_2(\varphi,\theta)}{R(b,\varphi,\theta)}-\tfrac{\rho_2(\varphi,\eta)}{R(b,\varphi,\eta)}\right)+\tfrac{\rho_2(\varphi,\theta)R^{2}(b,\varphi,\eta)+\rho_2(\varphi,\eta)R^{2}(b,\varphi,\theta)}{R(b,\varphi,\theta)R(b,\varphi,\eta)}.
	\end{align*}
	Using \eqref{estimate vr1} combined with the  law products stated in Lemma \ref{Lem-lawprod}, Lemma \ref{cheater lemma}-$(ii)$ and the smallness condition of Lemma \ref{lemma estimates vector field XP} we find that
	\begin{equation}\label{estimate vr1PP}
		\| d_{r}v_{r,1}^2(r)[\rho_2]\|_{q,H_{\varphi,\theta,\eta}^s}^{\gamma,\mathcal{O}}\lesssim \| \rho_2\|_{q,s}^{\gamma,\mathcal{O}}+\| r\|_{q,s+1}^{\gamma,\mathcal{O}}\| \rho_2\|_{q,s_0}^{\gamma,\mathcal{O}}.
	\end{equation}
	According to \eqref{estimate vr1PP},  \eqref{definition of mathbbK1M} and using Lemma \ref{Lem-lawprod}-$(iv)$-$(v)$, Lemma \ref{lemma symmetry and reversibility}-$(ii)$ and the smallness condition we obtain,
	\begin{align}\label{e-drL1}
		\nonumber \|\partial_{\theta}d_{r}\mathbf{L}_{r,1}(r)[\rho_{2}]\rho_{1}\|_{q,s}^{\gamma,\mathcal{O}}\lesssim&\|d_{r}\mathbf{L}_{r,1}(r)[\rho_{2}]\rho_{1}\|_{q,s+1}^{\gamma,\mathcal{O}}\\
		\lesssim&\|\rho_{1}\|_{q,s+1}^{\gamma,\mathcal{O}}\|\rho_{2}\|_{q,s_{0}+1}^{\gamma,\mathcal{O}}+\|\rho_{1}\|_{q,s_{0}}^{\gamma,\mathcal{O}}\big(\|\rho_{2}\|_{q,s+1}^{\gamma,\mathcal{O}}+\| r\|_{q,s+2}^{\gamma,\mathcal{O}}\|\rho_{2}\|_{q,s_{0}+1}^{\gamma,\mathcal{O}}\big).
	\end{align}
Now we shall move to the estimate of $d_{r}\mathbf{S}_{r,1}(r)[\rho_{2}]\rho_{1}(b,\varphi,\theta).$ 
By differentiating with respect to $r$ in \eqref{dcp Sr} and \eqref{scrK}, we obtain
\begin{align*}
	 d_{r}\mathbf{S}_{r,1}(r)[\rho_{2}]\rho_{1}(b,\varphi,\theta)&=\frac12\int_{\mathbb{T}}\rho_1(\varphi,\eta)\tfrac{\big(d_{r}B_{r}^2\big)[\rho_{2}](b,\varphi,\theta,\eta)}{B_{r}^2(b,\varphi,\theta,\eta)}d\eta.
\end{align*}
	In view of \eqref{formulb}, direct computations yield
	\begin{align*}
		\frac{1}{2}d_rB_r^2(r)[\rho_2](b,\varphi,\theta,\eta)&=\rho_2(\varphi,\theta)R^2(b,\varphi,\eta)+\rho_2(\varphi,\eta)R^2(b,\varphi,\theta)\\
		&\quad-\Big(\rho_2(\varphi,\theta)\tfrac{R(b,\varphi,\eta)}{R(b,\varphi,\theta)}+\rho_2(\varphi,\eta)\tfrac{R(b,\varphi,\theta)}{R(b,\varphi,\eta)}\Big)\cos(\eta-\theta).
	\end{align*}
Then, Lemma \ref{Lem-lawprod}-$(iv)$-$(v)$  and the smallness condition on $r$ imply
\begin{align*}
	\|d_rB_r^2(r)[\rho_2]\|_{q,H_{\varphi,\theta,\eta}^s}^{\gamma,\mathcal{O}}\lesssim\|\rho_2\|_{q,s}^{\gamma,\mathcal{O}}+\|r\|_{q,s}^{\gamma,\mathcal{O}}\|\rho_2\|_{q,s_0}^{\gamma,\mathcal{O}}.
\end{align*}
It follows from Lemma \ref{lemma symmetry and reversibility}-$(ii)$, that
\begin{align}\label{e-drS1}
	\nonumber \|\partial_{\theta}d_{r}\mathbf{S}_{r,1}(r)[\rho_{2}]\rho_{1}\|_{q,s}^{\gamma,\mathcal{O}}\lesssim&\|d_{r}\mathbf{S}_{r,1}(r)[\rho_{2}]\rho_{1}\|_{q,s+1}^{\gamma,\mathcal{O}}\\
	\lesssim&\|\rho_{1}\|_{q,s+1}^{\gamma,\mathcal{O}}\|\rho_{2}\|_{q,s_{0}+1}^{\gamma,\mathcal{O}}+\|\rho_{1}\|_{q,s_{0}}^{\gamma,\mathcal{O}}\big(\|\rho_{2}\|_{q,s+1}^{\gamma,\mathcal{O}}+\| r\|_{q,s+1}^{\gamma,\mathcal{O}}\|\rho_{2}\|_{q,s_{0}+1}^{\gamma,\mathcal{O}}\big).
\end{align}
Next we shall move to the estimate of $d_rV_r[\rho_2]$. From Lemma \ref{lemma general form of the linearized operator}, we can write
\begin{align*}
V_r=V_r^0+V_r^1+V_r^2,\quad \textnormal{with}\quad	V_r^0(b,\varphi,\theta)&:=-\tfrac{1}{2}\int_{\mathbb{T}}\tfrac{R^{2}(b,\varphi,\eta)}{R^{2}(b,\varphi,\theta)}d\eta,\\
	V_r^1(b,\varphi,\theta)&:=-\tfrac{1}{R(b,\varphi,\theta)}\int_{\mathbb{T}}\log\big(A_{r}(b,\varphi,\theta,\eta)\big)\partial_{\eta}\big(R(b,\varphi,\eta)\sin(\eta-\theta)\big)d\eta,\\
	V_r^2(b,\varphi,\theta)&:=-\tfrac{1}{R^{3}(b,\varphi,\theta)}\int_{\mathbb{T}}\log\big(B_{r}(b,\varphi,\theta,\eta)\big)\partial_{\eta}\big(R(b,\varphi,\eta)\sin(\eta-\theta)\big)d\eta.
\end{align*}
Differentiating $V_r^0$ with respect to $r$ in the direction $\rho_2$ yields
$$d_{r}V_{r}^0(r)[\rho_{2}](\theta)=-\int_{\mathbb{T}}\tfrac{\rho_2(\varphi,\theta)R^2(b,\varphi,\eta)-\rho_2(\varphi,\eta)R^2(b,\varphi,\theta)}{R^4(b,\varphi, \theta)}d\eta.$$
Law products in Lemma \ref{Lem-lawprod} and the smallness condition in $r$ then imply
\begin{equation}\label{e-dVr0}
	\|d_{r}V_{r}^0(r)[\rho_{2}]\|_{q,s}^{\gamma,\mathcal{O}}\lesssim\|\rho_2\|_{q,s}^{\gamma,\mathcal{O}}+\|r\|_{q,s}^{\gamma,\mathcal{O}}\|\rho_2\|_{q,s_0}^{\gamma,\mathcal{O}}.
\end{equation}
Differentiating $V_r^1$ with respect to $r$ in the direction $\rho_2$  gives
\begin{align*}
	d_{r}V_{r}^1(r)[\rho_{2}](\theta)&=-\int_{\mathbb{T}}\log\left( A_{r}(b,\varphi,\theta,\eta)\right)\partial_{\eta}d_rf_r[\rho_2](b,\varphi,\theta,\eta)d\eta\\
	&\quad- \frac12 \int_{\mathbb{T}}\tfrac{\big(d_{r}v_{r,1}^2\big)[\rho_{2}](b,\varphi,\theta,\eta)}{v_{r,1}^2(b,\varphi,\theta,\eta)} \partial_{\eta}f_r(b,\varphi,\theta,\eta)d\eta
	\\
	&:=\mathcal{I}_1(\theta)+\mathcal{I}_2(\theta),
\end{align*}
with
\begin{align*}
f_r(b,\varphi,\theta,\eta)&:=\tfrac{R(b,\varphi,\eta)}{R(b,\varphi,\theta)}\sin(\eta-\theta).
\end{align*}
Straightforward computations give
$$d_rf_r[\rho_2](b,\varphi,\theta)=\tfrac{\rho_{2}(\varphi,\eta)R^2(b,\varphi,\theta)-\rho_2(\varphi, \theta)R^2(b,\varphi,\eta)}{R^3(b,\varphi,\theta)R(b,\varphi,\eta)}\sin(\eta-\theta).$$
Then,  by law products and composition laws in Lemma \ref{Lem-lawprod} we immediately deduce that
	\begin{align}\label{estimate partialthetaeta}
	\| \partial_{\eta}{f}_{r}\|_{q,H^s_{\varphi,\theta,\eta}}^{\gamma,\mathcal{O}}&\lesssim 1+\| r\|_{q,s+1}^{\gamma,\mathcal{O}},\\
	\| \partial_{\eta}d_rf_r[\rho_2]\|_{q,H^s_{\varphi,\theta,\eta}}^{\gamma,\mathcal{O}}&\lesssim \big(1+\| r\|_{q,s_0+1}^{\gamma,\mathcal{O}}\big)\| \rho_2\|_{q,s+1}^{\gamma,\mathcal{O}}+\big(1+\| r\|_{q,s+1}^{\gamma,\mathcal{O}}\big)\| \rho_2\|_{q,s_0+1}^{\gamma,\mathcal{O}}.\label{estimate partialthetaeta2}
	\end{align}
The following estimate on $\mathcal{I}_2$ can be obtained combining \eqref{estimate vr1PP}, \eqref{estimate vr1} and \eqref{estimate partialthetaeta} together with Lemma \ref{Lem-lawprod}-$(iv)$-$(v)$ and the smallness property on $r$.
\begin{equation}\label{estimateFlambdaP-0}
	\| \mathcal{I}_2\|_{q,s}^{\gamma,\mathcal{O}}\lesssim  \| \rho_2\|_{q,s}^{\gamma,\mathcal{O}}+\| r\|_{q,s+1}^{\gamma,\mathcal{O}}\| \rho_2\|_{q,s_0}^{\gamma,\mathcal{O}}.
\end{equation}
As for   $\mathcal{I}_1$ we argue in a similar way to Lemma \ref{lemma estimates L S}   to get
\begin{equation}\label{estimateFlambdaP-1}
	\| \mathcal{I}_1\|_{q,s}^{\gamma,\mathcal{O}}\lesssim  \| \rho_2\|_{q,s+1}^{\gamma,\mathcal{O}}+\| r\|_{q,s+1}^{\gamma,\mathcal{O}}\| \rho_2\|_{q,s_0+1}^{\gamma,\mathcal{O}}.
\end{equation}
Putting together \eqref{estimateFlambdaP-0} and  \eqref{estimateFlambdaP-1}  yields
\begin{equation}\label{e-dVr1}
	\|d_{r}V_{r}^1(r)[\rho_{2}]\|_{q,s}^{\gamma,\mathcal{O}}\lesssim  \| \rho_2\|_{q,s+1}^{\gamma,\mathcal{O}}+\| r\|_{q,s+1}^{\gamma,\mathcal{O}}\| \rho_2\|_{q,s_0+1}^{\gamma,\mathcal{O}}.
\end{equation}
Differentiating $V_r^2$ with respect to $r$ in the direction $\rho_2$ yields
\begin{align*}
	d_{r}V_{r}^2(r)[\rho_{2}](b,\varphi,\theta)&=-\bigintsss_{\mathbb{T}}\log\big(B_r(b,\varphi,\theta,\eta)\big)\partial_{\eta}\left(\tfrac{\rho_{2}(\varphi,\eta)R^2(b,\varphi,\theta)-3\rho_{2}(\varphi,\theta)R^2(b,\varphi,\eta)}{R^5(b,\varphi,\theta)R(b,\varphi,\eta)}\sin(\eta-\theta)\right)d\eta\\
	&\quad-\frac{1}{2}\bigintsss_{\mathbb{T}}\tfrac{\big(d_{r}B_{r}^2\big)[\rho_{2}](b,\varphi,\theta,\eta)}{B_{r}^2(b,\varphi,\theta,\eta)}\partial_{\eta}\Big(\tfrac{R(b,\varphi,\eta)}{R^3(b,\varphi,\theta)}\sin(\eta-\theta)\Big)d\eta.
\end{align*}
Arguing in a similar way as above we find 
\begin{equation}\label{e-dVr2}
	\|d_{r}V_{r}^2(r)[\rho_{2}]\|_{q,s}^{\gamma,\mathcal{O}}\lesssim\|\rho_2\|_{q,s}^{\gamma,\mathcal{O}}+\|r\|_{q,s}^{\gamma,\mathcal{O}}\|\rho_2\|_{q,s_0}^{\gamma,\mathcal{O}}.
\end{equation} 
	Putting together  \eqref{e-dVr0}, \eqref{e-dVr1} and \eqref{e-dVr2} gives
	\begin{equation}\label{secondpart}
		\|d_{r}V_{r}(r)[\rho_{2}]\|_{q,s}^{\gamma,\mathcal{O}}\lesssim\|\rho_2\|_{q,s+1}^{\gamma,\mathcal{O}}+\|r\|_{q,s+1}^{\gamma,\mathcal{O}}\|\rho_2\|_{q,s_0+1}^{\gamma,\mathcal{O}}.
	\end{equation}
	Therefore,   according to the law products in  Lemma \ref{Lem-lawprod}, \eqref{secondpart} and the smallness condition we obtain
	\begin{align*}\big\| \partial_{\theta}\big(d_{r}V_{r}(r)[\rho_{2}]\rho_{1}\big)\big\|_{q,s}^{\gamma,\mathcal{O}}&\leqslant\|d_{r}V_{r}(r)[\rho_{2}]\rho_{1}\|_{q,s+1}^{\gamma,\mathcal{O}}\\
		&\lesssim\big\| d_{r}V_{r}(r)[\rho_{2}]\big\|_{q,s+1}^{\gamma,\mathcal{O}}\big\| \rho_{1}\big\|_{q,s_0}^{\gamma,\mathcal{O}}+\big\| d_{r}V_{r}(r)[\rho_{2}]\big\|_{q,s_0}^{\gamma,\mathcal{O}}\big\| \rho_{1}\big\|_{q,s+1}^{\gamma,\mathcal{O}}\\
		&\lesssim\|\rho_{1}\|_{q,s_{0}}^{\gamma,\mathcal{O}}\|\rho_{2}\|_{q,s+2}^{\gamma,\mathcal{O}}+\| r\|_{q,s+2}^{\gamma,\mathcal{O}}\|\rho_{1}\|_{q,s_{0}}^{\gamma,\mathcal{O}}\|\rho_{2}\|_{q,s_{0}+1}^{\gamma,\mathcal{O}}+\|\rho_{1}\|_{q,s+1}^{\gamma,\mathcal{O}}\|\rho_{2}\|_{q,s_{0}+1}^{\gamma,\mathcal{O}}.
	\end{align*}
	Combining the latter estimate with \eqref{Second-diff},  \eqref{e-drL1} and \eqref{e-drS1} allows to get
	$$\| d_r^{2}X_{P}(r)[\rho_{1},\rho_{2}]\|_{q,s}^{\gamma,\mathcal{O}}\lesssim\|\rho_{1}\|_{q,s_{0}}^{\gamma,\mathcal{O}}\|\rho_{2}\|_{q,s+2}^{\gamma,\mathcal{O}}+\| r\|_{q,s+2}^{\gamma,\mathcal{O}}\|\rho_{1}\|_{q,s_{0}}^{\gamma,\mathcal{O}}\|\rho_{2}\|_{q,s_{0}+1}^{\gamma,\mathcal{O}}+\|\rho_{1}\|_{q,s+1}^{\gamma,\mathcal{O}}\|\rho_{2}\|_{q,s_{0}+1}^{\gamma,\mathcal{O}}.$$
	Using Sobolev embeddings we get the desired result. This concludes the proof of Lemma \ref{lemma estimates vector field XP}.
\end{proof}

Notice in particular that Lemma \ref{lemma estimates vector field XP}-(i) implies that there is no singlarity in $\varepsilon$ for the rescaled vector field $X_{P_{\varepsilon}}$ defined in \eqref{perturbed hamiltonian}.
Based on the previous lemma we obtain tame estimates for the Hamiltonian vector field 
$$X_{\mathcal{P}_{\varepsilon}}=(\partial_{I}\mathcal{P}_{\varepsilon},-\partial_{\vartheta}\mathcal{P}_{\varepsilon},\Pi_{\mathbb{S}}^{\perp}\partial_{\theta}\nabla_{z}\mathcal{P}_{\varepsilon})$$
defined by \eqref{cNP} and \eqref{hamiltonian vector filed associated with Hepsilon}. The proof can be done in a similar way to  \cite[Lem. 5.1]{BM18}. 
\begin{lem}\label{tame estimates for the vector field XmathcalPvarepsilon}
	{Let $(\gamma,q,s_{0},s)$ satisfy \eqref{initial parameter condition}.
		There exists $\varepsilon_0\in(0,1)$ such that if 
		$$\varepsilon\leqslant\varepsilon_0\quad\textnormal{and}\quad\|\mathfrak{I}\|_{q,s_{0}+2}^{\gamma,\mathcal{O}}\leqslant 1,$$ 
		then  the  perturbed Hamiltonian vector field $X_{\mathcal{P}_{\varepsilon}}$ satisfies the following tame estimates,
		\begin{enumerate}[label=(\roman*)]
			\item $\| X_{\mathcal{P}_{\varepsilon}}(i)\|_{q,s}^{\gamma,\mathcal{O}}\lesssim 1+\|\mathfrak{I}\|_{q,s+2}^{\gamma,\mathcal{O}}.$
			\item $\big\| d_{i}X_{\mathcal{P}_{\varepsilon}}(i)[\,\widehat{i}\,]\big\|_{q,s}^{\gamma,\mathcal{O}}\lesssim \|\,\widehat{i}\,\|_{q,s+2}^{\gamma,\mathcal{O}}+\|\mathfrak{I}\|_{q,s+2}^{\gamma,\mathcal{O}}\|\,\widehat{i}\,\|_{q,s_{0}+1}^{\gamma,\mathcal{O}}.$
			\item $\big\| d_{i}^{2}X_{\mathcal{P}_{\varepsilon}}(i)[\,\widehat{i},\widehat{i}\,]\big\|_{q,s}^{\gamma,\mathcal{O}}\lesssim \|\,\widehat{i}\,\|_{q,s+2}^{\gamma,\mathcal{O}}\|\,\widehat{i}\,\|_{q,s_{0}+1}^{\gamma,\mathcal{O}}+\|\mathfrak{I}\|_{q,s+2}^{\gamma,\mathcal{O}}\big(\|\,\widehat{i}\,\|_{q,s_{0}+1}^{\gamma,\mathcal{O}}\big)^{2}.$
	\end{enumerate}}
\end{lem}
\section{Construction of an approximate right inverse}\label{sec red lin opb}
In order to apply a modified Nash-Moser scheme, we need to construct an approximate right inverse
of the linearized operator associated to the functional $\mathcal{F}$, that is 
\begin{equation}\label{Lineqrized-op-F}
	d_{(i, \alpha)} {\mathcal F}(i_0, \alpha_0 )[\widehat \imath \,, \widehat \alpha ] =
	\omega \cdot \partial_\varphi \widehat \imath - d_i X_{H_\varepsilon^{\alpha_0}} ( i_0 (\varphi) ) [\widehat \imath ] - (\widehat \alpha,0, 0 ).
\end{equation} 
where $\mathcal{F}$ is  defined in \eqref{main function}, $\alpha_0:\mathcal{O}\to \mathbb{R}^d$ is a vector-valued function  and   $i_0=(\vartheta_0,I_0,z_0)$ is an arbitrary torus close to the flat one and satisfying the reversibility condition 
\begin{equation}\label{rev cond i0}
	\vartheta_0(-\varphi)=-\vartheta_0(\varphi),\quad I_0(-\varphi)=I_0(\varphi)\quad\textnormal{and}\quad z_0(-\varphi)=(\mathscr{S}z_0)(\varphi).
\end{equation}
For this aim, we may follow the procedure introduced in \cite{BB15} and slightly  simplified  in \cite[Sec. 6]{HHM21}. The main idea consists in conjugating \eqref{Lineqrized-op-F} by a linear diffeomorphism of the toroidal phase space $\mathbb{T}^d \times \mathbb{R}^d \times  L^2_{\perp}$ to a triangular system in the action-angles-normal variables up to small fast decaying error terms and terms vanishing at an exact solution. Then, to solve the triangular system we are led to almost invert the linearized operator in the normal directions, given by 
 \begin{equation}\label{linearized:normal-directons}
		\widehat{\mathcal{L}}_{\omega}:=\Pi_{\mathbb{S}_0}^{\perp}\Big(\omega\cdot\partial_{\varphi}-\partial_{\theta}( \partial_{z}\nabla_z H_\varepsilon^{\alpha_0}) (i_0(\varphi)) -\varepsilon\partial_{\theta}\mathcal{R}(\varphi)\Big)\Pi_{\mathbb{S}_0}^{\perp},
	\end{equation}
	where $H_\varepsilon^{\alpha_0}$ is given by \eqref{H alpha},
    \begin{align}
		\mathcal{R}(\varphi)&:=L_{2}^{\top}(\varphi)\partial_{I}\nabla_I\mathcal{P}_{\varepsilon}(i_{0}(\varphi))L_{2}(\varphi)+L_{2}^{\top}(\varphi)\partial_{z}\nabla_{I}\mathcal{P}_{\varepsilon}(i_{0}(\varphi))+\partial_{I}\nabla_{z}\mathcal{P}_{\varepsilon}(i_{0}(\varphi))L_{2}(\varphi),\label{Y-101}
	\end{align}
	 $\mathcal{P}_{\varepsilon}$ is  defined by \eqref{cNP} and  
\begin{equation}\label{l1l20}
	\begin{aligned}
		L_2(\phi) &:= - [(\partial_\vartheta \widetilde{z}_0)(\vartheta_0(\phi))]^\top \partial_\theta ^{-1},\quad
		\widetilde{z}_0 (\vartheta) &:= z_0 (\vartheta_0^{-1} (\vartheta)).
	\end{aligned}
\end{equation}
Here, for any linear operator $A\in{\mathcal L}({\mathbb R}^d, L^2_{\perp})$ the transposed operator $A^{\top}:L^2_{\perp} \to\mathbb{R}^d$ is defined through  the duality relation
\begin{align}\label{Def-dua}
	\forall \, u\in L^2_{\perp}\, ,\quad\forall\,  v\in\mathbb{R}^d,\quad\big\langle A^{\top}u,v\big\rangle_{\mathbb{R}^d} =\big\langle u,  Av\big\rangle_{L^2(\mathbb{T}^d)}.
\end{align}
We point out the presence of the remainder term due to the  linear change of variables performed to  decouple the dynamics of the action-angle components from the normal ones. For more details we refer the reader to \cite[Sec. 6]{HHM21}.

\subsection{Linearized operator in the normal direction}
Our main goal here is to explore the structure of the linear operator $\widehat{\mathcal{L}}_{\omega}$, introduced in \eqref{linearized:normal-directons}. We have the following result.
The following lemma describes the asymptotic structure of $\widehat{\mathcal{L}}_{\omega}$
around the equilibrium state, described in Lemma \ref{lemma general form of the linearized operator}.

\begin{prop}\label{prop hat L omega} Let $(\gamma,q,d,s_{0})$ satisfy \eqref{initial parameter condition}.
	Then, 
	the operator $\widehat{\mathcal{L}}_{\omega}$ defined in \eqref{linearized:normal-directons} takes the form
	\begin{equation}\label{hat-L-omega}
		\widehat{\mathcal{L}}_{\omega}=\Pi_{\mathbb{S}_0}^{\perp}\Big(\mathcal{L}_{\varepsilon r}-\varepsilon\partial_{\theta}\mathcal{R}\Big)\Pi_{\mathbb{S}_0}^{\perp},
	\end{equation}
	where
	\begin{enumerate}[label=(\roman*)]
		\item the operator  $\mathcal{L}_{\varepsilon r}$ is given by
		\begin{equation}\label{L eps r}
			\mathcal{L}_{\varepsilon r}:=\omega\cdot\partial_{\varphi}+\partial_{\theta}\big(V_{\varepsilon r}\cdot\big)+\partial_{\theta}\mathbf{L}_{\varepsilon r}-\partial_{\theta}\mathbf{S}_{\varepsilon r},
		\end{equation}
	with $V_{\varepsilon r},$ $\mathbf{L}_{\varepsilon r}$ and $\mathbf{S}_{\varepsilon r}$ defined by \eqref{Vr}, \eqref{mathbfLr} and \eqref{mathbfSr}.
		\item the function  $r$ is given by
		\begin{equation}\label{defr}
			r(\varphi,\cdot)=A(i_{0}(\varphi)),
		\end{equation}
		satisfies the following symmetry property
		\begin{equation}\label{sym r}
			r(-\varphi,-\theta)=r(\varphi,\theta)
		\end{equation}
		and the following estimates
		\begin{align}
			\| r\|_{q,s}^{\gamma,\mathcal{O}}&\lesssim 1+\|\mathfrak{I}_{0}\|_{q,s}^{\gamma,\mathcal{O}},\label{estimate r and mathfrakI0}
			\\
			\|\Delta_{12}r\|_{q,s}^{\gamma,\mathcal{O}}&\lesssim\|\Delta_{12}i\|_{q,s}^{\gamma,\mathcal{O}}+\|\Delta_{12}i\|_{q,s_0}^{\gamma,\mathcal{O}}\max_{j\in\{1,2\}}\|\mathfrak{I}_j\|_{q,s}^{\gamma,\mathcal{O}}. \label{control difference r}
		\end{align}
		\item the operator $\mathcal{R}$, defined in \eqref{Y-101}, is an integral operator with kernel $J$ satisfying the symmetry property
		\begin{equation}\label{symmetry for the kernel J}
			J(-\varphi,-\theta,-\eta)=J(\varphi,\theta,\eta)
		\end{equation}
		and the following estimates: for all $\ell\in\mathbb{N}$,
		\begin{align}
			\sup_{\eta\in\mathbb{T}}\|(\partial_{\theta}^{\ell}J)(\ast,\cdot,\centerdot,\eta+\centerdot)\|_{q,s}^{\gamma,\mathcal{O}}&\lesssim 1+\|\mathfrak{I}_{0}\|_{q,s+3+\ell}^{\gamma,\mathcal{O}},\label{estimate J}\\
			\sup_{\eta\in\mathbb{T}}\|\Delta_{12}(\partial_{\theta}^{\ell}J)(\ast,\cdot,\centerdot,\eta+\centerdot)\|_{q,s}^{\gamma,\mathcal{O}}&\lesssim\|\Delta_{12}i\|_{q,s+3+\ell}^{\gamma,\mathcal{O}}+\|\Delta_{12}i\|_{q,s_0+3}^{\gamma,\mathcal{O}}\max_{j\in\{1,2\}}\|\mathfrak{I}_j\|_{q,s+3+\ell}^{\gamma,\mathcal{O}}.\label{differences J} 
		\end{align}
	\end{enumerate}
	where $*,\cdot,\centerdot,$ denote successively the variables $b,\varphi,\theta$ and $\displaystyle\mathfrak{I}_{j}(\varphi)=i_{j}(\varphi)-(\varphi,0,0).$
\end{prop}
\begin{proof}
	From \eqref{H alpha}, \eqref{Hepsilon}, \eqref{definition of A action-angle-normal} and   \eqref{HEE}    we obtain 
	\begin{align*}
		( \partial_{z}\nabla_z H_\varepsilon^{\alpha_0}) (i_0(\varphi))& =  \mathrm{L}(b)\Pi_{\mathbb{S}_0}^{\perp}+\varepsilon\partial_{z}\nabla_{z}\mathcal{P}_{\varepsilon}(i_{0}(\varphi))\\
		& =  \mathrm{L}(b)\Pi_{\mathbb{S}_0}^{\perp}+\varepsilon\Pi_{\mathbb{S}_0}^{\perp}\partial_{r}\nabla_{r}P_{\varepsilon}(A(i_{0}(\varphi)))\\
		& =  \Pi_{\mathbb{S}_0}^{\perp}\partial_{r}\nabla_{r}\mathcal{H}_{\varepsilon}(A(i_{0}(\varphi)))\\
		& =  \Pi_{\mathbb{S}_0}^{\perp}\partial_{r}\nabla_{r}H(\varepsilon A(i_{0}(\varphi))).
	\end{align*}
	According to the general form of the linearized operator stated  in Lemma \ref{lemma general form of the linearized operator} one has
	$$-\partial_\theta ( \partial_{z}\nabla_z H_\varepsilon^{\alpha_0}) (i_0(\varphi))= \Pi_{\mathbb{S}_0}^{\perp}\big(\partial_{\theta}\big(V_{\varepsilon r}(b,\varphi,\centerdot )\cdot\big)+\partial_{\theta}\mathbf{L}_{\varepsilon r}-\partial_{\theta}\mathbf{S}_{\varepsilon r}\big)\Pi_{\mathbb{S}_0}^{\perp}.$$
	Inserting this identity into \eqref{linearized:normal-directons} gives \eqref{hat-L-omega}. 
The operator $\mathcal{R}(\varphi)$ in \eqref{Y-101} may be written as 
	\begin{align*}
		\mathcal{R}(\varphi)=\mathcal{R}_{1}(\varphi)+\mathcal{R}_{2}(\varphi)+\mathcal{R}_{3}(\varphi), \quad {\rm with} \quad \mathcal{R}_{1}(\varphi)&:=L_{2}^{\top}(\varphi)\partial_{I}\nabla_I\mathcal{P}_{\varepsilon}(i_{0}(\varphi))L_{2}(\varphi),\\ \mathcal{R}_{2}(\varphi)&:=L_{2}^{\top}(\varphi)\partial_{z}\nabla_{I}\mathcal{P}_{\varepsilon}(i_{0}(\varphi)),\\
		\mathcal{R}_{3}(\varphi)&:=\partial_{I}\nabla_{z}\mathcal{P}_{\varepsilon}(i_{0}(\varphi))L_{2}(\varphi).
	\end{align*}
Notice that  $\mathcal{R}_{1}(\varphi),$ $\mathcal{R}_{2}(\varphi)$ and  $\mathcal{R}_{3}(\varphi)$    have a finite-dimensional rank. In fact, from \eqref{l1l20} and \eqref{Def-dua} one may write 
	$$L_{2}(\varphi)[\rho]=\sum_{k=1}^{d}\big\langle L_{2}(\varphi)[\rho],\underline{e}_{k}\big\rangle_{\mathbb{R}^{d}}\,\underline{e}_{k}=\sum_{k=1}^{d}\big\langle\rho,L_{2}^{\top}(\varphi)[\underline{e}_{k}]\big\rangle_{L^{2}(\mathbb{T})}\,\underline{e}_{k},$$
	with $\displaystyle(\underline{e}_{k})_{k=1}^d$ being the canonical basis of $\mathbb{R}^d.$
	Hence
	$$\begin{array}{ll}
		\displaystyle \mathcal{R}_{1}(\varphi)[\rho]=\sum_{k=1}^{d}\big\langle\rho,L_{2}^{\top}(\varphi)[\underline{e}_{k}]\big\rangle_{L^{2}(\mathbb{T})}A_{1}(\varphi)[\underline{e}_{k}] & \quad\mbox{with }\quad A_{1}(\varphi)=L_{2}^{\top}(\varphi)\partial_{I}\nabla_I\mathcal{P}_{\varepsilon}(i_{0}(\varphi)),\\
		\displaystyle \mathcal{R}_{3}(\varphi)[\rho]=\sum_{k=1}^{d}\big\langle\rho,L_{2}^{\top}(\varphi)[\underline{e}_{k}]\big\rangle_{L^{2}(\mathbb{T})}A_{3}(\varphi)[\underline{e}_{k}] &\quad \mbox{with }\quad A_{3}(\varphi)=\partial_{I}\nabla_{z}\mathcal{P}_{\varepsilon}(i_{0}(\varphi)).
	\end{array}$$
	Analogously,  since $A_{2}(\varphi):=\partial_{z}\nabla_{I}\mathcal{P}_{\varepsilon}(i_{0}(\varphi)):L^2_{\perp}\rightarrow\mathbb{R}^{d},$ then we may write
	$$
	\mathcal{R}_{2}(\varphi)[\rho]=\sum_{k=1}^{d}\big\langle\rho,A_{2}^{\top}(\varphi)[\underline{e}_{k}]\big\rangle_{L^{2}\,(\mathbb{T})}L_{2}^{\top}(\varphi)[\underline{e}_{k}].
	$$
	By setting
	\begin{align*}
	g_{k,1}(\varphi,\theta)=g_{k,3}(\varphi,\theta)=\chi_{k,2}(\varphi,\theta):=L_{2}^{\top}(\varphi)[\underline{e}_{k}](\theta),\qquad g_{k,2}(\varphi,\theta):=A_{2}^{\top}(\varphi)[\underline{e}_{k}](\theta),\\
 \chi_{k,1}(\varphi,\theta):=A_{1}(\varphi)[\underline{e}_{k}](\theta),\qquad \chi_{k,3}(\varphi,\theta):=A_{3}(\varphi)[\underline{e}_{k}](\theta),
	\end{align*}
	we can see that   the operator $\mathcal{R}$ takes the integral form
	\begin{align*}\mathcal{R}\rho(\varphi,\theta)
		&=\int_{\mathbb{T}}\rho(\varphi,\eta)J(\varphi,\theta,\eta)d\eta,\quad\textnormal{with}\quad J(\varphi,\theta,\eta):=\sum_{k'=1}^{3}\sum_{k=1}^{d}g_{k,k'}(\varphi,\eta)\chi_{k,k'}(\varphi,\theta).
	\end{align*}
	The symmetry property \eqref{symmetry for the kernel J} is a consequence of the definition of $r$ and the reversibility condition \eqref{rev cond i0} imposed on the torus $i_{0}.$ The estimates \eqref{differences J},    \eqref{estimate J}, \eqref{estimate r and mathfrakI0} and  \eqref{control difference r} are straightforward and follow in a similar way to  Proposition 6.1 in \cite{HR21}. 
\end{proof}

\subsection{Diagonalization of the linearized operator in the normal directions}
This section is devoted to the reduction of the linearized operator $\widehat{\mathcal{L}}_{\omega}$, defined in \eqref{hat-L-omega}, to constant coefficients. This procedure is done in three steps. First, we reduce the operator $\mathcal{L}_{\varepsilon r}$ introduced in \eqref{L eps r} up to smoothing reminders. Then we study the action of the localization in the normal directions. Finally, we almost eliminate the remainders by using a KAM reduction procedure. 
We fix the following parameters. 
\begin{equation}\label{param}
	\begin{array}{ll} 
		s_{l}:=s_0+\tau_1q+\tau_1+2,& \overline{\mu}_2:=4\tau_1q+6\tau_1+3, \\
		\overline{s}_{l}:=s_l+\tau_2q+\tau_2, & \overline{s}_{h}:=\frac{3}{2}\overline{\mu}_{2}+s_{l}+1.
	\end{array}
\end{equation}
\subsubsection{Leading orders reduction}
In this section, we shall straighten  the transport part by using a suitable quasi-periodic symplectic change of variables and look at its conjugation action on the non-local terms. The reduction of the transport part is done by  a KAM iterative scheme.
Such technique is now well-developed  in \cite{BBMH18,BM21,BFM21-1,FGMP19,HHM21,HR21}. The result reads as follows.
\begin{prop}\label{reduction of the transport part}
	Let $(\gamma,q,d,\tau_{1},s_{0},\overline{\mu}_2,s_l,\overline{s}_h,S)$ satisfy \eqref{initial parameter condition}, \eqref{setting tau1 and tau2} and \eqref{param}. Let  $\upsilon\in\left(0,\frac{1}{q+2}\right].$ We set
	\begin{equation}\label{sigma1}
		\sigma_{1}=s_0+\tau_{1}q+2\tau_{1}+4\quad\textnormal{and}\quad\sigma_2=s_0+\sigma_{1}+3.
	\end{equation}
For any $(\mu_2,\mathtt{p},s_h)$ satisfying
	\begin{equation}\label{param-trans}
		\mu_{2}\geqslant \overline{\mu}_{2}, \quad\mathtt{p}\geqslant 0,\quad s_{h}\geqslant\max\left(\frac{3}{2}\mu_{2}+s_{l}+1,\overline{s}_{h}+\mathtt{p}\right),
	\end{equation}
	there exists $\varepsilon_{0}>0$ such that if
	\begin{equation}\label{small-2}
		\varepsilon\gamma^{-1}N_{0}^{\mu_{2}}\leqslant\varepsilon_{0}\quad\textnormal{and}\quad \|\mathfrak{I}_{0}\|_{q,s_{h}+\sigma_{2}}^{\gamma,\mathcal{O}}\leqslant 1,
	\end{equation}
then following assertions hold true.
\begin{enumerate}
	\item There exist
	$$V_{i_{0}}^{\infty}\in W^{q,\infty,\gamma }(\mathcal{O},\mathbb{C})\quad\mbox{ and }\quad\beta\in W^{q,\infty,\gamma }(\mathcal{O},H^{S})$$
	such that with $\mathscr{B}$ defined in \eqref{definition symplectic change of variables} one gets the following results.
	\begin{enumerate}[label=(\roman*)]
		\item The function $V_{i_{0}}^{\infty}$ satisfies the estimate:
		\begin{equation}\label{estimate r1}
			\| V_{i_{0}}^{\infty}-\tfrac{1}{2}\|_{q}^{\gamma ,\mathcal{O}}\lesssim \varepsilon.
		\end{equation}
		\item The transformations $\mathscr{B}^{\pm 1},\mathcal{B}^{\pm 1}, {\beta}$ and $\widehat{\beta}$ satisfy the following estimates: for all $s\in[s_{0},S]$, 
		\begin{align}\label{estimate on the first reduction operator and its inverse}
			\|\mathscr{B}^{\pm 1}\rho\|_{q,s}^{\gamma ,\mathcal{O}}+\|\mathcal{B}^{\pm 1}\rho\|_{q,s}^{\gamma ,\mathcal{O}}&\lesssim\|\rho\|_{q,s}^{\gamma ,\mathcal{O}}+\varepsilon\gamma ^{-1}\| \mathfrak{I}_{0}\|_{q,s+\sigma_{1}}^{\gamma ,\mathcal{O}}\|\rho\|_{q,s_{0}}^{\gamma ,\mathcal{O}}, \\
			\|\widehat{\beta}\|_{q,s}^{\gamma,\mathcal{O}}\lesssim\|\beta\|_{q,s}^{\gamma ,\mathcal{O}}&\lesssim \varepsilon\gamma ^{-1}\left(1+\| \mathfrak{I}_{0}\|_{q,s+\sigma_{1}}^{\gamma ,\mathcal{O}}\right).\label{estimate beta and r}
		\end{align}
	Moreover, $\beta$ and $\widehat{\beta}$ satisfy the following symmetry condition:
	\begin{equation}\label{symmetry for beta}
		\beta(\mu,-\varphi,-\theta)=-\beta(\mu,\varphi,\theta)\quad\textnormal{and}\quad\widehat{\beta}(\mu,-\varphi,-\theta)=-\widehat{\beta}(\mu,\varphi,\theta).
	\end{equation}
		\item Let $n\in\mathbb{N}$, then in the truncated Cantor set
		\begin{equation}\label{def Cantor trans}
			\mathcal{O}_{\infty,n}^{\gamma,\tau_{1}}(i_{0})=\bigcap_{(l,j)\in\mathbb{Z}^{d}\times\mathbb{Z}\setminus\{(0,0)\}\atop|l|\leqslant N_{n}}\left\lbrace(b,\omega)\in \mathcal{O}\quad\textnormal{s.t.}\quad\big|\omega\cdot l+jV_{i_{0}}^{\infty}(b,\omega)\big|>\tfrac{4\gamma^{\upsilon}\langle j\rangle}{\langle l\rangle^{\tau_{1}}}\right\rbrace,
		\end{equation}
		we have  the decomposition
		$$\mathfrak{L}_{\varepsilon r}:=\mathscr{B}^{-1}\mathcal{L}_{\varepsilon r}\mathscr{B}=\omega\cdot\partial_{\varphi}+V_{i_{0}}^{\infty}\partial_{\theta}+\partial_{\theta}\mathcal{K}_{b}\ast\cdot+\partial_{\theta}\mathfrak{R}_{\varepsilon r}+\mathtt{E}_{n}^{0},$$
		where $\mathcal{L}_{\varepsilon r}$ is given by \eqref{L eps r}, $\mathcal{K}_{b}$ is defined in Lemma \ref{lemma linearized operator at equilibrium} and $\mathtt{E}_{n}^{0}=\mathtt{E}_{n}^{0}(b,\omega,i_{0})$ is a linear operator satisfying
		\begin{equation}\label{estim En0 s0}
			\|\mathtt{E}_{n}^{0}\rho\|_{q,s_{0}}^{\gamma,\mathcal{O}}\lesssim\varepsilon N_{0}^{\mu_{2}}N_{n+1}^{-\mu_{2}}\|\rho\|_{q,s_{0}+2}^{\gamma,\mathcal{O}}.
		\end{equation}
	The operator $\mathfrak{R}_{\varepsilon r}$ is a real and reversibility preserving integral operator satisfying
	\begin{align}
	\forall s\in[s_0,S],\quad \max_{k\in\{0,1,2\}}\|\partial_{\theta}^{k}\mathfrak{R}_{\varepsilon r}\|_{\textnormal{\tiny{O-d}},q,s}^{\gamma,\mathcal{O}}&\lesssim\varepsilon\gamma^{-1}\left(1+\|\mathfrak{I}_{0}\|_{q,s+\sigma_{2}}^{\gamma,\mathcal{O}}\right).\label{e-frakR}
\end{align}
	\end{enumerate}
	\item Given two tori $i_{1}$ and $i_{2}$ both satisfying \eqref{small-2}, we have 
\begin{align}
	\|\Delta_{12}V_{i}^{\infty}\|_{q}^{\gamma,\mathcal{O}}&\lesssim\varepsilon\| \Delta_{12}i\|_{q,\overline{s}_{h}+2}^{\gamma,\mathcal{O}},  \label{difference ci}
\\
	\|\Delta_{12}\beta\|_{q,\overline{s}_{h}+\mathtt{p}}^{\gamma,\mathcal{O}}+\|\Delta_{12}\widehat\beta\|_{q,\overline{s}_{h}+\mathtt{p}}^{\gamma,\mathcal{O}}&\lesssim\varepsilon\gamma^{-1}\|\Delta_{12}i\|_{q,\overline{s}_{h}+\mathtt{p}+\sigma_{1}}^{\gamma,\mathcal{O}}.\label{difference beta}
\end{align}
In addition, we have 
\begin{align}
	\max_{k\in\{0,1\}}\|\Delta_{12}(\partial_{\theta}^k\mathfrak{R}_{\varepsilon r})\|_{\textnormal{\tiny{O-d}},q,\overline{s}_{h}+\mathtt{p}}^{\gamma,\mathcal{O}}&\lesssim\varepsilon\gamma^{-1}\|\Delta_{12}i\|_{q,\overline{s}_{h}+\mathtt{p}+\sigma_{2}}^{\gamma,\mathcal{O}}.\label{e-dfrakR}
\end{align}
\end{enumerate}
\end{prop}
\begin{proof} Notice that along the proof, to simplify the notation, we shall omit  the dependence with respect to the parameters $b,\omega$ kipping in mind that the functions appearing actually depend on them.
	We begin by setting 
	\begin{equation}\label{V0 and f0}
		V_0=\tfrac{1}{2}\quad\textnormal{and}\quad f_{0}(\varphi,\theta):=V_{\varepsilon r}(\varphi,\theta)-\tfrac{1}{2},
	\end{equation}
	with $V_{\varepsilon r}$ defined by \eqref{Vr}. According to \eqref{sym r} and \eqref{symVr}, one gets
	\begin{equation}\label{sym f0}
		f_{0}(-\varphi,-\theta)=f_{0}(\varphi,\theta).
	\end{equation} 
	Notice that according to \eqref{Vr}, \eqref{secondpart} and  Taylor formula, one has
	\begin{equation}\label{es-f0}
		\|f_{0}\|_{q,s}^{\gamma,\mathcal{O}}\lesssim\varepsilon\left(1+\|\mathfrak{I}_{0}\|_{q,s+1}^{\gamma,\mathcal{O}}\right).
	\end{equation}
	These properties allow to apply \cite[Prop. 6.2]{HR21}, whose proof is based on a KAM iterative scheme reduction of the perturbation term $f_0$ and construct $\beta$ and $V_{i_0}^{\infty}.$ In particular, for any $n\in\mathbb{N},$ we are able to construct a Cantor set $\mathcal{O}_{\infty,n}^{\gamma,\tau_1}(i_0)$ in the form \eqref{def Cantor trans} in which the following reduction holds
	\begin{equation}\label{reduc lin trans op}
		\mathscr{B}^{-1}\big(\omega\cdot\partial_{\varphi}+\partial_{\theta}\big(V_{\varepsilon r}\cdot\big)\big)\mathscr{B}=\omega\cdot\partial_{\varphi}+V_{i_{0}}^{\infty}\partial_{\theta}+\mathtt{E}_{n}^{0},
	\end{equation}
where $\mathtt{E}_{n}^{0}$ is an operator enjoying the decay property stated in \eqref{estim En0 s0}.	
	\noindent Using \eqref{reduc lin trans op}, \eqref{dcp Lr}, \eqref{dcp Sr} and Lemma \ref{algeb1}-(i), one obtains in the Cantor set $\mathcal{O}_{\infty,n}^{\gamma,\tau_{1}}(i_{0})$ the following decomposition
	\begin{align*}
		\mathscr{B}^{-1}\mathcal{L}_{\varepsilon r}\mathscr{B}&=\mathscr{B}^{-1}\big(\omega\cdot\partial_{\varphi}+\partial_{\theta}\left(V_{\varepsilon r}\cdot\right)\big)\mathscr{B}+\mathscr{B}^{-1}\partial_{\theta}\mathbf{L}_{\varepsilon r}\mathscr{B}-\mathscr{B}^{-1}\partial_{\theta}\mathbf{S}_{\varepsilon r}\mathscr{B}\\
		&=\omega\cdot\partial_{\varphi}+V_{i_{0}}^{\infty}\partial_{\theta}+\mathscr{B}^{-1}\partial_{\theta}\big(\mathcal{K}_{1,b}\ast\cdot\big)\mathscr{B}+\mathscr{B}^{-1}\partial_{\theta}\mathbf{L}_{\varepsilon r,1}\mathscr{B}\\
		&\quad-\mathscr{B}^{-1}\partial_{\theta}\big(\mathcal{K}_{2,b}\ast\cdot\big)\mathscr{B}-\mathscr{B}^{-1}\partial_{\theta}\mathbf{S}_{\varepsilon r,1}\mathscr{B}+\mathtt{E}_{n}^{0}\\
		&=\omega\cdot\partial_{\varphi}+V_{i_{0}}^{\infty}\partial_{\theta}+\partial_{\theta}\mathcal{K}_{1,b}\ast\cdot-\partial_{\theta}\mathcal{K}_{2,b}\ast\cdot+\partial_{\theta}\mathfrak{R}_{\varepsilon r}+\mathtt{E}_{n}^{0},
	\end{align*}
	wehre
\begin{equation}\label{frk Rr}
	\mathfrak{R}_{\varepsilon r}:=\Big[\mathcal{B}^{-1}\big(\mathcal{K}_{1,b}\ast\cdot\big)\mathscr{B}-\mathcal{K}_{1,b}\ast\cdot\Big]-\Big[\mathcal{B}^{-1}\big(\mathcal{K}_{2,b}\ast\cdot\big)\mathscr{B}-\mathcal{K}_{2,b}\ast\cdot\Big]+\mathcal{B}^{-1}\mathbf{L}_{\varepsilon r,1}\mathscr{B}-\mathcal{B}^{-1}\mathbf{S}_{\varepsilon r,1}\mathscr{B}.
\end{equation}
Direct computations using \eqref{mathcalK1b} lead to
$$\mathcal{B}^{-1}(\mathcal{K}_{1,b}\ast(\mathscr{B}\rho))(\varphi,\theta)=\int_{\mathbb{T}}\rho(\varphi,\eta)\log(\mathscr{A}_{\widehat{\beta}}(\varphi,\theta,\eta))d\eta,$$
where
$$\mathscr{A}_{\widehat{\beta}}(\varphi,\theta,\eta):=\left|\sin\left(\tfrac{\eta-\theta}{2}+\widehat{h}(\varphi,\theta,\eta)\right)\right|\quad\textnormal{with}\quad \widehat{h}(\varphi,\theta,\eta):=\tfrac{\widehat{\beta}(\varphi,\theta)-\widehat{\beta}(\varphi,\eta)}{2}\cdot$$
 Using elementary trigonometric identities, we can write
$$\mathscr{A}_{\widehat{\beta}}(\varphi,\theta,\eta)=\left|\sin\left(\tfrac{\eta-\theta}{2}\right)\right|v_{\widehat{\beta}}(\varphi,\theta,\eta)\quad\textnormal{with}\quad v_{\widehat{\beta}}(\varphi,\theta,\eta):=\cos\big(\widehat{h}(\varphi,\theta,\eta)\big)+\tfrac{\sin\left(\widehat{h}(\varphi,\theta,\eta)\right)}{\tan\big(\tfrac{\eta-\theta}{2}\big)}\cdot$$
In view of \eqref{symmetry for beta}, one finds that $v_{\widehat{\beta}}$ enjoys the following symmetry property,
\begin{equation}\label{sym vbeta}
	v_{\widehat{\beta}}(-\varphi,-\theta,-\eta)=v_{\widehat{\beta}}(\varphi,\theta,\eta).
\end{equation}
Using the morphism property of the logarithm, one gets
$$\Big[\mathcal{B}^{-1}\big(\mathcal{K}_{1,b}\ast\rho\big)\mathscr{B}-\mathcal{K}_{1,b}\ast \rho\big] (\varphi,\theta)=\int_{\mathbb{T}}\rho(\varphi,\eta)\mathbb{K}_{\widehat{\beta},2}(\varphi,\theta,\eta)d\eta
$$
where 
\begin{equation}\label{def Kbeta2}
	\mathbb{K}_{\widehat{\beta},2}(\varphi,\theta,\eta):=\log\big(v_{\widehat{\beta}}(\varphi,\theta,\eta)\big).
\end{equation}
Notice that \eqref{def Kbeta2} and \eqref{sym vbeta} imply
\begin{equation}\label{sym bbK2}
	\mathbb{K}_{\widehat{\beta},2}(-\varphi,-\theta,-\eta)=\mathbb{K}_{\widehat{\beta},2}(\varphi,\theta,\eta)\in\mathbb{R}.
\end{equation}
Hence, we deduce from Lemma \ref{lemma symmetry and reversibility} that $\mathcal{B}^{-1}\big(\mathcal{K}_{1,b}\ast\cdot\big)\mathscr{B}-\mathcal{K}_{1,b}\ast\cdot$ is a real and reversibility preserving Toeplitz in time operator.
Writting 
$$v_{\widehat{\beta}}(\varphi,\theta,\eta)=1+\Big(\cos\big(\widehat{h}(\varphi,\theta,\eta)\big)-1\Big)+\tfrac{\sin\left(\widehat{h}(\varphi,\theta,\eta)\right)}{\tan\big(\tfrac{\eta-\theta}{2}\big)},$$
one finds, by Lemma \ref{Lem-lawprod}-(v), Lemma \ref{cheater lemma} and \eqref{estimate beta and r},
\begin{align*}
	\|v_{\widehat{\beta}}-1\|_{q,H_{\varphi,\theta,\eta}^{s}}^{\gamma,\mathcal{O}}&\lesssim\|\widehat{\beta}\|_{q,s+1}^{\gamma,\mathcal{O}}\\
	&\lesssim\varepsilon\gamma^{-1}\left(1+\|\mathfrak{I}_{0}\|_{q,s+1+\sigma_{1}}^{\gamma,\mathcal{O}}\right).
\end{align*}
Moreover, by \eqref{difference beta} and the Mean Value Theorem  (applied with $\mathtt{p}$ replaced by $\mathtt{p}+s_0+2$), we find
\begin{align*}
	\|\Delta_{12}v_{\widehat{\beta}}\|_{q,H_{\varphi,\theta,\eta}^{\overline{s}_h+\mathtt{p}+s_0+1}}^{\gamma,\mathcal{O}}&\lesssim\|\Delta_{12}\widehat{\beta}\|_{q,\overline{s}_h+\mathtt{p}+s_0+2}^{\gamma,\mathcal{O}}\\
	&\lesssim\varepsilon\gamma^{-1}\|\Delta_{12}i\|_{q,\overline{s}_h+\mathtt{p}+s_0+2+\sigma_{1}}.
\end{align*}
In a similar way, we deduce that
\begin{align}
	\|\mathbb{K}_{\widehat{\beta},2}\|_{q,H_{\varphi,\theta,\eta}^{s}}^{\gamma,\mathcal{O}}\lesssim\varepsilon\gamma^{-1}\left(1+\|\mathfrak{I}_{0}\|_{q,s+1+\sigma_{1}}^{\gamma,\mathcal{O}}\right), \label{e-Kb2h}\\
	\|\Delta_{12}\mathbb{K}_{\widehat{\beta},2}\|_{q,H_{\varphi,\theta,\eta}^{\overline{s}_h+\mathtt{p}+s_0+1}}^{\gamma,\mathcal{O}}\lesssim\varepsilon\gamma^{-1}\|\Delta_{12}i\|_{q,\overline{s}_h+\mathtt{p}+s_0+2+\sigma_{1}}^{\gamma,\mathcal{O}}.\label{ed-Kb2h}
\end{align}
In view of Lemma \ref{lemma symmetry and reversibility} we get, from \eqref{e-Kb2h} and \eqref{ed-Kb2h},
\begin{align}
	\max_{k\in\{0,1,2\}}\Big\|\partial_{\theta}^{k}\Big[\mathcal{B}^{-1}\big(\mathcal{K}_{1,b}\ast\cdot\big)\mathscr{B}-\mathcal{K}_{1,b}\ast\cdot\Big]\Big\|_{\textnormal{\tiny{O-d}},q,s}^{\gamma,\mathcal{O}}\lesssim\varepsilon\gamma^{-1}\left(1+\|\mathfrak{I}_0\|_{q,s+s_0+3+\sigma_{1}}^{\gamma,\mathcal{O}}\right),\label{e-cLr2}
\\
	\max_{k\in\{0,1\}}\Big\|\Delta_{12}\partial_{\theta}^{k}\Big[\mathcal{B}^{-1}\big(\mathcal{K}_{1,b}\ast\cdot\big)\mathscr{B}-\mathcal{K}_{1,b}\ast\cdot\Big]\Big\|_{\textnormal{\tiny{O-d}},q,\overline{s}_h+\mathtt{p}}^{\gamma,\mathcal{O}}\lesssim\varepsilon\gamma^{-1}\|\Delta_{12}i\|_{q,\overline{s}_h+\mathtt{p}+s_0+2+\sigma_{1}}^{\gamma,\mathcal{O}}.\label{ed-cLr2}
\end{align}
According to \eqref{mathcalK2b}, one finds
\begin{align*}
	\mathcal{B}^{-1}\big(\mathcal{K}_{2,b}\ast(\mathscr{B}\rho)\big)(\varphi,\theta)-\mathcal{K}_{2,b}\ast\rho(\varphi,\theta)=\int_{\mathbb{T}}\rho(\varphi,\eta)\mathscr{K}_{\widehat{\beta},2}(\varphi,\theta,\eta)d\eta,
\end{align*}
with
\begin{align*}
	\mathscr{K}_{\widehat{\beta},2}(\varphi,\theta,\eta)&:=\tfrac{1}{2}\Big[\log\big(1+b^4-2b^2\cos(\eta-\theta+\widehat{h}(\varphi,\theta,\eta))\big)-\log\big(1+b^4-2b^2\cos(\eta-\theta)\big)\Big]\\
	&=\tfrac{1}{2}\log\left(\tfrac{1+b^4-2b^2\cos\big(\eta-\theta+\widehat{h}(\varphi,\theta,\eta)\big)}{1+b^4-2b^2\cos(\eta-\theta)}\right).
\end{align*}
From \eqref{symmetry for beta}, we deduce that
\begin{equation}\label{sym bbK3}
	\mathscr{K}_{\widehat{\beta},2}(-\varphi,-\theta,-\eta)=\mathscr{K}_{\widehat{\beta},2}(\varphi,\theta,\eta)\in\mathbb{R}.
\end{equation}
It follows from Lemma \ref{lemma symmetry and reversibility} that $\mathcal{B}^{-1}\big(\mathcal{K}_{2,b}\ast\cdot\big)\mathscr{B}-\mathcal{K}_{2,b}\ast\cdot$ is a real and reversibility preserving Toeplitz in time operator. 
Arguying as for \eqref{e-scrK} and using \eqref{estimate beta and r}, we obtain
\begin{align}\label{e-Kb3h}
	\|\mathscr{K}_{\widehat{\beta},2}\|_{q,H_{\varphi,\theta,\eta}^{s}}^{\gamma,\mathcal{O}}&\lesssim\|\widehat{\beta}\|_{q,s}^{\gamma,\mathcal{O}}\nonumber\\
	&\lesssim\varepsilon\gamma^{-1}\left(1+\|\mathfrak{I}_0\|_{q,s+\sigma_{1}}^{\gamma,\mathcal{O}}\right).
\end{align}
Using Mean Value theorem, applied with $\mathtt{p}$ replaced by $\mathtt{p}+s_0+1$, one also gets by \eqref{difference beta}
\begin{align}\label{ed-Kb3h}
	\|\Delta_{12}\mathscr{K}_{\widehat{\beta},2}\|_{q,H_{\varphi,\theta,\eta}^{\overline{s}_h+\mathtt{p}+s_0+1}}^{\gamma,\mathcal{O}}&\lesssim\|\Delta_{12}\widehat{\beta}\|_{q,\overline{s}_h+\mathtt{p}+s_0+1}^{\gamma,\mathcal{O}}\nonumber\\
	&\lesssim\varepsilon\gamma^{-1}\|\Delta_{12}i\|_{q,\overline{s}_h+\mathtt{p}+s_0+1+\sigma_1}^{\gamma,\mathcal{O}}.
\end{align}
Consequently, in view of Lemma \ref{lemma symmetry and reversibility}, we get from \eqref{e-Kb3h}
\begin{equation}\label{e-cLr3}
	\max_{k\in\{0,1,2\}}\Big\|\partial_{\theta}^{k}\Big[\mathcal{B}^{-1}\big(\mathcal{K}_{2,b}\ast\cdot\big)\mathscr{B}-\mathcal{K}_{2,b}\ast\cdot\Big]\Big\|_{\textnormal{\tiny{O-d}},q,s}^{\gamma,\mathcal{O}}\lesssim\varepsilon\gamma^{-1}\left(1+\|\mathfrak{I}_0\|_{q,s+s_0+2+\sigma_{1}}^{\gamma,\mathcal{O}}\right)
\end{equation}
and from \eqref{ed-Kb3h}
\begin{equation}\label{ed-cLr3}
	\max_{k\in\{0,1\}}\Big\|\Delta_{12}\partial_{\theta}^{k}\Big[\mathcal{B}^{-1}\big(\mathcal{K}_{2,b}\ast\cdot\big)\mathscr{B}-\mathcal{K}_{2,b}\ast\cdot\Big]\Big\|_{\textnormal{\tiny{O-d}},q,\overline{s}_h+\mathtt{p}}^{\gamma,\mathcal{O}}\lesssim\varepsilon\gamma^{-1}\|\Delta_{12}i\|_{q,\overline{s}_h+\mathtt{p}+s_0+1+\sigma_{1}}^{\gamma,\mathcal{O}}.
\end{equation}
Next, putting together \eqref{symmetry kernel K1}, \eqref{symmetry for beta} and   Lemma \ref{lem CVAR kernel}, we infer that $\mathcal{B}^{-1}\mathbf{L}_{\varepsilon r,1}\mathscr{B}$ is a real and reversibility preserving Toeplitz in time operator. Moreover, we obtain from \eqref{e-odsBtB} in  Lemma \ref{lem CVAR kernel}, \eqref{estimate beta and r}, \eqref{estimate kernel mathbbK1}, \eqref{estimate r and mathfrakI0} and the smallness condition \eqref{small-2},
\begin{align}\label{e-cLr1}
	\max_{k\in\{0,1,2\}}\|\partial_{\theta}^{k}\mathcal{B}^{-1}\mathbf{L}_{\varepsilon r,1}\mathscr{B}\|_{\textnormal{\tiny{O-d}},q,s}^{\gamma,\mathcal{O}}&\lesssim\|\mathbb{K}_{\varepsilon r,1}\|_{q,H_{\varphi,\theta,\eta}^{s+s_0+2}}^{\gamma,\mathcal{O}}+\|\beta\|_{q,s+s_0+2}^{\gamma,\mathcal{O}}\|\mathbb{K}_{\varepsilon r,1}\|_{q,H_{\varphi,\theta,\eta}^{s_0}}^{\gamma,\mathcal{O}}\nonumber\\
	&\lesssim\varepsilon\gamma^{-1}\left(1+\|\mathfrak{I}_0\|_{q,s+s_0+2+\sigma_{1}}^{\gamma,\mathcal{O}}\right).
\end{align} 
Applying Lemma \ref{Lem-lawprod}, we get from \eqref{expression of Ar with vr1}, Lemma \ref{cheater lemma} and \eqref{control difference r},
\begin{align*}
	\|\Delta_{12}\mathbb{K}_{\varepsilon r,1}\|_{q,H_{\varphi,\theta,\eta}^s}^{\gamma,\mathcal{O}}&\lesssim\|\Delta_{12}v_{\varepsilon r,1}\|_{q,H_{\varphi,\theta,\eta}^s}^{\gamma,\mathcal{O}}\\
	&\lesssim\varepsilon\gamma^{-1}\Big(\|\Delta_{12}i\|_{q,s+1}^{\gamma,\mathcal{O}}+\|\Delta_{12}i\|_{q,s_0}^{\gamma,\mathcal{O}}\max_{j\in\{1,2\}}\|\mathfrak{I}_j\|_{q,s+1}^{\gamma,\mathcal{O}}\Big).
\end{align*}
Added to Lemma \ref{lem CVAR kernel}-(ii), \eqref{difference beta}, \eqref{estimate beta and r}, \eqref{estimate kernel mathbbK1} and \eqref{small-2}, we infer 
 \begin{align}\label{ed-cLr1}
 	\max_{k\in\{0,1\}}\|\Delta_{12}\partial_{\theta}^{k}\mathcal{B}^{-1}\mathbf{L}_{\varepsilon r,1}\mathscr{B}\|_{\textnormal{\tiny{O-d}},q,\overline{s}_h+\mathtt{p}}^{\gamma,\mathcal{O}}&\lesssim\varepsilon\gamma^{-1}\|\Delta_{12}i\|_{q,\overline{s}_h+\mathtt{p}+s_0+1+\sigma_{1}}^{\gamma,\mathcal{O}}.
 \end{align}
The next task is to estimate the term $\mathcal{B}^{-1}\mathbf{S}_{\varepsilon r,1}\mathscr{B}$ in \eqref{frk Rr}.  Note that \eqref{sym scrK}, \eqref{symmetry for beta} and Lemma \ref{lem CVAR kernel} imply that $\mathcal{B}^{-1}\mathbf{S}_{\varepsilon r,1}\mathscr{B}$ is a real and reversibility preserving Toeplitz in time operator. In addition, Lemma \ref{lem CVAR kernel} together with the estimates \eqref{e-scrK}, \eqref{estimate r and mathfrakI0} and \eqref{estimate beta and r} give
\begin{align}\label{e-cLr4}
	\max_{k\in\{0,1,2\}}\|\partial_{\theta}^{k}\mathcal{B}^{-1}\mathbf{S}_{\varepsilon r,1}\mathscr{B}\|_{\textnormal{\tiny{O-d}},q,s}^{\gamma,\mathcal{O}}&\lesssim\|\mathscr{K}_{\varepsilon r,1}\|_{q,H_{\varphi,\theta,\eta}^{s+2}}^{\gamma,\mathcal{O}}+\|\beta\|_{q,s+2}^{\gamma,\mathcal{O}}\|\mathscr{K}_{\varepsilon r,1}\|_{q,H_{\varphi,\theta,\eta}^{s_0}}^{\gamma,\mathcal{O}}\nonumber\\
	&\lesssim\varepsilon\gamma^{-1}\left(1+\|\mathfrak{I}_0\|_{q,s+s_0+2+\sigma_{1}}^{\gamma,\mathcal{O}}\right).
\end{align} 
Applying Lemma \ref{Lem-lawprod}, we get from \eqref{scrK} and \eqref{control difference r},
\begin{align*}
	\|\Delta_{12}\mathscr{K}_{\varepsilon r,1}\|_{q,H_{\varphi,\theta,\eta}^s}^{\gamma,\mathcal{O}}&\lesssim\varepsilon\gamma^{-1}\Big(\|\Delta_{12}i\|_{q,s}^{\gamma,\mathcal{O}}+\|\Delta_{12}i\|_{q,s_0}^{\gamma,\mathcal{O}}\max_{i\in\{1,2\}}\|\mathfrak{I}_i\|_{q,s}^{\gamma,\mathcal{O}}\Big).
\end{align*}
Then, combining Lemma \ref{lem CVAR kernel}-(ii), \eqref{difference beta}, \eqref{estimate beta and r}, \eqref{e-scrK} and \eqref{small-2}, we get
\begin{align}\label{ed-cLr4}
	\max_{k\in\{0,1\}}\|\Delta_{12}\partial_{\theta}^{k}\mathcal{B}^{-1}\mathbf{S}_{\varepsilon r,1}\mathscr{B}\|_{\textnormal{\tiny{O-d}},q,\overline{s}_h+\mathtt{p}}^{\gamma,\mathcal{O}}&\lesssim\varepsilon\gamma^{-1}\|\Delta_{12}i\|_{q,\overline{s}_h+\mathtt{p}+s_0+1+\sigma_{1}}^{\gamma,\mathcal{O}}.
\end{align}
In view of \eqref{frk Rr}, Lemma \ref{lem CVAR kernel} and the previous computations, we conclude that  $\mathfrak{R}_{\varepsilon r}$ is a real and reversibility preserving toeplitz in time integral operator which satisfies, by \eqref{e-cLr2}, \eqref{e-cLr3}, \eqref{e-cLr1} and \eqref{e-cLr4},
\begin{align*}
	\max_{k\in\{0,1,2\}}\|\partial_{\theta}^{k}\mathfrak{R}_{\varepsilon r}\|_{\textnormal{\tiny{O-d}},q,s}^{\gamma,\mathcal{O}}&\lesssim\varepsilon\gamma^{-1}\left(1+\|\mathfrak{I}_0\|_{q,s+s_0+3+\sigma_{1}}^{\gamma,\mathcal{O}}\right).
\end{align*}
In addition, combining \eqref{ed-cLr2}, \eqref{ed-cLr3}, \eqref{ed-cLr1} and \eqref{ed-cLr4} yields
\begin{align*}
	\max_{k\in\{0,1\}}\|\Delta_{12}\partial_{\theta}^{k}\mathfrak{R}_{\varepsilon r}\|_{\textnormal{\tiny{O-d}},q,\overline{s}_{h}+\mathtt{p}}^{\gamma,\mathcal{O}}&\lesssim\varepsilon\gamma^{-1}\|\Delta_{12}i\|_{q,\overline{s}_{h}+\mathtt{p}+s_0+2+\sigma_{1}}^{\gamma,\mathcal{O}}.
\end{align*}
This ends the proof of Proposition \ref{reduction of the transport part}.
\end{proof}
\subsubsection{Projection in the normal directions}
In this section, we study the effects of the localization in the normal directions for the reduction of the transport part. For that purpose, we consider the localized quasi-periodic symplectic change of coordinates defined by
$$\mathscr{B}_{\perp}=\Pi_{\mathbb{S}_0}^{\perp}\mathscr{B}\Pi_{\mathbb{S}_{0}}^{\perp}.$$
Then, the main result of this section reads as follows.
\begin{prop}\label{projection in the normal directions}
	Let $(\gamma,q,d,\tau_{1},s_{0},s_{h},\overline{s}_{h},\mathtt{p},S)$ satisfy the assumptions \eqref{initial parameter condition}, \eqref{setting tau1 and tau2} and \eqref{param-trans}.\\
	There exist $\varepsilon_0>0$ and $\sigma_{3}=\sigma_{3}(\tau_{1},q,d,s_{0})\geqslant\sigma_{2}$ such that if 
	\begin{equation}\label{small-3}
		\varepsilon\gamma^{-1}N_0^{\mu_2}\leqslant\varepsilon\quad\textnormal{and}\quad\|\mathfrak{I}_0\|_{q,s_h+\sigma_3}^{\gamma,\mathcal{O}}\leqslant 1,
	\end{equation}
	then the following assertions hold true.
	\begin{enumerate}[label=(\roman*)]
		\item The operators $\mathscr{B}_{\perp}^{\pm 1}$ satisfy the following estimate
		\begin{equation}\label{estimate for mathscrBperp and its inverse}
			\|\mathscr{B}_{\perp}^{\pm 1}\rho\|_{q,s}^{\gamma ,\mathcal{O}}\lesssim\|\rho\|_{q,s}^{\gamma ,\mathcal{O}}+\varepsilon\gamma ^{-1}\| \mathfrak{I}_{0}\|_{q,s+\sigma_{3}}^{\gamma ,\mathcal{O}}\|\rho\|_{q,s_{0}}^{\gamma ,\mathcal{O}}.
		\end{equation}
		\item For any $n\in\mathbb{N}^{*},$ in the Cantor set $\mathcal{O}_{\infty,n}^{\gamma,\tau_{1}}(i_{0})$ introduced in Proposition \ref{reduction of the transport part}, we have 
		\begin{align*}
			\mathscr{B}_{\perp}^{-1}\widehat{\mathcal{L}}_{\omega}\mathscr{B}_{\perp}&=\big(\omega\cdot\partial_{\varphi}+V_{i_{0}}^{\infty}\partial_{\theta}+\partial_{\theta}\mathcal{K}_{b}\ast\cdot\big)\Pi_{\mathbb{S}_0}^{\perp}+\mathscr{R}_{0}+\mathtt{E}_{n}^{1}\\
			&:=\omega\cdot\partial_{\varphi}\Pi_{\mathbb{S}_0}^{\perp}+\mathscr{D}_{0}+\mathscr{R}_{0}+\mathtt{E}_{n}^{1}\\
			&:=\mathscr{L}_{0}+\mathtt{E}_{n}^{1},
		\end{align*}
		where $\mathscr{R}_{0}=\Pi_{\mathbb{S}_0}^\perp \mathscr{R}_{0}\Pi_{\mathbb{S}_0}^\perp$ is reversible and $\mathscr{D}_{0}=\Pi_{\mathbb{S}_0}^\perp \mathscr{D}_{0}\Pi_{\mathbb{S}_0}^\perp$ is a reversible Fourier multiplier operator given by 
		$$
		\forall (l,j)\in \mathbb{Z}^{d}\times\mathbb{S}_{0}^{c},\quad  \mathscr{D}_{0}\mathbf{e}_{l,j}=\ii\,\mu_{j}^{0}\,\mathbf{e}_{l,j},$$
		with 
		\begin{equation}\label{def mu0 and r1}
			\mu_{j}^{0}(b,\omega,i_{0})=\Omega_{j}(b)+jr^{1}(b,\omega,i_{0})\quad\mbox{ and }\quad r^{1}(b,\omega,i_{0})=V_{i_{0}}^{\infty}(b,\omega)-\tfrac{1}{2}
		\end{equation}
		and such that
		\begin{equation}\label{differences mu0}
			\|r^1\|_{q}^{\gamma,\mathcal{O}}\lesssim \varepsilon \quad\textnormal{and}\quad \|\Delta_{12}r^{1}\|_{q}^{\gamma,\mathcal{O}}\lesssim \varepsilon \| \Delta_{12}i\|_{q,\overline{s}_{h}+2}^{\gamma,\mathcal{O}}.
		\end{equation}
		\item The operator $\mathtt{E}_{n}^{1}$ satisfies the following estimate
		\begin{equation}\label{En1-s0}
			\|\mathtt{E}_{n}^{1}\rho\|_{q,s_{0}}^{\gamma,\mathcal{O}}\lesssim \varepsilon N_{0}^{\mu_{2}}N_{n+1}^{-\mu_{2}}\|\rho\|_{q,s_{0}+2}^{\gamma,\mathcal{O}}.
		\end{equation}
		\item The operator $\mathscr{R}_{0}$ is a real and reversible Toeplitz in time operator satisfying 
		\begin{equation}\label{estimate mathscr R in off diagonal norm}
			\forall s\in [s_{0},S],\quad \max_{k\in\{0,1\}}\|\partial_{\theta}^{k}\mathscr{R}_{0}\|_{\textnormal{\tiny{O-d}},q,s}^{\gamma,\mathcal{O}}\lesssim\varepsilon\gamma^{-1}\left(1+\| \mathfrak{I}_{0}\|_{q,s+\sigma_{3}}^{\gamma,\mathcal{O}}\right)
		\end{equation}
		and
		\begin{equation}\label{estimate differences mathscr R in off diagonal norm}
			\|\Delta_{12}\mathscr{R}_{0}\|_{\textnormal{\tiny{O-d}},q,\overline{s}_{h}+\mathtt{p}}^{\gamma,\mathcal{O}}\lesssim\varepsilon\gamma^{-1}\| \Delta_{12}i\|_{q,\overline{s}_{h}+\mathtt{p}+\sigma_{3}}^{\gamma,\mathcal{O}}.
		\end{equation}
		\item Furthermore the operator $\mathscr{L}_{0}$ satisfies
		\begin{equation}\label{Taptap1}
			\forall s\in[s_{0},S],\quad\|\mathscr{L}_{0}\rho\|_{q,s}^{\gamma ,\mathcal{O}}\lesssim\|\rho\|_{q,s+1}^{\gamma ,\mathcal{O}}+\varepsilon\gamma ^{-1}\| \mathfrak{I}_{0}\|_{q,s+\sigma_{3}}^{\gamma ,\mathcal{O}}\|\rho\|_{q,s_{0}}^{\gamma ,\mathcal{O}}.
		\end{equation}
	\end{enumerate}
\end{prop}
\begin{proof}
	\textbf{(i)} Follows from \eqref{estimate on the first reduction operator and its inverse} and Lemma \ref{Lem-lawprod}-(ii).\\
	\textbf{(ii)} From \eqref{hat-L-omega} and the decomposition $\textnormal{Id}=\Pi_{\mathbb{S}_{0}}+\Pi_{\mathbb{S}_{0}}^{\perp}$  we write
	\begin{align*}
		\mathscr{B}_{\perp}^{-1}\widehat{\mathcal{L}}_{\omega}\mathscr{B}_{\perp}&=\mathscr{B}_{\perp}^{-1}\Pi_{\mathbb{S}_{0}}^{\perp}(\mathcal{L}_{\varepsilon r}-\varepsilon\partial_{\theta}\mathcal{R})\mathscr{B}_{\perp}
		\\
		&=\mathscr{B}_{\perp}^{-1}\Pi_{\mathbb{S}_{0}}^{\perp}\mathcal{L}_{\varepsilon r}\mathscr{B}\Pi_{\mathbb{S}_{0}}^{\perp}-\mathscr{B}_{\perp}^{-1}\Pi_{\mathbb{S}_{0}}^{\perp}\mathcal{L}_{\varepsilon r}\Pi_{\mathbb{S}_{0}}\mathscr{B}\Pi_{\mathbb{S}_{0}}^{\perp}-\varepsilon\mathscr{B}_{\perp}^{-1}\Pi_{\mathbb{S}_{0}}^{\perp}\partial_{\theta}\mathcal{R}\mathscr{B}_{\perp}.
	\end{align*}
	According to the definitions of $\mathfrak{L}_{\varepsilon r}$ and $\mathcal{L}_{\varepsilon r}$ seen  in Proposition \ref{reduction of the transport part} and in Lemma \ref{lemma general form of the linearized operator} and using \eqref{dcp Lr}, \eqref{dcp Sr} and \eqref{mathcalKb}, one has in the Cantor set $\mathcal{O}_{\infty,n}^{\gamma,\tau_{1}}(i_{0})$
	$$
	\mathcal{L}_{\varepsilon r}\mathscr{B}=\mathscr{B}\mathfrak{L}_{\varepsilon r}\quad\hbox{and}\quad
	\mathcal{L}_{\varepsilon r}=\omega\cdot\partial_{\varphi}+\partial_{\theta}\left(V_{\varepsilon r}\cdot\right)+\partial_{\theta}\mathcal{K}_{b}\ast\cdot+\partial_{\theta}\mathbf{L}_{\varepsilon r,1}-\partial_{\theta}\mathbf{S}_{\varepsilon r,1}$$
	and therefore
	\begin{align*}
		\mathscr{B}_{\perp}^{-1}\widehat{\mathcal{L}}_{\omega}\mathscr{B}_{\perp}=&\mathscr{B}_{\perp}^{-1}\Pi_{\mathbb{S}_{0}}^{\perp}\mathscr{B}\mathfrak{L}_{\varepsilon r}\Pi_{\mathbb{S}_{0}}^{\perp}-\mathscr{B}_{\perp}^{-1}\Pi_{\mathbb{S}_{0}}^{\perp}\left(\partial_{\theta}\left(V_{\varepsilon r}\cdot\right)+\partial_{\theta}\mathbf{L}_{\varepsilon r,1}-\partial_{\theta}\mathbf{S}_{\varepsilon r,1}\right)\Pi_{\mathbb{S}_{0}}\mathscr{B}\Pi_{\mathbb{S}_{0}}^{\perp}-\varepsilon\mathscr{B}_{\perp}^{-1}\partial_{\theta}\mathcal{R}\mathscr{B}_{\perp},
	\end{align*}
	where we have used the identities
	$$
	\mathscr{B}_{\perp}^{-1}\Pi_{\mathbb{S}_{0}}^{\perp}=\mathscr{B}_{\perp}^{-1}\quad\hbox{and}\quad [\Pi_{\mathbb{S}_{0}}^{\perp},T]=0=[\Pi_{\mathbb{S}_{0}},T], 
	$$
	for any  Fourier multiplier $T$. The structure of $\mathfrak{L}_{\varepsilon r}$ is detailed  in Proposition \ref{reduction of the transport part}, and from this we deduce that
	\begin{align*}\Pi_{\mathbb{S}_{0}}^{\perp}\mathscr{B}\mathfrak{L}_{\varepsilon r}\Pi_{\mathbb{S}_{0}}^{\perp}&=\Pi_{\mathbb{S}_{0}}^{\perp}\mathscr{B}\big(\omega\cdot\partial_{\varphi}+V_{i_{0}}^{\infty}\partial_{\theta}+\partial_{\theta}\mathcal{K}_{b}\ast\cdot+\partial_{\theta}\mathfrak{R}_{\varepsilon r}+\mathtt{E}_{n}^{0}\big)\Pi_{\mathbb{S}_{0}}^{\perp}\\
		&= \Pi_{\mathbb{S}_{0}}^{\perp}\mathscr{B}\Pi_{\mathbb{S}_{0}}^{\perp}\big(\omega\cdot\partial_{\varphi}+V_{i_{0}}^{\infty}\partial_{\theta}+\partial_{\theta}\mathcal{K}_{b}\ast\cdot\big)+\Pi_{\mathbb{S}_{0}}^{\perp}\mathscr{B}\partial_{\theta}\mathfrak{R}_{\varepsilon r}\Pi_{\mathbb{S}_{0}}^{\perp}+\Pi_{\mathbb{S}_{0}}^{\perp}\mathscr{B}\mathtt{E}_{n}^{0}\Pi_{\mathbb{S}_{0}}^{\perp}\\
		&=\mathscr{B}_\perp\big(\omega\cdot\partial_{\varphi}+V_{i_{0}}^{\infty}\partial_{\theta}+\partial_{\theta}\mathcal{K}_{b}\ast\cdot\big)+\Pi_{\mathbb{S}_{0}}^{\perp}\mathscr{B}\partial_{\theta}\mathfrak{R}_{\varepsilon r}\Pi_{\mathbb{S}_{0}}^{\perp}+\Pi_{\mathbb{S}_{0}}^{\perp}\mathscr{B}\mathtt{E}_{n}^{0}\Pi_{\mathbb{S}_{0}}^{\perp}.
	\end{align*}
	It follows that
	\begin{align*}
		\mathscr{B}_{\perp}^{-1}\Pi_{\mathbb{S}_{0}}^{\perp}\mathscr{B}\mathfrak{L}_{\varepsilon r}\Pi_{\mathbb{S}_{0}}^{\perp}
		&=\big(\omega\cdot\partial_{\varphi}+V_{i_{0}}^{\infty}\partial_{\theta}+\partial_{\theta}\mathcal{K}_{b}\ast\cdot\big)\Pi_{\mathbb{S}_{0}}^{\perp}+\mathscr{B}_{\perp}^{-1}\Pi_{\mathbb{S}_{0}}^{\perp}\mathscr{B}\partial_{\theta}\mathfrak{R}_{\varepsilon r}\Pi_{\mathbb{S}_{0}}^{\perp}+\mathscr{B}_{\perp}^{-1}\Pi_{\mathbb{S}_{0}}^{\perp}\mathscr{B}\mathtt{E}_{n}^{0}\Pi_{\mathbb{S}_{0}}^{\perp}\\
		&=\big(\omega\cdot\partial_{\varphi}+V_{i_{0}}^{\infty}\partial_{\theta}+\partial_{\theta}\mathcal{K}_{b}\ast\cdot\big)\Pi_{\mathbb{S}_{0}}^{\perp}+\Pi_{\mathbb{S}_{0}}^{\perp}\partial_{\theta}\mathfrak{R}_{\varepsilon r}\Pi_{\mathbb{S}_{0}}^{\perp}+\mathscr{B}_{\perp}^{-1}\mathscr{B}\Pi_{\mathbb{S}_{0}}\partial_{\theta}\mathfrak{R}_{\varepsilon r}\Pi_{\mathbb{S}_{0}}^{\perp}\\
		&\quad+\mathscr{B}_{\perp}^{-1}\Pi_{\mathbb{S}_{0}}^{\perp}\mathscr{B}\mathtt{E}_{n}^{0}\Pi_{\mathbb{S}_{0}}^{\perp}.
	\end{align*}
	Consequently, in the Cantor set $\mathcal{O}_{\infty,n}^{\gamma,\tau_{1}}(i_{0})$, one has the following reduction
	\begin{align}\label{Form1A}
		\nonumber \mathscr{B}_{\perp}^{-1}\widehat{\mathcal{L}}_{\omega}\mathscr{B}_{\perp}=&\big(\omega\cdot\partial_{\varphi}+V_{i_{0}}^{\infty}\partial_{\theta}+\partial_{\theta}\mathcal{K}_{b}\ast\cdot\big)\Pi_{\mathbb{S}_{0}}^{\perp}+\Pi_{\mathbb{S}_{0}}^{\perp}\partial_{\theta}\mathfrak{R}_{\varepsilon r}\Pi_{\mathbb{S}_{0}}^{\perp}+\mathscr{B}_{\perp}^{-1}\mathscr{B}\Pi_{\mathbb{S}_{0}}\partial_{\theta}\mathfrak{R}_{\varepsilon r}\Pi_{\mathbb{S}_{0}}^{\perp}\\
		\nonumber &-\mathscr{B}_{\perp}^{-1}\Pi_{\mathbb{S}_{0}}^{\perp}\left(\partial_{\theta}\left(V_{\varepsilon r}\cdot\right)+\partial_{\theta}\mathbf{L}_{\varepsilon r,1}-\partial_{\theta}\mathbf{S}_{\varepsilon r,1}\right)\Pi_{\mathbb{S}_{0}}\mathscr{B}\Pi_{\mathbb{S}_{0}}^{\perp}-\varepsilon\mathscr{B}_{\perp}^{-1}\partial_{\theta}\mathcal{R}\mathscr{B}_{\perp}+\mathscr{B}_{\perp}^{-1}\Pi_{\mathbb{S}_{0}}^{\perp}\mathscr{B}\mathtt{E}_{n}^{0}\Pi_{\mathbb{S}_{0}}^{\perp}\\
		:=&\omega\cdot\partial_{\varphi}\Pi_{\mathbb{S}_{0}}^{\perp}+\mathscr{D}_0+\mathscr{R}_0+\mathtt{E}_{n}^{1},
	\end{align}
	where we set
	$$\mathscr{D}_0:=\big(V_{i_{0}}^{\infty}\partial_{\theta}+\partial_{\theta}\mathcal{K}_{b}\ast\cdot\big)\Pi_{\mathbb{S}_{0}}^{\perp}$$
	and
	\begin{equation}\label{def En1}
		\mathtt{E}_{n}^{1}:=\mathscr{B}_{\perp}^{-1}\Pi_{\mathbb{S}_{0}}^{\perp}\mathscr{B}\mathtt{E}_{n}^{0}\Pi_{\mathbb{S}_{0}}^{\perp}.
	\end{equation}
\textbf{(iii)} Results from \eqref{def En1}, \eqref{estimate for mathscrBperp and its inverse}, \eqref{estimate on the first reduction operator and its inverse}, \eqref{estim En0 s0} and Lemma \ref{Lem-lawprod}-(ii).\\
\textbf{(iv)}
	For the estimates \eqref{estimate mathscr R in off diagonal norm} and \eqref{estimate differences mathscr R in off diagonal norm}, we refer to Lemma \cite[Prop 6.3 and Lem. 6.3]{HR21}. They are based on suitable duality representations of $\mathscr{B}_{\perp}^{\pm 1}$ linked to $\mathscr{B}^{\pm 1}.$\\
	\textbf{(v)} It is obtained by \eqref{e-dKb}, \eqref{estimate r1}, \eqref{estimate mathscr R in off diagonal norm} and Lemma \ref{properties of Toeplitz in time operators}-(iv).
\end{proof}
\subsubsection{Elimination of the remainder term}
We perform here the KAM reduction of the remainder $\mathscr{R}_0$ of Proposition \ref{projection in the normal directions}. This procedure allows to diagonalise the linearized operator in the normal directions, namely to conjugate it to a constant coefficients operator $\mathscr{L}_{\infty},$ up to fast decaying terms. 
\begin{prop}\label{reduction of the remainder term}
	Let $(\gamma,q,d,\tau_{1},\tau_2,s_{0}, s_l, \overline{s}_l,\overline{s}_h, \overline{\mu}_2,S)$ satisfy \eqref{initial parameter condition}, \eqref{setting tau1 and tau2}, \eqref{param}. For any $(\mu_2,s_h)$ satisfying 
	\begin{align}\label{Conv-T2}
		\mu_2\geqslant \overline{\mu}_2+2\tau_2q+2\tau_2,\quad\textnormal{and}\quad  s_h\geqslant \frac{3}{2}\mu_{2}+\overline{s}_{l}+1,
	\end{align}
	there exist $\varepsilon_{0}\in(0,1)$ and $\sigma_{4}=\sigma_{4}(\tau_1,\tau_2,q,d)\geqslant\sigma_{3}$ such that if 
	\begin{equation}\label{hypothesis KAM reduction of the remainder term}
		\varepsilon\gamma^{-2-q}N_{0}^{\mu_{2}}\leqslant \varepsilon_{0}\quad\textnormal{and}\quad\|\mathfrak{I}_{0}\|_{q,s_{h}+\sigma_{4}}^{\gamma,\mathcal{O}}\leqslant 1,
	\end{equation}
	then the following assertions hold true.
	\begin{enumerate}[label=(\roman*)]
		\item There exists a  family of invertible linear operator $\Phi_{\infty}:\mathcal{O}\to \mathcal{L}\big(H^{s}\cap L^2_{\perp}\big)$ satisfying the estimates
		\begin{equation}\label{estimate on Phiinfty and its inverse}
			\forall s\in[s_{0},S],\quad \mbox{ }\|\Phi_{\infty}^{\pm 1}\rho\|_{q,s}^{\gamma ,\mathcal{O}}\lesssim \|\rho\|_{q,s}^{\gamma,\mathcal{O}}+\varepsilon\gamma^{-2}\| \mathfrak{I}_{0}\|_{q,s+\sigma_{4}}^{\gamma,\mathcal{O}}\|\rho\|_{q,s_{0}}^{\gamma,\mathcal{O}}.
		\end{equation}
		There exists a diagonal operator $\mathscr{L}_\infty=\mathscr{L}_{\infty}(b,\omega,i_{0})$ taking the form
		\begin{align*}\mathscr{L}_{\infty}&=\omega\cdot\partial_{\varphi}\Pi_{\mathbb{S}_0}^{\perp}+\mathscr{D}_{\infty}
		\end{align*}
		where $\mathscr{D}_{\infty}=\mathscr{D}_{\infty}(b,\omega,i_{0})=\Pi_{\mathbb{S}_0}^\perp\mathscr{D}_\infty\Pi_{\mathbb{S}_{0}}^\perp$ is a reversible Fourier multiplier operator given by,
		$$
		\forall (l,j)\in \mathbb{Z}^{d}\times\mathbb{S}_{0}^{c},\quad  \mathscr{D}_{\infty}\mathbf{e}_{l,j}=\ii\,\mu_{j}^{\infty}\,\mathbf{e}_{l,j},$$
		with
		\begin{equation}\label{estim mujinfty}
			\forall j\in\mathbb{S}_{0}^{c},\quad\mu_{j}^{\infty}(b,\omega,i_{0})=\mu_{j}^{0}(b,\omega,i_{0})+r_{j}^{\infty}(b,\omega,i_{0})
		\end{equation}
		and 
		\begin{align}\label{estimate rjinfty}
			\sup_{j\in\mathbb{S}_{0}^{c}}|j|\| r_{j}^{\infty}\|_{q}^{\gamma ,\mathcal{O}}\lesssim\varepsilon\gamma^{-1}
		\end{align}
		such that in the Cantor set
		\begin{align*}
			\mathscr{O}_{\infty,n}^{\gamma,\tau_1,\tau_{2}}(i_{0}):=\bigcap_{\underset{|l|\leqslant N_{n}}{(l,j,j_{0})\in\mathbb{Z}^{d}\times(\mathbb{S}_0^{c})^{2}}\atop(l,j)\neq(0,j_{0})}\Big\{&(b,\omega)\in\mathcal{O}_{\infty,n}^{\gamma,\tau_{1}}(i_{0}),\big|\omega\cdot l+\mu_{j}^{\infty}(b,\omega,i_{0})-\mu_{j_{0}}^{\infty}(b,\omega,i_{0})\big|>\tfrac{2\gamma \langle j-j_{0}\rangle}{\langle l\rangle^{\tau_{2}}}\Big\}
		\end{align*}
		we have 
		\begin{align*}
			\Phi_{\infty}^{-1}\mathscr{L}_{0}\Phi_{\infty}&=\mathscr{L}_{\infty}+\mathtt{E}^2_n,
		\end{align*}
		and the linear operator  $\mathtt{E}_n^{2}$ satisfies the estimate
		\begin{equation}\label{Error-Est-2D}
			\|\mathtt{E}^2_n\rho\|_{q,s_0}^{\gamma ,\mathcal{O}}\lesssim \varepsilon\gamma^{-2}N_{0}^{{\mu}_{2}}N_{n+1}^{-\mu_{2}} \|\rho\|_{q,s_0+1}^{\gamma,\mathcal{O}}.
		\end{equation}
		Notice that the Cantor set $\mathcal{O}_{\infty,n}^{\gamma,\tau_{1}}(i_{0})$ was introduced in Proposition \ref{reduction of the transport part}, the operator $\mathscr{L}_{0}$ and the frequencies $\big(\mu_{j}^{0}(b,\omega,i_{0})\big)_{j\in\mathbb{S}_0^c}$ were stated in Proposition  \ref{projection in the normal directions}.
		\item Given two tori $i_{1}$ and $i_{2}$ both satisfying \eqref{hypothesis KAM reduction of the remainder term}, then
		\begin{equation}\label{diffenrence rjinfty}
			\forall j\in\mathbb{S}_{0}^{c},\quad\|\Delta_{12}r_{j}^{\infty}\|_{q}^{\gamma,\mathcal{O}}\lesssim\varepsilon\gamma^{-1}\|\Delta_{12}i\|_{q,\overline{s}_{h}+\sigma_{4}}^{\gamma,\mathcal{O}}
		\end{equation}
		
		\begin{equation}\label{estimate differences mujinfty}
			\forall j\in\mathbb{S}_{0}^{c},\quad\|\Delta_{12}\mu_{j}^{\infty}\|_{q}^{\gamma,\mathcal{O}}\lesssim\varepsilon\gamma^{-1}|j|\| \Delta_{12}i\|_{q,\overline{s}_{h}+\sigma_{4}}^{\gamma,\mathcal{O}}.
		\end{equation}
	\end{enumerate}	
\end{prop}
\begin{proof}
	The proof consists in a KAM reduction algorithm. We may refer for instance to \cite{BBM14,B19,FP14,HHM21,HR21} to see some implementations of this method. Here, the main reference which fits with our purpose is \cite[Prop 6.5]{HR21}. Hence, we only explain here the main lines of this procedure and refer to this work for the technical details. In Proposition \ref{projection in the normal directions}, we managed to find the following reduction valid in the Cantor set $\mathcal{O}_{\infty,n}^{\gamma,\tau_1}(i_0),$ 
	$$\mathscr{B}_{\perp}^{-1}\widehat{\mathcal{L}}_{\omega}\mathscr{B}_{\perp}=\mathscr{L}_{0}+\mathtt{E}_{n}^{1},$$
	where $\mathscr{L}_{0}$ is an operator admitting the following structure 
	\begin{equation}\label{mouka1}
		\mathscr{L}_0=\omega\cdot\partial_{\varphi}\Pi_{\mathbb{S}_0}^{\perp}+\mathscr{D}_0+\mathscr{R}_0,
	\end{equation}
	with $\mathscr{D}_0=\Pi_{\mathbb{S}_0}^\perp \mathscr{D}_0\Pi_{\mathbb{S}_0}^\perp=\big(\ii\mu_{j}^0(b,\omega)\big)_{j\in\mathbb{S}_{0}^c}$ a diagonal operator of pure imaginary spectrum and $\mathscr{R}_0=\Pi_{\mathbb{S}_0}^\perp \mathscr{R}_0\Pi_{\mathbb{S}_0}^\perp$ a real and reversible Toeplitz in time operator of zero order. The proof consists in a recursive elimination of the remainder term $\mathscr{R}_0$. First notice that if we denote
	$$\delta_{0}(s):=\gamma ^{-1}\|\mathscr{R}_{0}\|_{\textnormal{\tiny{O-d}},q,s}^{\gamma ,\mathcal{O}},$$
	then, in view of \eqref{estimate mathscr R in off diagonal norm}, we get
	\begin{align}\label{whab1}\delta_{0}(s)\leqslant C\varepsilon\gamma^{-2}\left(1+\|\mathfrak{I}_{0}\|_{q,s+\sigma_{3}}^{\gamma,\mathcal{O}}\right).
	\end{align}
	Hence, according to \eqref{Conv-T2}, \eqref{hypothesis KAM reduction of the remainder term} and the fact that $\sigma_4\geqslant \sigma_3,$ we infer 
	\begin{align}\label{Conv-P3}
		\nonumber N_{0}^{\mu_{2}}\delta_{0}(s_{h}) &\leqslant  C N_{0}^{\mu_{2}}\varepsilon\gamma^{-2}\\
		&\leqslant C\varepsilon_{0}.
	\end{align}
We construct recursively a sequence $(\mathscr{L}_m)_{m\in\mathbb{N}}$ of operators in the form  
\begin{align}\label{Op-Lm}
	\mathscr{L}_{m}:=\omega\cdot\partial_{\varphi}\Pi_{\mathbb{S}_0}^{\perp}+\mathscr{D}_{m}+\mathscr{R}_{m},
\end{align}
with $\mathscr{D}_m=\Pi_{\mathbb{S}_0}^{\perp}\mathscr{D}_m\Pi_{\mathbb{S}_0}^{\perp}=\big(\ii\mu_{j}^m(b,\omega)\big)_{j\in\mathbb{S}_0^c}$ a diagonal real reversible operator, that is,
\begin{align}\label{TUma1}
	\mathscr{D}_m\mathbf{e}_{l,j}=\ii\,\mu_j^m(b,\omega) \,\mathbf{e}_{l,j}\quad\hbox{and}\quad \mu_{-j}^m(b,\omega)=-\mu_{j}^m(b,\omega)
\end{align}
and $\mathscr{R}_m=\Pi_{\mathbb{S}_0}^\perp \mathscr{R}_m\Pi_{\mathbb{S}_0}^\perp$ a real and reversible Toeplitz in time operator of zero order quadratically smaller than the previous one in the Toeplitz topology introduced in Section \ref{section-ope}. To construct $\mathscr{D}_{m+1}$ and $\mathscr{R}_{m+1},$ we consider a linear invertible transformation close to the identity $$\Phi_m=\Pi_{\mathbb{S}_0}^{\perp}+\Psi_m:\mathcal{O}\rightarrow\mathcal{L}(H^{s}\cap L^2_{\perp}),$$
where $\Psi_m$ is small and depends on $\mathscr{R}_m$. Straightforward computations give
\begin{align*}
	\Phi_m^{-1}\mathscr{L}_m\Phi_m & =  \omega\cdot\partial_{\varphi}\Pi_{\mathbb{S}_0}^{\perp}+\mathscr{D}_m+\Phi_m^{-1}\Big(\big[\omega\cdot\partial_{\varphi}\Pi_{\mathbb{S}_0}^{\perp}+\mathscr{D}_m,\Psi_m\big]+P_{N_m}\mathscr{R}_m+P_{N_m}^{\perp}\mathscr{R}_m+\mathscr{R}_m\Psi_m\Big),
\end{align*}
		where the projector $P_{N_m}$ was defined in \eqref{definition of projections for operators}. Therefore, we shall define $\Psi_m$ so that it satisfies the following {\it homological equation}
		\begin{equation}\label{equation Psi}
			\big[\omega\cdot\partial_{\varphi}\Pi_{\mathbb{S}_0}^{\perp}+\mathscr{D}_m,\Psi_m\big]+P_{N_m}\mathscr{R}_m=\lfloor P_{N_m}\mathscr{R}_m\rfloor,
		\end{equation}
		where $\lfloor P_{N_m}\mathscr{R}_m\rfloor$ is the diagonal part of the operator $P_{N_m}\mathscr{R}_m$. We emphasize that the notation   $\lfloor  \mathcal{R}\rfloor$ with a general  operator $ \mathcal{R}$ is defined as follows, {for all $(l_{0},j_{0})\in\mathbb{Z}^{d}\times\mathbb{S}_{0}^{c}$},
		\begin{equation}\label{Dp1X}
			\mathcal{R}{\bf e}_{l_0,j_0}=\sum_{(l,j)\in\mathbb{Z}^{d}\times\mathbb{S}_{0}^{c}}\mathcal{R}_{l_0,j_0}^{l,j}{\bf e}_{l,j}\Longrightarrow  \lfloor\mathcal{R}\rfloor {\bf e}_{l_0,j_0}=\mathcal{R}_{l_0,j_0}^{l_0,j_0}{\bf e}_{l_0,j_0}=\big \langle \mathcal{R}{\bf e}_{l_0,j_0}, {\bf e}_{l_0,j_0}\big\rangle_{L^2(\T^{d+1})}\,{\bf e}_{l_0,j_0}.
		\end{equation}
		Recall that we denote ${\bf e}_{l_0,j_0}(\varphi,\theta)=e^{\ii(l_0\cdot\varphi+j_0\theta)}$. The equation \eqref{equation Psi} can be solved using the Fourier decomposition of operators. Indeed, since $\mathscr{R}_m$ is a real and reversible Toeplitz in time operator, then by virtue of Lemma \ref{characterization of real operator by its Fourier coefficients}, we can write
		 \begin{equation}\label{coefficients of the remainder operator R}
		 	(\mathscr{R}_m)_{l_{0},j_{0}}^{l,j}:=\ii\,r_{j_{0},m}^{j}(b,\omega,l_{0}-l)\in \ii\,\mathbb{R}\quad\mbox{ and }\quad (\mathscr{R}_m)_{-l_{0},-j_{0}}^{-l,-j}=-(\mathscr{R}_m)_{l_{0},j_{0}}^{l,j}.
		 \end{equation}
	 Therefore, we define the operator $\Psi_m$ by
	 \begin{equation}\label{Ext-psi-op}
	 	(\Psi_m)_{j_{0}}^{j}(b,\omega,l)=\left\lbrace\begin{array}{ll}
	 		-\varrho_{j_{0},m}^{j}(b,\omega,l)\,\, r_{j_{0},m}^{j}(b,\omega,l),& \mbox{if }\quad (l,j)\neq(0,j_{0})\\
	 		0, & \mbox{if }\quad (l,j)=(0,j_{0}),
	 	\end{array}\right.
	 \end{equation}
	 with
	 \begin{align}\label{varr-d}
	 	\varrho_{j_{0},m}^{j}(b,\omega,l):=\frac{\chi\left((\omega\cdot l+\mu_{j}^{m}(b,\omega)-\mu_{j_{0}}^{m}(b,\omega))(\gamma\langle j-j_{0}\rangle)^{-1}\langle l\rangle^{\tau_{2}}\right)}{\omega\cdot l+\mu_{j}^m(b,\omega)-\mu_{j_{0}}^m(b,\omega)},
	 \end{align}
  where the cut-off function $\chi\in C^{\infty}(\mathbb{R},[0,1])$ is even and defined by
  \begin{equation}\label{properties cut-off function first reduction}
  	\chi(\xi)=\left\lbrace\begin{array}{ll}
  		0 & \textnormal{if }|\xi|\leqslant\frac{1}{3}\\
  		1 & \textnormal{if }|\xi|\geqslant\frac{1}{2}.
  	\end{array}\right.
  \end{equation}
  supplemented by the normal conditions
 $$
 \forall \,l\,\in\mathbb{Z}^d,\,\forall\, j\,\,\hbox{or}\,\,\,j_0\in\mathbb{S}_0,\quad (\Psi_m)_{j_{0}}^{j}(b,\omega,l)=0.
 $$
Using Lemma \ref{Lem-lawprod}-(v), we can prove that 
 \begin{align}\label{link Psi and R}
 	\|\Psi_m\|_{\textnormal{\tiny{O-d}},q,s}^{\gamma ,\mathcal{O}}&
 	\leqslant C \gamma ^{-1}\| P_{N_m}\mathscr{R}_m\|_{\textnormal{\tiny{O-d}},q,s+\tau_{2} q+\tau_{2}}^{\gamma ,\mathcal{O}},
 \end{align}
provided that 
 \begin{align}\label{reg-G-10}
 	\forall\,(j,j_{0})\in(\mathbb{S}_0^c)^{2},\quad\max_{|\alpha|\in\llbracket 0, q\rrbracket}\sup_{(b,\omega)\in\mathcal{O}}\left|\partial_{b,\omega}^{\alpha}\left(\mu_{j}^m(b,\omega)-\mu_{j_{0}}^m(b,\omega)\right)\right|\leqslant C\,|j-j_{0}|.
 \end{align}

Remark that \eqref{reg-G-10} is satisfied for $m=0$ in view of Lemma \ref{properties omegajb}-$(iv)$ and \eqref{estimate r1}. By construction, $\Psi_m=\Pi_{\mathbb{S}_0}^{\perp}\Psi_m\Pi_{\mathbb{S}_0}^{\perp}$ is a real and reversibility preserving Toeplitz in time operator well-defined in the whole set of parameters $\mathcal{O}$ and solves the equation \eqref{equation Psi} when restricted to the Cantor set $\mathscr{O}_{m+1}^{\gamma}$ defined recursively by $\mathscr{O}_0^{\gamma}=\mathcal{O}$ and
\begin{equation}\label{Cantor-SX}
	\mathscr{O}_{m+1}^{\gamma}=\bigcap_{\underset{|l|\leqslant N_m}{(l,j,j_{0})\in\mathbb{Z}^{d }\times({\mathbb{S}}_{0}^{c})^{2}}\atop (l,j)\neq(0,j_0)}\left\lbrace(b,\omega)\in\mathscr{O}_m^{\gamma}\quad\textnormal{s.t.}\quad |\omega\cdot l+\mu_{j}^m(b,\omega)-\mu_{j_{0}}^m(b,\omega)|>\tfrac{\gamma\langle j-j_{0}\rangle }{\langle l\rangle^{\tau_{2}}}\right\rbrace.
\end{equation}
Consequently, in the Cantor set $\mathscr{O}_{m+1}^{\gamma},$ one has
\begin{equation}\label{Op-Lm1}
	\mathscr{L}_{m+1}:=\Phi_m^{-1}\mathscr{L}_m\Phi_m=\omega\cdot\partial_{\varphi}\Pi_{\mathbb{S}_0}^{\perp}+\mathscr{D}_{m+1}+\mathscr{R}_{m+1},
\end{equation}
with		
\begin{equation}\label{PU-RT}
	\mathscr{D}_{m+1}=\mathscr{D}_m+\lfloor P_{N_m}\mathscr{R}_m\rfloor\quad\textnormal{and}\quad \mathscr{R}_{m+1}=\Phi_m^{-1}\big(-\Psi_m\,\lfloor P_{N_m}\mathscr{R}_m\rfloor +P_{N_m}^{\perp}\mathscr{R}_m+\mathscr{R}_m\Psi_m\big).
\end{equation}
Remark  that $\mathscr{D}_{m}$ and $\lfloor P_{N_{m}}\mathscr{R}_{m}\rfloor$ are Fourier multiplier Toeplitz operators that can be identified to their spectra  $(\ii \mu_{j}^{m})_{j\in\mathbb{S}_{0}^{c}}$ and $(\ii r_{j}^{m})_{j\in\mathbb{S}_{0}^{c}}$, namely 
\begin{align}\label{Spect-T1}
	\forall (l,j)\in\mathbb{Z}^d\times\mathbb{S}_0^c,\quad \mathscr{D}_{m} {\bf e}_{l,j}=\ii\mu_{j}^{m}\,{\bf e}_{l,j}\quad\hbox{and}\quad \lfloor P_{N_{m}}\mathscr{R}_{m}\rfloor \mathbf{e}_{l,j}=\ii r_{j}^{m}\,{\bf e}_{l,j}.
\end{align}
By construction, we find
\begin{align}\label{Spect-T2}
	\mu_{j}^{m+1}=\mu_{j}^{m}+r_{j}^{m}.
\end{align}
	We set 
	$$\delta_m(s):=\gamma^{-1}\|\mathscr{R}_m\|_{\textnormal{\tiny{O-d}},q,s}^{\gamma ,\mathcal{O}}\quad \textnormal{and}\quad\widehat{\delta}_m(s):=\max\left(\gamma^{-1}\|\partial_{\theta}\mathscr{R}_m\|_{\textnormal{\tiny{O-d}},q,s}^{\gamma ,\mathcal{O}}\,,\delta_m(s)\right).$$
		By Lemma  \ref{properties of Toeplitz in time operators} and \eqref{link Psi and R}, we get for all $S\geqslant \overline{s}\geqslant s\geqslant s_{0}$,
		\begin{align}
			\delta_{m+1}(s)&\leqslant N_m^{s-\overline{s}}\delta_m(\overline{s})+CN_m^{\tau_{2} q+\tau_{2}}\delta_m(s_0)\delta_m(s),\label{KAM step remainder term}\\
			\widehat{\delta}_{m+1}(s)&\leqslant N_m^{s-\overline{s}}\widehat{\delta}_m(\overline{s})+CN_m^{\tau_{2} q+\tau_{2}+1}\widehat{\delta}_m(s_0)\widehat{\delta}_m(s).\label{KAM steph remainder term}
		\end{align}
	These recursive relations allow us to prove  by induction on $m$ that
	\begin{equation}\label{hypothesis of induction for deltamprime}
		\delta_{m}(\overline{s}_{l})\leqslant \delta_{0}(s_{h})N_{0}^{\mu_{2}}N_{m}^{-\mu_{2}}\quad \mbox{ and }\quad \delta_{m}(s_{h})\leqslant\left(2-\tfrac{1}{m+1}\right)\delta_{0}(s_{h}),
	\end{equation}
	\begin{align}\label{Ind-Ty1}
		\widehat{\delta}_{m}(s_{0})\leqslant\widehat{\delta}_{0}(s_{h})N_{0}^{\mu_{2}}N_{m}^{-\mu_{2}}\quad \mbox{ and }\quad \widehat{\delta}_{m}(s_{h})\leqslant\left(2-\tfrac{1}{m+1}\right)\widehat{\delta}_{0}(s_{h}),
	\end{align}
	with $\overline{s}_{l}$ and $s_h$ fixed by \eqref{param} and \eqref{Conv-T2}. Using the definition of the Fourier coefficients for an operator and the Toeplitz structure of $\mathscr{R}_m$, we find
	\begin{align*}
		\mu_{j}^{m+1}-\mu_{j}^{m}&=r_j^m\\
		&=-\ii\langle P_{N_{m}}\mathscr{R}_{m}\mathbf{e}_{l,j},\mathbf{e}_{l,j}\rangle_{L^{2}(\mathbb{T}^{d +1})}\\
		&=-\ii\langle P_{N_{m}}\mathscr{R}_{m}\mathbf{e}_{0,j},\mathbf{e}_{0,j}\rangle_{L^{2}(\mathbb{T}^{d +1})}.
	\end{align*} 
		By a duality argument combined with Lemma \ref{properties of Toeplitz in time operators}, \eqref{hypothesis of induction for deltamprime}, \eqref{whab1} and \eqref{hypothesis KAM reduction of the remainder term}, we infer 
		\begin{align}\label{Mahma1-R}
			\|\mu_{j}^{m+1}-\mu_{j}^{m}\|_{q}^{\gamma ,\mathcal{O}}
			&\leqslant C \varepsilon\gamma^{-1}N_0^{\mu_{2}}N_{m}^{-\mu_{2}}.
		\end{align}
		As the assumption \eqref{reg-G-10} is satisfied  with  $\mathscr{D}_{m}$, 
		then we obtain  by  \eqref{Mahma1-R}
		\begin{align*} \forall\,(j,j_{0})\in(\mathbb{S}_0^c)^{2},\quad\max_{|\alpha| \in\llbracket 0,q\rrbracket}\sup_{(\lambda,\omega)\in\mathcal{O}}\left|\partial_{\lambda,\omega}^{\alpha}\left(\mu_{j}^{m+1}(\lambda,\omega)-\mu_{j_{0}}^{m+1}(\lambda,\omega)\right)\right|\leqslant C\big(1+\varepsilon\gamma^{-1-q}N_0^{\mu_2}N_{m}^{-\mu_{2}}\big)\,|j-j_{0}|.
		\end{align*}
		Therefore, the sequence $\big(\mu_j^m\big)_{m\in\mathbb{N}}$ converges in $W^{q,\infty,\gamma}(\mathcal{O},\mathbb{C}).$ We denote $\mu_j^{\infty}$ its limit and consider the associated diagonal operator $\mathscr{D}_{\infty}=\big(\ii\mu_j^\infty\big)_{j\in\mathbb{S}_0^c}.$
		We introduce the sequence of operators 
		$\left(\widehat{\Phi}_m\right)_{m\in\mathbb{N}}$ by 
		\begin{equation}\label{Def-Phi}\widehat\Phi_0:=\Phi_0\quad\hbox{and}\quad \quad  \forall m\geqslant1,\,\, \widehat\Phi_m:=\Phi_0\circ\Phi_1\circ...\circ\Phi_m.
		\end{equation} 
		It is obvious  from the identity ${\Phi}_{m}=\hbox{Id}+\Psi_m$ that 
		\begin{equation}\label{rec Phih}
			\widehat\Phi_{m+1}=\widehat\Phi_{m}+\widehat\Phi_{m}\Psi_{m+1}.
		\end{equation}
		From \eqref{rec Phih}, \eqref{link Psi and R} and \eqref{hypothesis KAM reduction of the remainder term}, we can make the sequence $(\widehat{\Phi}_m)_{m\in\mathbb{N}}$ converge to an element $\Phi_\infty$  in the Toeplitz operator topology introduced in Section \ref{section-ope}. Moreover, arguing in a similar way to \cite{HR21} we may prove that
			\begin{align}\label{Conv-op-od}
\|\widehat\Phi_{m}^{-1}-\Phi_\infty^{-1}\|_{\textnormal{\tiny{O-d}},q,s_{0}}^{\gamma ,\mathcal{O}}+\|\widehat\Phi_{m}-\Phi_\infty\|_{\textnormal{\tiny{O-d}},q,s_{0}}^{\gamma ,\mathcal{O}}&\leqslant   C\,\delta_{0}(s_{h})N_{0}^{\mu_{2}}N_{m+1}^{-\mu_2}.
\end{align} 
In addition, for any $m\in\mathbb{N},$ we find in view of \eqref{Mahma1-R}
		\begin{align}\label{Spect-TD1}
			\|\mu_{j}^{\infty}-\mu_{j}^{m}\|_{q}^{\gamma ,\mathcal{O}}&\leqslant  \sum_{n=m}^\infty\|\mu_{j}^{n+1}-\mu_{j}^{n}\|_{q}^{\gamma ,\mathcal{O}}\nonumber\\
			&\leqslant C\gamma\,\delta_{0}(s_{h})N_{0}^{\mu_{2}}N_{m}^{-\mu_{2}}.
		\end{align}
		Therefore, we deduce 
		\begin{align}\label{dekomp}
			\nonumber \mu_{j}^{\infty}&=\mu_{j}^{0}+\sum_{m=0}^\infty\big(\mu_{j}^{m+1}-\mu_{j}^{m}\big)\\
			&:=\mu_{j}^{0}+ r_{j}^{\infty},
		\end{align}
		where $(\mu_{j}^{0})_{j\in\mathbb{S}_0^c}$ is  described in Proposition \ref{projection in the normal directions} and takes the form  
		$$\mu_{j}^{0}(b,\omega,i_{0})=\Omega_{j}(b)+j\big(V_{i_{0}}^{\infty}(b,\omega)-\tfrac{1}{2}\big).
		$$
		Define the diagonal  operator $\mathscr{D}_{\infty}$  by
		\begin{align}\label{Dinfty-op}
			\forall (l,j)\in\mathbb{Z}^d\times\mathbb{S}_0^c,\quad \mathscr{D}_{\infty} {\bf e}_{l,j}=\ii\mu_{j}^{\infty}{\bf e}_{l,j}.
		\end{align}
		By the norm definition we obtain
		\[
		\|\mathscr{D}_{m}-\mathscr{D}_{\infty}\|_{\textnormal{\tiny{O-d}},q,s_{0}}^{\gamma ,\mathcal{O}}=\displaystyle\sup_{j\in\mathbb{S}_{0}^{c}}\|\mu_{j}^{m}-\mu_{j}^{\infty}\|_{q}^{\gamma ,\mathcal{O}},
		\]
		which gives by virtue  of \eqref{Spect-TD1}
		\begin{align}\label{Conv-Dinf}
			\|\mathscr{D}_{m}-\mathscr{D}_{\infty}\|_{\textnormal{\tiny{O-d}},q,s_{0}}^{\gamma ,\mathcal{O}}\leqslant C\, \gamma\,\delta_{0}(s_{h})N_{0}^{\mu_{2}} N_{m}^{-\mu_{2}}.
		\end{align}
		Let us consider the following Cantor set for a given $n\in\mathbb{N},$ 
		$$\mathscr{O}_{\infty,n}^{\gamma,\tau_1,\tau_{2}}(i_{0}):=\bigcap_{(l,j,j_0)\in\mathbb{Z}^{d}\times(\mathbb{S}_0^c)^2\atop\underset{(l,j)\neq(0,j_0)}{\langle l,j-j_0\rangle\leqslant N_n}}\Big\{(b,\omega)\in\mathcal{O}_{\infty,n}^{\gamma,\tau_1}(i_0)\quad\textnormal{s.t.}\quad \big|\omega\cdot l+\mu_j^{\infty}(b,\omega)-\mu_{j_0}^{\infty}(b,\omega)\big|>\tfrac{2\langle j-j_0\rangle}{\langle l\rangle^{\tau_2}}\Big\},$$
		where  the intermediate Cantor sets are defined in \eqref{Cantor-SX}. A finite induction, in a similar way to \cite{HR21}, allows  to prove that
		$$\mathscr{O}_{\infty,n}^{\gamma,\tau_1,\tau_{2}}(i_{0})\subset\mathscr{O}_{n+1}^{\gamma},$$
		 Now define the operator
		$$\mathscr{L}_{\infty}:=\omega\cdot\partial_{\varphi}\Pi_{\mathbb{S}_0}^{\perp}+\mathscr{D}_{\infty}.$$
		Writing 
		$$\mathscr{L}_m-\mathscr{L}_{\infty}=\mathscr{D}_m-\mathscr{D}_{\infty}+\mathscr{R}_m,$$
		then using \eqref{Conv-Dinf} and \eqref{hypothesis of induction for deltamprime}, we can deduce the convergence of the sequence $(\mathscr{L}_m)_{m\in\mathbb{N}}$ to the operator $\mathscr{L}_{\infty}$ in the Toeplitz operator topology. In view of \eqref{Def-Phi} and \eqref{Op-Lm1} one obtains 
			\begin{align*}
				\forall (\lambda,\omega)\in \mathscr{O}_{n+1}^{\gamma },\quad\widehat{\Phi}_{n}^{-1}\mathscr{L}_{0}\widehat{\Phi}_{n}&=\omega\cdot\partial_{\varphi}\Pi_{\mathbb{S}_0}^{\perp}+\mathscr{D}_{n+1}+\mathscr{R}_{n+1}\\
				&=\mathscr{L}_{\infty}+\mathscr{D}_{n+1}-\mathscr{D}_{\infty}+\mathscr{R}_{n+1}.
			\end{align*} 
			It follows that for any $(\lambda,\omega)\in \mathscr{O}_{n+1}^{\gamma }$ 					\begin{align*}
				\Phi_{\infty}^{-1}\mathscr{L}_{0}\Phi_{\infty}&=\mathscr{L}_{\infty}+\mathscr{D}_{n+1}-\mathscr{D}_{\infty}+\mathscr{R}_{n+1}\\
				&+\Phi_{\infty}^{-1}\mathscr{L}_{0}\left(\Phi_{\infty}-\widehat{\Phi}_{n}\right)+\left(\Phi_{\infty}^{-1}-\widehat{\Phi}_{n}^{-1}\right)\mathscr{L}_{0}\widehat{\Phi}_{n}\\
				&:=\mathscr{L}_{\infty}+\mathtt{E}_{n}^2.
			\end{align*}
		According to Lemma \ref{properties of Toeplitz in time operators}, \eqref{Taptap1}, \eqref{Spect-TD1}, \eqref{hypothesis KAM reduction of the remainder term} and \eqref{Conv-op-od}  we get \eqref{Error-Est-2D}.
		Next let us see how to get the estimate  \eqref{estimate rjinfty}.
			Recall that
			$$r_{j}^{\infty}=\sum_{m=0}^{\infty}r_{j}^{m}\quad\quad \textnormal{with}\quad\quad r_{j}^{m}=-\ii\big\langle P_{N_{m}}\mathscr{R}_{m}\mathbf{e}_{0,j},\mathbf{e}_{0,j}\big\rangle_{L^{2}(\mathbb{T}^{d +1})}.$$
			An integration by parts gives
			\begin{align*}
				\big\langle P_{N_{m}}\mathscr{R}_{m}\mathbf{e}_{0,j},\mathbf{e}_{0,j}\big\rangle_{L^{2}(\mathbb{T}^{d +1})}&=\tfrac{\ii}{j}\big\langle P_{N_{m}}\mathscr{R}_{m}\mathbf{e}_{0,j},\partial_{\theta}\mathbf{e}_{0,j}\big\rangle_{L^{2}(\mathbb{T}^{d +1})}\\
				&=\tfrac{-\ii}{j}\big\langle P_{N_{m}}\partial_{\theta}\mathscr{R}_{m}\mathbf{e}_{0,j},\mathbf{e}_{0,j}\big\rangle_{L^{2}(\mathbb{T}^{d +1})}.
			\end{align*}
			Using a duality argument $H^{s_{0}}-H^{-s_{0}}$ combined with   Lemma \ref{properties of Toeplitz in time operators}, \eqref{Ind-Ty1}, \eqref{estimate mathscr R in off diagonal norm} and \eqref{hypothesis KAM reduction of the remainder term}, we obtain
			\begin{align*}
				\| r_{j}^{\infty}\|_{q}^{\gamma ,\mathcal{O}}& \lesssim\displaystyle|j|^{-1}\varepsilon\gamma^{-1}.
			\end{align*}
			The difference estimates \eqref{diffenrence rjinfty} and \eqref{estimate differences mujinfty} can be obtained by fixing $\mathtt{p}=4\tau_2q+4\tau_2$ as in \cite{HR21}. The proof of Proposition \ref{reduction of the remainder term} is now complete.
\end{proof}
\subsection{Construction and tame estimates for the approximate inverse}
At this step, we can construct an almost approximate right inverse for $\widehat{\mathcal{L}}_{\omega}$ defined in \eqref{hat-L-omega}. This enables  to find in turn an almost approximate right inverse for the whole operator $d_{i,\alpha}\mathcal{F}(i_0,\alpha_0)$ given by \eqref{Lineqrized-op-F}.

\begin{prop}\label{inversion of the linearized operator in the normal directions}
	Let $(\gamma,q,d,\tau_{1},s_{0},s_h,\mu_2,S)$ satisfying \eqref{initial parameter condition}, \eqref{setting tau1 and tau2} and \eqref{Conv-T2}. 
	There exists $\sigma:=\sigma(\tau_1,\tau_2,q,d)\geqslant\sigma_{4}$ such that if 
	\begin{equation}\label{small-1}
		\varepsilon\gamma^{-2-q}N_0^{\mu_2}\leqslant\varepsilon_0\quad\textnormal{and}\quad\|\mathfrak{I}_0\|_{q,s_h+\sigma}^{\gamma,\mathcal{O}}\leqslant 1,
	\end{equation}
then, the following assertions hold true.
	\begin{enumerate}[label=(\roman*)]
		\item Consider the operator  $\mathscr{L}_{\infty}$ defined in Proposition $\ref{reduction of the remainder term},$ then there exists a family of  linear reversible operators $\big(\mathtt{T}_n\big)_{n\in\mathbb{N}}$  defined in $\mathcal{O}$ satisfying the estimate
		$$
		\forall s\in[s_0,S],\quad \sup_{n\in\mathbb{N}}\|\mathtt{T}_{n}\rho\|_{q,s}^{\gamma ,\mathcal{O}}\lesssim \gamma ^{-1}\|\rho\|_{q,s+\tau_{1}q+\tau_{1}}^{\gamma ,\mathcal{O}}$$
		and such that for any $n\in\mathbb{N}$, in the Cantor set
		$$\Lambda_{\infty,n}^{\gamma,\tau_{1}}(i_{0})=\bigcap_{(l,j)\in\mathbb{Z}^{d }\times\mathbb{S}_{0}^{c}\atop |l|\leqslant N_{n}}\left\lbrace(b,\omega)\in\mathcal{O}\quad\textnormal{s.t.}\quad\left|\omega\cdot l+\mu_{j}^{\infty}(b,\omega,i_{0})\right|>\tfrac{\gamma \langle j\rangle }{\langle l\rangle^{\tau_{1}}}\right\rbrace,
		$$
		we have
		$$
		\mathscr{L}_{\infty}\mathtt{T}_n=\textnormal{Id}+\mathtt{E}^3_n,
		$$
		with 
		$$
		\forall s_{0}\leqslant s\leqslant\overline{s}\leqslant S, \quad \|\mathtt{E}^3_{n}\rho\|_{q,s}^{\gamma ,\mathcal{O}} \lesssim 
		N_n^{s-\overline{s}}\gamma^{-1}\|\rho\|_{q,\overline{s}+1+\tau_{1}q+\tau_{1}}^{\gamma ,\mathcal{O}}.
		$$
		\item 
		There exists a family of linear reversible operators $\big(\mathtt{T}_{\omega,n}\big)_{n\in\mathbb{N}}$ satisfying
		\begin{equation}\label{estimate mathcalTomega}
			\forall \, s\in\,[ s_0, S],\quad\sup_{n\in\mathbb{N}}\|\mathtt{T}_{\omega,n}\rho\|_{q,s}^{\gamma ,\mathcal{O}}\lesssim\gamma^{-1}\left(\|\rho\|_{q,s+\sigma}^{\gamma ,\mathcal{O}}+\| \mathfrak{I}_{0}\|_{q,s+\sigma}^{\gamma ,\mathcal{O}}\|\rho\|_{q,s_{0}+\sigma}^{\gamma,\mathcal{O}}\right)
		\end{equation}
		and such that in the Cantor set
		\begin{equation}\label{def calGn}
			\mathcal{G}_n(\gamma,\tau_{1},\tau_{2},i_{0}):=\mathcal{O}_{\infty,n}^{\gamma,\tau_{1}}(i_{0})\cap\mathcal{O}_{\infty,n}^{\gamma,\tau_{1},\tau_{2}}(i_{0})\cap\Lambda_{\infty,n}^{\gamma,\tau_{1}}(i_{0}),
		\end{equation}
		we have
		$$
		\widehat{\mathcal{L}}_{\omega}\mathtt{T}_{\omega,n}=\textnormal{Id}+\mathtt{E}_n,
		$$
		where $\mathtt{E}_n$ satisfies the following estimate
		\begin{align}\label{e-En-s0}
			\forall\, s\in [s_0,S],\quad  &\|\mathtt{E}_n\rho\|_{q,s_0}^{\gamma ,\mathcal{O}}
			\nonumber\lesssim N_n^{s_0-s}\gamma^{-1}\Big( \|\rho\|_{q,s+\sigma}^{\gamma,\mathcal{O}}+\varepsilon\gamma^{-2}\| \mathfrak{I}_{0}\|_{q,s+\sigma}^{\gamma,\mathcal{O}}\|\rho\|_{q,s_{0}}^{\gamma,\mathcal{O}} \Big)\nonumber\\
			&\qquad\qquad\quad+ \varepsilon\gamma^{-3}N_{0}^{{\mu}_{2}}N_{n+1}^{-\mu_{2}} \|\rho\|_{q,s_0+\sigma}^{\gamma,\mathcal{O}}.
		\end{align}
		Recall that  $\widehat{\mathcal{L}}_{\omega},$ $  \mathcal{O}_{\infty,n}^{\gamma,\tau_{1}}(i_{0})$ and $\mathcal{O}_{\infty,n}^{\gamma,\tau_{1},\tau_{2}}(i_{0})$ are given by \eqref{hat-L-omega} and Propositions \ref{reduction of the transport part} and \ref{reduction of the remainder term}, respectively.
		\item In the Cantor set $\mathcal{G}_{n}(\gamma,\tau_{1},\tau_{2},i_{0})$, we have the following splitting
		$$\widehat{\mathcal{L}}_{\omega}=\widehat{\mathtt{L}}_{\omega,n}+\widehat{\mathtt{R}}_{n}\quad\textnormal{with}\quad\widehat{\mathtt{L}}_{\omega,n}\mathtt{T}_{\omega,n}=\textnormal{Id}\quad\textnormal{and}\quad\widehat{\mathtt{R}}_{n}=\mathtt{E}_{n}\widehat{\mathtt{L}}_{\omega,n},$$
		where  $\widehat{\mathtt{L}}_{\omega,n}$ and $\widehat{\mathtt{R}}_{n}$ are  reversible operators defined in  $\mathcal{O}$ and satisfy the following estimates
		\begin{align}
			\forall s\in[s_{0},S],\quad& \sup_{n\in\mathbb{N}}\|\widehat{\mathtt{L}}_{\omega,n}\rho\|_{q,s}^{\gamma,\mathcal{O}}\lesssim\|\rho\|_{q,s+1}^{\gamma,\mathcal{O}}+\varepsilon\gamma^{-2}\|\mathfrak{I}_{0}\|_{q,s+\sigma}^{\gamma,\mathcal{O}}\|\rho\|_{q,s_{0}+1}^{\gamma,\mathcal{O}},\label{e-Lomeganh}\\
			\forall s\in[s_{0},S],\quad &\|\widehat{\mathtt{R}}_{n}\rho\|_{q,s_{0}}^{\gamma,\mathcal{O}}\lesssim N_{n}^{s_{0}-s}\gamma^{-1}\left(\|\rho\|_{q,s+\sigma}^{\gamma,\mathcal{O}}+\varepsilon\gamma^{-2}\|\mathfrak{I}_{0}\|_{q,s+\sigma}^{\gamma,\mathcal{O}}\|\rho\|_{q,s_{0}+\sigma}^{\gamma,\mathcal{O}}\right)\nonumber\\
			&\qquad\qquad\quad+\varepsilon\gamma^{-3}N_{0}^{\mu_{2}}N_{n+1}^{-\mu_{2}}\|\rho\|_{q,s_{0}+\sigma}^{\gamma,\mathcal{O}}.\label{e-Rnh}
		\end{align}
	\end{enumerate}
\end{prop}
\begin{proof}
	{\bf{(i)}} First recall from Proposition \ref{reduction of the remainder term} that
	\begin{align*}
		\mathscr{L}_{\infty}&=\omega\cdot\partial_{\varphi}\Pi_{\mathbb{S}_0}^{\perp}+\mathscr{D}_{\infty}.
	\end{align*}
	Using the projectors defined in \eqref{definition of projections for functions}, we can split this operator as follows
	\begin{align}\label{Tikl1}
		\nonumber\mathscr{L}_{\infty}&=\Pi_{N_n}\omega\cdot\partial_{\varphi}\Pi_{N_n}\Pi_{\mathbb{S}_0}^{\perp}+\mathscr{D}_{\infty}-\Pi_{N_n}^\perp\omega\cdot\partial_{\varphi}\Pi_{N_n}^\perp\Pi_{\mathbb{S}_0}^{\perp}\nonumber\\
		&:=\mathtt{L}_n-\mathtt{R}_n,
	\end{align}
	where 
	$$\mathtt{R}_{n}:=\Pi_{N_n}^\perp\omega\cdot\partial_{\varphi}\Pi_{N_n}^\perp\Pi_{\mathbb{S}_0}^{\perp}.$$
	According to the structure of $\mathscr{D}_{\infty}$ in Proposition \ref{reduction of the remainder term}, we obtain from \eqref{Tikl1},
	$$\forall(l,j)\in\mathbb{Z}^{d}\times\mathbb{S}_{0}^{c},\quad {\bf e}_{-l,-j}\mathtt{L}_n{\bf e}_{l,j}=\left\lbrace\begin{array}{ll}
		\ii\big(\omega\cdot l+\mu_j^\infty\big) & \hbox{if }  |l|\leqslant N_n\\
		\ii\,\mu_j^\infty&\hbox{if } |l|> N_n.
	\end{array}\right.$$
	Let us now consider the diagonal operator  $\mathtt{T}_n$ defined by 
	\begin{align*}
		\mathtt{T}_{n}\rho(b,\omega,\varphi,\theta):=&
		-\ii\sum_{(l,j)\in\mathbb{Z}^{d}\times\mathbb{S}_{0}^{c}\atop |l|\leqslant N_n }\tfrac{\chi\left((\omega\cdot l+\mu_{j}^{\infty}(b,\omega,i_{0}))\gamma ^{-1}\langle l\rangle^{\tau_{1}}\right)}{\omega\cdot l+\mu_{j}^{\infty}(b,\omega,i_{0})}\rho_{l,j}(b,\omega)\,e^{\ii(l\cdot\varphi+j\theta)}\\
		&-\ii\sum_{(l,j)\in\mathbb{Z}^{d}\times\mathbb{S}_{0}^{c}\atop|l|>N_n}\tfrac{\rho_{l,j}(b,\omega)}{\mu_{j}^{\infty}(b,\omega,i_{0})}\,e^{\ii(l\cdot\varphi+j\theta)},
	\end{align*}
	where $\chi$ is the cut-off function introduced in \eqref{properties cut-off function first reduction} and $\left(\rho_{l,j}(b,\omega)\right)_{l,j}$ are the Fourier coefficients of $\rho$. Now recall the expansion of the perturbed eigenvalues given by Proposition \ref{reduction of the remainder term}, namely
	$$\mu_{j}^{\infty}(b,\omega,i_{0})=\Omega_{j}(b)+jr^{1}(b,\omega)+r_{j}^{\infty}(b,\omega)\quad \mbox{ with }\quad r^{1}(b,\omega)=V_{i_{0}}^{\infty}(b,\omega)-\tfrac{1}{2}.$$
	In view of Lemma \ref{properties omegajb}-(iv), \eqref{differences mu0} and \eqref{estimate rjinfty}, they satisfy the following estimates  
	$$
	\forall j\in\mathbb{S}_{0}^{c},\quad \|\mu^\infty_{j}\|_{q}^{\gamma ,\mathcal{O}}\lesssim |j|.
	$$
	According  to Lemma \ref{properties omegajb}-(ii), \eqref{differences mu0}, \eqref{estimate rjinfty} and the smallness condition \eqref{small-1} we infer
	$$|j|\lesssim \|\mu^\infty_{j}\|_{0}^{\gamma ,\mathcal{O}}\leqslant  \|\mu^\infty_{j}\|_{q}^{\gamma ,\mathcal{O}}.$$
	Computations based on Lemma \ref{Lem-lawprod}-(vi) give  
	\begin{align}\label{Es-Mu1}
		\forall s\geqslant s_0,\quad  \|\mathtt{T}_{n}\rho\|_{q,s}^{\gamma ,\mathcal{O}}\lesssim \gamma ^{-1}\|\rho\|_{q,s+\tau_{1}q+\tau_{1}}^{\gamma ,\mathcal{O}}.
	\end{align}
	In addition, by construction 
	\begin{equation}\label{Inv-Ty1}
		\mathtt{L}_{n}\mathtt{T}_{n}=\textnormal{Id}\quad\textnormal{in }\Lambda_{\infty,n}^{\gamma,\tau_{1}}(i_{0})
	\end{equation} 
	since $\chi(\cdot)=1$ in this set. Gathering \eqref{Inv-Ty1} and \eqref{Tikl1} yields 
	\begin{align}\label{LinftyTn}
		\forall\, (b,\omega)\in \Lambda_{\infty,n}^{\gamma,\tau_{1}}(i_{0}),\quad \mathscr{L}_\infty \mathtt{T}_{n}&=\textnormal{Id}-\mathtt{R}_n\mathtt{T}_{n}\nonumber\\
		&:=\textnormal{Id}+\mathtt{E}_n^3.
	\end{align}
	Remark that by Lemma \ref{Lem-lawprod}-(ii),
	$$
	\forall \, s_0\leqslant s\leqslant \overline{s}\leqslant S,\quad \|\mathtt{R}_{n}\rho\|_{q,s}^{\gamma ,\mathcal{O}}\lesssim N_n^{s-\overline{s}}\|\rho\|_{q,\overline{s}+1}^{\gamma ,\mathcal{O}}.
	$$
	Putting this estimate with \eqref{Es-Mu1} implies
	\begin{align}\label{Mila-1}
		\forall \, s_0\leqslant s\leqslant \overline{s}\leqslant S,\quad \|\mathtt{E}^3_{n}\rho\|_{q,s}^{\gamma ,\mathcal{O}}&\lesssim 
		N_n^{s-\overline{s}}\gamma^{-1}\|\rho\|_{q,\overline{s}+1+\tau_{1}q+\tau_{1}}^{\gamma ,\mathcal{O}}.
	\end{align}
	\textbf{(ii)} We set  
	\begin{align}\label{tomega}
		\mathtt{T}_{\omega,n}:=\mathscr{B}_{\perp}\Phi_{\infty}\mathtt{T}_{n}\Phi_{\infty}^{-1}\mathscr{B}_{\perp}^{-1},
	\end{align}
	where the operators $ \mathscr{B}_{\perp}$ and $\Phi_{\infty}$ are defined in Propositions \ref{projection in the normal directions} and \ref{reduction of the remainder term} respectively. Notice that $\mathtt{T}_{\omega,n}$ is defined in the whole range of parameters $\mathcal{O}.$ Since the condition \eqref{small-1} is satisfied, then, both Propositions \ref{reduction of the transport part} and \ref{reduction of the remainder term} apply and the estimate \eqref{estimate mathcalTomega} is obtained combining \eqref{estimate for mathscrBperp and its inverse}, \eqref{estimate on Phiinfty and its inverse}, \eqref{Es-Mu1} and \eqref{small-1}.
	Now combining Propositions \ref{projection in the normal directions} and \ref{reduction of the remainder term}, we find that in the Cantor set  $\mathcal{O}_{\infty,n}^{\gamma,\tau_{1}}(i_{0})\cap\mathcal{O}_{\infty,n}^{\gamma,\tau_{1},\tau_{2}}(i_{0})$
	the following decomposition holds
	\begin{align*}
		\Phi^{-1}_{\infty}\mathscr{B}_{\perp}^{-1}\widehat{\mathcal{L}}_{\omega}\mathscr{B}_{\perp}\Phi_{\infty}&=\Phi^{-1}_{\infty}\mathscr{L}_{0}\Phi_{\infty}+\Phi^{-1}_{\infty}\mathtt{E}_{n}^1\Phi_{\infty}\\
		&=\mathscr{L}_{\infty}+\mathtt{E}_n^2+\Phi^{-1}_{\infty}\mathtt{E}_{n}^1\Phi_{\infty}.
	\end{align*}
	According to \eqref{LinftyTn}, one finds that in the Cantor set  $\mathcal{O}_{\infty,n}^{\gamma,\tau_{1}}(i_{0})\cap\mathscr{O}_{\infty,n}^{\gamma,\tau_{2}}(i_{0})\cap \Lambda_{\infty,n}^{\gamma,\tau_{1}}(i_{0})$ the following identity holds 
	\begin{align*}
		\Phi^{-1}_{\infty}\mathscr{B}_{\perp}^{-1}\widehat{\mathcal{L}}_{\omega}\mathscr{B}_{\perp}\Phi_{\infty}\mathtt{T}_{n}
		&=\textnormal{Id}+\mathtt{E}_n^3+\mathtt{E}_n^2\mathtt{T}_{n}+\Phi^{-1}_{\infty}\mathtt{E}_{n}^1\Phi_{\infty}\mathtt{T}_{n},
	\end{align*}
	which implies in turn, in view of \eqref{tomega}, the following identity in $\mathcal{G}_n(\gamma,\tau_{1},\tau_{2},i_{0})$
	\begin{align}\label{Id-Moh1}
		\widehat{\mathcal{L}}_{\omega}\mathtt{T}_{\omega,n}&=\textnormal{Id}+\mathscr{B}_{\perp}\Phi_{\infty}\big(\mathtt{E}_n^3+\mathtt{E}_n^2\mathtt{T}_{n}+\Phi^{-1}_{\infty}\mathtt{E}_{n}^1\Phi_{\infty}\mathtt{T}_{n}\big)\Phi^{-1}_{\infty}\mathscr{B}_{\perp}^{-1}
		\nonumber\\
		&:=\textnormal{Id}+\mathtt{E}_n.
	\end{align}
	Combining \eqref{Id-Moh1}, \eqref{En1-s0}, \eqref{Error-Est-2D}, \eqref{Mila-1}, \eqref{Es-Mu1}, \eqref{estimate for mathscrBperp and its inverse}, \eqref{estimate on Phiinfty and its inverse} and \eqref{small-1}, we get \eqref{e-En-s0}, up to take $\sigma$ large enough.\\
	\textbf{(iii)} By virtue of \eqref{Id-Moh1}, one can write in the Cantor set $\mathcal{G}_{n}(\gamma,\tau_{1},\tau_{2},i_{0})$
	\begin{equation}\label{Id-Moh11}
		\widehat{\mathcal{L}}_{\omega}=\mathtt{T}_{\omega,n}^{-1}+\mathtt{E}_{n}\mathtt{T}_{\omega,n}^{-1}.
	\end{equation}
	Putting together \eqref{tomega} and \eqref{Inv-Ty1}, one finds in the Cantor set $\mathcal{G}_n(\gamma,\tau_1,\tau_2,i_0)$
	$$\widehat{\mathtt{L}}_{\omega,n}:=\mathtt{T}_{\omega,n}^{-1}=\mathscr{B}_{\perp}\Phi_{\infty}\mathtt{L}_{n}\Phi_{\infty}^{-1}\mathscr{B}_{\perp}^{-1}.$$
	Therefore, \eqref{Id-Moh11} can be rewritten
	$$
		\widehat{\mathcal{L}}_{\omega}=\widehat{\mathtt{L}}_{\omega,n}+\widehat{\mathtt{R}}_{n}\quad\textnormal{with}\quad\widehat{\mathtt{R}}_{n}:=\mathtt{E}_{n}\widehat{\mathtt{L}}_{\omega,n}.
	$$
	The estimate \eqref{e-Lomeganh} is obtained gathering \eqref{Tikl1}, \eqref{estimate for mathscrBperp and its inverse}, \eqref{estimate on Phiinfty and its inverse} and \eqref{small-1}. Finally, \eqref{e-Lomeganh} together with \eqref{e-En-s0} implies \eqref{e-Rnh}. This ends the proof of Proposition \ref{inversion of the linearized operator in the normal directions}.  
\end{proof}

The following theorem, which can be found in \cite{BB15,BM18,HHM21}, states  that the linearized operator  $d_{i,\alpha} {\mathcal F}(i_0,\alpha_0 )$ in \eqref{Lineqrized-op-F} admits an approximate right  inverse on a suitable Cantor set.

\begin{theo}\label{thm ai}
	{\bf (Approximate inverse)}\\
	Let $(\gamma,q,d,\tau_{1},\tau_2,s_{0},s_h,\mu_{2})$ satisfy  \eqref{initial parameter condition}, \eqref{setting tau1 and tau2}, \eqref{param} and \eqref{Conv-T2}. 
	Then there exists $ \overline{\sigma}= \overline{\sigma}(\tau_1,\tau_2,d,q)>0$ and a family of reversible operators ${\rm T}_0:={\rm T}_{0,n}(i_0)$ such that if the smallness condition \eqref{small-1} holds, then
	for all $ g = (g_1, g_2, g_3) $, satisfying 
	$$g_1(\varphi)=g_1(\varphi),\quad g_2(-\varphi)=-g_2(\varphi)\quad\textnormal{and}\quad g_3(-\varphi)=(\mathscr{S}g_3)(\varphi),$$  
	the function  $ {\rm T}_0 g $  satisfies the following estimate  
	$$
		\forall s\in [s_0,S],\quad \| {\rm T}_0 g\|_{q,s}^{\gamma,\mathcal{O}}\lesssim\gamma^{-1}\left(\|g\|_{q,s+\overline{\sigma}}^{\gamma,\mathcal{O}}+\|\mathfrak{I}_{0}\|_{q,s+\overline{\sigma}}^{\gamma,\mathcal{O}}\|g\|_{q,s_{0}+\overline{\sigma}}^{\gamma,\mathcal{O}}\right).
	$$
	Moreover  ${\rm T}_0$ is an almost-approximate right inverse of $d_{i, \alpha} \mathcal{ F}(i_0,\alpha_0)$ in the Cantor set $\mathcal{G}_{n}(\gamma,\tau_{1},\tau_{2},i_{0})$ defined by \eqref{def calGn}. More precisely,   
	$$
		\forall (b,\omega)\in \mathcal{G}_{n}(\gamma,\tau_{1},\tau_{2},i_{0}), \quad d_{i , \alpha} \mathcal{ F} (i_0) \circ {\rm T}_0
		- {\rm Id} = \mathcal{E}^{(n)}_1+\mathcal{E}^{(n)}_2+\mathcal{E}^{(n)}_3, 
	$$
	where the operators $\mathcal{E}^{(n)}_1$, $\mathcal{E}^{(n)}_2$ and $\mathcal{E}^{(n)}_3$ are defined in the whole set $\mathcal{O}$ with the estimates
	\begin{align*}
		\|\mathcal{E}^{(n)}_1 g \|_{q,s_0}^{\gamma,\mathcal{O}} & \lesssim \gamma^{-1 } \| \mathcal{ F}(i_0, \alpha_0) \|_{q,s_0 +\overline{\sigma}}^{\gamma,\mathcal{O}} \| g \|_{q,s_0 + \overline{\sigma}}^{\gamma,\mathcal{O}},\\
		\forall b \geqslant 0,\quad \| \mathcal{E}^{(n)}_2 g \|_{q,s_0}^{\gamma,\mathcal{O}}& \lesssim 
		\gamma^{- 1} N_n^{- b } \big( \| g \|_{q,s_0+b+\overline{\sigma} }^{\gamma,\mathcal{O}}+\varepsilon
		\| \mathfrak{I}_{0} \|_{q,s_0
			+b+\overline{\sigma}    }^{\gamma,\mathcal{O}} \big \| g \|_{q,s_0 +\overline{\sigma}}^{\gamma,\mathcal{O}} \big)\,,
		\\
		\forall b\in [0,S], \quad \| \mathcal{E}^{(n)}_3 g \|_{q,s_0}^{\gamma,\mathcal{O}}& \lesssim N_n^{-b}\gamma^{-2}\Big( \|g\|_{q,s_0+b+\overline{\sigma}}^{\gamma,\mathcal{O}}+{\varepsilon\gamma^{-2}}\| \mathfrak{I}_{0}\|_{q,s_0+b+\overline{\sigma}}^{\gamma,\mathcal{O}}\|g\|_{q,s_0+\overline{\sigma}}^{\gamma,\mathcal{O}} \Big)\nonumber\\
		&\quad+ \varepsilon\gamma^{-4}N_{0}^{{\mu}_{2}}{N_{n}^{-\mu_{2}}} \|g\|_{q,s_0+\overline{\sigma}}^{\gamma,\mathcal{O}}. 
	\end{align*}
\end{theo}
\section{Nash-Moser iteration and measure of the final Cantor set}
In this last section, we shall find a non-trivial solution $(b,\omega)\mapsto(i_{\infty}(b,\omega),\alpha_{\infty}(b,\omega))$ to the equation
$$\mathcal{F}(i,\alpha,b,\omega,\varepsilon)=0,$$
where $\mathcal{F}$ is the functional defined in \eqref{main function}. This done by using a Nash-Moser scheme in a similar way to the series of papers \cite{BBMH18,BM18,HHM21,HR21}. The solutions are constructed for parameters $(b,\omega)$ belonging to the intersection of all the Cantor sets $\mathcal{G}_{\infty}^{\gamma}$ on which we are able to invert the linearized operator at the different steps. In order to find a solution to the original problem, we must rigidify the frequencies $\omega$ so that it coincides with the equilibrium frequencies. This amounts to consider a frequency curve $b\mapsto\omega(b,\varepsilon)$ implicitly  defined by the equation
$$\alpha_{\infty}\big(b,\omega(b,\varepsilon)\big)=-\omega_{\textnormal{Eq}}(b).$$ 
Considering the associated rigidified Cantor set
$$\mathcal{C}_{\infty}^{\varepsilon}=\Big\{b\in(b_0,b_1)\quad\textnormal{s.t.}\quad\big(b,\omega(b,\varepsilon)\big)\in\mathcal{G}_{\infty}^{\gamma}\Big\},$$
we have a solution to the original problem provided that the measure of $\mathcal{C}_{\infty}^{\varepsilon}$ is non-zero. This will be checked, in Section \ref{Section measure of the final Cantor set}, by perturbative arguments in the spirit of the previous works \cite{BBMH18,BBM14,BM18,HHM21,HR21}. This proves in particular Theorem \ref{thm QPS ED}.
\subsection{Nash-Moser iteration}
In this section we implement the Nash-Moser scheme, which is a modified Newton method consisting in a recursive construction of approximate solutions of the equation $\mathcal{F}\big(i,\alpha,b,\omega\big):=\mathcal{F}\big(i,\alpha,b,\omega,\varepsilon\big)=0$ where the functional  $\mathcal{F}$ is defined in \eqref{main function}. At each step of this procedure, we need to construct an approximate inverse of  the linearized operator at a state near the equilibrium by applying the reduction procedure developed in Section \ref{sec red lin opb}. This allows to get  Theorem \ref{thm ai}  with the suitable  tame estimates associated to  the final  loss of regularity $\overline{\sigma}$ that could be arranged to be large enough.  We point out that  $\overline{\sigma}$ depends only on the shape of the Cantor set through the  parameters $\tau_1,\tau_2,d$ and $q$ but  it is independent of  the regularity of the solutions that we want to construct.  Now, we shall fix the following parameters needed  to implement  the Nash-Moser scheme and related to the loss of regularity $\overline{\sigma}.$ 
\begin{equation}\label{param NM}
	\left\lbrace\begin{array}{rcl}
		\overline{a} & = & \tau_{2}+2\\
		\mu_1 & = & 3q(\tau_{2}+2)+6\overline{\sigma}+6\\
		a_{1} & = & 6q(\tau_{2}+2)+12\overline{\sigma}+15\\
		a_{2} & = & 3q(\tau_{2}+2)+6\overline{\sigma}+9\\
		\mu_{2} & = & 2q(\tau_{2}+2)+5\overline{\sigma}+7\\
		s_{h} & = & s_{0}+4q(\tau_{2}+2)+9\overline{\sigma}+11\\
		b_{1} & = & 2s_{h}-s_{0}.
	\end{array}\right.
\end{equation}
We shall now impose the constraint relating $\gamma$ and $N_{0}$ to $\varepsilon$
\begin{equation}\label{choice of gamma and N0 in the Nash-Moser}
	0<a<\tfrac{1}{\mu_{2}+q+2}, \quad \gamma:=\varepsilon^{a}, \quad N_{0}:=\gamma^{-1}.
\end{equation}
We consider the finite dimensional subspaces
\begin{align*}
	E_{n}:=&\Big\{\mathfrak{I}=(\Theta,I,z)\quad
	\textnormal{s.t.}\quad\Theta=\Pi_{n}\Theta,\quad I=\Pi_{n}I\quad\textnormal{and}\quad z=\Pi_{n}z\Big\},
\end{align*}
where $\Pi_{n}$ is the projector defined by
$$f(\varphi,\theta)=\sum_{(l,j)\in\mathbb{Z}^{d}\times\mathbb{Z}}f_{l,j}e^{\ii(l\cdot\varphi+j\theta)}\quad\Rightarrow\quad\Pi_{n}f(\varphi,\theta)=\sum_{\langle l,j\rangle\leqslant N_{n}}f_{l,j}e^{\ii(l\cdot\varphi+j\theta)}.$$
The main result of this section can be stated as follows. The proof is now standard and to avoid the technical details we refer to the works \cite{BM18,HHM21,HR21}. In our context, we may follow closely the version detailed in \cite[Prop. 7.1]{HR21}. In particular the roles of the parameters in \eqref{param NM} are explained  there.
\begin{prop}\label{Nash-Moser}
	\textbf{(Nash-Moser)}\\
	Let $(\tau_{1},\tau_{2},q,d,s_{0})$ satisfy \eqref{initial parameter condition} and \eqref{setting tau1 and tau2}. Consider the parameters fixed by \eqref{param NM} and \eqref{choice of gamma and N0 in the Nash-Moser}.
	There exist $C_{\ast}>0$ and $\varepsilon_{0}>0$ such that for any $\varepsilon\in[0,\varepsilon_{0}]$ we get for all $n\in\mathbb{N}$ the following properties.
		\begin{itemize}
			\item [$(\mathcal{P}1)_{n}$] There exists a $q$-times differentiable function 
			$$W_{n}:\begin{array}[t]{rcl}
				\mathcal{O} & \rightarrow &  E_{n-1}\times\mathbb{R}^{d}\times\mathbb{R}^{d+1}\\
				(b,\omega) & \mapsto & \big(\mathfrak{I}_{n},\alpha_{n}-\omega,0\big)
			\end{array}$$
			satisfying 
			$$W_{0}=0\quad\mbox{ and }\quad\hbox{for}\quad n\in\mathbb{N}^*,\quad\,\| W_{n}\|_{q,s_{0}+\overline{\sigma}}^{\gamma,\mathcal{O}}\leqslant C_{\ast}\varepsilon\gamma^{-1}N_{0}^{q\overline{a}}.$$
			By setting
			\begin{equation}\label{constru}
				U_{0}=\Big((\varphi,0,0),\omega,(b,\omega)\Big)\quad\textnormal{and}\quad  \hbox{for}\quad n\in\mathbb{N}^*,\quad U_{n}=U_{0}+W_{n}\quad \textnormal{and}\quad H_{n} =U_{n}-U_{n-1},
			\end{equation}
			then 
			\begin{equation}\label{estim Hm+1 tilde}
				\forall s\in[s_{0},S],\,\| H_{1}\|_{q,s}^{\gamma,\mathcal{O}}\leqslant \frac{1}{2}C_{\ast}\varepsilon\gamma^{-1}N_{0}^{q\overline{a}}\quad\mbox{ and }\quad\forall\, 2\leqslant k\leqslant n,\,\| H_{k}\|_{q,s_{0}+\overline{\sigma}}^{\gamma,\mathcal{O}}\leqslant C_{\ast}\varepsilon\gamma^{-1}N_{k-1}^{-a_{2}}.
			\end{equation}
			We also have for $n\geqslant 2,$
			\begin{equation}\label{e-Hn-diff} \|H_{n}\|_{q,\overline{s}_h+\sigma_{4}}^{\gamma,\mathcal{O}}\leqslant C_{\ast}\varepsilon\gamma^{-1}N_{n-1}^{-a_{2}}.
			\end{equation}
			\item [$(\mathcal{P}2)_{n}$] Define 
			\begin{equation}\label{def gamman}
				i_{n}=(\varphi,0,0)+\mathfrak{I}_{n},\quad \gamma_{n}=\gamma(1+2^{-n}),
			\end{equation}
			then $i_n$ satisfies the following reversibility condition
			\begin{equation}\label{rev in}
				\mathfrak{S}i_n(\varphi)=i_n(-\varphi),
			\end{equation}
		where $\mathfrak{S}$ is defined by \eqref{rev_aa}. Define also
			$$\mathcal{A}_{0}^{\gamma}=\mathcal{O}\quad \mbox{ and }\quad \mathcal{A}_{n+1}^{\gamma}=\mathcal{A}_{n}^{\gamma}\cap\mathcal{G}_{n}(\gamma_{n+1},\tau_{1},\tau_{2},i_{n})$$
			where $\mathcal{G}_{n}(\gamma_{n+1},\tau_{1},\tau_{2},i_{n})$ is defined in Proposition $\ref{inversion of the linearized operator in the normal directions}.$ Consider the open sets 
			$$
			\forall \mathtt{r}>0,\quad \mathrm{O}_{n}^\mathtt{r}:=\Big\{(b,\omega)\in\mathcal{O}\quad\textnormal{s.t.}\quad {\mathtt{dist}}\big((b,\omega),\mathcal{A}_{n}^{\gamma}\big)< \mathtt{r} N_{n}^{-\overline{a}}\Big\}
			$$
			where $\displaystyle {\mathtt{dist}}(x,A)=\inf_{y\in A}\|x-y\|$.  Then we have the following estimate 
			$$ \|\mathcal{F}(U_{n})\|_{q,s_{0}}^{\gamma,\mathrm{O}_{n}^{\gamma}}\leqslant C_{\ast}\varepsilon N_{n-1}^{-a_{1}}.
			$$
			\item[$(\mathcal{P}3)_{n}$] $\| W_{n}\|_{q,b_{1}+\overline{\sigma}}^{\gamma,\mathcal{O}}\leqslant C_{\ast}\varepsilon\gamma^{-1}N_{n-1}^{\mu_1}.$
	\end{itemize}
\end{prop}
A non trivial reversible quasi-periodic solution of our problem is obtained as the limit of the sequence $(U_n)_{n\in\mathbb{N}}$ according to the fast convergence stated in Proposition \ref{Nash-Moser}. This is explained in the following corollary.
\begin{cor}\label{Corollary NM}
	There exists $\varepsilon_0>0$ such that, for all $\varepsilon\in(0,\varepsilon_0),$ the following assertions hold true. We consider the Cantor set $\mathcal{G}_{\infty}^{\gamma}$, related to $\varepsilon$ through $\gamma$, and defined by
	$$\mathcal{G}_{\infty}^{\gamma}:=\bigcap_{n\in\mathbb{N}}\mathcal{A}_{n}^{\gamma}.$$
	There exists a function
	$$U_{\infty}:\begin{array}[t]{rcl}
		\mathcal{O} & \rightarrow & \left(\mathbb{T}^{d}\times\mathbb{R}^{d}\times L^2_\perp\cap  H^{s_0}\right)\times\mathbb{R}^{d}\times\mathbb{R}^{d+1}\\
		(b,\omega) & \mapsto & \big(i_{\infty}(b,\omega),\alpha_{\infty}(b,\omega),(b,\omega)\big)
	\end{array}$$
	such that 
	$$\forall(b,\omega)\in\mathcal{G}_{\infty}^{\gamma},\quad\mathcal{F}(U_{\infty}(b,\omega))=0.$$
	In addition, $i_{\infty}$ is reversible and $\alpha_{\infty}\in W^{q,\infty,\gamma}(\mathcal{O},\mathbb{R}^d)$ with
	\begin{equation}\label{alpha infty}
		\alpha_{\infty}(b,\omega)=\omega+\mathrm{r}_{\varepsilon}(b,\omega)\quad\mbox{ and }\quad\|\mathrm{r}_{\varepsilon}\|_{q}^{\gamma,\mathcal{O}}\lesssim\varepsilon\gamma^{-1}N_{0}^{q\overline{a}}.
	\end{equation}
	Moreover, there exists a $q$-times differentiable function $b\in(b_{0},b_{1})\mapsto\omega(b,\varepsilon)$ with
	\begin{equation}\label{alpha infty-1}
		\omega(b,\varepsilon)=-\omega_{\textnormal{Eq}}(b)+\bar{r}_{\varepsilon}(b),\quad \|\bar{r}_{\varepsilon}\|_{q}^{\gamma,\mathcal{O}}\lesssim\varepsilon\gamma^{-1}N_{0}^{q\overline{a}},
	\end{equation}
	and 
	$$\forall b\in \mathcal{C}_{\infty}^{\varepsilon},\quad \mathcal{F}\big(U_{\infty}(b,\omega(b,\varepsilon))\big)=0\quad\textnormal{and}\quad\alpha_{\infty}\big(b,\omega(b,\varepsilon)\big)=-\omega_{\textnormal{Eq}}(b),
	$$
	where the  Cantor set $\mathcal{C}_{\infty}^{\varepsilon}$ is defined by 
	\begin{equation}\label{def final Cantor set}
		\mathcal{C}_{\infty}^{\varepsilon}=\Big\{b\in(b_{0},b_{1})\quad\textnormal{s.t.}\quad\big(b,\omega(b,\varepsilon)\big)\in\mathcal{G}_{\infty}^{\gamma}\Big\}.
	\end{equation}
\end{cor}
\begin{proof}
In view of \eqref{constru} and \eqref{estim Hm+1 tilde}, we obtain
$$\|W_{n+1}-W_{n}\|_{q,s_{0}}^{\gamma,\mathcal{O}}=\|H_{n+1}\|_{q,s_{0}}^{\gamma,\mathcal{O}}\leqslant\|H_{n+1}\|_{q,s_{0}+\overline{\sigma}}^{\gamma,\mathcal{O}}\leqslant C_{\ast}\varepsilon\gamma^{-1}N_{n}^{-a_{2}}.$$
This implies the convergence of the sequence $\left(W_{n}\right)_{n\in\mathbb{N}}.$ Its limit is denoted by
$$W_{\infty}:=\lim_{n\rightarrow\infty}W_{n}:=(\mathfrak{I}_{\infty},\alpha_{\infty}-\omega,0,0)$$
and we set
$$U_{\infty}:=\big(i_{\infty},\alpha_{\infty},(b,\omega)\big)=U_0+W_{\infty}.$$
Taking $n\to\infty$ in \eqref{rev in} gives
$$\mathfrak{S}i_\infty(\varphi)=i_\infty(-\varphi).$$
According to Proposition \ref{Nash-Moser}-$(\mathcal{P}2)_{n}$, we get for small values of $\varepsilon$ 
\begin{equation}\label{sol before rigidity}
	\forall(b,\omega)\in\mathcal{G}_{\infty}^{\gamma},\quad\mathcal{F}\Big(i_{\infty}(b,\omega),\alpha_{\infty}(b,\omega),(b,\omega),\varepsilon\Big)=0,
\end{equation}
where $\mathcal{F}$ is the functional defined in \eqref{main function}. We emphasize that the Cantor set $\mathcal{G}_{\infty}^{\gamma}$ depends on $\varepsilon$ through $\gamma$ fixed in  \eqref{choice of gamma and N0 in the Nash-Moser}. Now, from Proposition \ref{Nash-Moser}-$(\mathcal{P}1)_{n}$, we deduce that 
$$\alpha_{\infty}(b,\omega)=\omega+\mathrm{r}_{\varepsilon}(b,\omega) \quad\mbox{ with }\quad\|\mathrm{r}_{\varepsilon}\|_{q}^{\gamma,\mathcal{O}}\lesssim\varepsilon\gamma^{-1}N_{0}^{q\overline{a}}.$$
Next we shall prove the second result and check the existence of solutions to the original Hamiltonian equation. First recall that the open set $\mathcal{O}$ is defined in \eqref{def initial parameters setb} by
$$\mathcal{O}=(b_{0},b_{1})\times\mathscr{U}\quad\textnormal{with}\quad\mathscr{U}=B(0,R_{0})\quad\textnormal{for some large }R_{0}>0,$$
where the ball $\mathscr{U}$ is taken to contain the equilibrium frequency vector $b\mapsto{\omega}_{\textnormal{Eq}}(b).$ In view of \eqref{alpha infty}, we obtain that for any $b\in(b_{0},b_{1}),$ the mapping $\omega\mapsto\alpha_{\infty}(b,\omega)$ is invertible from $\mathscr{U}$ into its image $\alpha_{\infty}(b,\mathscr{U})$ and we have
$$\widehat{\omega}=\alpha_{\infty}(b,\omega)=\omega+\mathrm{r}_{\varepsilon}(b,\omega)\Leftrightarrow\omega=\alpha_{\infty}^{-1}(b,\widehat{\omega})=\widehat{\omega}+\widehat{\mathrm{r}}_{\varepsilon}(b,\widehat{\omega}).$$
In particular,
$$\widehat{\mathrm{r}}_{\varepsilon}(b,\widehat{\omega})=-\mathrm{r}_{\varepsilon}(b,\omega).$$
Differentiating the previous relation and using \eqref{alpha infty}, we find 
\begin{equation}\label{estimate mathrm repsilon}
	\|\widehat{\mathrm{r}}_{\varepsilon}\|_{q}^{\gamma,\mathcal{O}}\lesssim\varepsilon\gamma^{-1}N_{0}^{q\overline{a}}.
\end{equation}
Now, we set 
$${\omega}(b,\varepsilon):=\alpha_{\infty}^{-1}(b,-{\omega}_{\textnormal{Eq}}(b))=-{\omega}_{\textnormal{Eq}}(b)+\overline{\mathrm{r}}_{\varepsilon}(b)\quad \mbox{ with }\quad \overline{\mathrm{r}}_{\varepsilon}(b):=\widehat{\mathrm{r}}_{\varepsilon}\big(b,-{\omega}_{\textnormal{Eq}}(b)\big)
$$
and consider the following Cantor set
$$\mathcal{C}_{\infty}^{\varepsilon}:=\Big\{b\in(b_{0},b_{1})\quad\textnormal{s.t.}\quad\big(b,\omega(b,\varepsilon)\big)\in\mathcal{G}_{\infty}^{\gamma}\Big\}.$$
Then, according to \eqref{sol before rigidity}, we get
$$\forall b\in\mathcal{C}_{\infty}^{\varepsilon},\quad\mathcal{F}\Big(U_{\infty}\big(b,\omega(b,\varepsilon)\big)\Big)=0.$$
This gives a nontrivial reversible solution for the original Hamiltonian equation provided that $b\in\mathcal{C}_{\infty}^{\varepsilon}.$
From Lemma \ref{properties omegajb}, we obtain that all the derivatives up to order $q$ of $\omega_{\textnormal{Eq}}$ are uniformly bounded on $[b_{0},b_{1}].$ As a consequence, the chain rule and \eqref{estimate mathrm repsilon} imply
\begin{equation}\label{estimate repsilon1}
	\|\overline{\mathrm{r}}_{\varepsilon}\|_{q}^{\gamma,\mathcal{O}}\lesssim\varepsilon\gamma^{-1}N_{0}^{q\overline{a}}\quad\textnormal{and}\quad\|\omega(\cdot,\varepsilon)\|_{q}^{\gamma,\mathcal{O}}\lesssim 1+\varepsilon\gamma^{-1}N_{0}^{q\overline{a}}\lesssim 1.
\end{equation}
This achieves the proof of Corollary \ref{Corollary NM}.
\end{proof}
\subsection{Measure estimates}\label{Section measure of the final Cantor set}
In this last section, we check that the Cantor set $\mathcal{C}_{\infty}^{\varepsilon},$ defined in \eqref{def final Cantor set}, of parameters generating non-trivial quasi-periodic solutions is non trivial. More precisely, we have the following proposition giving a lower bound measure for $\mathcal{C}_{\infty}^{\varepsilon}.$ 
\begin{prop}\label{lem-meas-es1}
	Let $q_{0}$ be defined as in Lemma \ref{lemma transversalityb} and impose  \eqref{param NM} and \eqref{choice of gamma and N0 in the Nash-Moser}  with $q=q_0+1.$ Assume the additional conditions
	\begin{equation}\label{choice tau 1 tau2 upsilon} 
		\left\lbrace\begin{array}{l}
			\tau_{1}> dq_{0}\\
			\tau_{2}>\tau_{1}+dq_{0}\\
			\upsilon=\frac{1}{q_{0}+3}.
		\end{array}\right.
	\end{equation}
	Then there exists $C>0$ such that 
	$$
	\big|\mathcal{C}_{\infty}^{\varepsilon}\big|\geqslant (b_1-b_0)-C \varepsilon^{\frac{a\upsilon}{q_{0}}}.
	$$
	In particular, 
	$$\displaystyle \lim_{\varepsilon\to0}\big|\mathcal{C}_{\infty}^{\varepsilon}\big|=b_1-b_0.$$
\end{prop}
\begin{remark}
	The constraints listed in \eqref{choice tau 1 tau2 upsilon} appear naturally in the proof, see \eqref{assum2-d} and \eqref{assum-co1}, for the convergence of series and for smallness conditions. Notice that these conditions agree with \eqref{setting tau1 and tau2} and Proposition \ref{reduction of the transport part}.
\end{remark}
\begin{proof}
According to Corollary \ref{Corollary NM}, we can decompose the Cantor set $\mathcal{C}_{\infty}^{\varepsilon}$ in the following intersection
\begin{equation}\label{definition of the final Cantor set}
	\mathcal{C}_{\infty}^{\varepsilon}:=\bigcap_{n\in\mathbb{N}}\mathcal{C}_{n}^{\varepsilon}\quad \mbox{ where }\quad \mathcal{C}_{n}^{\varepsilon}:=\Big\{b\in(b_{0},b_{1})\quad\hbox{s.t}\quad \big(b,{\omega}(b,\varepsilon)\big)\in\mathcal{A}_n^{\gamma}\Big\}.
\end{equation}
Recall that the intermediate sets $\mathcal{A}_n^{\gamma}$ and the perturbed frequency vector $\omega(b,\varepsilon)$ are respectively defined in Proposition \ref{Nash-Moser} and in \eqref{alpha infty}. Instead of measuring directly $\mathcal{C}_{\infty}^{\varepsilon}$, we rather estimate the measure of its complementary set in $(b_0,b_1).$ Thus, we write
\begin{equation}\label{decomp Cantor}
	(b_{0},b_{1})\setminus\mathcal{C}_{\infty}^{\varepsilon}=\big((b_{0},b_{1})\setminus\mathcal{C}_{0}^{\varepsilon}\big)\sqcup\bigsqcup_{n=0}^{\infty}\big(\mathcal{C}_{n}^{\varepsilon}\setminus\mathcal{C}_{n+1}^{\varepsilon}\big).
\end{equation}
Then, we have to measure all the sets appearing in the decomposition \eqref{decomp Cantor}. This can be done by using   Lemma \ref{lemma useful for measure estimates} together with some trivial inclusions allowing to link the time and space Fourier modes in order to make the series converge. For more details, we refer to Lemmata  \ref{lemm-dix1}, \ref{some cantor set are empty} and \ref{lemma Russeman condition for the perturbed frequencies}. 

	From  \eqref{choice of gamma and N0 in the Nash-Moser} and \eqref{alpha infty-1}, one obtains
	$$\sup_{b\in(b_0,b_1)}\left|\omega(b,\varepsilon)+\omega_{\textnormal{Eq}}(b)\right|\leqslant\|\overline{\mathrm{r}}_{\varepsilon}\|_{q}^{\gamma,\mathcal{O}}\leqslant C\varepsilon\gamma^{-1}N_{0}^{q\overline{a}}=C\varepsilon^{1-a(1+q\overline{a})}.$$
	Notice that the conditions \eqref{param NM} and \eqref{choice of gamma and N0 in the Nash-Moser} imply in particular
	$$0<a<\frac{1}{1+q\overline{a}}.$$
	Therefore, taking $\varepsilon$ small enough yields
	$$\sup_{b\in(b_0,b_1)}\left|\omega(b,\varepsilon)+\omega_{\textnormal{Eq}}(b)\right|\leqslant\|\overline{\mathrm{r}}_{\varepsilon}\|_{q}^{\gamma,\mathcal{O}}\leqslant 1.$$
	 Recall that   $\mathscr{U}=B(0,R_0)$, then, up to take $R_{0}$ large enough, we get
	$$\forall b\in(b_0,b_1),\; \forall \varepsilon\in[0,\varepsilon_0),\quad \omega(b,\varepsilon)\in \mathscr{U}=B(0,R_{0}).$$
	Recall that $\mathcal{A}_0^{\gamma}=\mathcal{O}=(b_0,b_1)\times \mathscr{U}$ then, from \eqref{definition of the final Cantor set},
	$$\mathcal{C}_{0}^{\varepsilon}=(b_0,b_1)$$
	and coming back to \eqref{decomp Cantor}, we find 
	\begin{align}\label{Mir-1}
		\nonumber \Big|(b_{0},b_{1})\setminus\mathcal{C}_{\infty}^{\varepsilon}\Big|&\leqslant\sum_{n=0}^{\infty}\Big|\mathcal{C}_{n}^{\varepsilon}\setminus\mathcal{C}_{n+1}^{\varepsilon}\Big|\\
		&:=\sum_{n=0}^{\infty}\mathcal{S}_{n}.
	\end{align}
	In accordance with the notations used in Propositions \ref{projection in the normal directions} and \ref{reduction of the remainder term}, we denote the perturbed frequencies associated with the reduced linearized operator at state $i_n$ in the following way  
	\begin{align}\label{asy-z1}
		\nonumber \mu_{j}^{\infty,n}(b,\varepsilon)&:=\mu_{j}^{\infty}\big(b,{\omega}(b,\varepsilon),i_{n}\big)\\
		&=\Omega_{j}(b)+jr^{1,n}(b,\varepsilon)+r_{j}^{\infty,n}(b,\varepsilon),
	\end{align}
	where
	\begin{align*}
		r^{1,n}(b,\varepsilon)&:=V_{n}^{\infty}(b,\varepsilon)-\tfrac{1}{2},\\
		V_{n}^{\infty}(b,\varepsilon)&:=V_{i_{n}}^{\infty}(b,{\omega}(b,\varepsilon)),\\
		r_{j}^{\infty,n}(b,\varepsilon)&:=r_{j}^{\infty}\big(b,{\omega}(b,\varepsilon),i_{n}\big).
	\end{align*}
	Now, according to \eqref{definition of the final Cantor set}, Propositions \ref{reduction of the remainder term}, \ref{inversion of the linearized operator in the normal directions} and \ref{reduction of the transport part} one can write for any $n\in\mathbb{N}$, 
	\begin{equation}\label{set-U0}
		\mathcal{C}_{n}^{\varepsilon}\setminus\mathcal{C}_{n+1}^{\varepsilon}=\bigcup_{(l,j)\in\mathbb{Z}^{d}\times\mathbb{Z}\setminus\{(0,0)\}\atop |l|\leqslant N_{n}}\mathcal{R}_{l,j}^{(0)}(i_{n})\bigcup_{(l,j,j_{0})\in\mathbb{Z}^{d}\times(\mathbb{S}_{0}^{c})^{2}\atop |l|\leqslant N_{n}}\mathcal{R}_{l,j,j_{0}}(i_{n})\bigcup_{(l,j)\in\mathbb{Z}^{d}\times\mathbb{S}_{0}^{c}\atop |l|\leqslant N_{n}}\mathcal{R}_{l,j}^{(1)}(i_{n}),
	\end{equation}
	where we denote
	\begin{align*}
		\mathcal{R}_{l,j}^{(0)}(i_{n})&:=\left\lbrace b\in\mathcal{C}_{n}^{\varepsilon}\quad\textnormal{s.t.}\quad\Big|{\omega}(b,\varepsilon)\cdot l+jV_{n}^{\infty}(b,\varepsilon)\Big|\leqslant\tfrac{4\gamma_{n+1}^{\upsilon}\langle j\rangle}{\langle l\rangle^{\tau_{1}}}\right\rbrace,\\
		\mathcal{R}_{l,j,j_{0}}(i_{n})&:=\left\lbrace b\in\mathcal{C}_{n}^{\varepsilon}\quad\textnormal{s.t.}\quad\Big|{\omega}(b,\varepsilon)\cdot l+\mu_{j}^{\infty,n}(b,\varepsilon)-\mu_{j_{0}}^{\infty,n}(b,\varepsilon)\Big|\leqslant\tfrac{2\gamma_{n+1}\langle j-j_{0}\rangle}{\langle l\rangle^{\tau_{2}}}\right\rbrace,\\
		\mathcal{R}_{l,j}^{(1)}(i_{n})&:=\left\lbrace b\in\mathcal{C}_{n}^{\varepsilon}\quad\textnormal{s.t.}\quad\Big|{\omega}(b,\varepsilon)\cdot l+\mu_{j}^{\infty,n}(b,\varepsilon)\Big|\leqslant\tfrac{\gamma_{n+1}\langle j\rangle}{\langle l\rangle^{\tau_{1}}}\right\rbrace.
	\end{align*}
	In view of the inclusion
	$$W^{q,\infty,\gamma}(\mathcal{O},\mathbb{C})\hookrightarrow C^{q-1}(\mathcal{O},\mathbb{C})$$
	and the fact that $q=q_0+1$, one obtains that for any $n\in\mathbb{N}$  the curves
	\begin{align*}
		&b\mapsto\omega(b,\varepsilon)\cdot l+j V_{n}^{\infty}(b,\varepsilon), \quad (l,j)\in\mathbb{Z}^{d}\times \mathbb{Z}\backslash\{(0,0)\}\\
		&b\mapsto\omega(b,\varepsilon)\cdot l+\mu_{j}^{\infty,n}(b,\varepsilon)-\mu_{j_0}^{\infty,n}(b,\varepsilon),\quad (l,j,j_0)\in\mathbb{Z}^{d}\times(\mathbb{S}_0^c)^{2}\\
		&b\mapsto\omega(b,\varepsilon)\cdot l+\mu_{j}^{\infty,n}(b,\varepsilon),\quad (l,j)\in\mathbb{Z}^{d}\times \mathbb{S}_0^c	
	\end{align*}
	are of regularity $C^{q_0}.$ Therefore, applying Lemma \ref{lemma useful for measure estimates} together with Lemma \ref{lemma Russeman condition for the perturbed frequencies} yields 
	\begin{align}\label{kio1}
		\Big|\mathcal{R}_{l,j}^{(0)}(i_{n})\Big|&\lesssim\gamma^{\frac{\upsilon}{q_{0}}}\langle j\rangle^{\frac{1}{q_{0}}}\langle l\rangle^{-1-\frac{\tau_{1}+1}{q_{0}}},\nonumber\\
		\Big|\mathcal{R}_{l,j}^{(1)}(i_{n})\Big|&\lesssim\gamma^{\frac{1}{q_{0}}}\langle j\rangle^{\frac{1}{q_{0}}}\langle l\rangle^{-1-\frac{\tau_{1}+1}{q_{0}}},\\
		\Big|\mathcal{R}_{l,j,j_{0}}(i_{n})\Big|&\lesssim\gamma^{\frac{1}{q_{0}}}\langle j-j_{0}\rangle^{\frac{1}{q_{0}}}\langle l\rangle^{-1-\frac{\tau_{2}+1}{q_{0}}}.\nonumber
	\end{align}
	We first estimate the measure of $\mathcal{S}_0$ and $\mathcal{S}_{1}$ defined in \eqref{Mir-1}. 	From Lemma  \ref{some cantor set are empty}, we have some trivial inclusions allowing us to write for $n\in\{0,1\}$,
	\begin{align}\label{set-U5}
		\mathcal{S}_{n}\lesssim  \sum_{\underset{|j|\leqslant C_{0}\langle l\rangle, |l|\leqslant N_{n}}{(l,j)\in\mathbb{Z}^{d}\times\mathbb{Z}\setminus\{(0,0)\}}}\Big|\mathcal{R}_{l,j}^{(0)}(i_{n})\Big|+\sum_{\underset{\underset{\min(|j|,|j_{0}|)\leqslant c_{2}\gamma_{1}^{-\upsilon}\langle l\rangle^{\tau_{1}}}{|j-j_{0}|\leqslant C_{0}\langle l\rangle, |l|\leqslant N_{n}}}{(l,j,j_{0})\in\mathbb{Z}^{d}\times(\mathbb{S}_{0}^{c})^{2}}}\Big|\mathcal{R}_{l,j,j_{0}}(i_{n})\Big|+\sum_{\underset{|j|\leqslant C_{0}\langle l\rangle, |l|\leqslant N_{n}}{(l,j)\in\mathbb{Z}^{d}\times\mathbb{S}_{0}^{c}}}\Big|\mathcal{R}_{l,j}^{(1)}(i_{n})\Big|.
	\end{align}
	Inserting \eqref{kio1} into \eqref{set-U5} implies that for $n\in\{0,1\}$,
	\begin{align*}
		\mathcal{S}_{n}&\lesssim  \gamma^{\frac{1}{q_{0}}}\Bigg(\sum_{|j|\leqslant C_{0}\langle l\rangle}|j|^{\frac{1}{q_{0}}}\langle l\rangle^{-1-\frac{\tau_{1}+1}{q_{0}}}+\sum_{\underset{\min(|j|,|j_{0}|)\leqslant c_{2}\gamma^{-\upsilon}\langle l\rangle^{\tau_{1}}}{|j-j_{0}|\leqslant C_{0}\langle l\rangle}} |j-j_{0}|^{\frac{1}{q_{0}}}\langle l\rangle^{-1-\frac{\tau_{2}+1}{q_{0}}}\Bigg)\\
		&\quad+\gamma^{\frac{\upsilon}{q_0}}\sum_{{|j|\leqslant C_{0}\langle l\rangle}}|j|^{\frac{1}{q_{0}}}\langle l\rangle^{-1-\frac{\tau_{1}+1}{q_{0}}}.
	\end{align*}
	The first two conditions listed in \eqref{choice tau 1 tau2 upsilon} write
	\begin{align}\label{assum2-d}
		\tau_1>d \,q_0\quad\hbox{and}\quad \tau_2>\tau_1+ d\, q_0.
	\end{align}
	Hence, we can make the series appearing in the following expression converge and write
	\begin{align}\label{set-U6}
		\max_{n\in\{0,1\}}\mathcal{S}_{n}&\lesssim  \gamma^{\frac{1}{q_{0}}}\Bigg(\sum_{l\in\mathbb{Z}^d}\langle l\rangle^{-\frac{\tau_{1}}{q_{0}}}+\gamma^{-\upsilon}\sum_{l\in\mathbb{Z}^d}\langle l\rangle^{\tau_1-1-\frac{\tau_{2}}{q_{0}}}\Bigg)
		+\gamma^{\frac{\upsilon}{q_0}}\sum_{l\in\mathbb{Z}^d}\langle l\rangle^{-\frac{\tau_{1}}{q_{0}}}\\
		\nonumber &\lesssim  \gamma^{\min\left(\frac{\upsilon}{q_{0}},\frac{1}{q_{0}}-\upsilon\right)}.
	\end{align}
	Let us now move to the estimate of $\mathcal{S}_{n}$ for $n\geqslant 2$ defined by \eqref{Mir-1}. Using Lemma \ref{lemm-dix1} and Lemma \ref{some cantor set are empty}, we infer
	\begin{align*}
		\mathcal{S}_{n}\leqslant \sum_{\underset{|j|\leqslant C_{0}\langle l\rangle,N_{n-1}<|l|\leqslant N_{n}}{(l,j)\in\mathbb{Z}^{d}\times\mathbb{Z}\setminus\{(0,0)\}}}\Big|\mathcal{R}_{l,j}^{(0)}(i_{n})\Big|+\sum_{\underset{\underset{\min(|j|,|j_{0}|)\leqslant c_{2}\gamma_{n+1}^{-\upsilon}\langle l\rangle^{\tau_{1}}}{|j-j_{0}|\leqslant C_{0}\langle l\rangle,N_{n-1}<|l|\leqslant N_{n}}}{(l,j,j_{0})\in\mathbb{Z}^{d}\times(\mathbb{S}_{0}^{c})^{2}}}\Big|\mathcal{R}_{l,j,j_{0}}(i_{n})\Big|+\sum_{\underset{|j|\leqslant C_{0}\langle l\rangle,N_{n-1}<|l|\leqslant N_{n}}{(l,j)\in\mathbb{Z}^{d}\times\mathbb{S}_{0}^{c}}}\Big|\mathcal{R}_{l,j}^{(1)}(i_{n})\Big|.
	\end{align*}
	Notice that if $|j-j_{0}|\leqslant C_{0}\langle l\rangle$ and $\min(|j|,|j_{0}|)\leqslant\gamma_{n+1}^{-\upsilon}\langle l\rangle^{\tau_{1}}$, then 
	$$\max(|j|,|j_{0}|)=\min(|j|,|j_{0}|)+|j-j_{0}|\leqslant\gamma_{n+1}^{-\upsilon}\langle l\rangle^{\tau_{1}}+C_{0}\langle l\rangle\lesssim\gamma^{-\upsilon}\langle l\rangle^{\tau_{1}}.$$
	Hence, we deduce from \eqref{kio1} that
	\begin{align*}
		\mathcal{S}_{n}\lesssim 
		\gamma^{\frac{1}{q_{0}}}\Bigg(\sum_{|l|>N_{n-1}}\langle l\rangle^{-\frac{\tau_{1}}{q_{0}}}+\gamma^{-\upsilon}\sum_{|l|>N_{n-1}}\langle l\rangle^{\tau_1-1-\frac{\tau_{2}}{q_{0}}}\Bigg)
		+\gamma^{\frac{\upsilon}{q_0}}\sum_{|l|>N_{n-1}}\langle l\rangle^{-\frac{\tau_{1}}{q_{0}}}.
	\end{align*}
	Now according to \eqref{assum2-d}, we obtain
	\begin{align}\label{dtu-c1}
		\sum_{n=2}^\infty \mathcal{S}_{n}&\lesssim  \gamma^{\min\left(\frac{\upsilon}{q_{0}},\frac{1}{q_{0}}-\upsilon\right)}.
	\end{align}
	Inserting \eqref{dtu-c1} and \eqref{set-U6} into \eqref{Mir-1} yields
	$$\Big|(b_{0},b_{1})\setminus\mathcal{C}_{\infty}^{\varepsilon}\Big|  \lesssim  \gamma^{\min\left(\frac{\upsilon}{q_{0}},\frac{1}{q_{0}}-\upsilon\right)}.$$
	Remark also that \eqref{choice tau 1 tau2 upsilon} implies
	$$\min\left(\tfrac{\upsilon}{q_{0}},\tfrac{1}{q_{0}}-\upsilon\right)=\tfrac{\upsilon}{q_{0}}.$$
	Consequently, using the fact that $\gamma=\varepsilon^a$ due to \eqref{choice of gamma and N0 in the Nash-Moser}, we finally get
	$$\Big|(b_{0},b_{1})\setminus\mathcal{C}_{\infty}^{\varepsilon}\Big|\lesssim\varepsilon^{\frac{a\upsilon}{q_{0}}}.$$
	This ends the proof of Proposition \ref{lem-meas-es1}. 
\end{proof}
We shall now prove Lemmata \ref{lemm-dix1}, \ref{some cantor set are empty} and \ref{lemma Russeman condition for the perturbed frequencies} used in the proof of Proposition \ref{lem-meas-es1}.
\begin{lem}\label{lemm-dix1}
	{Let $n\in\mathbb{N}\setminus\{0,1\}$ and  $l\in\mathbb{Z}^{d}$ such that $|l|\leqslant N_{n-1}.$ Then the following assertions hold true.
		\begin{enumerate}[label=(\roman*)]
			\item For $ j\in\mathbb{Z}$ with $(l,j)\neq(0,0)$, we get  $\,\,\mathcal{R}_{l,j}^{(0)}(i_{n})=\varnothing.$
			\item For  $ (j,j_{0})\in(\mathbb{S}_{0}^{c})^{2}$ with $(l,j)\neq(0,j_0),$ we get $\,\,\mathcal{R}_{l,j,j_{0}}(i_{n})=\varnothing.$
			\item For  $j\in\mathbb{S}_{0}^{c}$, we get $\,\,\mathcal{R}_{l,j}^{(1)}(i_{n})=\varnothing.$
			\item For any $n\in\mathbb{N}\setminus\{0,1\},$
			\begin{equation*}
				\mathcal{C}_{n}^{\varepsilon}\setminus\mathcal{C}_{n+1}^{\varepsilon}=\bigcup_{\underset{N_{n-1}<|l|\leqslant N_{n}}{(l,j)\in\mathbb{Z}^{d}\times\mathbb{Z}\setminus\{(0,0)\}}}\mathcal{R}_{l,j}^{(0)}(i_{n})\cup\bigcup_{\underset{N_{n-1}<|l|\leqslant N_{n}}{(l,j,j_{0})\in\mathbb{Z}^{d}\times(\mathbb{S}_{0}^{c})^{2}}}\mathcal{R}_{l,j,j_{0}}(i_{n})\cup\bigcup_{\underset{N_{n-1}<|l|\leqslant N_{n}}{(l,j)\in\mathbb{Z}^{d}\times\mathbb{S}_{0}^{c}}}\mathcal{R}_{l,j}^{(1)}(i_{n}).
			\end{equation*}
	\end{enumerate}}
\end{lem}
\begin{proof}
	The following estimate, obtained from \eqref{e-Hn-diff}, turns to be very useful in the sequel. For any $n\geqslant 2$, we have
	\begin{align}\label{Bio-X1}
		\nonumber \| i_{n}-i_{n-1}\|_{q,\overline{s}_{h}+\sigma_{4}}^{\gamma,\mathcal{O}}&\leqslant\| U_{n}-U_{n-1}\|_{q,\overline{s}_{h}+\sigma_{4}}^{\gamma,\mathcal{O}}\\
		&\leqslant\| H_{n}\|_{q,s_{h}+\sigma_{4}}^{\gamma,\mathcal{O}}\nonumber\\
		&\leqslant C_{\ast}\varepsilon\gamma^{-1}N_{n-1}^{-a_{2}}.
	\end{align}
	Since \eqref{Bio-X1} is only true for $n\geqslant 2$, we had to estimate the measures of $\mathcal{S}_{0}$ and $\mathcal{S}_{1}$ differently in the proof of Proposition \ref{lem-meas-es1}.\\
	\noindent\textbf{(i)} Assume that $|l|\leqslant N_{n-1}$ and $(l,j)\neq (0,0).$ Let us prove that
	\begin{equation}\label{inc R0 in-R0in-1}
		\mathcal{R}_{l,j}^{(0)}(i_{n}) \subset\mathcal{R}_{l,j}^{(0)}(i_{n-1}).
	\end{equation}
	Take $b\in\mathcal{R}_{l,j}^{(0)}(i_{n}).$ In view of \eqref{set-U0}, we have in particular that $b\in\mathcal{C}_{n}^{\varepsilon}\subset \mathcal{C}_{n-1}^{\varepsilon}.$ In addition, the triangle inequality gives
	\begin{align*}
		\big|{\omega}(b,\varepsilon)\cdot l+jV_{n-1}^{\infty}(b,\varepsilon)\big| & \leqslant \big|{\omega}(b,\varepsilon)\cdot l+jV_{n}^{\infty}(b,\varepsilon)\big|+|j|\big|V_{n}^{\infty}(b,\varepsilon)-V_{n-1}^{\infty}(b,\varepsilon)\big|\\
		&\leqslant \displaystyle\tfrac{4\gamma_{n+1}^{\upsilon}\langle j\rangle}{\langle l\rangle^{\tau_{1}}}+C|j|\|V_{ i_{n}}^{\infty}-V_{i_{n-1}}^{\infty}\|_{q}^{\gamma,\mathcal{O}}.
	\end{align*}
	Thus, putting together \eqref{difference ci}, \eqref{Bio-X1}, \eqref{choice of gamma and N0 in the Nash-Moser} and the fact that $\sigma_{4}\geqslant 2$, we obtain
	\begin{align*}
		\big|{\omega}(b,\varepsilon)\cdot l+jV_{n-1}^{\infty}(b,\varepsilon)\big| & \leqslant  \displaystyle\tfrac{4\gamma_{n+1}^{\upsilon}\langle j\rangle}{\langle l\rangle^{\tau_{1}}}+C\varepsilon\langle j\rangle\| i_{n}-i_{n-1}\|_{q,\overline{s}_{h}+2}^{\gamma,\mathcal{O}}\\
		& \leqslant \displaystyle\tfrac{4\gamma_{n+1}^{\upsilon}\langle j\rangle}{\langle l\rangle^{\tau_{1}}}+C\varepsilon^{2-a}\langle j\rangle N_{n-1}^{-a_{2}}.
	\end{align*}
	According the definition of $\gamma_n$ in Proposition \ref{Nash-Moser}-$(\mathcal{P}2)_n,$ we infer
	$$
	\exists c_0>0,\quad \forall n\in \mathbb{N},\quad \gamma_{n+1}^{\upsilon}-\gamma_{n}^{\upsilon}\leqslant - c_0\,\gamma^{\upsilon} 2^{-n}.
	$$
	Notice that \eqref{choice tau 1 tau2 upsilon},  \eqref{param NM} and \eqref{choice of gamma and N0 in the Nash-Moser} give 
	\begin{align}\label{assum-co1}
		2-a-a \upsilon>1\quad\hbox{and}\quad a_2>\tau_1,
	\end{align}
which implies in turn
$$\displaystyle \sup_{n\in\mathbb{N}}2^{n}N_{n-1}^{-a_2+\tau_1}<\infty.$$
Consequently, for $\varepsilon$ small enough  and $|l|\leqslant N_{n-1}$,  
	\begin{align*}
		\big|{\omega}(b,\varepsilon)\cdot l+jV_{n-1}^{\infty}(b,\varepsilon)\big| & \leqslant \displaystyle\tfrac{4\gamma_{n}^{\upsilon}\langle j\rangle}{\langle l\rangle^{\tau_{1}}}+C\tfrac{\langle j\rangle\gamma^{\upsilon}}{2^n\langle l\rangle^{\tau_{1}}}\Big(-4c_0+C\varepsilon2^{n}N_{n-1}^{-a_2+\tau_1}\Big)\\
		& \leqslant \displaystyle\tfrac{4\gamma_{n}^{\upsilon}\langle j\rangle}{\langle l\rangle^{\tau_{1}}}.
	\end{align*}
	It follows that $b\in \mathcal{R}_{l,j}^{(0)}(i_{n-1})$ and this proves \eqref{inc R0 in-R0in-1}. Now, from \eqref{set-U0} we deduce  
	$$
	\mathcal{R}_{l,j}^{(0)}(i_{n}) \subset\mathcal{R}_{l,j}^{(0)}(i_{n-1})\subset \mathcal{C}_{n-1}^{\varepsilon}\setminus\mathcal{C}_{n}^{\varepsilon}.
	$$
	In view of \eqref{inc R0 in-R0in-1} and \eqref{set-U0}, we get $\mathcal{R}_{l,j}^{(0)}(i_{n}) \subset\mathcal{C}_{n}^{\varepsilon}\setminus\mathcal{C}_{n+1}^{\varepsilon}$ and thus we conclude
	$$
	\mathcal{R}_{l,j}^{(0)}(i_{n}) \subset\big(\mathcal{C}_{n}^{\varepsilon}\setminus\mathcal{C}_{n+1}^{\varepsilon}\big)\cap \big(\mathcal{C}_{n-1}^{\varepsilon}\setminus\mathcal{C}_{n}^{\varepsilon}\big)=\varnothing.
	$$
	This proves the first point.\\
	\noindent \textbf{(ii)} Let $(j,j_{0})\in(\mathbb{S}_{0}^{c})^{2}$ and  $(l,j)\neq(0,j_0).$ If $j=j_0$ then by construction $\mathcal{R}_{l,j_0,j_{0}}(i_{n})=\mathcal{R}_{l,0}^{(0)}(i_{n})$ and then the result is an immediate consequence of the first point. Then, we restrict the discussion to the case $j\neq j_0.$ In a similar way to the point (i), we only have to check that 
	$$\mathcal{R}_{l,j,j_{0}}(i_{n})\subset \mathcal{R}_{l,j,j_{0}}(i_{n-1}).$$
	Take $b\in\mathcal{R}_{l,j,j_{0}}(i_{n}).$ Then coming back to \eqref{set-U0}, we deduce from the triangle inequality that  $b\in\mathcal{C}_{n}^{\varepsilon}\subset \mathcal{C}_{n-1}^{\varepsilon}$ and 
	\begin{equation}\label{poiH1}
		\big|{\omega}(b,\varepsilon)\cdot l+\mu_{j}^{\infty,n-1}(b,\varepsilon)-\mu_{j_{0}}^{\infty,n-1}(b,\varepsilon)\big|  \leqslant  \displaystyle\tfrac{2\gamma_{n+1}\langle j-j_{0}\rangle}{\langle l\rangle^{\tau_{2}}}+\varrho^n_{j,j_0}(b,\varepsilon),
	\end{equation}
	where 
	$$ \varrho^n_{j,j_0}(b,\varepsilon):=\big|\mu_{j}^{\infty,n}(b,\varepsilon)-\mu_{j_0}^{\infty,n}(b,\varepsilon)-\mu_{j}^{\infty,n-1}(b,\varepsilon)+\mu_{j_{0}}^{\infty,n-1}(b,\varepsilon)\big|.
	$$ 
	According to \eqref{asy-z1}, one obtains
	\begin{align}\label{asy-z2}
		\nonumber \varrho^n_{j,j_0}(b,\varepsilon)\leqslant |j-j_0|\big|r^{1,n}(b,\varepsilon)-r^{1,n-1}(b,\varepsilon)\big|&+\big|r_{j}^{\infty,n}(b,\varepsilon)-r_{j}^{\infty,n-1}(b,\varepsilon)\big|\\
		&+\big|r_{j_0}^{\infty,n}(b,\varepsilon)-r_{j_0}^{\infty,n-1}(b,\varepsilon)\big|.
	\end{align}
	From \eqref{differences mu0}, \eqref{Bio-X1}, \eqref{choice of gamma and N0 in the Nash-Moser} and the fact that $\sigma_{4}\geqslant\sigma_{3},$ we deduce that 
	\begin{align*}
		\big|r^{1,n}(b,\varepsilon)-r^{1,n-1}(b,\varepsilon)\big|&\lesssim \varepsilon\| i_{n}-i_{n-1}\|_{q,\overline{s}_{h}+\sigma_{3}}^{\gamma,\mathcal{O}}\\
		&\lesssim \varepsilon^{2}\gamma^{-1}N_{n-1}^{-a_2}\\
		&\lesssim \varepsilon^{2-a} N_{n-1}^{-a_2}.
	\end{align*}
	Similarly, \eqref{diffenrence rjinfty}, \eqref{Bio-X1} and \eqref{choice of gamma and N0 in the Nash-Moser} imply
	\begin{align*}
		\big|r_{j}^{\infty,n}(b,\varepsilon)-r_{j}^{\infty,n-1}(b,\varepsilon)\big|&\lesssim \varepsilon\gamma^{-1}\| i_{n}-i_{n-1}\|_{q,\overline{s}_{h}+\sigma_{4}}^{\gamma,\mathcal{O}}\\
		&\lesssim \varepsilon^{2}\gamma^{-2}N_{n-1}^{-a_2}\\
		&\lesssim \varepsilon^{2(1-a)}\langle j-j_0\rangle N_{n-1}^{-a_2}.
	\end{align*}
	Plugging the preceding two estimates into \eqref{asy-z2} yields
	\begin{align}\label{asy-z3}
		\varrho^n_{j,j_0}(b,\varepsilon)\lesssim \varepsilon^{2(1-a)}\langle j-j_0\rangle N_{n-1}^{-a_{2}}.
	\end{align}
	Gathering \eqref{asy-z3} and \eqref{poiH1} and using $\gamma_{n+1}=\gamma_{n}-\varepsilon^a 2^{-n-1},$ we obtain
	\begin{align*}
		\big|{\omega}(b,\varepsilon)\cdot l+\mu_{j}^{\infty,n-1}(b,\varepsilon)-\mu_{j_{0}}^{\infty,n-1}(b,\varepsilon)\big| 
		&\leqslant \displaystyle\tfrac{2\gamma_{n}\langle j-j_0\rangle }{\langle l\rangle^{\tau_{2}}}-{\varepsilon^a \langle j-j_{0}\rangle}2^{-n}\langle l\rangle ^{-\tau_2}\\
		&\quad+C\varepsilon^{2(1-a)}\langle j-j_0\rangle N_{n-1}^{-a_{2}}.
	\end{align*} 
	Using the fact that $|l|\leqslant N_{n-1}$, we deduce
	\begin{align*}
		-{\varepsilon^a }2^{-n}\langle l\rangle ^{-\tau_2}+C\varepsilon^{2(1-a)}N_{n-1}^{-a_{2}}\leqslant {\varepsilon^a }2^{-n}\langle l\rangle ^{-\tau_2}\Big(-1+C\varepsilon^{2-3a}2^{n}N_{n-1}^{-a_{2}+\tau_{2}}\Big).
	\end{align*}
	Notice that \eqref{param NM} and \eqref{choice of gamma and N0 in the Nash-Moser} imply in particular
	\begin{align}\label{Condor-1}
		a_2>\tau_{2}\quad\hbox{and}\quad a<\tfrac{2}{3}.
	\end{align}
	Therefore, for $\varepsilon$ small enough, we get
	\begin{align*}
		\forall\, n\in\mathbb{N},\quad -1+C\varepsilon^{2-3a}2^{n}N_{n-1}^{-a_{2}+\tau_{2}}\leqslant 0,
	\end{align*}
	which implies in turn 
	$$\big|{\omega}(b,\varepsilon)\cdot l+\mu_{j}^{\infty,n-1}(b,\varepsilon)-\mu_{j_{0}}^{\infty,n-1}(b,\varepsilon)\big| \leqslant  \displaystyle\tfrac{2\gamma_{n}\langle j-j_0\rangle }{\langle l\rangle^{\tau_{2}}}\cdot$$
	Finally, $b \in  \mathcal{R}_{l,j,j_{0}}(i_{n-1}).$ This achieves the proof of the second point.\\
	\noindent \textbf{(iii)} Let $j\in\mathbb{S}_{0}^{c}.$ In particular, one has $(l,j)\neq (0,0).$ In a similar line to the first point, we shall prove that if $|l|\leqslant N_{n-1}$ and then 
	$$\mathcal{R}_{l,j}^{(1)}(i_{n}) \subset\mathcal{R}_{l,j}^{(1)}(i_{n-1}),$$
	where the set $\mathcal{R}_{l,j}^{(1)}(i_{n})$ is defined below \eqref{set-U0}. Take $b\in\mathcal{R}_{l,j}^{(1)}(i_{n}).$ Then, by construction, $b\in \mathcal{C}_{n}^{\varepsilon}\subset \mathcal{C}_{n-1}^{\varepsilon}.$ By using the triangle inequality, \eqref{estimate differences mujinfty}, \eqref{Bio-X1} and the choice $\gamma=\varepsilon^a$, we obtain 
	\begin{align*}
		\big|{\omega}(b,\varepsilon)\cdot l+\mu_{j}^{\infty,n-1}(b,\varepsilon)\big| & \leqslant  \big|{\omega}(b,\varepsilon)\cdot l+\mu_{j}^{\infty,n}(b,\varepsilon)\big|+|\mu_{j}^{\infty,n}(b,\varepsilon)-\mu_{j}^{\infty,n-1}(b,\varepsilon)|\\
		& \leqslant \displaystyle\tfrac{\gamma_{n+1}\langle j\rangle }{\langle l\rangle^{\tau_{1}}}+C\varepsilon\gamma^{-1}|j|\| i_{n}-i_{n-1}\|_{q,\overline{s}_{h}+\sigma_{4}}^{\gamma,\mathcal{O}}\\
		& \leqslant  \displaystyle\tfrac{\gamma_{n+1}\langle j\rangle }{\langle l\rangle^{\tau_{1}}}+C\varepsilon^{2(1-a)}\langle j\rangle N_{n-1}^{-a_{2}}.
	\end{align*}
	Now recalling that $\gamma_{n+1}=\gamma_{n}-\varepsilon^a 2^{-n-1}$ and $|l|\leqslant N_{n-1}$, we get 
	$$\big|{\omega}(b,\varepsilon)\cdot l+\mu_{j}^{\infty,n-1}(b,\varepsilon)\big| \leqslant \displaystyle\tfrac{\gamma_{n}\langle j\rangle }{\langle l\rangle^{\tau_{1}}}+\tfrac{\langle j\rangle \varepsilon^a}{2^{n+1}\langle l\rangle^{\tau_1}}\Big(-1+\varepsilon^{2-3a} 2^{n+1}N_{n-1}^{-a_{2}+\tau_{1}}\Big).$$
	As a byproduct of \eqref{Condor-1}, we infer 
	\begin{align}\label{samedi-1}
		a_2>\tau_{1}\quad\hbox{and}\quad a<\tfrac{2}{3}.
	\end{align}
	Therefore, up to take $\varepsilon$ small enough, we deduce
	$$
	\forall\,n\in \mathbb{N},\quad -1+\varepsilon^{2-3a} 2^{n+1}N_{n-1}^{-a_{2}+\tau_{1}}\leqslant 0,
	$$
	which implies in turn that 
	$$
	\big|{\omega}(b,\varepsilon)\cdot l+\mu_{j}^{\infty,n-1}(b,\varepsilon)\big|  \leqslant \tfrac{\gamma_{n}\langle j\rangle }{\langle l\rangle^{\tau_{1}}}.
	$$
	Finally, $b\in \mathcal{R}_{l,j}^{(1)}(i_{n-1})$ and the proof of the third point is now complete. 
	\\
	\noindent \textbf{(iv)} Follows immediately from \eqref{set-U0} and the points (i)-(ii)-(iii).
\end{proof}
The following lemma provides necessary constraints on the time and space Fourier modes so that the sets in \eqref{set-U0} are not void.
\begin{lem}\label{some cantor set are empty}
	There exists $\varepsilon_0$ such that for any $\varepsilon\in[0,\varepsilon_0]$ and $n\in\mathbb{N}$ the following assertions hold true. 
	\begin{enumerate}[label=(\roman*)]
		\item Let $(l,j)\in\mathbb{Z}^{d}\times\mathbb{Z}\setminus\{(0,0)\}.$ If $\,\displaystyle\mathcal{R}_{l,j}^{(0)}(i_{n})\neq\varnothing,$ then $|j|\leqslant C_{0}\langle l\rangle.$
		\item Let $(l,j,j_{0})\in\mathbb{Z}^{d}\times(\mathbb{S}_{0}^{c})^{2}.$ If $\,\displaystyle\mathcal{R}_{l,j,j_{0}}(i_{n})\neq\varnothing,$ then $|j-j_{0}|\leqslant C_{0}\langle l\rangle.$
		\item Let $(l,j)\in\mathbb{Z}^{d}\times\mathbb{S}_{0}^{c}.$ If $\,\displaystyle \mathcal{R}_{l,j}^{(1)}(i_{n})\neq\varnothing,$ then $|j|\leqslant C_{0}\langle l\rangle.$
		\item Let $(l,j,j_{0})\in\mathbb{Z}^{d}\times(\mathbb{S}_{0}^{c})^{2}.$ There exists $c_{2}>0$ such that if $\displaystyle \min(|j|,|j_{0}|)\geqslant c_{2}\gamma_{n+1}^{-\upsilon}\langle l\rangle^{\tau_{1}},$ then 
		$$\mathcal{R}_{l,j,j_{0}}(i_{n})\subset\mathcal{R}_{l,j-j_{0}}^{(0)}(i_{n}).$$
	\end{enumerate}
\end{lem}
\begin{proof}
	\textbf{(i)} Let us assume that $\mathcal{R}_{l,j}^{(0)}(i_{n})\neq\varnothing.$ Then, there exists $b\in(b_{0},b_{1})$ such that
	$$|\omega(b,\varepsilon)\cdot l+j V_n^{\infty}(b,\varepsilon)|\leqslant\tfrac{4\gamma_{n+1}^{\upsilon}\langle j\rangle}{\langle l\rangle^{\tau_{1}}}.$$ 
	From triangle and Cauchy-Schwarz inequalities, \eqref{def gamman} and \eqref{choice of gamma and N0 in the Nash-Moser}, we deduce 
	\begin{align*}
		|V_{n}^{\infty}(b,\varepsilon)||j|&\leqslant 4|j|\gamma_{n+1}^{\upsilon}\langle l\rangle^{-\tau_{1}}+|{\omega}(b,\varepsilon)\cdot l|\\
		&\leqslant 4|j|\gamma_{n+1}^{\upsilon}+C\langle l\rangle\\
		&\leqslant 8\varepsilon^{a \upsilon}|j|+C\langle l\rangle.
	\end{align*}
Remark that we used the fact that $(b,\varepsilon)\mapsto \omega(b, \varepsilon)$ is bounded. Also notice that the identity
	$$V_{n}^{\infty}(b,\varepsilon)=\tfrac{1}{2}+r^{1,n}(b,\varepsilon)
	$$
	together with \eqref{estimate r1}, \eqref{estimate rjinfty} and Proposition \ref{Nash-Moser}-$(\mathcal{P}1)_{n}$ imply
	\begin{align}\label{uniform estimate r1}
		\forall k\in\llbracket 0,q\rrbracket,\quad\sup_{n\in\mathbb{N}}\sup_{b\in(b_{0},b_{1})}|\partial_{b}^{k}r^{1,n}(b,\varepsilon)|&\leqslant\gamma^{-k}\sup_{n\in\mathbb{N}}\| r^{1,n}\|_{q}^{\gamma,\mathcal{O}}\nonumber\\
		&\lesssim\varepsilon\gamma^{-k}\nonumber\\
		&\lesssim\varepsilon^{1-ak}.
	\end{align}
	Hence, taking $\varepsilon$ small enough, we infer
	$$\inf_{n\in\mathbb{N}}\inf_{b\in(b_{0},b_{1})}|V_{n}^{\infty}(b,\varepsilon)|\geqslant \tfrac{1}{4}.$$
	Therefore, up to choose $\varepsilon$ small enough we can ensure  $|j|\leqslant C_{0}\langle l\rangle$ for some $C_{0}>0.$\\
	\noindent\textbf{(ii)} In the case $j=j_0$ we get by definition $\mathcal{R}_{l,j_0,j_{0}}(i_{n})=\mathcal{R}^{(0)}_{l,0}(i_{n})$, so this case can be treated by the first point. Then, we shall restrict the discussion to the case $j\neq j_0.$ Let us assume that $\mathcal{R}_{l,j,j_{0}}(i_{n})\neq\varnothing.$ Then, there exists $b\in(b_{0},b_{1})$ such that 
	$$|\omega(b,\varepsilon)\cdot l+\mu_j^{\infty,n}(b,\varepsilon)-\mu_{j_0}^{\infty,n}(b,\varepsilon)|\leqslant\tfrac{2\gamma_{n+1}|j-j_0|}{\langle l\rangle^{\tau_2}}.$$
	By using triangle and Cauchy-Schwarz inequalities, \eqref{def gamman} and \eqref{choice of gamma and N0 in the Nash-Moser}, we get
	\begin{align*}
		|\mu_{j}^{\infty,n}(b,\varepsilon)-\mu_{j_{0}}^{\infty,n}(b,\varepsilon)|&\leqslant 2\gamma_{n+1}|j-j_{0}|\langle l\rangle^{-\tau_{2}}+|{\omega}(b,\varepsilon)\cdot l|\\
		&\leqslant 2\gamma_{n+1}|j-j_{0}|+C\langle l\rangle\\
		&\leqslant 4\varepsilon^a|j-j_{0}|+C\langle l\rangle.
	\end{align*}
	In a similar way to \eqref{uniform estimate r1}, we may obtain
	\begin{align}\label{uniform estimate rjinfty}
		\forall k\in\llbracket 0,q\rrbracket,\quad\sup_{n\in\mathbb{N}}\sup_{j\in\mathbb{S}_{0}^{c}}\sup_{b\in(b_{0},b_{1})}|j||\partial_{b}^{k}r_{j}^{\infty,n}(b,\varepsilon)|&\leqslant\gamma^{-k}\sup_{n\in\mathbb{N}}\sup_{j\in\mathbb{S}_{0}^{c}}|j|\| r_{j}^{\infty,n}\|_{q}^{\gamma,\mathcal{O}}\nonumber\\
		&\lesssim\varepsilon\gamma^{-1-k}\nonumber\\
		&\lesssim\varepsilon^{1-a(1+k)}.
	\end{align} 
	From the triangle inequality, Lemma \ref{properties omegajb}-(iii), \eqref{uniform estimate r1} and \eqref{uniform estimate rjinfty} we infer for $j\neq j_0$,
	\begin{align*}
		|\mu_{j}^{\infty,n}(b,\varepsilon)-\mu_{j_{0}}^{\infty,n}(b,\varepsilon)| & \geqslant  |\Omega_{j}(b)-\Omega_{j_{0}}(b)|-|r^{1,n}(b,\varepsilon)||j-j_{0}|-|r_{j}^{\infty,n}(b,\varepsilon)|-|r_{j_{0}}^{\infty,n}(b,\varepsilon)|\\
		& \geqslant  \big(\tfrac{b_0^2}{6}-C\varepsilon^{1-a}\big)|j-j_{0}|\\
		& \geqslant \tfrac{b_0^2}{12}|j-j_{0}|.
	\end{align*}
	Notice that the last inequality is obtained for  $\varepsilon$ sufficiently small. Gathering the previous inequalities implies that, up to choose $\varepsilon $ small enough, we can ensure $|j-j_{0}|\leqslant C_{0}\langle l\rangle,$ for some $C_{0}>0.$\\
	\noindent \textbf{(iii)} First notice that the case $j=0$ is obvious. Now  for $j\neq 0$ we assume that  $\mathcal{R}_{l,j}^{(1)}(i_{n})\neq\varnothing.$ Then, there exists $b\in(b_{0},b_{1})$ such that 
	$$|\omega(b,\varepsilon)\cdot l+\mu_j^{\infty,n}(b,\varepsilon)|\leqslant\tfrac{\gamma_{n+1}|j|}{\langle l\rangle^{\tau_1}}.$$
	Thus, triangle and Cauchy-Schwarz inequalities, \eqref{def gamman} and \eqref{choice of gamma and N0 in the Nash-Moser} imply
	\begin{align*}
		|\mu_{j}^{\infty,n}(b,\varepsilon)|&\leqslant \gamma_{n+1}|j|\langle l\rangle^{-\tau_{1}}+|{\omega}(b,\varepsilon)\cdot l|\\
		&\leqslant  2\varepsilon^a|j|+C\langle l\rangle.
	\end{align*}
	According to the definition \eqref{asy-z1} together with the triangle inequality, Lemma \ref{properties omegajb}-(ii), \eqref{uniform estimate r1} and \eqref{uniform estimate rjinfty}, we obtain
	\begin{align*}
		|\mu_{j}^{\infty,n}(b,\varepsilon)|&\geqslant \tfrac{b_0^2}{2}|j|-|j||r^{1,n}(b,\varepsilon)|-|r_{j}^{\infty,n}(b,\varepsilon)|\\
		&
		\geqslant \tfrac{b_0^2}{2}|j|-C\varepsilon^{1-a}|j|.
	\end{align*}
	Putting together the previous two inequalities and the second condition in \eqref{samedi-1} yields
	\begin{align*}
		\big( \tfrac{b_0^2}{2}-C\varepsilon^{1-a}-2\varepsilon^a\big)|j|
		&\leqslant  C\langle l\rangle.
	\end{align*}
	Finally, by choosing $\varepsilon$ small enough we get $|j|\leqslant C_0\langle  l\rangle,$ for some $C_{0}>0.$\\
	\noindent \textbf{(iv)} First remark that the case $j=j_0$ is obvious as a direct consequence of the definition \eqref{set-U0}. Let $j\neq j_0.$ In view of the symmetry property $\mu_{-j}^{\infty,n}=-\mu_j^{\infty,n}$ of the perturbed eigenvalues, we can always assume that $0<j<j_0.$ Take $b\in\mathcal{R}_{l,j,j_{0}}(i_{n}).$ Then by construction
	$$\big|{\omega}(b,\varepsilon)\cdot l+\mu_{j}^{\infty,n}(b,\varepsilon)\pm\mu_{j_{0}}^{\infty,n}(b,\varepsilon)\big|\leqslant\tfrac{2\gamma_{n+1}\langle j\pm j_{0}\rangle}{\langle l\rangle^{\tau_{2}}}.
	$$
	Putting together \eqref{asy-z1}, \eqref{Omegajb} and the triangle inequality, we find
	\begin{align*}
		\big|{\omega}(b,\varepsilon)\cdot l+(j\pm j_{0})V_{n}^{\infty}(b,\varepsilon)\big|  &\leqslant  \big|{\omega}(b,\varepsilon)\cdot l+\mu_{j}^{\infty,n}(b,\varepsilon)\pm\mu_{j_{0}}^{\infty,n}(b,\varepsilon)\big|+\tfrac12|b^{2j}\pm b^{2j_0}|\\
		&\quad+\tfrac12\big|(j-1)\pm (j_0-1)-(j\pm j_0) \big|+\big|r_{j}^{\infty,n}(b,\varepsilon)\pm r_{j_{0}}^{\infty,n}(b,\varepsilon)\big|.
	\end{align*}
	Hence, we deduce
	\begin{align}\label{ZaraX1}
		 \big|{\omega}(b,\varepsilon)\cdot l+(j\pm j_{0})V_{n}^{\infty}(b,\varepsilon)\big|  &\leqslant \tfrac{2\gamma_{n+1}\langle j\pm j_{0}\rangle}{\langle l\rangle^{\tau_{2}}}+\tfrac12|b^{2j}\pm b^{2j_0}|\nonumber\\
		 &\quad+\tfrac12\big|(j-1)\pm (j_0-1)-(j\pm j_0) \big|+\big|r_{j}^{\infty,n}(b,\varepsilon)\pm r_{j_{0}}^{\infty,n}(b,\varepsilon)\big|.
	\end{align}
	Notice that
	\begin{align*}
	b^{2j}+b^{2j_0}&\leqslant C\tfrac{\langle j+j_0\rangle}{j}.
\end{align*}
In addition, Taylor formula implies
	\begin{align*}
		b^{2j}-b^{2j_0}&\leqslant-2\ln(b)\int_{j}^{j_0}b^{2x}dx
 \leqslant \tfrac{c_1\langle j-j_{0}\rangle}{j},
	\end{align*}
where $c_1=\displaystyle\sup_{j\in\mathbb{N}, b\in(0,1)}\big(-2\ln(b)jb^{2j}\big)>0.$
On the other hand, one has
$$ 
 \big|(j-1)\pm (j_0-1)-(j\pm j_0) \big| =1\pm 1\leqslant \tfrac{\langle j+j_0\rangle}{j}.
$$
	Applying  \eqref{estimate rjinfty}, we find for $j\neq j_{0}$,
	\begin{align*}
		\big|r_{j}^{\infty,n}(b,\varepsilon)\pm r_{j_{0}}^{\infty,n}(b,\varepsilon)\big|  \leqslant &C \varepsilon^{1-a}\big(|j|^{-1}+|j_0|^{-1}\big)\\
		\leqslant & C \varepsilon^{1-a} \tfrac{\langle j\pm j_{0}\rangle}{j}\cdot
	\end{align*}
	Plugging the preceding estimates into \eqref{ZaraX1} yields 
	\begin{align*}
		\nonumber \big|{\omega}(b,\varepsilon)\cdot l+(j\pm j_{0})V_{n}^{\infty}(b,\varepsilon)\big|  \leqslant & \tfrac{2\gamma_{n+1}\langle j\pm j_{0}\rangle}{\langle l\rangle^{\tau_{2}}}+ C \tfrac{\langle j\pm j_{0}\rangle}{j}\cdot
	\end{align*}
	Therefore,  if we assume $\displaystyle j\geqslant \tfrac{1}{2} C\gamma_{n+1}^{-\upsilon}\langle l\rangle^{\tau_{1}}$  and $\tau_{2}>\tau_{1}$, then we deduce 
	$$\big|{\omega}(b,\varepsilon)\cdot l+(j\pm j_{0})V_{n}^{\infty}(b,\varepsilon)\big| \leqslant  \tfrac{4\gamma_{n+1}^{\upsilon}\langle j\pm j_{0}\rangle}{\langle l\rangle^{\tau_{1}}}\cdot$$
	This achieves the proof of Lemma \ref{some cantor set are empty}, taking $c_{2}=\frac{C}{2}.$
\end{proof}
We shall now establish that the perturbed frequencies $\omega(b,\varepsilon)$ satisfy the Rüssemann conditions. This is done by a perturbation argument on the transversality conditions of  the equilibrium linear frequencies $\omega_{\textnormal{Eq}}(b)$ stated in Lemma \ref{lemma transversalityb}. 
\begin{lem}\label{lemma Russeman condition for the perturbed frequencies}
	Let $q_{0}$, $C_{0}$ and $\rho_{0}$ as in Lemma \ref{lemma transversalityb}. There exist $\varepsilon_{0}>0$ small enough such that for any   $\varepsilon\in[0,\varepsilon_{0}]$ the following assertions hold true.
		\begin{enumerate}[label=(\roman*)]
			\item For all $l\in\mathbb{Z}^{d}\setminus\{0\}, $ we have 
			$$\inf_{b\in[b_{0},b_{1}]}\max_{k\in\llbracket 0,q_{0}\rrbracket}\big|\partial_{b}^{k}\left({\omega}(b,\varepsilon)\cdot l\right)\big|\geqslant\tfrac{\rho_{0}\langle l\rangle}{2}.$$
			\item For all $(l,j)\in\mathbb{Z}^{d+1}\setminus\{(0,0)\}$ such that $|j|\leqslant C_{0}\langle l\rangle,$ we have
			$$\forall n\in\mathbb{N},\quad\inf_{b\in[b_{0},b_{1}]}\max_{k\in\llbracket 0,q_{0}\rrbracket}|\partial_{b}^{k}\big({\omega}(b,\varepsilon)\cdot l+jV_{n}^{\infty}(b,\varepsilon)\big)|\geqslant\tfrac{\rho_{0}\langle l\rangle}{2}.$$
			\item For all $(l,j)\in\mathbb{Z}^{d}\times\mathbb{S}_{0}^{c}$ such that $|j|\leqslant C_{0}\langle l\rangle,$ we have
			$$\forall n\in\mathbb{N},\quad\inf_{b\in[b_{0},b_{1}]}\max_{k\in\llbracket 0,q_{0}\rrbracket}\big|\partial_{b}^{k}\big({\omega}(b,\varepsilon)\cdot l+\mu_{j}^{\infty,n}(b,\varepsilon)\big)\big|\geqslant\tfrac{\rho_{0}\langle l\rangle}{2}.$$
			\item For all $(l,j,j_{0})\in\mathbb{Z}^{d}\times(\mathbb{S}_{0}^{c})^{2}$ such that $|j-j_{0}|\leqslant C_{0}\langle l\rangle,$ we have
			$$\forall n\in\mathbb{N},\quad\inf_{b\in[b_{0},b_{1}]}\max_{k\in\llbracket 0,q_{0}\rrbracket}\big|\partial_{b}^{k}\big({\omega}(b,\varepsilon)\cdot l+\mu_{j}^{\infty,n}(b,\varepsilon)-\mu_{j_{0}}^{\infty,n}(b,\varepsilon)\big)\big|\geqslant\tfrac{\rho_{0}\langle l\rangle}{2}.$$
	\end{enumerate}
\end{lem}
\begin{proof}
	\textbf{(i)} From the triangle and Cauchy-Schwarz  inequalities together with  \eqref{estimate repsilon1}, \eqref{choice of gamma and N0 in the Nash-Moser} and Lemma \ref{lemma transversalityb}-(i), we deduce
	\begin{align*}
		\displaystyle\max_{k\in\llbracket 0,q_{0}\rrbracket}|\partial_{b}^{k}\left({\omega}(b,\varepsilon)\cdot l\right)| & \geqslant \displaystyle\max_{k\in\llbracket 0,q_{0}\rrbracket}|\partial_{b}^{k}\left({\omega}_{\textnormal{Eq}}(b)\cdot l\right)|-\max_{k\in\llbracket 0,q\rrbracket}|\partial_{b}^{k}\left(\overline{\mathrm{r}}_{\varepsilon}(b)\cdot l\right)|\\
		& \geqslant\displaystyle\rho_{0}\langle l\rangle-C\varepsilon\gamma^{-1-q}N_{0}^{q\overline{a}}\langle l\rangle\\
		& \geqslant \displaystyle\rho_{0}\langle l\rangle-C\varepsilon^{1-a(1+q+q\overline{a})}\langle l\rangle\\
		& \geqslant \displaystyle\tfrac{\rho_{0}\langle l\rangle}{2}
	\end{align*}
	provided that $\varepsilon$ is small enough and 
	\begin{equation}\label{condition param Russ}
		1-a(1+q+q\overline{a})>0.
	\end{equation}
	Notice that the condition \eqref{condition param Russ} is automatically satisfied by \eqref{choice of gamma and N0 in the Nash-Moser} and \eqref{param NM}.\\
	\textbf{(ii)} As before, using the triangle and Cauchy-Schwarz inequalities combined with  \eqref{estimate repsilon1}, \eqref{uniform estimate r1}, Lemma \ref{lemma transversalityb}-(ii) and the fact that $|j|\leqslant C_{0}\langle l\rangle$, we get
	\begin{align*}
		\displaystyle\max_{k\in\llbracket 0,q_{0}\rrbracket}|\partial_{b}^{k}\left({\omega}(b,\varepsilon)\cdot l+jV_{n}^{\infty}(b,\varepsilon)\right)| & \geqslant \displaystyle\max_{k\in\llbracket 0,q_{0}\rrbracket}\big|\partial_{b}^{k}\big({\omega}_{\textnormal{Eq}}(b)\cdot l+\tfrac{j}{2}\big)\big|-\max_{k\in\llbracket 0,q\rrbracket}|\partial_{b}^{k}\left(\overline{\mathrm{r}}_{\varepsilon}(b)\cdot l+jr^{1,n}(b,\varepsilon)\right)|\\
		& \geqslant \displaystyle\rho_{0}\langle l\rangle-C\varepsilon^{1-a(1+q+q\overline{a})}\langle l\rangle-C\varepsilon^{1-aq}|j|\\
		& \geqslant \displaystyle\tfrac{\rho_{0}\langle l\rangle}{2}
	\end{align*}
	for $\varepsilon$ small enough and with the condition \eqref{condition param Russ}.\\
	\textbf{(iii)} As before, using triangle  and Cauchy-Schwarz inequalities combined with \eqref{estimate repsilon1}, \eqref{uniform estimate r1}, \eqref{uniform estimate rjinfty}, Lemma \ref{lemma transversalityb}-(iii) and the fact that  $|j|\leqslant C_{0}\langle l\rangle$, we get
	\begin{align*}
		\displaystyle\max_{k\in\llbracket 0,q_{0}\rrbracket}\big|\partial_{b}^{k}\big({\omega}(b,\varepsilon)\cdot l+\mu_{j}^{\infty,n}(b,\varepsilon)\big)\big| & \geqslant  \displaystyle\max_{k\in\llbracket 0,q_{0}\rrbracket}|\partial_{b}^{k}\left({\omega}_{\textnormal{Eq}}(b)\cdot l+\Omega_{j}(b)\right)|\\
		& \displaystyle\quad-\max_{k\in\llbracket 0,q\rrbracket}\big|\partial_{b}^{k}\big(\overline{\mathrm{r}}_{\varepsilon}(b)\cdot l+jr^{1,n}(b,\varepsilon)+r_{j}^{\infty,n}(b,\varepsilon))\big)\big|\\
		& \geqslant  \displaystyle\rho_{0}\langle l\rangle-C\varepsilon^{1-a(1+q+q\overline{a})}\langle l\rangle-C\varepsilon^{1-a(1+q)}|j|\\
		& \geqslant  \displaystyle\tfrac{\rho_{0}\langle l\rangle}{2}
	\end{align*}
	for $\varepsilon$ small enough with the condition \eqref{condition param Russ}.\\
	\textbf{(iv)} Arguing as in the preceding point, using  \eqref{uniform estimate r1}, \eqref{uniform estimate rjinfty}, Lemma \ref{lemma transversalityb}-(iv) and the fact that $0<|j-j_{0}|\leqslant C_{0}\langle l\rangle$ (notice that the case $j=j_0$ is trivial), we have 
	\begin{align*}
		\max_{k\in\llbracket 0,q_{0}\rrbracket}\big|\partial_{b}^{k}\big({\omega}(b,\varepsilon)\cdot l&+\mu_{j}^{\infty,n}(b,\varepsilon)-\mu_{j_{0}}^{\infty,n}(b,\varepsilon)\big)\big|
		\geqslant\displaystyle\max_{k\in\llbracket 0,q_{0}\rrbracket}\big|\partial_{b}^{k}\big({\omega}_{\textnormal{Eq}}(b)\cdot l+\Omega_{j}(b)-\Omega_{j_{0}}(b)\big)\big|\displaystyle\\
		&-\max_{k\in\llbracket 0,q\rrbracket}\big|\partial_{b}^{k}\big(\overline{\mathrm{r}}_{\varepsilon}(b)\cdot l+(j-j_{0})r^{1,n}(b,\varepsilon)+r_{j}^{\infty,n}(b,\varepsilon)-r_{j_{0}}^{\infty,n}(b,\varepsilon)\big)\big|\\
		&\geqslant\displaystyle\rho_{0}\langle l\rangle-C\varepsilon^{1-a(1+q+q\overline{a})}\langle l\rangle-C\varepsilon^{1-a(1+q)}|j-j_{0}|\\
		&\geqslant \displaystyle\tfrac{\rho_{0}\langle l\rangle}{2}
	\end{align*}
	for $\varepsilon$ small enough. This ends the proof of Lemma \ref{lemma Russeman condition for the perturbed frequencies}.
\end{proof}

\textbf{Acknowledgements :}
The work of Z. Hassainia is supported by Tamkeen under the NYU Abu Dhabi Research Institute grant of the center SITE.\\
The authors would like to thank Taoufik Hmidi for introducing the problem and for his precious advices when writing this document.

\end{document}